\documentclass[12pt,a4paper,english,twoside,openright]{report}
\usepackage[latin1]{inputenc}
\usepackage[T1]{fontenc}
\usepackage{amsmath,amssymb,amsthm}
\usepackage{babel}
\usepackage{epsfig}
\usepackage{epstopdf}
\usepackage[all]{xy}

\makeatletter
\def\ps@myheadings{%
\def\@oddfoot{}%
\def\@evenfoot{}%
\def\@oddhead{\underline{\protect\rule[-0.5ex]{0cm}{1ex}\protect\makebox[\textwidth][l]{
\hbox{}%
\rightmark \hfill \hfill  \rm\thepage}}}%
\def\@evenhead{\underline{\protect\rule[-0.5ex]{0cm}{1ex}\protect\makebox[\textwidth][l]
{\hbox{}%
{\rm\thepage}  \hfill \hfill \leftmark}}}%
}

\def\thesection {\thechapter.\arabic{section}} 

\mark{{}{}} 
\if@twoside
\def\ps@headings{\let\@mkboth\markboth%
\def\@oddfoot{}\def\@evenfoot{}%    No feet
\def\@evenhead{\underline{\protect\rule[-0.5ex]{0cm}{1ex}\protect\makebox[\textwidth][l]
{\hbox{}%
{\rm\thepage}  \hfill \hfill \leftmark}}}%
\def\@oddhead{\underline{\protect\rule[-0.5ex]{0cm}{1ex}\protect\makebox[\textwidth][l]{
\hbox{}%
\rightmark \hfill \hfill  \rm\thepage}}}%
\def\chaptermark##1{\markboth {\sc{\ifnum \c@secnumdepth >\z@
    \thechapter\hskip 1em\relax \fi ##1}}%
{\sc{\ifnum \c@secnumdepth >\z@
    \thechapter\hskip 1em\relax \fi ##1}}}%
\def\sectionmark##1{\markright{\ifnum \c@secnumdepth >\@ne
          \thesection\hskip 1em\relax \fi ##1}}}
\else \def\ps@headings{%
\def\@oddfoot{}%
\def\@evenfoot{}%
\def\@evenhead{\underline{\protect\rule[-0.5ex]{0cm}{1ex}\protect\makebox[\textwidth][l]
{\hbox{}%
{\rm\thepage}  \hfill \hfill \leftmark}}}%
\def\@oddhead{\underline{\protect\rule[-0.5ex]{0cm}{1ex}\protect\makebox[\textwidth][l]{
\hbox{}%
\rightmark \hfill \hfill  \rm\thepage}}}%
\def\chaptermark##1{\markboth{\sc{\ifnum \c@secnumdepth >\z@
    \thechapter\hskip 1em\relax \fi ##1}}%
{\sc{\ifnum \c@secnumdepth >\z@
    \thechapter\hskip 1em\relax \fi ##1}}}%
\def\sectionmark##1{\markright{\ifnum \c@secnumdepth >\@ne
          \thesection\hskip 1em\relax \fi ##1}}}
\fi
\makeatother

\makeatletter
\renewcommand*{\cleardoublepage}{%
  \clearpage%
  \thispagestyle{empty}% <- addition
  \if@twoside%
  \ifodd\c@page\else\hbox{}\newpage%
  \if@twocolumn\hbox{}\newpage%
  \fi\fi\fi%
}
\makeatother

\newcommand{\suchthat}{\,\vert\,}

\newcommand{\te}[1]{\ensuremath{#1^\mathrm{th}}}

\newcommand{\pd}[1]{{\smash{(#1)}}}   % a divided power

\newcommand{\A}{\mathcal{A}}          % an algebra A
\newcommand{\D}{\mathcal{D}}          % basis D of differentials
\newcommand{\E}{\mathcal{E}}          % basis E of differentials
\newcommand{\J}{\mathcal{J}}          % partition J
\newcommand{\K}{k}                    % the ground field
\newcommand{\m}{\mathsf{m}}           % a maximal ideal
\newcommand{\cO}{\mathcal{O}}         % structure sheaf
\newcommand{\p}{\partial}             % partial differentiation
\newcommand{\SG}{\Sigma}              % the symmetric group
\newcommand{\T}{\mathsf{T}\!}         % transpose of a matrix
\newcommand{\W}{\mathfrak{W}}         % space W
\newcommand{\Z}{\mathbb{Z}}           % the ring of integers

\newcommand{\F}{\mathcal{F}}          % resolution F
\newcommand{\G}{\mathcal{G}}          % resolution G
\newcommand{\HH}{\mathcal{H}}         % resolution H
\newcommand{\M}{\mathcal{M}}          % resolution M
\newcommand{\N}{\mathcal{N}}          % resolution N
\newcommand{\MN}{\mathcal{MN}}        % resolution MN

\newcommand{\bA}{\mathbb{A}}          % affine A
\newcommand{\bN}{\mathbb{N}}          % the positive integers
\newcommand{\bP}{\mathbb{P}}          % projective P
\newcommand{\cR}{\mathcal{R}}         % ring R of polynomials
\newcommand{\cS}{\mathcal{S}}         % ring S of polynomials
\newcommand{\cT}{\mathcal{T}}         % ring T or tangent space

\DeclareMathOperator*{\dsum}{\oplus}  % direct sum
\DeclareMathOperator*{\isect}{\cap}   % intersection
\DeclareMathOperator*{\sprod}{\Pi}    % small product
\DeclareMathOperator*{\union}{\cup}   % union

\DeclareMathOperator{\ann}{ann}       % the annihilator ideal
\DeclareMathOperator{\Cat}{Cat}       % a catalecticant
\DeclareMathOperator{\chr}{char}      % characteristic (of a field)
\DeclareMathOperator{\Coid}{Coid}     % a set of coids
\DeclareMathOperator{\codim}{codim}   % codimension
\DeclareMathOperator{\Ext}{Ext}       % Ext group
\DeclareMathOperator{\GL}{GL}         % the general linear group
\DeclareMathOperator{\Grass}{Grass}   % the Grassmanian
\DeclareMathOperator{\Hom}{Hom}       % homomorphisms
\DeclareMathOperator{\id}{id}         % identity map
\DeclareMathOperator{\im}{im}         % the image
\DeclareMathOperator{\indx}{index}    % index of nilpotency
\DeclareMathOperator{\iso}{\cong}     % isomorphism
\DeclareMathOperator{\lc}{lc}         % leading coefficient
\DeclareMathOperator{\Mat}{Mat}       % the matrix algebra
\DeclareMathOperator{\nil}{nil}       % the nil radical
\DeclareMathOperator{\rank}{rank}     % the rank (of a matrix)
\DeclareMathOperator{\Reg}{Reg}       % regular splittings
\DeclareMathOperator{\sat}{sat}       % saturation
\DeclareMathOperator{\Socle}{Socle}   % the socle
\DeclareMathOperator{\bGor}{\mathbf{Gor}}       % parameter scheme
\DeclareMathOperator{\bPGor}{\mathbf{PGor}}     % parameter scheme
\DeclareMathOperator{\Gor}{Gor}       % parameterizing affine set
\DeclareMathOperator{\PGor}{PGor}     % parameterizing projective set
\DeclareMathOperator{\Split}{Split}   % splitting subfamily of Gor
\DeclareMathOperator{\PSplit}{PSplit} % splitting subfamily of PGor
\renewcommand{\implies}{\Rightarrow}
\renewcommand{\iff}{\Leftrightarrow}

\makeatletter
\newtheoremstyle{thm_style}{.5\baselineskip}{.5\baselineskip}{\slshape}{}{\bfseries\scshape}{:}{\newline}{}
\newtheoremstyle{lem_style}{.5\baselineskip}{.5\baselineskip}{\slshape\thm@notefont{}}{}{\bfseries\slshape}{:}{.5em}{}
\newtheoremstyle{def_style}{.5\baselineskip}{.5\baselineskip}{}{}{\bfseries}{:}{.5em}{}
\newtheoremstyle{case_style}{3pt}{3pt}{}{}{\slshape}{:}{ }{}
\makeatother

\theoremstyle{thm_style}
\newtheorem{thm}{Theorem}[chapter]

\theoremstyle{lem_style}
\newtheorem{lem}[thm]{Lemma}
\newtheorem{prop}[thm]{Proposition}
\newtheorem{cor}[thm]{Corollary}
\newtheorem{que}[thm]{Question}

\theoremstyle{def_style}
\newtheorem{defn}[thm]{Definition}

\newtheorem{rem}[thm]{Remark}
\newtheorem{exmp}[thm]{Example}

\theoremstyle{case_style}

\begin{document}

\bibliographystyle{amsalpha}
\renewcommand{\proofname}{Proof:}
\renewcommand{\chaptername}{\scshape Chapter\upshape}

%% \input{titlepg}
%% Forside
\pagestyle{empty}
\begin{titlepage}
\begin{center}
  {\huge\bfseries
    Additive Splittings\\
    of Homogeneous Polynomials\\}

  \vspace{1.5cm}
  {\Large\scshape
    Thesis for the Degree of\\
    Doctor Scientiarum\\}

  \vspace{2.5cm}
  {\Large\slshape\bfseries Johannes Kleppe}

  \vspace{2.5cm}
  \begin{figure}[htb]
    \begin{center}
      \epsfig{height=5cm,clip=false,file=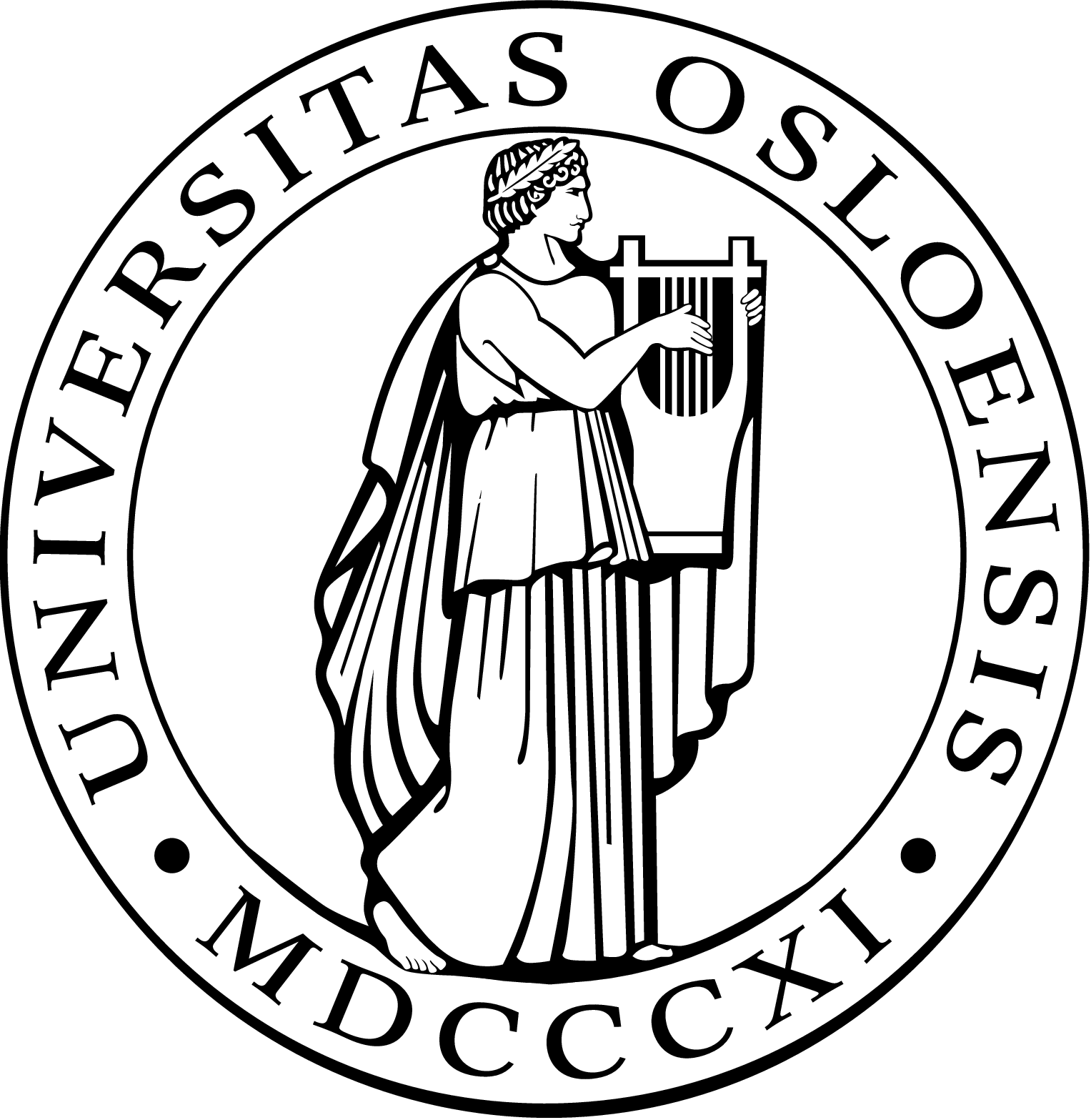}
    \end{center}
  \end{figure}

  \vspace{1cm}
  {\Large\scshape
    Department of Mathematics\\
    University of Oslo\\
    Spring 2005\\}
\end{center}
\end{titlepage}

\cleardoublepage
%% \input{abstract}
%% Abstract
%\addcontentsline{toc}{chapter}{Abstract}
\renewcommand{\abstractname}{\LARGE Abstract}
\begin{abstract}
  \vspace*{\baselineskip}
  The purpose of this paper is to study when a homogeneous polynomial $f$
  ``splits additively''. By this we mean that it is possible to write $f = g +
  h$ where $g$ and $h$ are polynomials in independent sets of variables. To
  help us we will define a matrix algebra $M_f$. The study of $M_f$ will be
  important to us, and we prove that it determines the (regular) splitting
  properties of $f$. Next, assuming the corresponding information is known
  about the additive components of $f$, we will compute the minimal free
  resolution of the graded Artinian Gorenstein quotient $R/\ann f$, the
  dimension of a ``splitting subfamily'' of the parameter space $\PGor (H)$,
  and the dimension of the tangent space to $\PGor (H)$ at the point
  corresponding to $f$.

  We will also study degenerations of polynomials that split and see how they
  relate to $M_f$. This situation is more difficult, but we are able to prove
  several partial results that together cover many interesting cases. In
  particular, we prove that $f$ has a regular or degenerate splitting if and
  only if the ideal $\ann f$ has at least one generator in its socle degree.
  Finally, we look at some generalizations of $M_f$ and corresponding results.
\end{abstract}
%% The following was used as an abstract for a talk at Blindern the autumn of 2003:
%%
%% Title: Polynomials that spilt additively
%%
%% Abstract: First I will define precisely what I mean by an additive splitting
%% of a homogeneous polynomial. As a basic example, a form f of degree d in
%% k[x,y] splits additive if it equals x^d + y^d after a suitable base
%% change. I will explain how to find any (regular) splitting that a given f
%% has, and I will say something about degenerations of such f. If time
%% permits, I will compute the resolution of the associated Artinian
%% Gorenstein quotient R/ann f, the dimension of a "splitting subfamily" of
%% PGor(H) and the dimension of the tangent space at the point corresponding
%% to f, assuming the corresponding information is known for the additive
%% components of f.

\cleardoublepage
\pagenumbering{roman}
\pagestyle{plain}
%% \input{preface}
%% Preface / Acknowledgments / History
\addcontentsline{toc}{chapter}{Preface}
\chapter*{Preface}

This thesis has been written for the degree of Doctor Scientiarum (dr.
scient.) at the Department of Mathematics, University of Oslo. My supervisor
has been professor Kristian Ranestad.

I would like to thank the University of Oslo for financial support. Special
thanks go to my supervisor, professor Kristian Ranestad, for helping me
through this period of time, and to professor Frank-Olaf Schreyer who first
introduced me to the problem of ``additive splittings'' during my stay at the
University of Bayreuth, Germany, the fall of 2000. I would also like to thank
my father, professor Jan Oddvar Kleppe, for interesting discussions.

The problem professor Schreyer originally gave me is the following. Consider a
homogeneous polynomial $f$ of degree $3$ (we were looking at double hyperplane
sections of canonical curves) that is a sum of two polynomials in separate
variables, that is $f = g+h$ with $g \in \K[x_1, \dots x_s]$ and $h \in
\K[x_{s+1}, \dots, x_r]$ up to base change. The minimal resolution of the
ideal
\begin{displaymath}
  ( \{ \p_i \p_j \suchthat i = 1, \dots, s, \: j = s+1, \dots, r \} )
  \subseteq R = \K [\p_1, \dots, \p_r]
\end{displaymath}
will be part of any resolution of $\ann f$. Therefore the graded Betti number
$\beta_{r-1,r}$ of $R/\ann f$ will be nonzero. He asked if I could prove that
this was an equivalence.

After computing some examples, I realized degree three did not matter much,
and I wondered if something stronger might be true. Could $1+\beta_{r-1,r}$ be
the maximal length of an ``additive splitting'' of $f$? It was also clear that
I had to allow degenerations of such splittings. I decided to take the simple
approach of definition \ref{def:limsplit} and restrict my attention to
``deformations'' defined over a polynomial ring. In the end it turned out that
$1+\beta_{r-1,r}$ does not always count the length of a maximal degenerate
splitting.

Chapter 1 contains a brief discussion of background material. In chapter 2 I
define precisely want I mean by regular and degenerate additive splittings. I
also define a matrix algebra $M_f$, which probably is the most important new
object in this thesis, and I give some basic results about $M_f$ and additive
splittings.

In chapter 3 I effectively determine all regular splittings, and I use this to
calculate the minimal free resolution of $R/ \ann f$ and its graded Betti
numbers. I also discuss some consequences for $\bPGor (H)$, the scheme
parameterizing all graded Artinian Gorenstein quotients of $R$. Chapter 4
studies degenerate splittings. The central question is whether we can use all
of $M_f$ to construct generalizations of $f$ that splits $\beta_{r-1,r}$
times. I give some conditions that implies a positive answer, and I construct
several counter examples in general. Finally, chapter 5 generalizes $M_f$ and
some results about it.

\vspace*{2\baselineskip}\noindent June, 2005\hfill Johannes Kleppe

%% Redefining \cleardoublepage (again) wuch that empty pages have "plain" style
\makeatletter
\renewcommand*{\cleardoublepage}{%
  \clearpage%
  \thispagestyle{plain}% <- addition
  \if@twoside%
  \ifodd\c@page\else\hbox{}\newpage%
  \if@twocolumn\hbox{}\newpage%
  \fi\fi\fi%
}
\makeatother

%% Table of contents after Preface:
\tableofcontents

%% Body of paper

\cleardoublepage
\pagenumbering{arabic}
\pagestyle{headings}

%% \input{intro}
%% Chapter 1: Introduction of problem and method
\chapter{Introduction}

It is well known that any homogeneous polynomial $f$ of degree two in $r$
variables over a field of characteristic $\ne 2$ can be written as a linear
combination of $n = \rank H \le r$ squares. Here $H = (\p_i\p_j f)$ is the
Hessian matrix of $f$. The usual way to generalize this to higher degrees is
to ask how to write a form $f$ of degree $d$ as a sum of powers of linear
forms, $f = \sum_{i=1}^n c_i l_i^d$, and how small $n$ can be. This is usually
called Waring's problem, and has been studied by many people and has been
solved for general $f$.

There is, however, a different way to generalize the sum of squares theorem.
If we write $f = \sum_{i=1}^n c_i l_i^2$ with $n$ minimal, then $l_1, \dots,
l_n$ are necessarily linearly independent. For higher degrees, when $f =
\sum_{i=1}^n c_i l_i^d$ and $d \ge 3$, the $l_i$'s can no longer be linearly
independent, except for very special $f$. With this in mind, we see that there
is another question that naturally generalizes of the sum of squares theorem:
When is it possible to write $f$ as a sum of several homogeneous polynomials
in independent sets of variables? We will call this a \emph{regular splitting}
of $f$ (definition \ref{def:regsplit}). Some examples of polynomials that
split regularly are $f = x_1^3 + x_2 x_3 x_4$, $f = x_1 x_2^6 + x_3^2 x_4^5 +
x_4^3 x_5^4$ and $f = (x_1 + x_2)^8 + x_2^3 (x_2 + x_3)^5$. Sometimes there
exist more than one regular splitting of the same polynomial, like $f = x_1^2
+ x_2^2 = \tfrac 12 \bigl( (x_1 + x_2)^2 + (x_1 - x_2)^2 \bigr)$.

To make the theory work in positive characteristics we assume that $f$ is a
homogeneous polynomial in the divided power algebra $\cR = \K [x_1, \dots,
x_r]^{DP}$. The polynomial ring $R = \K [\p_1, \dots, \p_r]$ acts on $\cR$ by
partial differentiation. An important ideal in $R$ will be $\ann_R f$, the set
of $D \in R$ that annihilates $f$, i.e. $D(f) = 0$. It is well known that $R/
\ann_R f$ is a Gorenstein ring of dimension zero, and furthermore that every
graded Artinian Gorenstein quotient arises this way, cf. lemma \ref{lem:mac}.

To study the splitting behavior of a homogeneous polynomial $f$ of degree $d$,
we associate to $f$ the following set of $r \times r$-matrices.

\medskip\noindent
{\bfseries Definition~\ref{def:Mf}:}
Given $f \in \cR_d$, define
\begin{displaymath}
  M_f = \{ A \in \Mat_\K (r,r) \suchthat I_2 (\p\: A\p) \subseteq \ann_R f \}.
\end{displaymath}

Here $\p = [\p_1, \dots, \p_r]^\T$ is a column vector, thus $(\p\: A\p)$ is
the $r \times 2$ matrix consisting of the two columns $\p$ and $A\p$, and $I_2
(\p\: A\p)$ is the ideal generated by its $2 \times 2$ minors. The study of
$M_f$ has a central position in this paper. One goal is figure out what $M_f$
can tell us about $f$. To transfer matrices $A \in M_f$ back into polynomials
$g \in \cR$, we also define a $\K$-linear map
\begin{displaymath}
  \gamma_f : M_f \to \cR_d
\end{displaymath}
sending $A \in M_f$ to the unique $g \in \cR_d$ that satisfies $\p g = A \p f$
(definition \ref{def:gamma_f}). An important property of $M_f$ is the
following.

% The first important property is the existence of the map $\gamma_f$.

\medskip\noindent {\bfseries\slshape Proposition~\ref{prop:Mf}:} {\slshape Let
  $d \ge 3$ and $f \in \cR_d$. $M_f$ is a $\K$-algebra, and all commutators
  belong to $\ker \gamma_f$. In particular, $M_f$ is commutative if $\ann
  (f)_1 = 0$. }

\medskip In chapter \ref{chapter:regular} we analyze the situation of regular
splittings completely. In particular, we prove that the idempotents in $M_f$
determine all regular splittings of $f$ in the following precise way.

\medskip\noindent
{\bfseries\scshape Theorem~\ref{thm:regsplit}:}\newline
{\slshape
  Assume $d \ge 2$, $f \in \cR_d$ and $\ann_R (f)_1 = 0$. Let $\Coid (M_f)$ be
  the set of all complete sets $\{ E_1, \dots, E_n \}$ of orthogonal
  idempotents in $M_f$, and let
  \begin{displaymath}
    \Reg (f) = \{ \{ g_1, \dots, g_n \} \suchthat f = g_1 + \dots + g_n \text{
      is a regular splitting of } f \}.
  \end{displaymath}
  The map $\{ E_i \}_{i=1}^n \mapsto \{ g_i = \gamma_f (E_i) \}_{i=1}^n$
  defines a bijection
  \begin{displaymath}
    \Coid (M_f) \to \Reg (f).
  \end{displaymath}
  In particular, there is a unique maximal regular splitting of $f$ when $d
  \ge 3$.
}

\medskip We also give an extended version of this theorem. In the
generalization (theorem \ref{thm:regsplitE}) we also prove that, loosely
speaking, $M_f = \dsum_{i=1}^n M_{g_i}$, if these algebras are computed inside
the appropriate rings. Note in particular the uniqueness when $d=3$, which is
not there when $d=2$.

In the last two sections of chapter \ref{chapter:regular} we examine a regular
splitting $f = \sum_{i=1}^n g_i$ more carefully. For each $i$, the additive
component $g_i$ is a polynomial in some divided power subring $\cS_i \subseteq
\cR$. The definition of a regular splitting requires that these subrings are
independent in the sense that $(\cS_i)_1 \isect \sum_{j \ne i} (\cS_j)_1 = 0$
for all $i$. We let $S_i$ be a polynomial subring of $R$ dual to $\cS_i$.
Assuming the minimal free resolutions of every $S_i/ \ann_{S_i} (g_i)$ is
known, then we are able to compute the minimal free resolution of $R/ \ann_R
f$. Theorem~\ref{thm:fgh} does this for the case $n=2$. The induction process
to get $n \ge 2$ is carried out for the shifted graded Betti numbers (see
equation \eqref{eq:shiftB} below), culminating in the following theorem.

\medskip\noindent
{\bfseries\scshape Theorem~\ref{thm:betti}:}\newline
{\slshape
  Let $d \ge 2$ and $f, g_1, \dots, g_n \in \cR_d$. Suppose $f = g_1 + \dots +
  g_n$ is a regular splitting of $f$. Let $s_i = \dim_\K R_{d-1} (g_i)$ for
  every $i$. Let $s = \sum_{i=1}^n s_i$, and define
  \begin{displaymath}
    \nu_{nk} = (n-1) \binom{r}{k+1} + \binom{r-s}{k+1} - \sum_{i=1}^n
    \binom{r-s_i}{k+1}.
  \end{displaymath}
  Denote by $\smash{\hat{\beta}^f_{kj}}$ and $\smash{\hat{\beta}^{g_i}_{kj}}$
  the shifted graded Betti numbers of $R/ \ann_R (f)$ and $R/ \ann_R (g_i)$,
  respectively. Then
  \begin{equation*}
    \hat{\beta}^f_{kj} = \sum_{i=1}^n \hat{\beta}^{g_i}_{kj} + \nu_{nk}
    \delta_{1j} + \nu_{n,r-k} \delta_{d-1,j}
  \end{equation*}
  for all $0<j<d$ and all $k \in \Z$. Here the symbol $\delta_{ij}$ is defined
  by $\delta_{ii} = 1$ for all $i$, and $\delta_{ij} = 0$ for all $i \ne j$.
}

\medskip We proceed to study some consequences for $\bPGor (H)$, the
quasi-projective scheme parameterizing all graded Artinian Gorenstein
quotients $R/I$ with Hilbert function $H$. We define a subset $\PSplit (H_1,
\dots, H_n) \subseteq \bPGor (H)$ that parametrizes all quotients $R/ \ann_R
f$ such that $f$ has a regular splitting $f = \sum_{i=1}^n g_i$ such that the
Hilbert function of $R/ \ann_R (g_i)$ is $H_i$ for all $i$, and we are able to
prove under some conditions that its closure $\overline{\PSplit (H_1, \dots,
  H_n)}$ is an irreducible, generically smooth component of $\bPGor (H)$
(theorem \ref{thm:gor}).

\medskip In chapter 4 we turn our attention to degenerate splittings, i.e.
polynomials that are specializations of polynomials that split regularly. A
simple example is $f = x^\pd 2 y = \tfrac 1t \bigl( (x+y)^\pd 3 - x^\pd 3
\bigr)$. The main question that we are trying to shed some light upon, is the
following.

\medskip\noindent
{\bfseries{\slshape Question}~\ref{que:Q}: }
{\slshape
  Given $f \in \cR_d$, $d \ge 3$, is it possible to find $f_t \in \cR_d [t_1,
  \dots, t_n]$ such that $f_0 = f$ and $f_t$ splits regularly $\dim_k M_f - 1$
  times over $\K (t_1, \dots, t_n)$?
}

\medskip By lemma \ref{lem:upbound}, $\dim_\K M_f$ is an upper bound for the
length of a regular splitting of $f_t$. Thus the question asks when this upper
bound is achieved. This would mean that $M_f$ not only determines the regular
splittings of $f$, but that we are able to use all of $M_f$ to construct
degenerate splittings as well.

% The working hypothesis was that $M_f$ should determine all splitting
% properties of $f$. For regular splittings, this is achieved, but the
% situation is more complicated for degenerate splittings.

We first prove that we can construct an $f_t$ with the desired properties
using all powers of a single nilpotent matrix $A$. This is
theorem~\ref{thm:onematrix}. In particular it gives a positive answer to
question \ref{que:Q} in case $M_f$ is generated by $A$ alone as a
$\K$-algebra.

\medskip\noindent
{\bfseries\scshape Theorem~\ref{thm:onematrix}:}\newline
{\slshape
  Let $d \ge 3$ and $f \in \cR_d$. Assume that $M_f$ contains a non-zero
  nilpotent matrix $A \in \Mat_\K (r,r)$, and let $n = \indx (A) - 1 \ge 1$.
  Then $f$ is a specialization of some $f_t \in \cR_d [t_1, \dots, t_n]$ that
  splits regularly $n$ times inside $\cR_d (t_1, \dots, t_n)$.
}

\medskip We later give a generalized version of this theorem. A careful
analysis shows that this covers most cases with $r \le 4$, and we are able to
solve the rest by hand. Hence we get the following result.

\medskip\noindent
{\bfseries\scshape Theorem~\ref{thm:rle4}:}\newline
{\slshape
  Assume that $r \le 4$ and $\bar{\K} = \K$. Let $f \in \cR_d$, $d \ge 3$,
  satisfy $\ann_R (f)_1 = 0$. Then for some $n \ge 1$ there exists $f_t \in
  \cR_d [t_1, \dots, t_n]$ such that $f_0 = f$ and $f_t$ splits regularly
  $\dim_\K M_f - 1$ times over $\K (t_1, \dots, t_n)$.
}

\medskip The rest of chapter 4 is devoted to constructing examples where
question \ref{que:Q} has a negative answer. We are able to do this for all
$(r,d)$ with $r \ge 5$ and $d \ge 3$, except the six pairs $(5,3), (6,3),
(7,3), (8,3), (5,4)$ and $(6,4)$.

Finally, in chapter 5, we consider some generalizations of $M_f$. We do not
yet have a particular use for these generalizations. However, $M_f$ proved
very useful to us, and we show how to define two similar algebras and prove
some basic results about them.

% An alternative way to write these ``out of place'' theorems.
%\newcounter{oldthm}
%\setcounter{oldthm}{\value{thm}}
%\setcounter{thm}{10}
%\begin{thm}
%  Assume $d \ge 3$ and $f \in \cR_d$. $f$ splits as $g_1 + \dots + g_n$ if
%  and only if, for all $i$, there is a matrix $E_i \in \Mat_\K (r,r)$ such
%  that $\p g_i = E_i \p f$, and $\{ E_1, \dots, E_n \}$ is a complete set of
%  orthogonal idempotents for $M_f$. In particular, there is a unique maximal
%  regular splitting of $f$.
%\end{thm}
%\setcounter{thm}{\value{oldthm}}

% \chapter{Preliminaries}

\section{Polynomials and divided powers}

Let $R = \K[\p_1, \dots, \p_r]$ be a polynomial ring in $r$ variables with the
standard grading over a field $\K$. As usual, we denote by $R_d$ the
$\K$-vector space spanned by all monomials of total degree $d$. Then $R =
\dsum_{d \ge 0} R_d$, and elements in $\union_{d \ge 0} R_d$ are called
homogeneous. An ideal $I$ in $R$ is homogeneous if $I = \dsum_d I_d$ where
$I_d = I \isect R_d$. The unique maximal homogeneous ideal in $R$ is $\m_R =
(\p_1, \dots, \p_r)$.

The graded Betti numbers $\beta_{ij}$ of a homogeneous ideal $I$ are the
coefficients that appear in a graded minimal free resolution of $I$. We will
often speak of the ``shifted'' graded Betti numbers, by which we mean
$\hat{\beta}_{ij} = \beta_{i,i+j}$. So if $0 \to F_c \to \dots \to F_1$ is a
graded minimal free resolution of $I$, then the \te{i} term is
\begin{equation} \label{eq:shiftB}
  F_i \iso \dsum_{j \ge i} \beta_{ij} R(-j) = \dsum_{j \ge 0} \hat{\beta}_{ij}
  R(-i-j)
\end{equation}
In particular, $\beta_{1j}$ is the minimal number of generators of $I$ of
degree $j$.

Let $\cR = \dsum_{d \ge 0} \cR_d$ be the graded dual of $R$, i.e. $\cR_d =
\Hom_\K (R_d, \K)$. It is called a ring of divided powers, and we write $\cR =
\K [x_1, \dots, x_r]^{DP}$. Let $\bN_0$ denote the non-negative integers. The
divided power monomials
\begin{displaymath}
  \biggl\{ x^\pd \alpha = \prod_{i=1}^r x_i^\pd {\alpha_i} \,\bigg\vert\:
  \alpha \in \bN_0^r \text{ and } |\alpha| = \sum_{i=1}^r \alpha_i = d
  \biggr\}
\end{displaymath}
form a basis for $\cR_d$ as a $\K$-vector space. This basis is dual to the
standard monomial basis for $R_d$, i.e. $\{ \p^\beta = \sprod_i
\p_i^{\alpha_i} \suchthat \beta \in \bN_0^r \text{ and } |\beta| = d \}$, in
the sense that $x^\pd \alpha \bigl( \p^\alpha \bigr) = 1$ and $x^\pd \alpha
\bigl( \p^\beta \bigr) = 0$ for $\alpha \ne \beta$. The ring structure of
$\cR$ is the natural one generated by
\begin{displaymath}
  x_i^\pd a \cdot x_i^\pd b = \tbinom{a+b}{a} \, x_i^\pd {a+b},
\end{displaymath}
see \cite[Section A2.4]{Eis} or \cite[Appendix A]{IK} for details. We will
refer to elements of $\cR_d$ simply as homogeneous polynomials or forms of
degree $d$. If $\chr \K = 0$, we may identify $\cR$ with the regular
polynomial ring $\K [x_1, \dots, x_r]$ by letting $x_i^\pd d = x_i^d / d!$

Let $R$ act on $\cR$ by
\begin{displaymath}
  \p^\beta \bigl( x^\pd \alpha \bigr) = x^\pd {\alpha - \beta},
\end{displaymath}
i.e. the action generated by $\p_i \bigl( x_i^\pd d \bigr) = x_i^\pd {d-1}$
and $\p_j \bigl( x_i^\pd d \bigr) = 0$ for all $i \ne j$. The reason for our
notation is that $\p_i$ is indeed a derivation, which follows by bilinearity
from
\begin{align*}
  \p_i \bigl( x_i^\pd {a} \bigr) \cdot x_i^\pd {b} + x_i^\pd {a} & \cdot \p_i
  \bigl( x_i^\pd {b} \bigr) \\
  & = x_i^\pd {a-1} \cdot x_i^\pd {b} + x_i^\pd {a} \cdot x_i^\pd {b-1} \\
  & = \tbinom{a+b-1}{a-1} \, x_i^\pd {a+b-1} + \tbinom{a+b-1}{a} \, x_i^\pd
  {a+b-1} \\
  & = \tbinom{a+b}{a} \, x_i^\pd {a+b-1} = \tbinom{a+b}{a} \, \p_i \bigl(
  x_i^\pd {a+b} \bigr) \\
  & = \p_i \bigl( x_i^\pd {a} \cdot x_i^\pd {b} \bigr).
\end{align*}
Under the identification $x_i^\pd d = x_i^d / d!$ when $\chr \K = 0$, the
action of $\p_i$ becomes normal partial differentiation with respect to $x_i$.

% For any vector space $V$, we say that a column vector $v = [ v_1, \dots, v_n
% ]^\T$ is an \emph{ordered basis} for $V$ if $\{ v_1, \dots, v_n \}$ is a
% basis for $V$. Usually, we simply say that $v$ is a basis for $V$.

Arrange the elements of the standard monomial bases for $\cR_d$ and $R_d$ into
column vectors $h$ and $D$ using the same ordering. The fact that they are
dual can then be expressed as $D h^\T = I$, the identity matrix. If $\{ f_1,
\dots, f_N \}$ is any basis for $\cR_d$, $N = \dim_\K \cR_d =
\binom{r-1+d}{d}$, then there is a dual basis for $R_d$. Indeed, there exists
an $N \times N$ invertible matrix $P$ such that $f = [ f_1, \dots, f_N ]^\T =
P^\T h$. Let $E = P^{-1} D$. Then $Ef^\T = P^{-1} Dh^\T P = I$, hence $E$ is
the dual basis of $f$ (as column vectors).

If $S$ is any ring, let $\Mat_S (a,b)$ be the set of $a \times b$ matrices
defined over $S$, and let $\GL_r (S)$ be the invertible $r \times r$ matrices.
When $S = \K$, we usually just write $\GL_r$. We will frequently make use of
the following convention.
\begin{quote}
  If $v \in S^b$ is any vector and $A \in \Mat_S (a,b)$ any matrix, we denote
  by $v_i$ the \te{i} entry of $v$ and by $A_{ij}$ the \te{(i,j)} entry of
  $A$.
\end{quote}
In particular, $(Av)_i = \sum_{j=1}^b A_{ij} v_j$ is the \te{i} entry of the
vector $Av$, and the \te{(i,j)} entry of the rank one matrix $(Av)(Bv)^\T$ is
$(Avv^\T B^\T)_{ij} = (Av)_i(Bv)_j$.

% In terms of standard bases we have $v = \sum_i v_i e_i$ and $A = \sum_{ij}
% A_{ij} E_{ij}$.

For any $P \in \GL_r$, define $\phi_P : \cR \to \cR$ to be the $\K$-algebra
homomorphism induced by $x_i \mapsto \sum_{j=1}^r P_{ji} x_j$ for all $i$. We
usually let $x$ denote the column vector $x = [x_1, \dots x_r]^\T$, thus
$\phi_P$ is induced by $x \mapsto P^\T x$. The ``dual'' map $R \to R$, which
we also denote by $\phi_P$, is induced by $\p \mapsto P^{-1} \p$, where $\p =
[\p_1, \dots, \p_r]^\T$. For any $D \in R$ and $f \in \cR$, it follows that
\begin{displaymath}
  %% CHANGED 2013-07-13: removed comma at display end
  \phi_P (Df) = (\phi_P D) (\phi_P f)
\end{displaymath}
and in particular, $\ann_R (\phi_P f) = \phi_P (\ann_R f)$.

If $D \in \Mat_R (a,b)$ and $h \in \Mat_\cR (b,c)$, then $Dh$ denotes the $a
\times c$ matrix whose \te{(i,j)} entry is $(Dh)_{ij} = \sum_{k=1}^b D_{ik}
(h_{kj}) \in \cR$. Of course, this is nothing but the normal matrix product,
where multiplication is interpreted as the action of $R$ and $\cR$. We already
used this notation when discussing dual bases. Also, for any $f \in \cR$, we
let $D(f)$ (or simply $Df$) denote the $a \times b$ matrix whose \te{(i,j)}
entry is $(Df)_{ij} = D_{ij} (f) \in \cR$. It follows that if $E \in \Mat_R
(a',a)$, then $E(D(f)) = (ED)(f)$.

If $A \in \Mat_R (a,b)$ and $v_i \in R^a$ is the \te{i} column vector in $A$,
then we let $I_k (A) = I_k (v_1 \: \cdots \: v_b)$ be the ideal generated by
all $k \times k$ minors of $A$ ($k \le a,b$). Of course, this only depends on
$\im A = \langle v_1, \dots, v_b \rangle = \{ \sum_{i=1}^b c_i v_i \suchthat
c_1, \dots, c_b \in \K \}$.

% Note that Euler's identity $\sum_{i=1}^r x_i \p_i f = d f$ for all $f \in
% \cR_d$ also holds in the divided power ring. But if $\chr \K \divided d$,
% then the right-hand side is zero, and we cannot divide by $d$. But we can
% still lift a set $\{ g_i \}_{i=1}^r$ satisfying $\p_i g_j = \p_j g_i$ to a
% $g$ such that $\p_i g = g_i$ for all $i$.

\section{Annihilator ideals and Gorenstein quotients}

Given any $\K$-vector subspace $V \subseteq \cR_d$, define its orthogonal
$V^\perp \subseteq R_d$ by
\begin{displaymath}
  V^\perp = \{ D \in R_d \suchthat Df=0 \:\forall\: f \in V \}.
\end{displaymath}
Similarly, if $U \subseteq R_d$, define $U^\perp = \{ f \in \cR_d \suchthat
Df=0 \:\forall\: D \in U \}$.

% The dual of $V$, is $V^* = \Hom_\K (V,\K) = R_d / V^\perp$. Because $R_d =
% \Hom_\K (\cR_d, \K)$, hence $R_d \to \Hom_\K (V,\K)$ is surjective, and its
% kernel is clearly $V^\perp$. Thus $R_d \iso V^* \dsum V^\perp$.

Let $n=\dim_\K V$ and $N = \dim_\K \cR_d = \dim_\K R_d$. Pick a basis $\{f_1,
\dots, f_n\}$ for $V$, and expand it to a basis $\{f_1, \dots, f_N\}$ for
$\cR_d$. Let $\{D_1, \dots, D_N\}$ be the dual basis for $R_d$. Clearly,
$V^\perp = \langle D_{n+1}, \dots, D_N \rangle $, the $k$-vector subspace of
$R_d$ spanned by $D_{n+1}, \dots, D_N$. Therefore,
\begin{displaymath}
  \dim_\K V + \dim_\K V^\perp = \dim_\K R_d.
\end{displaymath}

By symmetry, this equation is true also when applied to $V^\perp$, that is, we
get $\dim_\K V^\perp + \dim_\K V^{\perp\perp} = \dim_\K R_d$. Hence it follows
that $\dim_\K V^{\perp\perp} = \dim_\K V$. Since
\begin{math}
  V^{\perp\perp} = \{ g \in \cR_d \suchthat Dg = 0 \:\forall\: D \in V^\perp
  \}
\end{math}
obviously contains $V$, we have in fact $V^{\perp\perp} = V$. Note in
particular that $\cR_d^\perp = 0$ and $R_d^\perp = 0$. This says precisely
that the pairing ($\K$-bilinear map) $R_d \times \cR_d \to \K$ defined by
$(D,f) \mapsto D(f)$ is non-degenerate.

\begin{defn} \label{def:ann}
  For any $f \in \cR_d$, $d \ge 0$, the \emph{annihilator ideal} in $R$ of $f$
  is defined to be
  \begin{displaymath}
    \ann_R (f) = \{ D \in R \suchthat Df=0 \}.
  \end{displaymath}
\end{defn}

% Could allow inhomogeneous $f$ in this definition. Also, even though the
% definition uses globally defined $\cR$ and $R$, it applies as long as $R$
% acts on $\cR_d$. Alternative def: for any $U \subseteq R$ and $V \subseteq
% \cR$, let $\ann_U V = \{ D \in U \suchthat Df = 0 \:\forall\: f \in V \}$.

Since $f$ is homogeneous, $\ann_R (f)$ is a homogeneous ideal in $R$. We
notice that its degree $d$ part $\ann_R (f)_d$ is equal to $\langle f
\rangle^\perp$ as defined above. The annihilator ideals have several nice
properties.

First, consider the homomorphism $R_e \to \cR_{d-e}$ defined by $D \mapsto
D(f)$. We denote its image by
\begin{displaymath}
  R_e (f) = \{ D(f) \suchthat D \in R_e \},
\end{displaymath}
and its kernel is by definition $\ann_R (f)_e$. We observe that if $R_e (f) =
0$ for some $e < d = \deg f$, then $R_d(f)=0$ because $R_d = R_{d-e} \cdot
R_e$. Since $R_d \times \cR_d \to \K$ is non-degenerate, this implies $f=0$.
Thus the contraction map $R_e \times \cR_d \to \cR_{d-e}$ is also
non-degenerate. The $R$-module $R(f) = \dsum_{e \ge 0} R_e (f)$ is called the
module of contractions.

\begin{lem} \label{lem:ann}
  Let $d,e \ge 0$ and $f \in \cR_d$. The ideal $\ann_R (f) \subseteq R$
  satisfies:
  \begin{enumerate}
    \setlength{\itemsep}{0pt}
    \setlength{\parskip}{2pt}
    \renewcommand{\theenumi}{\alph{enumi}}
    \renewcommand{\labelenumi}{\normalfont(\theenumi)}
  \item If $0 \le k \le e \le d$, then the degree $k$ part $\ann_R (f)_k$ is
    determined by the degree $e$ part $\ann_R (f)_e$ by ``saturation'', that
    is, $D \in \ann_R (f)_k$ if and only if $ED \in \ann_R (f)_e$ for all $E
    \in R_{e-k}$.
  \item $R_e (f) \iso R_e / \ann_R (f)_e$ and $R_e (f)^\perp = \ann_R
    (f)_{d-e}$.
  \item $\dim_\K (R/ \ann_R (f))_e = \dim_\K R_e (f) = \dim_k (R/ \ann_R
    (f))_{d-e}$.
  \item $\isect_{D\in R_e} \ann_R (Df) = \ann_R (f) + R_d + \dots +
    R_{d-e+1}$.

    In particular, $\isect_{D\in R_e} \ann_R (Df)_{d-e} = \ann_R (f)_{d-e}$.
  \end{enumerate}
\end{lem}

\begin{proof}
  To prove (a), let $D \in R_k$. Since $R_{d-e} \times \cR_{d-e} \to \K$ is
  non-degenerate, it follows for any $E \in R_{e-k}$ that $ED(f) = 0$ if and
  only if $E'ED(f) = 0$ for all $E' \in R_{d-e}$. Therefore, $ED(f) = 0$ for
  all $E \in R_{e-k}$ if and only if $E''D(f) = 0$ for all $E'' \in R_{d-k}$,
  which is equivalent to $D(f)=0$ since $R_{d-k} \times \cR_{d-k} \to \K$ is
  non-degenerate. Thus
  \begin{displaymath}
    \ann_R (f)_k = \{ D \in R_k \suchthat R_{e-k} \cdot D \subseteq \ann_R
    (f)_e \},
  \end{displaymath}
  i.e. $\ann_R (f)_k$ is determined by $\ann_R (f)_e$ by ``saturation''.

  The first part of (b) follows immediately from the exact sequence
  \begin{displaymath}
    0 \to \ann_R (f)_e \to R_e \to R_e (f) \to 0.
  \end{displaymath}
  Since $R_e (f) \subseteq \cR_{d-e}$, it follows from (a) that
  \begin{align*}
    R_e (f)^\perp & = \{ D \in R_{d-e} \suchthat D(Ef) = 0 \text{ for all } E
    \in R_e \} \\
    & = \{ D \in R_{d-e} \suchthat D(f) = 0 \} = \ann_R (f)_{d-e}.
  \end{align*}
  And (c) follows from (b) by taking dimensions of the two equalities. Note
  that
  \begin{align*}
    \isect_{D \in R_e} \ann_R (Df)_{d-e} & = \{ E \in R_{d-e} \suchthat
    E(Df) = 0 \text{ for all } D \in R_e \} \\
    & = R_e (f)^\perp = \ann_R (f)_{d-e}.
  \end{align*}
  Now (d) follows by ``saturating downwards'' due to (a). (Obviously, it is
  enough to use a basis for $R_e$ in the intersection.)
\end{proof}

% We see that $E \in \isect_{D\in R_e} \ann_R (Df)$ if and only if $DE(f) = 0$
% for all $D \in R_e$, which is equivalent to $E(f) \in \cR_{e-1}$. So either
% $\deg E > d-e$ or $E(f)=0$, which is (d). We note that it is enough to use a
% basis for $R_e$ in the intersection.

Let $f \in \cR_d$. The \emph{Hilbert function} $H_f = H (R/ \ann_R f)$ of $R/
\ann_R (f)$ computes the dimensions of the graded components of $R/ \ann_R
(f)$, i.e.
\begin{displaymath}
  H_f (e) = \dim_\K (R/ \ann_R f)_e \text{ for all } e \ge 0.
\end{displaymath}
Note that (c) implies that the Hilbert function of $R/ \ann_R (f)$ is
symmetric about $d/2$. Since $H_f (e) = 0$ for all $e > d$, we will often
abuse notation and write $H_f = (h_0, \dots, h_d)$ where $h_e = H_f (e)$.
Written this way, $H_f$ is sometimes called the \emph{$h$-vector} of $R/
\ann_R f$.

A finitely generated $\K$-algebra $A$ is \emph{Artinian} if and only if it has
finite dimension as a $\K$-vector space. Let $I \subseteq R$ be a homogeneous
ideal. Then $A = R/I$ is Artinian if and only if $I_e = R_e$ for all $e \gg
0$. Its \emph{socle} is defined by $\Socle (R/I) = (0:\m_R)$, i.e. $\Socle
(R/I) = \dsum_{e \ge 0} \Socle_e (R/I)$ where $\Socle_e (R/I) = \{ D \in R_e
\suchthat \p_i D \in I_{e+1} \text{ for all } i = 1, \dots, r \} / I_e$.
Furthermore, $\Hom_\K (-,\K)$ is a \emph{dualizing functor} for $A$, hence its
\emph{canonical module} is
\begin{displaymath}
  \omega_A = \Hom_\K (A, \K) = \dsum_{e \ge 0} \Hom_\K (A_e, \K).
\end{displaymath}
$A$ is called \emph{Gorenstein} if $\omega_A \iso A$ (up to a twist). By
\cite[proposition 21.5]{Eis}, $A = R/I$ is Gorenstein if and only if its socle
is simple, i.e. $\dim_\K \Socle (R/I) = 1$. By \cite[proposition 21.16]{Eis}
this is equivalent to the minimal free resolution of $A$ being self-dual.

% The socle of a local Artinian ring $A$ is defined to be $\Socle A = (0:\m) =
% \{ a \in A \suchthat a\m = 0 \}$.

% Gorenstein: $\omega_A \iso A(d)$, i.e. the twist is the socle degree.

\begin{lem}[Macaulay] \label{lem:mac}
  There is a one-to-one correspondence between graded Artinian Gorenstein
  quotients $R/I$ having socle degree $d$, and non-zero poly\-nomials $f \in
  \cR_d$ up to a scalar multiplum. The correspondence is given by $I = \ann_R
  f$ and $\langle f \rangle = (I_d)^\perp$.
\end{lem}

\begin{proof}
  See \cite[Theorem 21.6 and Exercise 21.7]{Eis} or \cite[Lemma 2.14]{IK}.
  Macaulay's original proof in \cite[chapter IV]{Mac} uses \emph{inverse
    systems}.
\end{proof}

Note that it is customary to call $\ann_R (f)$ a \emph{Gorenstein ideal} since
the quotient $R/\ann_R f$ is Gorenstein. We conclude these preliminaries with
the following fundamental lemma. It expresses the effect of dualizing ($V
\mapsto V^\perp$) an inclusion $U \subseteq V$ in terms of annihilator ideals.

\begin{lem}[Apolarity] \label{lem:apolar}
  Let $f_1, \dots, f_n, g_1, \dots, g_m \in \cR_d$ be forms of the same degree
  $d$. Then the following statements are equivalent:
  \begin{enumerate}
    \setlength{\itemsep}{2pt}
    \setlength{\parskip}{0pt}
    \renewcommand{\theenumi}{\alph{enumi}}
    \renewcommand{\labelenumi}{\normalfont(\theenumi)}
% Originally: \item $f \in \langle g_1, \dots, g_s \rangle$
  \item $\langle f_1, \dots, f_n \rangle \subseteq \langle g_1, \dots, g_m
    \rangle$
  \item $\isect_{i=1}^n \ann_R (f_i) \supseteq \isect_{i=1}^m \ann_R (g_i)$
  \item $\isect_{i=1}^n \ann_R (f_i)_d \supseteq \isect_{i=1}^m \ann_R
    (g_i)_d$
  \end{enumerate}
\end{lem}

\begin{proof}
  (a) just says that all $f_i$ can be written as $f_i = \sum_{j=1}^m c_{ij}
  g_j$ for suitable $c_{ij} \in \K$. So if $D \in R$ annihilates all $g_j$, it
  necessarily kills all $f_i$, which proves (a) $\implies$ (b). (b) $\implies$
  (c) is trivial, and (c) $\implies$ (a) follows from $V^{\perp\perp} = V$ and
  \begin{displaymath}
    \langle f_1, \dots, f_n \rangle ^\perp = \{ D \in R_d \suchthat D(f_i) = 0
    \:\forall\: i \} = \isect_{i=1}^n \ann_R (f_i)_d.\qedhere
  \end{displaymath}
\end{proof}

\begin{rem}
  What is more often called the \emph{apolarity lemma}, for example
  \cite[Lemma 1.15]{IK}, follows from lemma \ref{lem:apolar} by letting $n=1$
  and $g_i = l_{p_i}^\pd d$, $l_{p_i} = \sum_j p_{ij} x_j$, with the
  additional observation that $D (l_{p_i}^\pd d) = D(p_i) l_{p_i}^\pd {d-e}$
  for all $D \in R_e$.
\end{rem}

%% \input{defmf}
%% Chapter 2: Finding additive splittings (def of M_f)
\chapter{Additive splitting}

\section{What is an additive splitting?}

We would like to say that a polynomial like $f = x_1^\pd 2 x_2^\pd 2 + x_3^\pd
4$ \emph{splits} since it is a sum of two polynomials, $x_1^\pd 2 x_2^\pd 2$
and $x_3^\pd 4$, that do not share any variable. Of course, we want to allow a
change of variables. Therefore, we need to make the idea of ``polynomials in
separate variables'' more precise.

\begin{defn} \label{def:separate}
  Let $g_1, \dots, g_n \in \cR$ be homogeneous polynomials, and for all $i$
  let $d_i = \deg g_i$. We say that $g_1, \dots, g_n$ are polynomials in
  (linearly) \emph{independent sets of variables} if
  \begin{displaymath}
    R_{d_i-1} (g_i) \isect \biggl( \sum_{j \ne i} R_{d_j-1} (g_j) \biggr) = 0
  \end{displaymath}
  as subspaces of $\cR_1$ for all $i = 1, \dots, n$.
\end{defn}

\begin{rem} \label{rem:native}
  Let $f \in \cR_d$. It is natural to say that $R_{d-1} (f)$ contains the
  ``native'' variables of $f$ for the following reason. If $V \subseteq \cR_1$
  is a $\K$-vector subspace, denote by $\K [V]^{DP}$ the $\K$-subalgebra of
  $\cR$ generated by $V$. If $v_1, \dots, v_n$ is any basis for $V$, then $\K
  [V]^{DP} = \K [v_1, \dots, v_n]^{DP}$. In particular, $\K [V]^{DP}_0 = \K$
  and $\K [V]^{DP}_1 = V$. For all $\delta \in R_{d-1} (f)^\perp \subseteq
  R_1$ and all $D \in R_{d-1}$, it follows that $D \delta f \in \delta
  (R_{d-1} (f)) = 0$. Hence $\delta f = 0$ for all $\delta \in R_{d-1}
  (f)^\perp$, and therefore
  \begin{displaymath}
    f \in \K [R_{d-1} (f)]^{DP}.
  \end{displaymath}
  Thus definition \ref{def:separate} simply requires that the sets of native
  variables of $g_1, \dots, g_n$ are linearly independent, that is, if
  $\sum_{i=1}^n c_i v_i = 0$ for some $v_i \in R_{d_i-1} (g_i)$ and $c_i \in
  \K$, then $c_i = 0$ for all $i$.
\end{rem}

\begin{rem} \label{rem:separate}
  We note that definition \ref{def:separate} implies that
  \begin{displaymath}
    R_{d_i-e} (g_i) \isect \biggl( \sum_{j \ne i} R_{d_j-e} (g_j) \biggr) = 0
  \end{displaymath}
  for all $i = 1, \dots, n$ and all $e>0$. Indeed, if $h \in R_{d_i-e} (g_i)
  \isect \bigl( \sum_{j \ne i} R_{d_j-e} (g_j) \bigr)$, then $D(h) \in
  R_{d_i-1} (g_i) \isect \bigl( \sum_{j \ne i} R_{d_j-1} (g_j) \bigr) = 0$ for
  all $D \in R_{e-1}$, hence $h=0$.
\end{rem}

\begin{defn} \label{def:regsplit}
  Let $f \in \cR_d$. We say that $f$ \emph{splits regularly} $n-1$ times if
  $f$ is a sum of $n$ non-zero forms of degree $d$ in independent sets of
  variables. That is, if there exist non-zero $g_1, \dots, g_n \in \cR_d$ such
  that
  \begin{displaymath}
    f = g_1 + \dots + g_n,
  \end{displaymath}
  and for all $i$, $R_{d-1} (g_i) \isect \bigl( \sum_{j \ne i} R_{d-1} (g_j)
  \bigr) = 0$ as subspaces of $\cR_1$. In this situation, we call the $g_i$'s
  \emph{additive components} of $f$, and we say that the expression $f = g_1 +
  \dots + g_n$ is a \emph{regular splitting of length $n$}.
\end{defn}

% We may say that $f$ has length $n$ if it has a splitting of length $n$,
% i.e. $n$ additive components.

Clearly, this concept is uninteresting for $d=1$. For $d=2$ and $\chr \K \ne
2$ it is well known that any $f \in \cR_2$ can be written as a sum of $n =
\rank (\p\p^\T f)$ squares. (When $\chr \K = 2$ it is in general only a limit
of a sum of $n$ squares). Consequently, we will concentrate on $d \ge 3$.

\begin{exmp}
  Let $\chr \K \ne 2$ and $f = x^\pd 3 + x y^\pd 2 \in \K [x,y]^{DP}$. Then
  \begin{displaymath}
    f = \tfrac{1}{2} \bigl( (x+y)^\pd 3 + (x-y)^\pd 3 \bigr)
  \end{displaymath}
  is a regular splitting of $f$ of length 2. Indeed, $R_2 \bigl( (x+y)^\pd 3
  \bigr) = \langle x+y \rangle$ and $R_2 \bigl( (x-y)^\pd 3 \bigr) = \langle
  x-y \rangle$, and their intersection is zero.
\end{exmp}

\begin{rem} \label{rem:regsplit}
  When $f$ splits regularly, it is possible to separate the variables of its
  components by a suitable ``rectifying'' automorphism. More precisely, $f \in
  \cR_d$ splits regularly $n-1$ times if and only if there exists $\J_1,
  \dots, \J_n \subseteq \{1, \dots, r \}$ such that $\J_i \isect \J_j =
  \varnothing$ for all $i \ne j$, a graded automorphism $\phi : \cR \to \cR$
  and nonzero polynomials $h_i \in \cS^i_d$ where $\cS^i = \K [ \{ x_j
  \suchthat j \in \J_i \} ] \subseteq \cR$, such that $\phi (f) = h_1 + \dots
  + h_n$.
  
  To prove this, assume that $f = g_1 + \dots + g_n$ is a regular splitting of
  $f$. By definition, $R_{d-1} (g_i) \isect \bigl( \sum_{j \ne i} R_{d-1}
  (g_j) \bigr) = 0$ for all $i$. This simply means that $R_{d-1} (g_1), \dots,
  R_{d-1} (g_n)$ are linearly independent subspaces of $\cR_1$, that is, if
  $\sum_{i=1}^n c_i v_i = 0$ for some $v_i \in R_{d-1} (g_i)$ and $c_i \in
  \K$, then $c_i = 0$ for all $i$. Let $s_i = \dim_\K R_{d-1} (g_i)$. Then in
  particular, $\sum_{i=1}^n s_i = \dim_\K \bigl( \sum_{i=1}^n R_{d-1} (g_i)
  \bigr) \le r$. Hence we may choose $\J_1, \dots \J_n \subseteq \{ 1, \dots,
  r \}$ such that $| \J_i | = s_i$ and $\J_i \isect \J_j = \varnothing$ for
  all $i \ne j$. Now, choose a graded automorphism $\phi : \cR \to \cR$ such
  that $\{ \phi^{-1} (x_j) \suchthat j \in \J_i \}$ is a basis for $R_{d-1}
  (g_i)$ for all $i$, and let $h_i = \phi (g_i) \in \cR_d$ and $\cS^i = \K [
  \{ x_j \suchthat j \in \J_i \} ]^{DP}$. Obviously, $h_i \ne 0$ and $\phi (f)
  = \sum_i \phi (g_i) = \sum_i h_i$. Thus we only have to prove that $h_i \in
  \cS^i$ for all $i$. We note that
  \begin{displaymath}
    \cS^i_1 = \phi ( R_{d-1} (g_i) ) = (\phi R_{d-1}) (\phi g_i) = R_{d-1}
    (h_i).
  \end{displaymath}
  Therefore, for all $j \notin \J_i$ and $D \in R_{d-1}$, we have $\p_j D
  (h_i) \in \p_j (\cS^i_1) = 0$. This implies that $\p_j h_i = 0$ for all $i$
  and $j \notin \J_i$, and we are done.

% cf. remark \ref{rem:native}

  For the converse, we immediately get $f = \sum_{i=1}^n g_i$ with $g_i =
  \phi^{-1} (h_i)$. Note that $R_{d-1} (g_i) = \phi^{-1} ( R_{d-1} (h_i) )$.
  Since $R_{d-1} (h_i) \subseteq \cS^i_1$, and $\cS^1_1, \dots, \cS^n_1$
  obviously are linearly independent subspaces of $\cR_1$, so are $R_{d-1}
  (g_1), \dots, R_{d-1} (g_n)$. Thus $f = \sum_{i=1}^n g_i$ is a regular
  splitting.
\end{rem}

We will also investigate how the regular splitting property specializes. For
this purpose we give the following definition.

\begin{defn}
  \label{def:limsplit}
  Let $f \in \cR_d$. We say that $f$ has a \emph{degenerate splitting of
    length $m$} if there for some $n \ge 1$ exists an $f_t \in \cR_d [t_1,
  \dots, t_n]$ such that $f_0 = f$ and $f_t$ splits regularly $m-1$ times
%% CHANGED 2013-07-13: endret [\p_1, \dots, \p_r]_d til [x_1, \dots, x_r]^{DP}_d
  inside $\cR_d (t_1, \dots, t_n) = \K (t_1, \dots, t_n) [x_1, \dots,
  x_r]^{DP}_d$.
\end{defn}

% Is it better to write: $f_t$ splits over $\K (t_1, \dots, t_n)$ ? It would
% be better to allow splittings over $\overline{\K (t_1, \dots, t_n)}$, or
% even more general $T$ like $\kappa (\K [t_1, \dots, t_n]/P)$.

% $n=0$ is the case of regular splittings.

\begin{exmp}
  Let $f = x^\pd {d-1} y \in \K [x,y]^{DP}$, $d \ge 3$. Clearly $\ann_R f =
  (\p_y^2, \p_x^d)$. If $f$ splits regularly, then it must be in the $\GL_2$
  orbit of $g = x^\pd d + y^\pd d$, and this implies that $\ann_R g$ is in the
  $\GL_2$ orbit of $\ann_R f$. But $\ann_R (f)_2 = \langle \p_y^2 \rangle$ and
  $\ann_R (g)_2 = \langle \p_x \p_y \rangle$, hence this is impossible.

  Still, even though $f$ does not split regularly, $f$ has a degenerate
  splitting. For example, $f$ is a specialization of
  \begin{displaymath}
    f_t = \tfrac{1}{t} \bigl[ (x + t y)^\pd d - x^\pd d \bigr] = x^\pd {d-1} y
    + t x^\pd {d-2} y^\pd 2 + \dots,
  \end{displaymath}
  and $f_t$ splits inside $\K (t) [x,y]^{DP}$.
\end{exmp}

\section{The associated algebra $M_f$} \label{section:Mf}

The starting point of this section is the definition of a regular splitting.
We will see how this naturally leads to the definition of a $\K$-vector space
$M_f$ associated to $f \in \cR_d$. $M_f$ consists of $r \times r$-matrices,
and we prove that $M_f$ is closed under multiplication when $d = \deg f \ge
3$. We start with a fundamental observation.

\begin{lem} \label{lem:basic}
  Let $f = g_1 + \dots + g_n$ be a regular splitting of some $f \in \cR_d$.
  Then
  \begin{displaymath}
    \ann_R (f)_e = \isect_{i=1}^n \ann_R (g_i)_e \text{ for all } e<d.
  \end{displaymath}
\end{lem}

% It would be enough to prove that $\ann(f)_{d-1} \subseteq \ann(g_i)_{d-1}$.
% Lemma \ref{lem:ann} gives the rest by ``saturation''.

\begin{proof}
  Let $e<d$. Lemma \ref{lem:apolar} gives the inclusion $\ann_R (f)_e
  \supseteq \isect_{i=1}^n \ann_R (g_i)_e$, so we only need to prove that
  $\ann_R (f)_e \subseteq \ann_R (g_i)_e$ for all $i$. To do this, let $D \in
  \ann_R (f)_e$. Applying $D$ to $f = \sum_{i=1}^n g_i$ gives $D(g_1) + \dots
  + D(g_n) = 0$. Since $D(g_1), \dots, D(g_n)$ are homogeneous polynomials of
  positive degree in separate rings, this implies $D(g_i) = 0$ for all $i$.
  Indeed, $D(g_i) = - \sum_{j \ne i} D(g_j)$ is an element of both $R_e (g_i)$
  and $\sum_{j \ne i} R_e (g_j)$, and since their intersection is zero by
  remark \ref{rem:separate}, it follows that $D(g_i) = 0$. This proves that
  $\ann_R (f)_e \subseteq \ann_R (g_i)_e$ for all $i$ and all $e<d$, and we
  are done.
\end{proof}

At first sight, one might think that there exist additional regular splittings
of a polynomial $f \in \cR_d$ if we allow ``dummy'' variables, i.e. if $\ann_R
(f)_1 \ne 0$. But it is not so when $d \ge 2$, as we prove next. For this
reason, we may freely assume $\ann_R (f)_1 = 0$ when studying regular
splittings.

\begin{cor} \label{cor:basic}
  Let $d \ge 2$ and $f \in \cR_d$. Every regular splitting of $f$ takes place
  inside the subring $\K [ R_{d-1} (f) ]^{DP} \subseteq \cR$.
\end{cor}

\begin{proof}
  Let $f = g_1 + \dots + g_n$ be a regular splitting of $f$. By remark
  \ref{rem:native}, $g_i \in \K [ R_{d-1} (g_i) ]^{DP}$. Lemma \ref{lem:basic}
  tells us that $\ann_R (f)_1 \subseteq \ann_R (g_i)_1$, and by duality (lemma
  \ref{lem:apolar}) we get $R_{d-1} (g_i) \subseteq R_{d-1} (f)$. Thus each
  additive component is an element of $\K [ R_{d-1} (f) ]^{DP}$.
\end{proof}

% Put differently, if $f = g_1 + \dots + g_n$ is any splitting of $f$ with $d
% \ge 2$, then in particular $\ann (f)_1 \subseteq \ann (g_i)_1$ for all $i$.
% If $\dim_\K \ann_R (f)_1 = s$, then after a base change we may assume that
% $f \in \cS := \K [ x_1, \dots, x_{r-s} ]$, and any splitting of $f$ inside
% $\cR$ actually happens inside $\cS$.

\begin{rem} \label{rem:basic}
  Let $f = g_1 + \dots + g_n$ be a regular splitting of $f \in \cR_d$. Lemma
  \ref{lem:basic} tells us that $\ann_R (f)_e = \isect_{i=1}^n \ann_R (g_i)_e$
  for all $e<d$. Using duality, this is equivalent to $R_{d-e} (f) =
  \sum_{i=1}^n R_{d-e} (g_i)$ for all $e<d$. In particular, we have $R_{d-1}
  (f) = R_{d-1} (g_1) + \dots + R_{d-1} (g_n)$ when $d \ge 2$.

  Let $\cS = \K [ R_{d-1} (f) ]^{DP}$ and $\cS^i = \K [ R_{d-1} (g_i) ]^{DP}$
  for $i = 1, \dots, n$. Since $R_{d-1} (g_i) \isect \bigl( \sum_{j \ne i}
  R_{d-1} (g_j) \bigr) = 0$ and $\sum_i R_{d-1} (g_i) = R_{d-1} (f)$, we get
  \begin{displaymath}
    \cS^1 \otimes_k \dots \otimes_k \cS^n = \cS \subseteq \cR.
  \end{displaymath}
  Obviously, $f \in \cS^1_d \dsum \dots \dsum \cS^n_d$. Hence we have another
  characterization of a regular splitting: An $f \in \cR_d$ splits regularly
  $n-1$ times if and only if there exist non-zero $\K$-vector subspaces $V_1,
  \dots, V_n \subseteq \cR_1$ such that $V_i \isect \bigl( \sum_{j \ne i} V_j
  \bigr) = 0$ for all $i$ and $\sum_{i=1}^n V_i = R_{d-1} (f)$, and $f \in
  \cS^1_d \dsum \dots \dsum \cS^n_d$ where $\cS^i = \K [ V_i ]^{DP}$.
\end{rem}

% Just different focus: on $\{ \cS^i \}$ instead of $\{ g_i \}$.
% The requirement above that $\sum_{i=1}^n V_i = R_{d-1} (f)$ implies that
% every component $f|_{\cS^i}$ is non-zero.

By lemma \ref{lem:basic}, if we want to split an $f \in \cR_d$, we have to
look for $g \in \cR_d$ such that $\ann (f)_e \subseteq \ann (g)_e$ for all
$e<d$. The next lemma investigates this relationship. Recall that $\p$ denotes
the column vector $\p = [\p_1, \dots, \p_r]^\T$, thus $\p f = [\p_1 f, \dots,
\p_r f]^\T$.

\begin{lem} \label{lem:main1}
  Given $f,g \in \cR_d$, the following are equivalent:
  \begin{enumerate}
    \setlength{\itemsep}{2pt}
    \setlength{\parskip}{0pt}
    \renewcommand{\theenumi}{\alph{enumi}}
    \renewcommand{\labelenumi}{\normalfont(\theenumi)}
  \item $\ann (f)_e \subseteq \ann (g)_e$ for all $e<d$,
  \item $\ann (f)_{d-1} \subseteq \ann (g)_{d-1}$,
  \item there exists a matrix $A \in \Mat_\K (r,r)$ such that $\p g = A \p f$,
  \item $R_1 \cdot \ann (f)_{d-1} \subseteq \ann (g)_d$,
  \item $\m \cdot \ann (f) \subseteq \ann (g)$.
  \end{enumerate}
\end{lem}

\begin{proof}
  (a) $\Leftrightarrow$ (b) is immediate by lemma \ref{lem:ann}. The same
  lemma also tells us that $\isect_{i=1}^r \ann (\p_i f) = \ann (f) + R_d$,
  which means that (b) just says that
  \begin{displaymath}
    \isect_{i=1}^r \ann (\p_i f)_{d-1} \subseteq \isect_{i=1}^r \ann (\p_i
    g)_{d-1}.
  \end{displaymath}
  By lemma \ref{lem:apolar}, this is equivalent to $\langle \p_1 g, \dots \p_r
  g \rangle \subseteq \langle \p_1 f, \dots \p_r f \rangle$, and (c) just
  expresses this in vector form. (b) $\Leftrightarrow$ (d) since $R_1^{-1}
  \ann_R (g)_d = \ann_R (g)_{d-1}$, again by lemma \ref{lem:ann}a. Finally,
  lemma \ref{lem:ann}a also shows that (d) $\Leftrightarrow$ (e), since $(\m
  \cdot \ann (f))_e = \sum_k \m_k \cdot \ann (f)_{e-k} = R_1 \cdot \ann
  (f)_{e-1}$.
\end{proof}

% Note that (c) is equivalent to the corresponding relationship between the
% catalecticant matrices of $f$ and $g$, $\Cat_{d-1,1} (g) = A \Cat_{d-1,1}
% (f)$.

Let $f \in \cR_d$. Both the previous lemma and the next lemma study the
equation $\p g = A \p f$. In the previous we gave equivalent conditions on $g
\in \cR_d$ for $A \in \Mat_\K (r,r)$ to exist. The next lemma tells us when
%% CHANGED 2013-07-13: is -> if
$g$ exists given $A$. Recall that if $B$ is any matrix, then $I_k (B)$ denotes
the ideal generated by all $k \times k$-minors of $B$.

\begin{lem} \label{lem:main2}
  Let $f \in \cR_d$ and $A \in \Mat_\K (r,r)$. The following are equivalent:
  \begin{enumerate}
    \setlength{\itemsep}{2pt}
    \setlength{\parskip}{0pt}
    \renewcommand{\theenumi}{\alph{enumi}}
    \renewcommand{\labelenumi}{\normalfont(\theenumi)}
  \item There exists $g \in \cR_d$ such that $\p g = A \p f$,
  \item $A \p\p^\T (f)$ is a symmetric matrix,
  \item $I_2 (\p\:A\p) \subseteq \ann f$.
  \end{enumerate}
  Furthermore, if $d>0$, then a $g \in \cR_d$ satisfying $\p g = A \p f$ is
  necessarily unique.
\end{lem}

\begin{proof}
  It is well known that a set $\{g_i\}_{i=1}^r$ can be lifted to a $g$ such
  that $\p_i g = g_i$ if and only if $\p_j g_i = \p_i g_j$ for all $i,j$. This
  condition simply says that $\p [g_1, \dots, g_r]$ is a symmetric matrix.
  Let $g_i = (A \p f)_i$, that is, $g_i$ is the \te{i} coordinate of the
  column vector $A \p f$. Then the existence of $g$ is equivalent to $A
  \p\p^\T f$ being a symmetric matrix. Thus (a) $\Leftrightarrow$ (b).
  
  Since $(A \p\p^\T)^\T = \p\p^\T A^\T$, it follows that $A \p\p^\T (f)$ is
  symmetric if and only if $(A \p\p^\T - \p\p^\T A^\T) (f)=0$. Thus (b) $\iff$
  (c), since the \te{(i,j)} entry of the matrix $(A\p)\p^\T - \p(A\p)^\T$ is
  $(A\p)_i \p_j - \p_i (A\p)_j$, the $2 \times 2$ minor of the $2 \times r$
  matrix $(\p\:A\p)$ corresponding to the \te{i} and \te{j} row (up to sign).
  The last statement is trivial.
\end{proof}

% In other words: $I_1 (A \p\p^\T - \p\p^\T A^\T) = I_2 (\p\:A\p)$.

Note that the $2 \times 2$ minors of $(\p\:A\p)$ are elements of $R_2$, so (c)
is really a condition on $\ann (f)_2$. Combining lemma \ref{lem:basic} with
lemmas \ref{lem:main1} and \ref{lem:main2}, we see that a regular splitting $f
= g_1 + \dots, g_n$ implies the existence of matrices $A$ satisfying $I_2
(\p\: A\p) \subseteq \ann_R f$. These matrices will in fact enable us to find
both regular and degenerate splittings. Thus we are naturally lead to the
following definition.

\begin{defn} \label{def:Mf}
  Given $f \in \cR_d$, define
  \begin{displaymath}
    M_f = \{ A \in \Mat_\K (r,r) \suchthat I_2 (\p\: A\p) \subseteq \ann_R f
    \}.
  \end{displaymath}
\end{defn}

\begin{exmp} \label{ex:r2d3a} % r=2, d=3, part a
  The notation $I_2 (\p\:A\p)$ might be confusing, so we will consider an
  example with $r=2$. Let $\cR = \K [x,y]^{DP}$ and $f = x^\pd 3 + xy^\pd 2
  \in \cR_3$. A quick calculation of the partials of $f$ proves that $\ann_R f
  = ( \p_x^2 - \p_y^2, \p_y^3 )$. We will show that the $2 \times 2$ matrix
  \begin{math}
    A = \bigl(
    \begin{smallmatrix}
      0 & 1 \\ 1 & 0
    \end{smallmatrix}
    \bigr)
  \end{math}
  belongs to $M_f$. Obviously,
  \begin{displaymath}
    A \p =
    \begin{pmatrix}
      0 & 1 \\ 1 & 0
    \end{pmatrix}
    \cdot
    \begin{pmatrix}
      \p_x \\ \p_y
    \end{pmatrix}
    =
    \begin{pmatrix}
      \p_y \\ \p_x
    \end{pmatrix}\!.
  \end{displaymath}
  The matrix $(\p\:A\p)$ has $\p$ as its first column and $A\p$ as its second
  column, so
  \begin{displaymath}
    (\p\:A\p) =
      \begin{pmatrix}
        \p_x & \p_y \\
        \p_y & \p_x
      \end{pmatrix}\!.
  \end{displaymath}
  Its only $2 \times 2$ minor is its determinant, $D = \p_x^2 - \p_y^2$, and
  since $D \in \ann_R f$, it follows by definition that $A \in M_f$.

  Let us determine $M_f$. We need to find all matrices
  \begin{math}
    B = \bigl(
    \begin{smallmatrix}
      a & b \\ c & d
    \end{smallmatrix}
    \bigr)
  \end{math}
  such that $I_2 (\p\:B\p) \subseteq \ann_R f$. Since
  \begin{displaymath}
    (\p\:B\p) =
      \begin{pmatrix}
        \p_x & a\p_x + b\p_y \\
        \p_y & c\p_x + d\p_y
      \end{pmatrix}\!,
  \end{displaymath}
  we get $I_2 (\p\:B\p) = \bigl( c \p_x^2 + (d-a) \p_x \p_y - b \p_y^2
  \bigr)$. Hence $\p_x^2 - \p_y^2$ must divide $c \p_x^2 + (d-a) \p_x \p_y - b
  \p_y^2$, which is equivalent to $a=d$ and $b=c$. Therefore, $M_f$ consists
  of all matrices $B$ with $a=d$ and $b=c$, that is, $M_f = \langle I,A
  \rangle$.
\end{exmp}

Almost everything that we are going to study in this paper is connected to
$M_f$. One goal is to find out what $M_f$ can tell us about $f$. Before we can
do this, we need investigate what properties $M_f$ itself possesses. We will
in particular show that $M_f$ is closed under matrix multiplication when $\deg
f \ge 3$. Obviously, $d \le 1$ implies $\ann_R (f)_2 = R_2$, and therefore
$M_f = \Mat_\K (r,r)$. The case $d=2$ is different, and not all of our results
will apply to this case. We start with another definition.

\begin{defn} \label{def:gamma_f}
  Suppose $d>0$ and $f \in \cR_d$. Define a map
  \begin{displaymath}
    \gamma_f : M_f \to \cR_d
  \end{displaymath}
  by sending $A \in M_f$ to the unique $g \in \cR_d$ satisfying $\p g = A \p
  f$, cf. lemma \ref{lem:main2}.
\end{defn}

Note that $\p \gamma_f(A) = A \p f$ by definition. If $\chr \K \nmid d$, then
the Euler identity ($x^\T \p f = d \mspace{1mu} f$) implies that $\gamma_f (A)
= \tfrac{1}{d} x^\T A \p f$. By lemmas \ref{lem:basic} and \ref{lem:main1},
the image of $\gamma_f$ contains in particular all additive components of $f$.
We will in chapter \ref{chapter:regular} see how to extract the regular
splitting properties of $f$ from $M_f$ explicitly.

\begin{lem} \label{lem:Mf}
  Let $d>0$ and $f \in \cR_d$, $f \ne 0$. Let $\beta_{1e}$ be the minimal
  number of generators of $\ann_R (f)$ of degree $e$.
  \begin{enumerate}
    \setlength{\itemsep}{2pt}
    \setlength{\parskip}{0pt}
    \renewcommand{\theenumi}{\alph{enumi}}
    \renewcommand{\labelenumi}{\normalfont(\theenumi)}
  \item $M_f$ is a $\K$-vector space containing the identity matrix $I$.
  \item $\gamma_f : M_f \to \cR_d$ is $\K$-linear.
% \item
%   \begin{math}
%     \ker \gamma_f = \{ A \in M_f \suchthat A \p f = 0 \} = \{ A \in \Mat_\K
%     (r,r) \suchthat (A \p)_i \in \ann_R (f)_1 \:\forall\: i \}.
%   \end{math}
  \item $\dim_\K \ker \gamma_f = r \cdot \beta_{11}$ and $\dim_\K \im \gamma_f
    = 1 + \beta_{1d}$.
  \item $\dim_\K M_f = 1 + \beta_{1d} + r \cdot \beta_{11}$.
  \end{enumerate}
\end{lem}

\begin{proof}
  Obviously, $I \in M_f$, so $M_f$ is nonempty. And since the determinant is
  linear in each column, it follows that $M_f$ is a $\K$-vector space.
  Alternatively, let $A,B \in M_f$. Since $\p \gamma_f (A) = A \p f$, it
  follows for any $a,b \in \K$ that
  \begin{displaymath}
    \p \bigl( a \gamma_f (A) + b \gamma_f (B) \bigr) = a \p \gamma_f (A) + b
    \p \gamma_f (B) = (aA+bB) \p f.
  \end{displaymath}
  This implies that $aA+bB \in M_f$ for all $a,b \in \K$, which proves (a),
  and furthermore that $\gamma_f (aA+bB) = a \gamma_f (A) + b \gamma_f (B)$,
  thus $\gamma_f$ is $\K$-linear.
  
  Of course, $\gamma_f (A) = 0$ if and only if $A \p f = 0$. For any $A \in
  \Mat_\K (r,r)$, the equation $A \p f = 0$ implies that $A \in M_f$, hence
  the kernel of $\gamma_f$ consists of all such $A$. Recall that $(A\p)_i$
  denotes the \te{i} coordinate of the column vector $A\p$, that is, $(A\p)_i
  = a_i^\T \p$ where $a_i^\T$ is the \te{i} row of $A$. Thus
  \begin{displaymath}
    \ker \gamma_f = \{ A \in M_f \suchthat A \p f = 0 \} = \{ A \in \Mat_\K
    (r,r) \suchthat (A \p)_i \in \ann_R (f)_1 \:\forall\: i \},
  \end{displaymath}
  and therefore $\dim_\K \ker \gamma_f = r \cdot \dim_\K \ann (f)_1 = r
  \beta_{11}$.

  Furthermore, by lemma \ref{lem:main1}, the image of $\gamma_f$ are precisely
  those $g \in \cR_d$ that satisfy $R_1 \cdot \ann (f)_{d-1} \subseteq \ann
  (g)_d$, which is equivalent to $\langle g \rangle \subseteq (R_1 \cdot \ann
  (f)_{d-1})^\perp$ by lemma \ref{lem:apolar}. Since $\dim_\K (R/ \ann (f))_d
  = 1$, and $R_1 \cdot \ann (f)_{d-1}$ is a subspace of $\ann (f)_d$ of
  codimension $\dim_\K (\ann (f)_d / R_1 \cdot \ann (f)_{d-1}) = \beta_{1d}$,
  it follows that $\dim_\K \im \gamma_f = \codim_\K (\m \cdot \ann (f))_d = 1
  + \beta_{1d}$. This finishes part (c). (d) follows immediately.
\end{proof}

% $\im \gamma_f = \{ \gamma_f (A) \suchthat A \in M_f \} = \{ g \in \cR_d
% \suchthat \p g = A \p f \text{ for some } A \}$.

% Must at least have $\K$ infinite in the next remark. And it should be true
% also for $d=3$, but I have yet no proof of it.

\begin{rem}
  We would like to point out that $M_f$ is ``large'' only for special $f$. In
  fact, when $\K = \bar{\K}$ and $d \ge 4$, a general $f \in \cR_d$ will
  satisfy $\beta_{11} = \beta_{1d} = 0$ (see for example \cite[Proposition
  3.12]{IK}), which implies $M_f = \langle I \rangle$. In particular, $M_f =
  M_g$ does not say very much by itself.
\end{rem}

\begin{exmp} \label{ex:r2d3b}
  Let us reconsider example \ref{ex:r2d3a}. Since $\ann_R f = ( \p_x^2 -
  \p_y^2, \p_y^3 )$, we see that $\beta_{11} = 0$ and $\beta_{13} = 1$. Lemma
  \ref{lem:Mf} implies that $\dim_\K M_f = 1 + 1 = 2$. As before,
  \begin{math}
    A = \bigl(
    \begin{smallmatrix}
      0 & 1 \\ 1 & 0
    \end{smallmatrix}
    \bigr) \in M_f
  \end{math}
  since $I_2 (\p\: A\p) = ( \p_x^2 - \p_y^2 ) \subseteq \ann_R f$. It follows
  that $M_f = \langle I,A \rangle$.

  Let us also determine $\im \gamma_f$. Letting $g = x^\pd 2 y + y^\pd 3 \in
  \cR_3$, we see that
  \begin{displaymath}
    \p g =
    \begin{pmatrix}
      xy \\ x^\pd 2 + y^\pd 2
    \end{pmatrix}
    = A \p f.
  \end{displaymath}
  Thus $A \in M_f$ and $\gamma_f (A) = g$. Obviously, $\gamma_f (I) = f$,
  hence $\im \gamma_f = \langle f,g \rangle$. This image consists of all $h
  \in \cR_3$ such that $\ann_R (f)_2 \subseteq \ann_R (h)_2$. Thus another way
  to compute $\im \gamma_f$ is $(R_1 \cdot \ann_R (f)_2)^\perp = \langle
  \p_x^3 - \p_x \p_y^2, \p_x^2 \p_y - \p_y^3 \rangle^\perp = \langle f,g
  \rangle$.
\end{exmp}

\begin{rem} \label{rem:compute}
  Before we move on, we would like to point out that there are several ways to
  compute $M_f$. One is to use the definition directly and find all $A \in
  \Mat_\K (r,r)$ such that every $2 \times 2$ minor of $(\p\: A\p)$ is
  contained in $\ann_R (f)_2$. This can be effective when $\ann_R (f)_2$ is
  simple enough, as in example \ref{ex:r2d3a}. In particular, if $\dim_\K
  \ann_R (f)_2 < r-1$, then $M_f = \langle I \rangle$. Another direct approach
  is to solve the system of linear equations that is contained in the
  statement ``$A \p\p^\T f$ is symmetric''. We will do this when we prove
  proposition \ref{prop:counterex}.
  
  Alternatively, we can find $\dim_\K M_f$ by computing $\ann_R f$ and
  counting the number of generators of degree $d$, and then explicitly find
  the correct number of linearly independent matrices $A$ satisfying $I_2
  (\p\: A\p) \subseteq \ann_R (f)$. In fact, most examples in this paper are
  constructed by first choosing $M \subseteq \Mat_\K (r,r)$ and then finding
  $f \in \cR$ such that $M \subseteq M_f$. Having done so, if we thereafter
  are able to show that $\ann_R f$ has no generators of degree 1 and $\dim_\K
  M - 1$ generators of degree $d$, then it follows that $M_f = M$.
\end{rem}

Note in particular that the $M_f$ in example \ref{ex:r2d3b} is closed under
matrix multiplication. This is in fact always true when $\deg f \ge 3$. We
will now prove this important and a bit surprising fact about $M_f$.

\begin{prop} \label{prop:Mf}
  Let $d \ge 3$ and $f \in \cR_d$. $M_f$ is a $\K$-algebra, and all
  commutators belong to $\ker \gamma_f$. In particular, $M_f$ is commutative
  if $\ann (f)_1 = 0$.
\end{prop}

\begin{proof}
  We use lemmas \ref{lem:main1} and \ref{lem:main2} several times. Let $A,B
  \in M_f$. Since $B \in M_f$, there exists $g \in \cR_d$ such that $\p g = B
  \p f$. Now $I_2 (\p\: A\p) \subseteq R \ann (f)_2$, and $\ann (f)_2
  \subseteq \ann (g)_2$ since $d \ge 3$. Hence $A \in M_g$, and there exists
  $h \in \cR_d$ such that $\p h = A \p g$. Then $\p h = AB \p f$, thus $AB \in
  M_f$. Furthermore, since $A \p\p^\T (f)$, $B \p\p^\T (f)$ and $AB \p\p^\T
  (f)$ are all symmetric, we get
  \begin{displaymath}
    AB \p\p^\T (f) = \p\p^\T (f) (AB)^\T = \p\p^\T (f) B^\T A^\T = B \p\p^\T
    (f) A^\T = BA \p\p^\T (f).
  \end{displaymath}
  Hence $(AB-BA) \p\p^\T f = 0$. Note that $C \p\p^\T f = 0 \iff (C\p)_i \p_j
  f = 0$ for all $i,j \iff C \p f = 0$. Thus $(AB-BA) \p f = 0$, and therefore
  $\gamma_f (AB-BA) = 0$. If $\ann (f)_1 = 0$, then it follows that $AB=BA$.
\end{proof}

\begin{rem}
  When $d \ge 3$ it also follows for all $A,B \in M_f$ that
  \begin{displaymath}
    A\p\p^\T(f)B^\T = AB\p\p^\T(f) = \p\p^\T(f)B^\T A^\T = B\p\p^\T(f)A^\T.
  \end{displaymath}
  Thus $(A\p)(B\p)^\T (f)$ is symmetric, which implies that $I_2 (A\p \: B\p)
  \subseteq \ann f$, cf. lemma \ref{lem:main2}.
\end{rem}

\begin{exmp} \label{ex:A3=0}
  Let $r=3$, $d \ge 3$ and $f = x_1^\pd {d-1} x_3 + x_1^\pd {d-2} x_2^\pd 2$.
  First, let us determine $\ann_R f$. Clearly, $\ann_R (f)_1 = 0$, and a
  straightforward computation shows that $\ann_R (f)_2 = \langle \p_3^2, \p_2
  \p_3, \p_1 \p_3 - \p_2^2 \rangle$. We note that these polynomials are the
  maximal minors of
  \begin{displaymath}
    \begin{pmatrix}
      \p_1 & \p_2 & \p_3 \\
      \p_2 & \p_3 & 0
    \end{pmatrix}\!.
  \end{displaymath}
  
  By Hilbert-Burch the ideal $J = R \ann_R (f)_2$ defines a scheme of length 3
  in $\bP^2$. Indeed, $\p_2^3 = \p_1 (\p_2 \p_3) - \p_2 (\p_1 \p_3 - \p_2^2)
  \in J$, and this implies for every $e \ge 2$ that $(R/J)_e$ is spanned by
  (the images of) $\p_1^e$, $\p_1^{e-1} \p_2$ and $\p_1^{e-2} \p_2^2$. Since
  $\p_1^e (f)$, $\p_1^{e-1} \p_2 (f)$ and $\p_1^{e-2} \p_2^2 (f)$ are linearly
  independent for all $2 \le e < d$, it follows that $\dim_\K (R/J)_e = 3$ for
  all $e>1$, and that $\ann_R (f)_e = J_e$ for all $1<e<d$. Thus $\ann_R f$
  needs exactly two generators of degree $d$, and we get
  \begin{displaymath}
    \ann_R f = (\p_3^2, \p_2 \p_3, \p_1 \p_3 - \p_2^2, \p_1^{d-1} \p_2,
    \p_1^d).
  \end{displaymath}

  Let
  \begin{displaymath}
    A = \left(
      \begin{smallmatrix}
        0 & 1 & 0 \\
        0 & 0 & 1 \\
        0 & 0 & 0
      \end{smallmatrix}
    \right).
  \end{displaymath}
  We have just seen that $\ann_R (f)_2$ is generated by the $2 \times 2$
  minors of $(\p\: A\p)$, hence $A \in M_f$. Because $M_f$ is closed under
  multiplication, we also have $A^2 \in M_f$. By looking at $\ann_R f$, we see
  that $\beta_{11} = 0$ and $\beta_{1d} = 2$. Thus $\dim_\K M_f = 3$, and it
  follows that $M_f = \langle I, A, A^2 \rangle$.
\end{exmp}

\begin{rem}
  The ``formula'' for the annihilator ideal $\ann_R f$ in example
  \ref{ex:A3=0} is true even for $d=2$. In this case $\ann_R f$ has five
  generators of degree 2, thus $M_f$ will be 6-dimensional. In fact, since in
  this case
  \begin{displaymath}
    \p\p^\T f = \left(
      \begin{smallmatrix}
        0 & 0 & 1 \\
        0 & 1 & 0 \\
        1 & 0 & 0
      \end{smallmatrix}
    \right)\!,
  \end{displaymath}
  it follows that $M_f$ consists of all matrices that are symmetric about the
  anti-diagonal. Thus $M_f$ is no longer closed under multiplication.
\end{rem}

\begin{rem}
  With $A$ as in example \ref{ex:A3=0}, it is easy to determine all $g \in
  \cR_d$ such that $A \in M_g$. Indeed, if $I_2 (\p\: A\p) \subseteq \ann_R g$
  for some $g \in \cR_d$, then $\ann_R (g)_e \supseteq \ann_R (f)_e$ for all
  $e<d$ since the $2 \times 2$ minors of $(\p\: A\p)$ are the only generators
  of $\ann_R f$ of degree less than $d$. It follows that
  \begin{displaymath}
    \{ g \in \cR_d \suchthat A \in M_g \} = \im \gamma_f = \{ a f + b x_1^\pd
    {d-1} x_2 + c x_1^\pd d \suchthat a,b,c \in \K \}.
  \end{displaymath}
  If in addition $\ann_R (g)_1 = 0$, then $a \ne 0$, implying that $g$ is in
  the $\GL_3$ orbit of $f$ ($\chr \K \nmid d$).
\end{rem}

%\begin{que}
%  Which subalgebras of $\Mat_\K (r,r)$ arise as $M_f$ for different $f \in
%  \cR_d$?
%\end{que}

One natural question to ask is the following:
\begin{quote}
  Which subalgebras of $\Mat_\K (r,r)$ arise as $M_f$ for different $f \in
  \cR_d$?
\end{quote}
We have not been able to determine this in general, but we will in the
remainder of this chapter point out some restrictions on $M_f$. We start with
the following result, which holds even for $d<3$.

% We have not found a simple characterization of which algebras that arise as
% $M_f$ for different $f \in \cR_d$, $d \ge 3$. Nevertheless, there are some
% things we can say.

\begin{prop} \label{prop:Mexpo}
  Suppose $d \ge 0$ and $f \in \cR_d$. Let $A,B \in \Mat_\K (r,r)$ and $C \in
  M_f$. Assume that $AC, BC \in M_f$ and $BAC = ABC$. Then $A^i B^j C \in M_f$
  for all $i,j \ge 0$. In particular, $M_f$ is always closed under
  exponentiation.
\end{prop}

\begin{proof}
  Lemma \ref{lem:main2} says that $A \in M_f$ if and only if $A \p\p^\T f$ is
  symmetric. Thus all three matrices $C \p\p^\T f$, $AC \p\p^\T f$ and $BC
  \p\p^\T f$ are symmetric. It follows that
  \begin{displaymath}
    ABC \p\p^\T f = A \p\p^\T f C^\T B^\T = AC \p\p^\T f B^\T = \p\p^\T f C^\T
    A^\T B^\T = \p\p^\T f (ABC)^\T,
  \end{displaymath}
  hence $ABC \in M_f$, and we are done by induction. The last statement
  follows by letting $B=C=I$. Note that we have not assumed $d \ge 3$ here.
\end{proof}

When $d \ge 3$ one might wonder if the assumptions $C, AC \in M_f$ actually
implies that $A \in M_f$. If so, the conclusion of the previous proposition
would immediately follow from the fact that $M_f$ is closed under
multiplication when $d \ge 3$. But $M_f$ does not support division, in the
sense that $C, AC \in M_f$ does not generally imply $A \in M_f$, as seen in
the following example.

\begin{exmp}
  Let $r=4$ and $f = x_1^\pd {d-1} x_4 + x_1^\pd {d-2} x_2 x_3 + x_2^\pd d$.
  Then
  \begin{displaymath}
    \ann_R f = (\p_1 \p_4 - \p_2 \p_3, \p_2 \p_4, \p_3^2, \p_3 \p_4, \p_4^2,
    \p_1 \p_2^2, \p_1^{d-2} \p_3 - \p_2^{d-1}, \p_1^d, \p_1^{d-1} \p_2).
  \end{displaymath}
  This implies that $\dim_\K M_f = 3$ when $d \ge 4$. Let
  \begin{displaymath}
    A = \left(
      \begin{smallmatrix}
        0 & 1 & 0 & 0 \\
        0 & 0 & 1 & 0 \\
        0 & 0 & 0 & 1 \\
        0 & 0 & 0 & 0
      \end{smallmatrix}
    \right).
  \end{displaymath}
  It is easy to verify that $A^2, A^3 \in M_f$, thus $M_f = \langle I, A^2,
  A^3 \rangle$ when $d \ge 4$. In particular, $A \notin M_f$, even though
  $A^2, A^3 \in M_f$.
\end{exmp}

% Above: $\ann_R f$ has three generators of degree $d$ when $d=3$. In this
% case, $M_f = \langle I, A, A^2, A^3 \rangle = \K [A]$.

We will finish this section with a result computing some special elements of
$\ann_R f$. We start with a lemma.

\begin{lem} \label{lem:UV}
  Let $d \ge 0$ and $f \in \cR_d$. Pick $A_1, \dots, A_m, B_1, \dots, B_n \in
  M_f$, and let $u \in \sum_{k=1}^m \im A_k^\T + \sum_{k=1}^n \ker B_k^\T$ and
  $v \in \bigl( \isect_{k=1}^m \ker A_k^\T \bigr) \isect \bigl( \isect_{k=1}^n
  \im B_k^\T \bigr)$. Then
  \begin{displaymath}
    (u^\T \p) \cdot (v^\T \p) \in \ann_R f.
  \end{displaymath}
\end{lem}

\begin{proof}
  The proof rests on the following equation. If $A \in M_f$ and $b = [b_1,
  \dots, b_r]^\T$ and $c = [c_1, \dots, c_r]^\T$ are two vectors, then
  \begin{displaymath}
    \begin{vmatrix}
      b^\T \p & b^\T A \p \\
      c^\T \p & c^\T A \p
    \end{vmatrix}
    =
    \begin{vmatrix}
      \sum_i b_i \p_i & \sum_i b_i (A \p)_i \\
      \sum_j c_j \p_j & \sum_j c_j (A \p)_j
    \end{vmatrix}
    = \sum_{i,j=1}^r b_i c_j
    \begin{vmatrix}
      \p_i & (A \p)_i \\
      \p_j & (A \p)_j
    \end{vmatrix},
  \end{displaymath}
  and therefore
  \begin{equation} \label{eq:UV}
    (b^\T \p) \cdot (c^\T A \p) - (b^\T A \p) \cdot (c^\T \p) \in \ann_R f.
  \end{equation}
  
  By definition of $u$ there exist $a_1, \dots a_m, b_1, \dots b_n \in \K^r$
  such that $B_k^\T b_k = 0$ and $u = \sum_{k=1}^m A_k^\T a_k + \sum_{k=1}^n
  b_k$. Furthermore, $A_k^\T v = 0$ and $v = B_1^\T c_1 = \dots = B_n^\T c_n$
  for some $c_1, \dots, c_n \in \K^r$. Putting $(A,b,c) = (A_k,a_k,v)$ in
  \eqref{eq:UV}, and using $A_k^\T v = 0$, implies $(a_k^\T A_k \p) (v^\T \p)
  \in \ann_R f$. Letting $(A,b,c) = (B_k,b_k,c_k)$ gives $(b_k^\T \p) (v^\T
  \p) \in \ann_R f$ since $B_k^\T b_k = 0$ and $B_k^\T c_k = v$. Adding these
  equations together proves that $(u^\T \p) \cdot (v^\T \p) \in \ann_R f$.
\end{proof}

The next proposition gives us a restriction on $M_f$ when $\ann_R (f)_1 = 0$.
We will use this in chapter \ref{chapter:limits}.

\begin{prop} \label{prop:UV}
  Let $d \ge 2$ and $f \in \cR_d$. Pick $A_1, \dots, A_m, B_1, \dots, B_n \in
  M_f$, and define
  \begin{displaymath}
    U = \sum_{k=1}^m \im A_k^\T + \sum_{k=1}^n \ker B_k^\T \quad \text{ and }
    \quad V = \biggl( \bigcap_{k=1}^m \ker A_k^\T \biggr) \isect \biggl(
    \bigcap_{k=1}^n \im B_k^\T \biggr).
  \end{displaymath}
  Assume that (a) $U+V = \K^r$ and $U \isect V \ne 0$, or (b) $\dim_\K U =
  r-1$ and $\dim_\K V \ge 2$. Then $\ann_R (f)_1 \ne 0$.
\end{prop}

\begin{proof}
  (a) Let $u \in U \isect V$. Since $u \in U$, lemma \ref{lem:UV} implies for
  all $v \in V$ that $(u^\T \p) \cdot (v^\T \p) \in \ann_R f$. Because $u \in
  V$, we get $(u^\T \p) \cdot (v^\T \p) \in \ann_R f$ for all $v \in U$ by the
  same lemma. Now $U+V=\K^r$ implies that $(u^\T \p) \cdot R_1 \in \ann_R f$,
  hence $(u^\T \p) \in \ann_R f$.

  (b) If $V \nsubseteq U$, then $U+V = \K^r$, and we are done by part (a).
  Thus we assume that $V \subseteq U$. Choose $u_1, u_2 \in V$, $u_1
  \nparallel u_2$. Expand this to a basis $\{ u_1, \dots, u_{r-1} \}$ for $U$,
  and choose $u_r \notin U$. Then $\{ u_1^\T \p, \dots, u_r^\T \p \}$ is a
  basis for $R_1$. Let $\{ l_1, \dots, l_r \}$ be the dual basis for $\cR_1$.
  Since $(u^\T \p) (u_1^\T \p) \in \ann_R f$ for all $u \in U$, it follows
  that $u_1^\T \p f = c_1 l_r^\pd {d-1}$ for some $c_1 \in \K$. Similarly,
  $u_2^\T \p f = c_2 l_r^\pd {d-1}$. Thus $(c_2 u_1 - c_1 u_2)\!^\T \p f = 0$,
  and $\ann_R (f)_1 \ne 0$.
\end{proof}

% $f = (c_1 l_1 + c_2 l_2) l_r^{d-1} + g (l_3, \dots, l_r)$

\begin{exmp} \label{ex:UV}
  We will give an example of each of the two cases of proposition
  \ref{prop:UV}. In both cases, let $r=3$, $d \ge 2$ and $f \in \cR_d$.
  \begin{enumerate}
    \setlength{\itemsep}{2pt}
    \setlength{\parskip}{0pt}
  \item Let
    \begin{math}
      B_1 = \left(
        \begin{smallmatrix}
          0 & 0 & 1 \\
          0 & 0 & 0 \\
          0 & 0 & 0
        \end{smallmatrix}
      \right)
    \end{math}
    and
    \begin{math}
      B_2 = \left(
        \begin{smallmatrix}
          0 & 0 & 0 \\
          0 & 0 & 1 \\
          0 & 0 & 0
        \end{smallmatrix}
      \right)\!,
    \end{math}
    and assume that $B_1,B_2 \in M_f$. Then
    \begin{displaymath}
      \ann_R f \supseteq I_2 (\p\: B_1\p) + I_2 (\p\: B_2\p)
      = ( \p_1 \p_3, \p_2 \p_3, \p_3^2 ) = \p_3 \cdot \m_R.
    \end{displaymath}
    Hence $\p_3 \in \ann_R (f)_1$, and $\ann_R (f)_1 \ne 0$. This belongs to
    case (a) of proposition \ref{prop:UV} (with $A_i = 0$ for all $i$).

  \item Let
    \begin{math}
      A_1 = \left(
        \begin{smallmatrix}
          0 & 1 & 0 \\
          0 & 0 & 0 \\
          0 & 0 & 0
        \end{smallmatrix}
      \right)
    \end{math}
    and
    \begin{math}
      A_2 = \left(
        \begin{smallmatrix}
          0 & 0 & 1 \\
          0 & 0 & 0 \\
          0 & 0 & 0
        \end{smallmatrix}
      \right)\!,
    \end{math}
    and assume that $A_1,A_2 \in M_f$. Then
    \begin{displaymath}
      \ann_R f \supseteq I_2 (\p\: A_1\p) + I_2 (\p\: A_2\p) = ( \p_2^2, \p_2
      \p_3, \p_3^2 ) = (\p_2, \p_3)^2.
    \end{displaymath}
    Thus $f = c_1 x_1^\pd d + c_2 x_1^\pd {d-1} x_2 + c_3 x_1^\pd {d-1} x_3$,
    and therefore, $\ann_R (f)_1 \ne 0$. This is case (b) of proposition
    \ref{prop:UV} (with $B_i = I$ for all $i$).
  \end{enumerate}
\end{exmp}

$M_f$ has other properties that further restrict the subalgebras that arise as
$M_f$, and we will say a little more about this in the next section.

\section{Determinantal ideals}

We mentioned in remark \ref{rem:compute} that most examples in this paper are
constructed by first choosing a subset (usually a subalgebra) $M \subseteq
\Mat_\K (r,r)$. Having chosen $M$, we can compute $X_d = \{ f \in \cR_d
\suchthat M_f \supseteq M \}$, and finally choose one of these $f$ to present
as the example.

We now take a closer look at this method. Given a subset $M \subseteq \Mat_\K
(r,r)$, we will define an ideal $I(M)$ and an $R$-module $X(M)$. Studying
$I(M)$ and $X(M)$ can be thought of as an alternative to studying all $M_f$
that contain $M$, and we will make this connection precise. However, the first
half of this section will only deal with $I(M)$ and a related ideal $\check
I(M)$.

\begin{defn} \label{def:IM}
  Let $M$ be any subset of $\Mat_\K (r,r)$. Let $I(M)$ and $\check I(M)$ be
  the ideals in $R$ defined by
  \begin{displaymath}
    I(M) = \sum_{A \in M} I_2 (\p\: A\p) \quad \text{and} \quad
    \check I(M) = \sum_{A,B \in M} I_2 (A\p\: B\p).
  \end{displaymath}
\end{defn}

Note that the ideal $I(M_f)$ is the part of $\ann_R f$ that determines $M_f$.
Obviously, if $M$ is a $\K$-vector space, and $A_1, \dots, A_n$ is a basis for
$M$, then
\begin{displaymath}
  \check I(M) = \sum_{i<j} I_2 (A_i\p\: A_j\p) = I_2 (A_1\p\: A_2\p\: \dots\:
  A_n\p).
\end{displaymath}
Thus $\check I(M)$ is the ideal generated by the $2 \times 2$ minors of a
matrix of linear forms. Conversely, if $\varphi$ is any matrix of linear forms,
then $I_2 (\varphi) = \check I(M)$ for suitable $M$. We realize that $\check
I(M)$ is a very general object. In this section we will usually require that
the identity matrix $I$ is in $M$. (Actually, it would be enough to assume
that $M$ contains an invertible matrix, but this is not important to us.) We
start with a result relating $\check I(M)$ and the simpler object $I(M)$.

\begin{lem} \label{lem:IM}
  Assume $I \in M \subseteq \Mat_\K (r,r)$. Then $I(M) \subseteq \check I(M) =
  I(M^2)$ and $I(M)_e = \check I(M)_e$ for all $e \ge 3$. In particular, if
  $M$ is closed under matrix multiplication, then $I(M) = \check I(M)$.
\end{lem}

\begin{proof}
  $I(M) \subseteq \check I(M)$ is immediate when $I \in M$. Let $A,B \in M$,
  and consider the determinant
  \begin{displaymath}
    D = \left|
      \begin{smallmatrix}
        \p_i & (A\p)_i & (B\p)_i \\
        \p_j & (A\p)_j & (B\p)_j \\
        \p_k & (A\p)_k & (B\p)_k
      \end{smallmatrix}
    \right|.
  \end{displaymath}
  By expanding along the third column, we get $D \in I(M)$. Thus expansion
  along the first row shows that
  \begin{displaymath}
    \p_i \cdot
    \begin{vmatrix}
      (A\p)_j & (B\p)_j \\
      (A\p)_k & (B\p)_k
    \end{vmatrix}
    \in I(M) \text{ for all } i,j \text{ and } k.
  \end{displaymath}
  Therefore, $\m_R \, \check I(M) \subseteq I(M)$. Since $\check I(M)$ is
  generated in degree 2, it follows that $\check I(M)_e = I(M)_e$ for all $e
  \ge 3$. Furthermore, since $(A\p)_j = \sum_{k=1}^r A_{jk} \p_k$, we get
  \begin{displaymath}
    \sum_{k=1}^r A_{jk}
    \begin{vmatrix}
      \,\p_i & (B\p)_i \\
      \,\p_k & (B\p)_k
    \end{vmatrix}
    =
    \begin{vmatrix}
      \p_i & (B\p)_i \\
      (A\p)_j & (AB\p)_j
    \end{vmatrix}
    = \p_i \cdot (AB\p)_j - (A\p)_j \cdot (B\p)_i
  \end{displaymath}
  and therefore,
  \begin{equation} \label{eq:AB}
    \sum_{k=1}^r A_{jk} \left|
      \begin{array}{@{}c@{\;}c@{}}
        \,\p_i & (B\p)_i \\
        \,\p_k & (B\p)_k
      \end{array}
    \right| - \sum_{k=1}^r A_{ik} \left|
      \begin{array}{@{}c@{\;}c@{}}
        \,\p_j & (B\p)_j \\
        \,\p_k & (B\p)_k
      \end{array}
    \right| = \left|
      \begin{array}{@{}c@{\;}c@{}}
        \,\p_i & (AB\p)_i \\
        \,\p_j & (AB\p)_j
      \end{array}
    \right| + \left|
      \begin{array}{@{}c@{\;}c@{}}
        (A\p)_i & (B\p)_i \\
        (A\p)_j & (B\p)_j
      \end{array}
    \right|\!.
  \end{equation}
  Hence, if $B \in M$, then $I_2 (A\p\: B\p) \subseteq I(M)$ if and only if
  $I_2 (\p\: AB\p) \subseteq I(M)$. In particular, $\check I(M) = I(M^2)$,
  since $I \in M$ implies $M \subseteq M^2$. If $M$ is closed under
  multiplication, then also $M^2 \subseteq M$, implying $I(M) = \check I(M)$.
\end{proof}

We note that $\check I(M) = I(M)$ when $M$ is closed under multiplication. If
$M$ is not closed, it is natural to ask if we can close $M$ and not change the
ideal $\check I(M)$. This is true, as the following proposition shows.

\begin{prop} \label{prop:K[M]}
  Assume $I \in M \subseteq \Mat_\K (r,r)$. Let $M'$ be the $\K$-subalgebra of
  $\Mat_\K (r,r)$ generated by $M$. Then $I(M') = \check I(M)$.
\end{prop}

\begin{proof}
  We have not assumed that $M$ is a $\K$-vector space. It is just any subset
  of $\Mat_\K (r,r)$ containing the identity matrix $I$. Therefore, its powers
  are defined as $M^k = \{ \sprod_{i=1}^k A_i \suchthat A_i \in M \text{ for
    all } i \}$, and not the linear span. Note that $M^k \subseteq M^{k+1}$
  since $I \in M$. Because $\Mat_\K (r,r)$ is a finite-dimensional vector
  space, it follows that $M' = \langle M^k \rangle$, the linear span of $M^k$,
  for large $k$. Since a minor is linear in each column, we get $I(\langle M^k
  \rangle) = I(M^k)$. Thus to prove that $I(M') = \check I(M)$, it is enough
  to show that $I_2 (\p\: A\p) \subseteq \check I(M)$ for all $A \in M^k$ for
  all $k \gg 0$.

  For every $A,B \in \Mat_\K (r,r)$ and all $1 \le i < j \le r$, define
  $(A,B)_{ij} \in R_2$ by $(A,B)_{ij} = (A\p)_i \cdot (B\p)_j$. We will
  usually suppress the subscripts. Note that
  \begin{displaymath}
    \begin{vmatrix}
      (AC\p)_i & (AD\p)_i \\
      (BC\p)_j & (BD\p)_j
    \end{vmatrix}
    = \sum_{k,l=1}^r A_{ik} B_{jl}
    \begin{vmatrix}
      (C\p)_k & (D\p)_k \\
      (C\p)_l & (D\p)_l
    \end{vmatrix}.
  \end{displaymath}
  Thus $(AC,BD) - (AD,BC) \in I_2 (C\p\: D\p)$, and if $I_2 (C\p\: D\p)
  \subseteq \check I(M)$, then
  \begin{equation} \label{eq:ABCD}
    (AC,BD) = (AD,BC) \mod \check I(M).
  \end{equation}
  
  Assume that $I_2 (X\p\: Y\p) \subseteq \check I(M)$ for all $X,Y \in
  \{I,A,B,C\}$. We want to show that $I_2 (\p\: ABC\p) \subseteq \check I(M)$.
  This is equivalent to $(ABC,I)_{ij} = (I,ABC)_{ij}$ mod $\check I(M)$ for
  all $i$ and $j$. To prove this, we will use equation \eqref{eq:ABCD} eight
  times, and each time one of the matrices will be $I$. Indeed, modulo $\check
  I(M)$ we have
  \begin{align*}
    (ABC,I) & = (AB,C) = (A,CB) = (B,CA) \\ & = (BA,C)
    = (BC,A) = (B,AC) = (C,AB) = (I,ABC).
  \end{align*}
  
  The rest is a simple induction. We know that $I_2 (\p\: A\p) \subseteq
  \check I(M)$ for all $A \in M^2$. Assume for some $k \ge 2$ that $I_2 (\p\:
  A\p) \subseteq \check I(M)$ for all $A \in M^k$. Then by equation
  \eqref{eq:AB} also $I_2 (A\p\: B\p) \subseteq \check I(M)$ for all $A \in
  M^i$ and $B \in M^j$ as long as $i+j \le k$. Pick $A' = \sprod_{i=1}^{k+1}
  A_i \in M^{k+1}$. Let $A=A_1$, $B = \sprod_{i=2}^k A_i$ and $C = A_{k+1}$ so
  that $ABC = A'$. The induction hypothesis and the previous paragraph imply
  that $I_2 (\p\: A'\p) \subseteq \check I(M)$. Hence we are done by induction
  on $k$.
\end{proof}

One consequence of lemma \ref{lem:IM} and proposition \ref{prop:K[M]} is that
$\{ I(M) \}$ does not change much if we restrict our attention to subsets $M
\subseteq \Mat_\K (r,r)$ that are $\K$-algebras. Indeed, if $M \subseteq
\Mat_\K (r,r)$ is any subset containing the identity matrix $I$, and $M'$ is
the $\K$-algebra generated by $M$, then $I(M)_e = I(M')_e$ for all $e \ge 3$.
Thus these ideals can only be different in degree two.

Another consequence is the following corollary.

\begin{cor} \label{cor:K[M]}
  Let $A_1, \dots, A_n \in \Mat_\K (r,r)$ and $M = \K [A_1, \dots, A_n]$. Then
  \begin{displaymath}
    I(M) = I_2 (\p\: A_1\p\: \cdots\: A_n\p).
  \end{displaymath}
\end{cor}

\begin{proof}
  $M$ is the $\K$-algebra generated by $\{ I, A_1, \dots, A_n \} \subseteq
  \Mat_\K (r,r)$, and the result follows from proposition \ref{prop:K[M]}.
\end{proof}

We now associate to any subset $M \subseteq \Mat_\K (r,r)$ a graded $R$-module
$X(M)$. When we defined $M_f = \{ A \in \Mat_\K (r,r) \suchthat I_2 (\p\: A\p)
\subseteq \ann_R f \}$ in definition \ref{def:Mf}, we required $f$ to a
homogeneous polynomial. To simplify the following definition and results, we
will allow any $f \in \cR$. Of course, if $f = \sum_{k \ge 0} f_k$ and $f_k
\in \cR_k$, then $M_f = \isect_{k \ge 0} M_{f_k}$, since $I_2 (\p\: A\p)$ is a
homogeneous ideal.

\begin{defn} \label{def:XM}
  Let $M \subseteq \Mat_\K (r,r)$. Define the graded $R$-module $X(M)$ by
  \begin{displaymath}
    X(M) = \{ f \in \cR \suchthat M \subseteq M_f \}.
  \end{displaymath}
\end{defn}

The discussion before the definition explains why $X(M)$ is a graded
$\K$-vector subspace of $\cR$. Note that $\ann_R (f) \subseteq \ann_R (Df)$
for any $D \in R$. This implies that $M_f \subseteq M_{Df}$, thus $X(M)$ is
indeed an $R$-module. $X(M)$ is closely connected to $I(M)$, as seen in the
following lemma.

\begin{lem} \label{lem:XM}
  Let $M \subseteq \Mat_\K (r,r)$ be any subset. Then
  \begin{enumerate}
    \setlength{\itemsep}{2pt}
    \setlength{\parskip}{0pt}
    \renewcommand{\theenumi}{\alph{enumi}}
    \renewcommand{\labelenumi}{\normalfont(\theenumi)}
  \item $M \subseteq M_f$ if and only if $I(M) \subseteq \ann_R f$,
% if and only if $I(M) \subseteq I(M_f)$
  \item $X_d (M) = \{ f \in \cR_d \suchthat R_{d-2} (f) \subseteq X_2 (M) \}$
    for all $d \ge 3$,
  \item $I(M)_d^\perp = X_d (M)$ for all $d \ge 0$,
  \item $I(M) = \isect_{f \in X(M)} I(M_f) = \isect_{f \in X(M)} \ann_R f$.
  \end{enumerate}
  In particular, $I_2 (\p\: A\p) \subseteq I(M)$ if and only if $A \in M_f$
  for all $f \in X(M)$.
\end{lem}

\begin{proof}
  Clearly, $I(M) \subseteq \ann_R f$ if and only if $I_2 (\p\: A\p) \subseteq
  \ann_R f$ for all $A \in M$, which is equivalent to $M \subseteq M_f$. This
  is (a).

  Let $X = X(M)$. Pick $f \in \cR_d$, $d \ge 3$. Since $I(M)$ is generated in
  degree two and $\ann_R (f)_{d-1} = \isect_{i=1}^r \ann_R (\p_i f)_{d-1}$, it
  follows that $M_f = \isect_{i=1}^r M_{\p_i f}$. Hence $f \in X_d$ if and
  only if $\p_i f \in X_{d-1}$ for all $i$, and by induction this is
  equivalent to $Df \in X_2$ for all $D \in R_{d-2}$. This proves (b).

  For all $d \ge 0$ we have $I(M)_d^\perp = \{ f \in \cR_d \suchthat Df = 0
  \:\forall\: D \in I(M) \}$, which equals $X_d$ by (a). For any $f \in X$ we
  note that $I(M) \subseteq I(M_f) \subseteq \ann_R f$, hence $I(M) \subseteq
  \isect_{f \in X} I (M_f) \subseteq \isect_{f \in X} \ann_R f$. Furthermore,
  by (c),
  \begin{displaymath}
    I(M)_d = X_d^\perp = \{ D \in R_d \suchthat Df = 0 \:\forall\: f \in X_d
    \} = \isect_{f \in X_d} \ann_R (f)_d.
  \end{displaymath}
  Thus $I(M)_d \supseteq ( \isect_{f \in X} \ann_R f )_d$, which implies (d).
  In particular, it follows that $I_2 (\p\: A\p) \subseteq I(M)$ if and only
  if $I_2 (\p\: A\p) \subseteq \ann_R f$ for all $f \in X$, and this is
  equivalent to $A \in M_f$ for all $f \in X$.
\end{proof}

\begin{rem} \label{rem:XM}
  A consequence of lemma \ref{lem:XM} is that results about $M_f$ often
  correspond to results about $I(M)$. For example, we know that $M_f$ is a
  $\K$-algebra for all $f \in \cR_d$, $d \ge 3$ (proposition \ref{prop:Mf}).
  This corresponds to the fact that $I(M^2)_d \subseteq I(M)_d$ for all $d \ge
  3$ when $I \in M$ (lemma \ref{lem:IM}).

% This is true even without $I \in M$.
  
  To prove this, let $d \ge 3$ and $f \in \cR_d$, and pick $A,B \in M_f$.
  Consider $M = \{ I,A,B \} \subseteq M_f$. We have $I(M^2)_d \subseteq I(M)_d
  \subseteq \ann_R (f)_d$. Since $\ann_R (f)_2$ is determined by $\ann_R
  (f)_d$ by lemma \ref{lem:ann}a, and $I_2 (\p\: AB\p) \subseteq I(M^2)$, we
  get $I_2 (\p\: AB\p) \subseteq \ann_R f$. Hence $AB \in M_f$.
  
  Conversely, let $A,B \in M$. Then $A,B \in M_f$ for all $f \in X = X(M)$,
  implying $AB \in M_f$ for all $f \in X_d$, $d \ge 3$. Hence $I_2 (\p\:
  AB\p)_d \subseteq \isect_{f \in X_d} \ann_R (f)_d = I(M)_d$ for all $d \ge
  3$, that is, $I(M^2)_d \subseteq I(M)_d$. Thus even though the proofs of
  these two results look very different, they actually imply each other.
\end{rem}

As promised, we give another result that restricts which algebras that arise
as $M_f$. The conclusion of this proposition does not in general follow from
the other results we have proven about $M_f$.

\begin{prop} \label{prop:M+}
  Suppose $A_0, \dots, A_n \in M_f$. Let $a_{ij}$ be the \te{j} column of
  $A_i^\T$. (So $A_i = [ a_{i1}, \dots, a_{ir} ]^\T$, i.e. $(A_i)_{jk} =
  (a_{ij})_k$ for all $i,j,k$.) Let $s<r$. Assume that $a_{ij} = 0$ for all $i
  \ge 1$ and $j \le s$, and that $a_{0j} \in \langle a_{1j}, \dots, a_{nj}
  \rangle$ for all $j>s$. Then $B = [ a_{01}, \dots, a_{0s}, 0 \dots, 0 ]^\T
  \in M_f$.
\end{prop}

% This result originally said: suppose $A_0, \dots, A_n \in M$, then $I_2
% (\p\: B\p) \subseteq I(M)$. An equivalent formulation, since: Assume $A_i
% \in M_f$ for all $i$. Then $I_2 (\p\: B\p) \subseteq I( \{A_i\} ) \subseteq
% \ann_R f$, hence $B \in M_f$. Conversely, assume that $A_i \in M_f$ for all
% $i$ implies that $B \in M_f$. Let $M = \{ A_i \}$ and $X = X(M)$. For any $f
% \in X$ we have $I_2 (\p\: A_i\p) \subseteq I(M) \subseteq \ann_R f$, hence
% $A_i \in M_f$ for all $i$. Thus also $B \in M_f$. Therefore, $I_2 (\p\: B\p)
% \subseteq \isect_{f \in X} I(M_f) = I(M)$.

\begin{proof}
  Let $M = \{ A_0, \dots, A_n \}$. We want to prove that $I(M)$ contains every
  $2 \times 2$ minor of $(\p\: B\p)$. If $i,j \le s$, then
  \begin{displaymath}
    \begin{vmatrix}
      \p_i & (B\p)_i \\
      \p_j & (B\p)_j
    \end{vmatrix}
    =
    \begin{vmatrix}
      \p_i & (A_0\p)_i \\
      \p_j & (A_0\p)_j
    \end{vmatrix}
    \in I(M).
  \end{displaymath}
  If $i,j > s$, then this minor is obviously zero. So we are left with the
  case $i \le s$ and $j>s$. By assumption $a_{0j} \in \langle a_{1j}, \dots,
  a_{nj} \rangle$, thus $a_{0j} = \sum_{k=1}^n c_{kj} a_{kj}$ for suitable
  $c_{kj} \in \K$. It follows that
  \begin{align*}
    \begin{vmatrix}
      \p_i & (B\p)_i \\
      \p_j & (B\p)_j
    \end{vmatrix}
    =
    \begin{vmatrix}
      \p_i & a_{0i}^\T \p \\
      \p_j & 0
    \end{vmatrix}
    & =
    \begin{vmatrix}
      \p_i & a_{0i}^\T \p \\
      \p_j & a_{0j}^\T \p
    \end{vmatrix}
    - \sum_{k=1}^n c_{kj}
    \begin{vmatrix}
      \p_i & 0 \\
      \p_j & a_{kj}^\T \p
    \end{vmatrix}
    \\ & =
    \begin{vmatrix}
      \p_i & (A_0\p)_i \\
      \p_j & (A_0\p)_j
    \end{vmatrix}
    - \sum_{k=1}^n c_{kj}
    \begin{vmatrix}
      \p_i & (A_k\p)_i \\
      \p_j & (A_k\p)_j
    \end{vmatrix}
    \in I(M).
  \end{align*}
  Therefore, $I_2 (\p\: B\p) \subseteq I(M)$. Since $I(M) \subseteq \ann_R f$,
  this implies $B \in M_f$.
\end{proof}

\begin{exmp}
  The assumptions in proposition \ref{prop:M+} might seem a bit strange. One
  situation where it can be used, is the following. Let $c_1 + c_2 + c_3 = r$.
  For $i=1,2,3$, pick $C_i \in \Mat_\K (c_i,r)$, and define $B_i \in \Mat_\K
  (r,r)$ by
  \begin{displaymath}
    B_1 = \left(
      \begin{smallmatrix}
        C_1 \\ 0 \\ 0
      \end{smallmatrix}
    \right) \qquad B_2 = \left(
      \begin{smallmatrix}
        C_1 \\ C_2 \\ 0
      \end{smallmatrix}
    \right) \qquad B_3 = \left(
      \begin{smallmatrix}
        C_1 \\ 0 \\ C_3
      \end{smallmatrix}
    \right).
  \end{displaymath}
  Assume that $B_2, B_3 \in M_f$. If we apply proposition \ref{prop:M+} with
  $A_0 = B_2$ and $A_1 = B_2 - B_3$, we get $B_1 \in M_f$. A special case when
  $r=6$ is
  \begin{displaymath}
    B_1 = \left(
      \begin{smallmatrix}
        C & 0 & 0 \\
        0 & 0 & 0 \\
        0 & 0 & 0
      \end{smallmatrix}
    \right) \qquad B_2 = \left(
      \begin{smallmatrix}
        C & 0 & 0 \\
        0 & C & 0 \\
        0 & 0 & 0
      \end{smallmatrix}
    \right) \qquad B_3 = \left(
      \begin{smallmatrix}
        C & 0 & 0 \\
        0 & 0 & 0 \\
        0 & 0 & C
      \end{smallmatrix}
    \right),
  \end{displaymath}
  where
  \begin{math}
    C = \bigl(
    \begin{smallmatrix}
      0 & 1 \\ 0 & 0
    \end{smallmatrix}
    \bigr).
  \end{math}
  As above, $B_2, B_3 \in M_f$ implies $B_1 \in M_f$.
\end{exmp}

We will end this section with a result identifying $Z(I(M)) \subseteq
\bP^{r-1}$, the set of closed points of the projective scheme determined by
$I(M)$.

\begin{prop} \label{prop:Z(IM)}
  Suppose $M \subseteq \Mat_\K (r,r)$. Then
  \begin{displaymath}
    Z(I(M)) = \{ v \in \K^r \setminus \{0\} \suchthat v \text{ is an
    eigenvector for every } A \in M \}/\K^*.
  \end{displaymath}
\end{prop}

\begin{proof}
  By definition, $I(M) = \sum_{A \in M} I_2 (\p\: A\p)$. Thus a $v \in \K^r$
  satisfies $D(v)=0$ for all $D \in I(M)$ if and only if
  \begin{displaymath}
    \begin{vmatrix}
      v_i & (Av)_i \\
      v_j & (Av)_j
    \end{vmatrix}
    = 0 \text{ for all $i<j$ and all $A \in M$.}
  \end{displaymath}
  This is equivalent to $v$ being an eigenvector for every $A \in M$. Thus
  $Z(I(M))$ is simply the projectivization of the union of the eigenspaces.
\end{proof}

% If $J \subseteq R$ is an ideal, we may define a $\K$-vector space $M_J = \{
% A \in \Mat_\K (r,r) \suchthat I_2 (\p\: A\p) \subseteq J \}$. Then $I (M_J)
% \subseteq J$ for all $J$, and $M_{I(M)} \supseteq M$ for all $M$. This
% implies that $M_{I(M_J)} = M_J$ and $I(M_{I(M)}) = I(M)$. This gives a
% one-to-one correspondence between ideals in $R$ of the form $I(M)$ and
% subspaces of $\Mat_\K (r,r)$ of the form $M_J$.

%% \input{regular}
%% Chapter 3: Regular splittings
\chapter{Regular splittings}\label{chapter:regular}

This chapter covers our work on regular splittings. The first half deals with
how to find such splittings. Then in section \ref{section:minimal} we study
how a regular splitting affects the Artinian Gorenstein quotient $R/ \ann_R
f$. In fact, if $f = \sum_{i=1}^n g_i$ is a regular splitting of $f$, then we
express the Hilbert function and, more generally, the (shifted) graded Betti
numbers of $R/ \ann_R f$ in terms of those for $R/ \ann_R (g_i)$, $i = 1,
\dots, n$. To get there, we calculate the minimal free resolution of $R/
\ann_R f$.

Section \ref{section:pgor} concerns $\bPGor (H)$, the space parameterizing all
graded Artinian Gorenstein quotients $R/I$ with Hilbert function $H$. We
define a subset parameterizing those $R/ \ann_R f$ where $f$ splits regularly,
and we compute its dimension and the dimension of the tangent space to $\bPGor
(H)$ at the point $R/ \ann_R f$.

One goal of this paper is to study what $M_f$ can tell us about $f \in \cR_d$,
and in section \ref{section:idempot} we show how to extract from $M_f$ the
regular splitting properties of $f$. By corollary \ref{cor:basic}, any regular
splitting of $f$ happens inside the subring $\K [ R_{d-1} (f) ]^{DP} \subseteq
\cR$. Thus we may assume that $\ann_R (f)_1 = 0$ by performing a suitable base
change and reducing the number of variables, if necessary. If in addition $d
\ge 3$, proposition \ref{prop:Mf} tells us that $M_f$ is a commutative
$\K$-algebra. This will allow us to find all regular splittings. It turns out
that the idempotents in $M_f$ determine the regular splittings, so we start by
studying these.

\section{Idempotents and matrix algebras}

This section discusses idempotents in general, and in particular how they
relate to matrix algebras. We will see how eigenvalues and eigenspaces are
connected to idempotents. We start with some elementary definitions.

Let $A$ be a ring with unity. A nonzero element $e$ in $A$ is called an
\emph{idempotent} if $e^2 = e$. A subset $\{ e_1, \dots, e_n \} \subseteq A$
is a set of \emph{orthogonal idempotents} in $A$ if $e_i^2 = e_i \ne 0$ for
all $i$ and $e_i e_j = 0$ for all $i \ne j$. The set is \emph{complete} if in
addition $\sum_{i=1}^n e_i = 1$. If $\{ e_1, \dots, e_n \}$ is not complete,
let $e_0 = 1 - \sum_{i=1}^n e_i \ne 0$. Then
\begin{displaymath}
  e_i e_0 = e_i - e_i \sum_{j=1}^n e_j = e_i - e_i^2 = 0 = e_0 e_i
\end{displaymath}
for all $i>0$, and $e_0^2 = (1 - \sum_{i=1}^n e_i) e_0 = e_0$. Thus $e_0$ is
an idempotent, and $\{ e_0, \dots, e_n \}$ is a complete set of orthogonal
idempotents.

% Since $\sum_i e_i = 1$, it is also a partition of unity.
% Originally I allowed a coid to contain (several) elements equal to zero.

We define a \emph{coid} to be a set $\E = \{ e_1, \dots, e_n \}$ of nonzero
elements of $A$ such that $e_i e_j = 0$ for all $i \ne j$ and $\sum_{i=1}^n
e_i = 1$. This implies $e_i = e_i \sum_{j=1}^n e_j = e_i^2$, thus $\E$ is a
complete set of orthogonal idempotents (hence the name coid). We define its
\emph{length} to be $l(\E) = n$, the size of $\E$ as a set.

% We only need the idempotents to commute.

Assume in addition that $A$ is a commutative ring. Let $\E = \{ e_1, \dots,
e_n \}$ and $\E' = \{ e'_1, \dots, e'_m \}$ be two coids. For all $1 \le i \le
n$ and $1 \le j \le m$, let $e_{ij} = e_i e'_j$. Then $\sum_{i,j} e_{ij} = (
\sum_{i=1}^n e_i ) ( \sum_{j=1}^m e'_j ) = 1$, and for all $(i,j) \ne (k,l)$,
we have $e_{ij} e_{kl} = e_i e'_j e_k e'_l = (e_i e_k) (e'_j e'_l) = 0$. Thus,
if $e_{ij}$ and $e_{kl}$ are nonzero, then they are orthogonal idempotents. In
particular, they are not equal. This shows that
\begin{displaymath}
  \E \otimes \E' = \{ e_{ij} \suchthat e_{ij} \ne 0 \}
\end{displaymath}
is another coid, which we call the \emph{product} coid. This product has the
following properties.

\begin{lem} \label{lem:productcoid}
  Suppose $A$ is a commutative ring with unity. Let $\E = \{ e_1, \dots, e_n
  \}$ and $\E' = \{ e'_1, \dots, e'_m \}$ be two coids. Then $l (\E \otimes
  \E') \ge l (\E)$, and $l (\E \otimes \E') = l (\E)$ if and only if $\E
  \otimes \E' = \E$. Furthermore, if $\E \otimes \E' = \E$, then $\E$
  \emph{refines} $\E'$ in the sense that there exists a partition $\{ \J_1,
  \dots, \J_m \}$ of $\{ 1, \dots, n \}$ such that $e'_j = \sum_{i \in \J_j}
  e_i$.
\end{lem}

% Let $\du$ denote disjoint union. When we write $J_1 \du \dots \du J_m = N$,
% it means that $\union_{i=1}^m J_i = N$ and $J_i \isect J_j = \varnothing$
% for all $i \ne j$. Thus the $J_i$'s are subsets of $N$, and $\{ J_i
% \}_{i=1}^m$ is a partition of $N$.

\begin{proof}
  For each $i = 1, \dots, n$, at least one of $e_i e'_1, \dots, e_i e'_m$ must
  be nonzero, since $\sum_{j=1}^m e_i e'_j = e_i \ne 0$. This proves that $l
  (\E \otimes \E') \ge l (\E)$. It also shows that, if $l (\E \otimes \E') = l
  (\E)$, then for every $i$ there exists a unique $j_i$ such that $e_i
  e'_{j_i} \ne 0$. Then $e_i = \sum_{j=1}^m e_i e'_j = e_i e'_{j_i}$, hence
  $\E \otimes \E'$ and $\E$ are equal. For every $j = 1, \dots, m$, let $\J_j
  = \{ i \suchthat j_i = j \}$. Then $\J_j \isect \J_k = \varnothing$ for all
  $j \ne k$, and $\J_1 \union \dots \union \J_m = \{ 1, \dots, n \}$. Thus $\{
  \J_j \}$ is a partition of $\{ 1, \dots, n \}$, and $e'_j = \sum_{i=1}^n e_i
  e'_j = \sum_{i \in \J_j} e_i$.
\end{proof}

% If $i \in \J_j \isect \J_k$, then $j_i = j = k$. And $i \in \J_{j_i}$.

The next proposition contains what we will need to know about idempotents.
First, note the following. Let $V$ be any $\K$-vector space, and $V_1, \dots,
V_n \subseteq V$ be subspaces. When we write $V = \dsum_{i=1}^n V_i$, we mean
that the natural map $\dsum_{i=1}^n V_i \to V$ defined by $(v_i) \mapsto
\sum_{i=1}^n v_i$ is an isomorphism. This is equivalent to $\sum_{i=1}^n V_i =
V$ and $V_i \isect \bigl( \sum_{j \ne i} V_j \bigr) = 0$ for all $i$.

We say that $A$ contains a unique maximal coid if it contains a coid $\E$ of
maximal length and every coid refines into $\E$, cf. lemma
\ref{lem:productcoid}.

\begin{prop} \label{prop:idempot}
  Let $A$ be a commutative ring with unity.
  \begin{enumerate}
    \setlength{\itemsep}{2pt}
    \setlength{\parskip}{0pt}
    \renewcommand{\theenumi}{\alph{enumi}}
    \renewcommand{\labelenumi}{\normalfont(\theenumi)}
  \item For every coid $\{ e_1, \dots, e_n \}$, the natural map $A \to e_1 A
    \dsum \dots \dsum e_n A$ is an isomorphism of rings. Furthermore, every
    ring-isomorphism $A \to \dsum_{i=1}^n A_i$ arise this way up to
    isomorphisms of the summands $A_i$.
  \item Assume in addition that $A$ is Noetherian. Then $A$ contains a unique
    maximal coid $\E = \{ e_1, \dots, e_n \}$. In particular, the idempotents
    in $A$ are precisely the elements $e = \sum_{i \in I} e_i$ with
    $\varnothing \ne I \subseteq \{ 1, \dots, n \}$.
% Hence there are exactly $2^n-1$ idempotents in $A$.
  \item Let $A$ also be Artinian, and let $\{ e_1, \dots, e_n \}$ be the
    unique maximal coid. For every $i$, the ring $A_i = e_i A$ is local
    Artinian, and its maximal ideal is $A_i^{\nil} = \{ a \in A_i \suchthat
    a^k = 0 \text{ for some } k \}$, the set of nilpotent elements in $A_i$.
    In particular, $A$ contains exactly $n$ prime ideals.
  \end{enumerate}
\end{prop}

\begin{proof}
  We note that if $e \in A$ is an idempotent, then the ideal
  \begin{displaymath}
    eA = \{ ea \suchthat a \in A \} \subseteq A
  \end{displaymath}
  is itself a commutative ring, with identity $e$. The map $a \mapsto (e_1 a,
  \dots, e_n a)$ is obviously a homomorphism of rings. Since $(a_1, \dots,
  a_n) \mapsto \sum_{i=1}^n a_i$ is an inverse, it is an isomorphism.

  Assume that $A \to \dsum_{i=1}^n A_i$ is an isomorphism of rings. For every
  $i = 1, \dots, n$, let $e_i \in A$ be the element mapped to $1 \in A_i$ and
  $0 \in A_j$ for all $j \ne i$. Then for all $i \ne j$, we have $e_i e_j
  \mapsto 0$ in every component, thus $e_i e_j = 0$. And $\sum_{i=1}^n e_i =
  1$ since $1 \mapsto (1, \dots, 1)$. Hence $\{ e_1, \dots, e_n \}$ is a coid,
  and $A \to \dsum_{i=1}^n A_i$ restricts to isomorphisms $e_i A \to A_i$.
  Thus the map $A \to \dsum_{i=1}^n A_i$ factors through the natural map $A
  \to \dsum_{i=1}^n e_i A \iso \dsum_{i=1}^n A_i$. This proves (a).

  We will now prove (b) in several steps. First, suppose that $A$ contains an
  idempotent $e \ne 1$. Then $1-e$ is also idempotent. Let
  \begin{displaymath}
    \Upsilon = \Upsilon (A) = \{ a \in A \suchthat a^2 = a \},
  \end{displaymath}
  and note that $\Upsilon (eA) = \{ ea \suchthat ea^2 = ea \} = e \Upsilon$.
  It follows that the isomorphism $A \to eA \dsum (1-e)A$ restricts to a
  bijection $\Upsilon \to e \Upsilon \times (1-e) \Upsilon$.

  Assume that $A$ contains infinitely many idempotents. Thus $\Upsilon$ is
  infinite, and for every idempotent $e$, at least one of $e \Upsilon$ and
  $(1-e) \Upsilon$ must be infinite. Pick $e_1 \in \Upsilon \setminus \{ 0,1
  \}$ such that $(1-e_1) \Upsilon$ is infinite. Since $(1-e_1)A$ has
  infinitely many idempotents, we may choose $e_2 \in (1-e_1) \Upsilon
  \setminus \{ 0,1-e_1 \}$ such that $(1-e_2)(1-e_1) \Upsilon$ is infinite.
  Since $e_2 \in (1-e_1) \Upsilon$, we get $e_1 e_2 = 0$. We may repeat this
  process as many times as we like, producing elements $e_1, e_2, \ldots \in
  A$ such that $e_i^2 = e_i \ne 0$ for all $i$ and $e_i e_j = 0$ for all $i
  \ne j$. If $e_k = \sum_{i<k} a_i e_i$ for some $a_i \in A$, then $e_k^2 =
  \sum_{i<k} a_i e_i e_k = 0$, which is a contradiction. Hence we have
  produced a non-terminating, ascending sequence of ideals
  \begin{displaymath}
    (e_1) \subsetneq (e_1, e_2) \subsetneq (e_1, e_2, e_3) \subsetneq \dots,
  \end{displaymath}
  contradicting the Noetherian hypothesis.

% proving that $A$ cannot be Noetherian.

  Since $A$ has only finitely many idempotents, there is a coid $\E$ of
  maximal length. If $\E'$ is any coid, we know that $l(\E \otimes \E') \ge
  l(\E)$. By the maximality of $\E$, it must be an equality, implying $\E
  \otimes \E' = \E$. Furthermore, $l(\E') \le l(\E \otimes \E') = l(\E)$, with
  equality if and only if $\E' = \E$. Hence $\E$ is the unique coid of maximal
  length. Moreover, $\E$ is a refinement of $\E'$, so any coid is obtained
  from $\E$ by ``grouping'' some of its elements as in lemma
  \ref{lem:productcoid}. In particular, if $e \ne 1$ is any idempotent, then
  $\{ e, 1-e \}$ can be refined to $\E = \{ e_1, \dots, e_n \}$, implying that
  there is a non-empty subset $I \subseteq \{ 1, \dots, n \}$ such that $e =
  \sum_{i \in I} e_i$.

  To prove (c), assume that $A$ is Artinian, and let $a \in A$. Since
  \begin{displaymath}
    (1) \supseteq (a) \supseteq (a^2) \supseteq (a^3) \supseteq \dots
  \end{displaymath}
  becomes stationary, there is an $n \ge 0$ such that $(a^n) = (a^{n+1})$.
  Hence there exists $b \in A$ such that $a^n = b a^{n+1}$. It follows that
  $a^{n+k} b^k = a^n$ for all $k \ge 1$, and therefore, $(ab)^{2n} = (ab)^n$.
  If $(ab)^n = 0$, then $a^n = a^{2n} b^n = 0$. Thus either $a$ is nilpotent,
  or $(ab)^n \ne 0$ is idempotent.

% Or $a^{n^2} = (a^n)^n = a^{n(n+1)} b^n = 0$.

  The ring $A_i = e_i A$ contains no non-trivial idempotents because $\{ e_1,
  \dots, e_n \}$ is maximal. Let $P \subseteq A_i$ be a prime ideal.
  Obviously, $P$ contains all nilpotents. But if $a \in A_i$ is not nilpotent,
  then we have just proven that $a$ must be invertible. Thus
  \begin{displaymath}
    P = A_i^{\nil} = \{ a \in A_i \suchthat \text{$a$ is nilpotent} \}.
  \end{displaymath}

  Clearly, an ideal $P \subseteq A = \dsum_{i=1}^n A_i$ is prime if and only
  if $P = P_1 \dsum \dots \dsum P_n$ and there exists $j$ such that $P_j$ is a
  prime ideal in $A_j$ and $P_i = A_i$ for all $i \ne j$. Since $A_j$ has a
  unique prime ideal, it follows that $P_j = A_j^{\nil}$. Thus $A$ has exactly
  $n$ prime ideals.
\end{proof}

\begin{rem}
  Continuing with the notation of the proof of part (c), we see that $A_P \iso
  (A_j)_{P_j} \iso A_j$. Hence the decomposition $A = \dsum_{i=1}^n A_i$ is
  the one that is obtained in \cite[section 2.4]{Eis} using filtrations and
  localizations.
  
  Note that the ideal $A_i^{\nil}$ is nilpotent. Since $A_i$ is Noetherian,
  $A_i^{\nil}$ is finitely generated, say by $a_1, \dots, a_q$. Since every
  $a_k$ is nilpotent, there exists $m_k$ such that $a_k^{m_k} = 0$. The ideal
  $(A_i^{\nil})^m$ is generated by products $\sprod_{j=1}^m \bigl(
  \sum_{k=1}^q c_{jk} a_k \bigr)$. When $m > \sum_{k=1}^q (m_k-1)$, every
  monomial in the expansion is necessarily zero. Thus the product is zero,
  proving that $A_i^{\nil}$ is a nilpotent ideal.
\end{rem}

\begin{rem}
  Note that the commutativity of $A$ in (b) is necessary. Indeed, $\Mat_\K
  (r,r)$ contains infinitely many idempotents when $r \ge 2$ and $\K$ is
  infinite. For instance,
  \begin{math}
    A = \bigl(
    \begin{smallmatrix}
      1 & a \\ 0 & 0
    \end{smallmatrix}
    \bigr)
  \end{math}
  is idempotent for all $a \in \K$.
\end{rem}

In this paper, when we apply proposition \ref{prop:idempot}, the ring $A$ will
usually be a matrix algebra $M$. In this case, the idempotents in $M$ are
closely related to the eigenspaces of $M$. Before we prove this, we give some
definitions.

Let $M$ be a commutative subalgebra $M \subseteq \Mat_\K (r,r)$, and assume
that $M$ contains the identity matrix $I$. We say that $v \in V = \K^r$ is an
\emph{eigenvector for $M$} if it is an eigenvector for all $A \in M$, that is,
if for every $A \in M$ there exists $\lambda_A \in \K$ such that $Av =
\lambda_A v$. Obviously, $v=0$ is an eigenvector.

Fix an eigenvector $v \ne 0$. Then $Av = \lambda_A v$ determines $\lambda_A$
uniquely. Consider the map $\lambda : M \to \K$ defined by $\lambda (A) =
\lambda_A$. Let $A,B \in M$. Since $M$ is a $\K$-algebra, we get $aA+bB \in M$
for all $a,b \in \K$, and $AB \in M$. It follows that
\begin{displaymath}
  \lambda (aA+bB) v = (aA+bB) v = aAv + bBv = \bigl( a \lambda (A) + b \lambda
  (B) \bigr) v.
\end{displaymath}
% = a \lambda (A) v + b \lambda (B) v
Since $v \ne 0$, this implies that $\lambda (aA+bB) = a \lambda (A) + b
\lambda (B)$, and similarly, $\lambda (AB) = \lambda (A) \lambda (B)$.
Moreover, $\lambda (I) = 1$. Thus $\lambda$ is a homomorphism of
$\K$-algebras. We call $\lambda$ an \emph{eigenvalue function} for $M$.

For every homomorphism $\lambda : M \to \K$ of $\K$-algebras, we define
\begin{displaymath}
  U_\lambda = \{ v \in V \suchthat Av = \lambda (A) v \text{ for all } A \in M
  \}.
\end{displaymath}
Clearly, $\lambda$ is an eigenvalue function for $M$ if and only if $U_\lambda
\ne 0$. $U_\lambda$ is the \emph{eigenspace} associated to $\lambda$. Let
$U_\lambda^0 = 0$. Define $U_\lambda^k$ for $k \ge 1$ inductively by
\begin{displaymath}
  U_\lambda^k = \bigl\{ v \in V \,\big\vert\: \bigl( A - \lambda (A) I \bigr)
  \mspace{1mu} v \in U_\lambda^{k-1} \text{ for all } A \in M \bigr\}.
\end{displaymath}
In particular, $U_\lambda^1 = U_\lambda$, the eigenspace associated to
$\lambda$. Obviously, the sequence $0 \subseteq U_\lambda^1 \subseteq
U_\lambda^2 \subseteq \dots$ must stabilize since $V$ is of finite dimension.
Define $V_\lambda = \sum_{k \ge 0} U_\lambda^k$, that is, $V_\lambda =
U_\lambda^k$ for all $k \gg 0$. We call $V_\lambda$ the \emph{generalized
  eigenspace} associated to $\lambda$.

The following proposition is a spectral theorem for $M$.

\begin{prop} \label{prop:eigen}
  Let $M \subseteq \Mat_\K (r,r)$ be a commutative subalgebra containing the
  identity matrix $I$.
  \begin{enumerate}
    \setlength{\itemsep}{2pt}
    \setlength{\parskip}{0pt}
    \renewcommand{\theenumi}{\alph{enumi}}
    \renewcommand{\labelenumi}{\normalfont(\theenumi)}
  \item $M$ has a unique maximal complete set of orthogonal idempotents $\{
    E_i \}_{i=1}^n$.
  \item $M_i = E_i M$ is local Artinian, and its unique prime ideal is
    \begin{displaymath}
      M_i^{\nil} = \{ A \in M_i \suchthat A \text{ is nilpotent} \}.
    \end{displaymath}
  \item $M = M_1 \dsum \dots \dsum M_n$.
  \item $\K^r = \im E_1 \dsum \dots \dsum \im E_n$.
  \item Let $I = \{ i \suchthat M_i = \langle E_i \rangle \dsum M_i^{\nil}
    \}$. There are exactly $|I|$ homomorphism $M \to \K$ of $\K$-algebras.
    Indeed, for each $i \in I$, the $\K$-linear map $\lambda_i : M \to \K$
    defined by $\lambda_i (E_i) = 1$ and $\lambda_i (A) = 0$ for all $A \in
    M_i^{\nil} \dsum \bigl( \dsum_{j \ne i} M_j \bigr)$ is a homomorphism of
    $\K$-algebras, and there are no others. Each $\lambda_i$ is an eigenvalue
    function, and $V_{\lambda_i} = \im E_i$.
  \item $M_i = \langle E_i \rangle \dsum M_i^{\nil}$ for all $i$ if and only
    if $\K$ contains every eigenvalue of each $A \in M$.
  \end{enumerate}
\end{prop}

% Every homomorphism $\lambda : M \to \K$ of $\K$-algebras is an eigenvalue
% function.

\begin{proof}
  Since $M$ has finite dimension as a $\K$-vector space, it is Artinian.
  Hence (a), (b) and (c) follow immediately from proposition
  \ref{prop:idempot}.

  To prove (d), that is, $\K^r = \im E_1 \dsum \dots \dsum \im E_n$, it is
  enough to note that $v \mapsto (E_1 v, \dots, E_n v)$ and $(v_1, \dots, v_n)
  \mapsto \sum_{i=1}^n v_i$ are $\K$-linear maps and inverses of each other.

% For all $i$, let $V_i = \im E_i = \{ E_i (v) \suchthat v \in V \}$. Since
% $\sum_i E_i = I$, it follows that $V = \K^r = \sum_i V_i$. Pick $v \in V_i
% \isect ( \sum_{j \ne i} V_j )$. Since $v \in V_i$, there exists $u_i \in V$
% such that $v = E_i (u_i)$. Since $v \in \sum_{j \ne i} V_j$, there are $u_j
% \in V$, $j \ne i$, with $v = \sum_{j \ne i} E_j (u_j)$. But then
% \begin{displaymath}
%   v = E_i (u_i) = E_i^2 (u_i) = E_i (v) = E_i \biggl( \sum_{j \ne i} E_j
%   (u_j) \biggr) = \sum_{j \ne i} E_i E_j (u_j) = 0.
% \end{displaymath}
% Therefore, $V_i \isect ( \sum_{j \ne i} V_j ) = 0$ for all $i$, and $V =
% \dsum_i V_i$. This proves (d).

  Clearly, each $\lambda_i$ is a homomorphism of $\K$-algebras. If $\lambda :
  M \to K$ is any $\K$-algebra homomorphism onto some subfield $K$ of
  $\bar{\K}$, then $\ker \lambda$ must be a maximal ideal in $M$. Thus $\ker
  \lambda = M_i^{\nil} \dsum \bigl( \dsum_{j \ne i} M_j \bigr)$ for some $i$.
  If $K = \K$, then this kernel must have codimension one as a $\K$-vector
  subspace of $M$, which implies that $M_i = \langle E_i \rangle \dsum
  M_i^{\nil}$. Obviously, $\lambda (E_i) = \lambda \bigl( \sum_j E_j \bigr) =
  1$, hence $\lambda = \lambda_i$.

  To prove that $\lambda_i$ is an eigenvalue function, we must find a nonzero
  $v \in \K^r$ such that $Av = \lambda_i (A) v$ for all $A \in M$. We shall in
  fact prove that $V_{\lambda_i} \ne 0$, since this implies $U_{\lambda_i} \ne
  0$. Since $E_i \ne 0$, it is enough to prove that $V_{\lambda_i} = \im E_i$.

  Let $v \in U_{\lambda_i}^k$. For every $j \ne i$ we have $\lambda_i (E_j) =
  0$, and thus $E_j v \in U_{\lambda_i}^{k-1}$. Then $E_j^k v \in
  U_{\lambda_i}^0 = 0$ by induction. But $E_j^k = E_j$, hence $v \in \ker
  E_j$. From $v \in \isect_{j \ne i} \ker E_j$, it follows that $v = \sum_j
  E_j v = E_i v \in \im E_i$. We also note for all $j \ne i$ and $A \in M_j$
  that $A v = A E_j v = 0$. Thus
  \begin{align*}
    U_{\lambda_i}^k & = \biggl\{ v \in \K^r \,\bigg\vert\: \prod_{j=1}^k
    \bigl( A_j - \lambda_i (A_j) I \bigr) v = 0 \text{ for all } A_1, \dots,
    A_k \in M \biggr\} \\
    & = \biggl\{ v \in \im E_i \,\bigg\vert\: \biggl( \prod_{j=1}^k A_j
    \biggr) (v) = 0 \text{ for all } A_1, \dots, A_k \in M_i^{\nil} \biggr\}
    \\
    & = \biggl\{ v \in \im E_i \,\bigg\vert\: A v = 0 \text{ for all } A \in
    \bigl( M_i^{\nil} \smash{\bigr)}^k \biggr\}.
  \end{align*}
  Since $M_i^{\nil}$ is nilpotent, this implies $V_{\lambda_i} = \im E_i$, and
  finishes the proof of (e).

  To prove (f), assume that $M_i = \langle E_i \rangle \dsum M_i^{\nil}$ for
  all $i$. Pick $A \in M$. For all $i$, since $E_i A \in M_i$, there exists
  $\lambda_i \in \K$ such that $E_i A - \lambda_i E_i \in M_i^{\nil}$. Hence
  there exists $m_i \ge 1$ such that $(E_i A - \lambda_i E_i)^{m_i} = 0$. It
  follows that $E_j \sprod_{i=1}^n (A - \lambda_i I)^{m_i} = 0$ for all $j$.
  Therefore, $\sprod_{i=1}^n (A - \lambda_i I)^{m_i} = 0$. Thus the minimal
  polynomial of $A$ divides $\sprod_{i=1}^n (\lambda - \lambda_i)^{m_i}$.
  Hence $\lambda_1, \dots, \lambda_n$ are all of $A$'s eigenvalues, and they
  are all in $\K$.

% In fact, if each $m_i$ is chosen minimally, then the minimal polynomial of
% $A$ is $\sprod_{i=1}^n (\lambda - \lambda_i)^{m_i}$.

  Conversely, let $A \in M_i$. $A$ has at least one eigenvalue $\lambda \in
  \bar{\K}$, and by assumption, $\lambda \in \K$. Thus $A - \lambda E_i \in
  M_i$ is not invertible. Since $M_i$ is local, $A - \lambda E_i$ must be
  nilpotent, i.e. $A \in \langle E_i \rangle \dsum M_i^{\nil}$. Since this is
  true for every $A \in M_i$, it follows that $M_i = \langle E_i \rangle \dsum
  M_i^{\nil}$.
  % And it follows that $\lambda = \lambda_i (A)$.
\end{proof}

% Since $M_i^{\nil}$ is maximal, $M_i/M_i^{\nil}$ is a finite field extension
% of $\K$. If $\K = \bar{\K}$, it must be $\K$, hence $M_i = \langle E_i
% \rangle \dsum M_i^{\nil}$ as a $\K$-vector space.

\begin{rem} \label{rem:diagEi}
  If $\{ E_1, \dots, E_n \}$ is a coid in $\Mat_\K (r,r)$, then $E_1, \dots,
  E_n$ can easily be diagonalized simultaneously. Indeed, let $s_i = \rank
  E_i$ for all $i$, and
  \begin{displaymath}
    \J_i = \biggl\{ j \in \Z \,\bigg\vert\: \sum_{k<i} s_k < j \le \sum_{k \le
      i} s_k \biggr\}.
  \end{displaymath}
  Choose a basis $\{ v_j \suchthat j \in \J_i \}$ for $\im E_i$. Since $\K^r =
  \dsum_{i=1}^n \im E_i$, it follows that $\sum_{i=1}^n s_i = r$, and that $\{
  v_1, \dots, v_r \}$ is a basis for $\K^r$. Hence $\{ \J_1, \dots, \J_n \}$
  is a partition of $\{ 1, \dots, r \}$, and $P = [ v_1, \dots, v_r]$ is
  invertible.

  Note that $E_i^2 = E_i$ is equivalent to $E_i v = v$ for all $v \in \im
  E_i$. Hence $E_i v_j = v_j$ for all $j \in \J_i$. Similarly, since $E_i E_j
  = 0$ for all $i \ne j$, we get $E_i v_j = 0$ for all $j \notin \J_i$. It
  follows that
  \begin{displaymath}
%% CHANGED 2013-07-13: v_1 -> v_r
    P^{-1} E_i P = \bigl[ P^{-1} E_i v_1, \dots, P^{-1} E_i v_r \bigr] =
    \begin{pmatrix}
      0 & 0 & 0 \\
      0 & I & 0 \\
      0 & 0 & 0
    \end{pmatrix}\!,
  \end{displaymath}
  where $I$ is the $s_i \times s_i$ identity matrix. Thus every $E'_i = P^{-1}
  E_i P$ is a diagonal matrix, with diagonal entries $(E'_i)_{jj} = 1$ if $j
  \in \J_i$ and $(E'_i)_{jj} = 0$ otherwise.

  Also note that a matrix $A \in \Mat_\K (r,r)$ commutes with every $E'_i$, $i
  = 1, \dots, n$, if and only if $A$ can be written in block diagonal form
  \begin{displaymath}
    A = \left(
      \begin{array}{@{}c@{\mspace{12mu}}c@{\mspace{12mu}}c@{}}
        A_1 & 0 & 0 \\
        0 & \smash{\ddots} & 0 \\
        0 & 0 & A_n
      \end{array}
    \right)\!,
  \end{displaymath}
  where $A_i$ is an $s_i \times s_i$ matrix. Furthermore,
  \begin{displaymath}
    E'_i \Mat_\K (r,r) E'_i = \left\{ \left.
        \begin{pmatrix}
          0 & 0 & 0 \\
          0 & A_i & 0 \\
          0 & 0 & 0
        \end{pmatrix}
        \:\right\vert\: A_i \in \Mat_\K (s_i,s_i) \right\},
  \end{displaymath}
  hence
  \begin{math}
    \bigl\{ A \in \Mat_\K (r,r) \,\big\vert\: A E'_i = E'_i A \text{ for all }
    i \bigr\} = \dsum_{i=1}^n E'_i \Mat_\K (r,r) E'_i.
  \end{math}
\end{rem}

\section{Idempotents and regular splittings} \label{section:idempot}

We are now ready to prove that the idempotents in $M_f$ determine the regular
splittings of $f$, and how they do it. The bridge between $M_f$ and the
additive components of $f$ is the map $\gamma_f$. Recall that $\gamma_f = M_f
\to \cR_d$ sends $A \in M_f$ to the unique $g \in \cR_d$ that satisfies $\p g
= A \p f$ ($d>0$). This map connects the idempotents in $M_f$ to the additive
components of $f$, and establishes a bijection between the complete sets of
orthogonal idempotents in $M_f$ and the regular splittings of $f$.

\begin{thm} \label{thm:regsplit}
  Assume $d \ge 2$, $f \in \cR_d$ and $\ann_R (f)_1 = 0$. Let $\Coid (M_f)$ be
  the set of all complete sets $\{ E_1, \dots, E_n \}$ of orthogonal
  idempotents in $M_f$, and let
  \begin{displaymath}
    \Reg (f) = \bigl\{ \{ g_1, \dots, g_n \} \,\big\vert\: f = g_1 + \dots +
      g_n \text{ is a regular splitting of } f \bigr\}.
  \end{displaymath}
  The map $\{ E_i \}_{i=1}^n \mapsto \{ g_i = \gamma_f (E_i) \}_{i=1}^n$
  defines a bijection
  \begin{displaymath}
    \Coid (M_f) \to \Reg (f).
  \end{displaymath}
  In particular, there is a unique maximal regular splitting of $f$ when $d
  \ge 3$.
\end{thm}

% \begin{align*}
%   \Coid (M_f) & = \{ \, \text{all complete sets } \{ E_i \}_{i=1}^n
%   \text{ of orthogonal idempotents in } M_f \}, \\
%   \Reg (f) & = \{ \{ g_1, \dots, g_n \} \suchthat f = g_1 + \dots + g_n
%   \text{ is a regular splitting of } f \}.
% \end{align*}
% The map $\{ E_i \}_{i=1}^n \mapsto \{ g_i = \gamma_f (E_i) \}_{i=1}^n$
% defines a bijection $\Coid (M_f) \to \Reg (f)$.

% Assume $d \ge 2$ and $f \in \cR_d$. Let $\cS = \K [ R_{d-1} (f) ]^{DP}$, and
% choose $S \subseteq R$ such that $S \iso \cS^*$. Let $\Coid (M_f^S)$ be the
% set of all complete sets of orthogonal idempotents in $M_f^S$, and let
% $\Reg (f) = \{ \{ g_1, \dots, g_n \} \suchthat f = g_1 + \dots + g_n \text{
% is a regular splitting of } f \}$. The map $\{ E_i \}_{i=1}^n \mapsto \{
% \gamma_f (E_i) \}_{i=1}^n$ defines a bijection $\Coid (M_f^S) \to \Reg (f)$.
% In particular, there is a unique maximal regular splitting of $f$ when $d
% \ge 3$.

% $g$ is an additive component of $f$ iff $f = g + (f-g)$ is an additive
% splitting, iff there is an idempotent $E \in M_f$ such that $\p g = E \p f$.

Similar to our usage in the last section, when we here say that there is a
unique maximal regular splitting of $f$, we mean that there is a unique
regular splitting of maximal length and that every other regular splitting is
obtained from the maximal one by ``grouping'' some of its summands, cf.
proposition \ref{prop:idempot}b.

\begin{proof}
  First, note that $\ann_R (f)_1 = 0$ is equivalent to $R_{d-1} (f) = \cR_1$,
  that is, $\{ \p Df \suchthat D \in R_{d-1} \} = \K^r$. Hence, if $\p g_i =
  E_i \p f$, then
  \begin{displaymath}
    \{ \p Dg_i \suchthat D \in R_{d-1} \} = \{ E_i \p Df \suchthat D \in
    R_{d-1} \} = \im E_i.
  \end{displaymath}
  Since $\p (v^\T x) = v$, this implies that
  \begin{equation} \label{eq:vTx}
    R_{d-1} (g_i) = \{ v^\T x \suchthat v \in \im E_i \} \subseteq \cR_1.
  \end{equation}
  (Recall that $x$ denotes the column vector $x = [x_1, \dots, x_r]^\T$.)

  Assume that $\{ E_1, \dots, E_n \} \subseteq M_f$ is a complete set of
  orthogonal idempotents. For each $i$, let $g_i = \gamma_f (E_i) \in \cR_d$,
  that is, $\p g_i = E_i \p f$. Note that $g_i \ne 0$ because $E_i \ne 0$ and
  $\ann_R (f)_1 = 0$. Since $\sum_{i=1}^n E_i = I$, we get $\sum_{i=1}^n g_i =
  f$. Furthermore, for all $i$, equation \eqref{eq:vTx} implies that
  \begin{equation} \label{eq:indep}
    R_{d-1} (g_i) \isect \biggl( \sum_{j \ne i} R_{d-1} (g_j) \biggr) =
    \biggl\{ v^\T x  \,\bigg\vert\: v \in \im E_i \isect \biggl( \sum_{j \ne
      i} \im E_j \biggr) \biggr\}
  \end{equation}
  But the $E_i$'s are orthogonal idempotents, thus $\im E_i \isect \bigl(
  \sum_{j \ne i} \im E_j \bigr) = 0$ by proposition \ref{prop:eigen}d. Hence
  $f = g_1 + \dots + g_n$ is a regular splitting of $f$.

  Conversely, assume that $f$ splits regularly as $f = g_1 + \dots + g_n$. By
  lemmas \ref{lem:basic} and \ref{lem:main1} there exists for every $i$ a
  matrix $E_i \in M_f$ such that $\p g_i = E_i \p f$. $E_i$ is unique since
  $\ann_R (f)_1 = 0$, and $\gamma_f (E_i) = g_i$ by definition of $\gamma_f$.
  Furthermore, $\p f = \sum_{i=1}^n \p g_i = \sum_{i=1}^n E_i \p f$ implies
  $\sum_{i=1}^n E_i = I$.

  Because $f = \sum_i g_i$ is a regular splitting, we know for all $i$ that
  \begin{displaymath}
    R_{d-1} (g_i) \isect \biggl( \sum_{j \ne i} R_{d-1} (g_j) \biggr) = 0.
  \end{displaymath}
  Combined with equation \eqref{eq:indep}, this implies $ \im E_i \isect
  \bigl( \sum_{j \ne i} \im E_j \bigr) = 0$. For all $v \in \K^r$ and all $j$
  we know that $E_j v = \sum_{k=1}^n E_k E_j v$. For any $i \ne j$, we
  rearrange this equation and get $E_i E_j v = E_j (v - E_j v) - \sum_{k \ne
    i,j} E_k E_j v$. This is an element of $\im E_i \isect \bigl( \sum_{j \ne
    i} \im E_j \bigr)$, and must therefore be zero. Hence $E_i E_j v = 0$ for
  all $v \in \K^r$, implying $E_i E_j = 0$ for all $i \ne j$. This proves that
  $\{ E_1, \dots, E_n \}$ is a complete set of orthogonal idempotents in
  $M_f$.
  
  When $d \ge 3$, $M_f$ is a commutative $\K$-algebra, and has therefore a
  unique maximal complete set of orthogonal idempotents, by proposition
  \ref{prop:idempot}. It follows that $f$ has a unique regular splitting of
  maximal length, and that every other regular splitting of $f$ is obtained
  from the maximal one by ``grouping'' some of the summands.
\end{proof}

% If $\ann_R (f)_1 \ne 0$, then $\gamma_f (E_i)$ might be zero. Apart from
% this, the map is still surjective, but not injective.

% The proof when $d \ge 3$ is somewhat simpler: $\ann_R (f)_1 = 0$ implies
% that $M_f$ is a commutative $\K$-algebra. Therefore, $E_i E_j v = E_j E_i v
% \in \im E_i \isect \im E_j$ for all $v$. Since this intersection is zero, we
% get $E_i E_j = 0$ for all $i \ne j$.

\begin{rem}
  To sum up, theorem \ref{thm:regsplit} tells us that there is a
  correspondence between regular splittings $f = g_1 + \dots + g_n$ and
  complete sets of orthogonal idempotents $\{ E_1, \dots, E_n \} \subseteq
  M_f$ given by the equation $\p g_i = E_i \p f$. The correspondence is
  one-to-one because $\p g_i = E_i \p f$ determines $g_i$ uniquely given $E_i$
  since $d>0$, and it determines $E_i$ uniquely given $g_i$ because $\ann_R
  (f)_1 = 0$.
\end{rem}

\begin{rem}
  We want to point out that $d \ge 3$ is very different from $d=2$ when we
  work with regular splittings. If $f \in \cR_d$ and $d \ge 3$, then $M_f$
  contains a unique maximal complete set of orthogonal idempotents, and $f$
  has therefore a unique maximal splitting. This is in stark contrast to
  $d=2$, when the representation of $f$ as a sum of squares is far from
  unique. The explanation for this difference is that $M_f$ does not have a
  unique maximal complete set of orthogonal idempotents when $d=2$, and the
  reason for this is that $M_f$ is not closed under multiplication.
\end{rem}

Theorem \ref{thm:regsplit} is not as complete as we would like it to be. It
tells us how to find a regular splitting $f = \sum_{i=1}^n g_i$, but it does
not say how $M_{g_i}$ is related to $M_f$. This is something we would like to
know, since $M_f$ can contain matrices that are not idempotent. If these
matrices are not found in one of the $M_{g_i}$'s, it would mean that we loose
some information about $f$ (contained in $M_f$) when we pass to the additive
components $\{ g_1, \dots, g_n \}$.

Fortunately, this is not the case, as theorem \ref{thm:regsplitE} will tell
us. It would be nice if the relationship between $M_f$ and the $M_{g_i}$'s was
as simple as $M_f = \dsum_{i=1}^n M_{g_i}$. But it is not, because there is an
important difference between $f$ and the $g_i$'s. In theorem
\ref{thm:regsplit} we assumed $\ann_R (f)_1 = 0$, an assumption which was
justified by corollary \ref{cor:basic}. But if $f = g_1 + \dots + g_n$ is a
non-trivial regular splitting (i.e. $n \ge 2$), then necessarily $\ann_R
(g_i)_1 \ne 0$ for all $i$. This affects $M_{g_i}$, and we have to adjust for
this effect. Thus in order to state and prove theorem \ref{thm:regsplitE}, we
need to understand what happens to $M_f$ if $\ann_R (f)_1 \ne 0$. After the
adjustment, the simple relationship between $M_f$ and the $M_{g_i}$'s is in
fact restored.

\begin{rem} \label{rem:scs}
  In the following we will often choose a subspace $W \subseteq \cR_1$ and
  consider the divided power subalgebra $\cS = \K [W]^{DP} \subseteq \cR$.
  (The most important example is $W = R_{d-1} (f)$. If $\ann_R (f)_1 \ne 0$,
  then $W \subsetneq \cR_1$ and $\cS \subsetneq \cR$.) We note that $D(g) \in
  \cS$ for all $g \in \cS$ and $D \in R$. Thus for any subset $S \subseteq R$,
  the action of $R$ on $\cR$ restricts to an action of $S$ on $\cS$. We
  usually want a polynomial ring $S = \K [V]$ with $V \subseteq R_1$ acting as
  the dual of $\cS$ (i.e. $S \iso \cS^*$).

% Actually, if $V \subseteq R_1$, then by $\K [V]$ we mean the image of a
% polynomial ring $S$ in $\dim_\K V$ variables by an injection $S \to R$
% mapping $S_1$ onto $V_1$. (The map is not unique, but the image is.)

% Let $\cS \subseteq \cR$ be any subset. If $\cS$ itself is a divided power
% $\K$-algebra then $D(f) \in \cS$ for all $f \in \cS$ and $D \in R$. If $\cS$
% is a maximal subset of $\cR$ satisfying $D(f) \in \cS$ for all $f \in \cS$
% and $D \in R$, then $\cS$ is a divided power $\K$-algebra.

  To ensure that the choice of $V \subseteq R_1$ implies $S \iso \cS^*$, we
  need $V \iso W^*$. Note that $R_1 \iso W^\perp \dsum W^*$. Thus choosing $S
  = \K [V] \subseteq R$ such that $S \iso \cS^*$ with the action induced by
  $R$, is equivalent to choosing $V \subseteq R_1$ such that $R_1 = W^\perp
  \dsum V$. Note that $\cS \subseteq \cR$ determines the ideal $\ann_R \cS =
  \{ D \in R \suchthat Dg = 0 \text{ for all } g \in \cS \}$, which equals
  $(W^\perp)$, the ideal in $R$ generated by $W^\perp$. Since $R = (W^\perp)
  \dsum S$ as graded $\K$-vector spaces, $\cS$ determines $S$ only as a direct
  summand.
\end{rem}

% Note that $\cR \ne \cS \dsum (V^\perp)$ when $\chr \K > 0$ due to the ring
% structure in $\cR$, but still $\cR_e = \cS_e \dsum (S_e)^\perp$.

% There are many polynomial subalgebras of $R = \K [\p_1, \dots, \p_r]$ that
% are dual to $\cS = \K [x_1, \dots, x_s]^{DP}$. They are all related to $S =
% \K [\p_1, \dots, \p_s]$ by a base change of the form
% \begin{math}
%   \p \mapsto \left(
%     \begin{smallmatrix}
%       I & P \\
%       0 & I
%     \end{smallmatrix}
%   \right) \p.
% \end{math}
% This corresponds to the fact that $\cS$ is invariant under the dual base
% change of $\cR$,
% \begin{math}
%   x \mapsto \left(
%     \begin{smallmatrix}
%       I & 0 \\ -P & I
%     \end{smallmatrix}
%   \right) x.
% \end{math}

\begin{rem} \label{rem:Escs}
  Note that $E \in \Mat_\K (r,r)$ is idempotent if and only if $E$ acts as the
  identity on its image and $\K^r = \im E \dsum \ker E$. Hence specifying $E$
  is equivalent to choosing subspaces $\im E, \ker E \subseteq \K^r$ such that
  $\K^r = \im E \dsum \ker E$.

  A pair $(W \subseteq \cR_1, V \subseteq R_1)$ satisfying $W \dsum V^\perp =
  \cR_1$ determines an idem\-potent $E \in \Mat_\K (r,r)$ by the equations
  \begin{displaymath}
    \im E = \{ v \in \K^r \suchthat v^\T x \in W \} \quad \text{ and } \quad
    \ker E = \{ v \in \K^r \suchthat v^\T x \in V^\perp \}.
  \end{displaymath}
% I.e. $\im E = \p (W)$ and $\ker E = \p (V^\perp)$.
  Note that by remark \ref{rem:scs}, a pair $(W,V)$ satisfying $W \dsum
  V^\perp = \cR_1$ is equivalent to a pair $(\cS = \K [W]^{DP} \subseteq \cR,
  S = \K [V] \subseteq R)$ satisfying $S = \cS^*$.

  Conversely, an idempotent $E \in \Mat_\K (r,r)$ determines $\cS = \K
  [W]^{DP} \subseteq \cR$ and $S = \K [V] \subseteq R$ by the equations
  \begin{alignat*}{2}
    W & = \{ v^\T x \suchthat v \in \im E \} && = \{ x^\T Eu \suchthat u \in
    \K^r \} \subseteq \cR_1, \\
    V & = \{ v^\T \p \suchthat v \in \im E^\T \} && = \{ u^\T E \p \suchthat u
    \in \K^r \} \subseteq R_1.
  \end{alignat*}
% Note that $\im E = \Cols E$ and $\im E^\T = \Rows E$.
  We note that
  \begin{displaymath}
    V^\perp = \{ v^\T x \suchthat (u^\T E \p) (v^\T x) = u^\T E v = 0
    \:\forall\: u \in \K^r \} = \{ v^\T x \suchthat v \in \ker E \}.
  \end{displaymath}
  Since $E$ is idempotent, we know that $\K^r = \im E \dsum \ker E$. This
  implies that $W \dsum V^\perp = \cR_1$ and $W^\perp \dsum V = R_1$, and
  therefore $S \iso \cS^*$.
  
  Let $s = \rank E = \dim_\K V = \dim_\K W$. Choose a basis $\{ v_1, \dots,
  v_s \}$ for $\im E$, and a basis $\{ v_{s+1}, \dots, v_r \}$ for $\ker E$.
  Since $\im E \dsum \ker E = \K^r$, it follows that the matrix $P = [v_1,
  \dots, v_r]$ is invertible. Furthermore,
  \begin{displaymath}
    P^{-1} E P = P^{-1} [v_1, \dots, v_s, 0, \dots, 0] =
      \begin{pmatrix}
        I & 0 \\ 0 & 0
      \end{pmatrix}\!,
  \end{displaymath}
  where $I$ is the $s \times s$ identity matrix, cf. remark \ref{rem:diagEi}.
  The similarity transformation $E \mapsto P^{-1}EP$ corresponds to a change
  of variables in $\cR$ and $R$, transforming $\cS$ into $\cS = \K [x_1,
  \dots, x_s]^{DP}$ and $S$ into $S = \K [\p_1, \dots, \p_s]$, cf. proposition
  \ref{prop:MfE}.
\end{rem}

It is usually more convenient to specify an idempotent $E$ instead of
specifying a pair $(\cS = \K [W]^{DP}, S = \K [V])$ such that $R_1 = W^\perp
\dsum V$. Therefore, we will formulate most of the results in this section
using idempotents, and we will define and use $\cS$ and $S$ only when we have
to. If $f \in \cR$ and $\cS = \K [W]^{DP} \subseteq \cR$, then we will often
need to know when $f \in \cS$. Since $f \in \K [R_{d-1} (f)]^{DP}$, this is
equivalent to $R_{d-1} (f) \subseteq W$. The next lemma allows us to express
this in terms of the idempotent $E$.

\begin{lem} \label{lem:Epf}
  Assume $d>0$ and $f \in \cR_d$. Let $E \in \Mat_\K (r,r)$ be idempotent, and
  define $W = \{ v^\T x \suchthat v \in \im E \}$. Then
  \begin{enumerate}
    \setlength{\itemsep}{2pt}
    \setlength{\parskip}{0pt}
    \renewcommand{\theenumi}{\alph{enumi}}
    \renewcommand{\labelenumi}{\normalfont(\theenumi)}
  \item $R_{d-1} (f) \subseteq W$ if and only if $E \p f = \p f$,
  \item $R_{d-1} (f) = W$ if and only if $E \p f = \p f$ and $\rank E =
    \dim_\K R_{d-1} (f)$.
  \end{enumerate}
\end{lem}

\begin{proof}
  Clearly, $R_{d-1} (f) = \{ Df \suchthat D \in R_{d-1} \} \subseteq W$ if and
  only if $\{ \p Df \suchthat D \in R_{d-1} \} \subseteq \{ \p h \suchthat h
  \in W \} = \im E$. Since $E$ is idempotent, this is equivalent to $E \p Df =
  \p Df$ for all $D \in R_{d-1}$, i.e. $E \p f = \p f$. This proves (a). (b)
  follows immediately, since $\rank E = \dim_\K W$. Note that $E \p f = \p f$
  implies that $\rank E \ge \dim_\K R_{d-1} (f)$ by (a), thus (b) is the case
  of minimal rank.
\end{proof}

When $f \in \cS \subsetneq \cR$, the definition of $M_f$ is ambiguous in the
following way.

\begin{rem} \label{rem:M'f}
  Let $\cS = \K [ x_1, \dots, x_s ]^{DP}$ and $S = \K [\p_1, \dots, \p_s]$.
  Assume $s<r$, so that $\cS \subsetneq \cR$ and $S \subsetneq R$. Let $\p' =
  [ \p_1, \dots, \p_s ]^\T$. There are two ways to interpret definition
  \ref{def:Mf} when $f \in \cS$. We may consider $f$ to be an element of
  $\cR$, giving $M_f = \{ A \in \Mat_\K (r,r) \suchthat I_2 (\p\: A\p)_2
  \subseteq \ann_R f \}$. Or we may think of $f$ as an element of $\cS$, in
  which case $M'_f = \{ A \in \Mat_\K (s,s) \suchthat I_2 (\p'\, A\p')_2
  \subseteq \ann_S f \}$.

  Notice that we choose to write $I_2 (\p\, A\p)_2$. This is the degree two
  part of the ideal $I_2 (\p\, A\p)$ and generates the ideal. The reason for
  doing this is that $I_2 (\p'\, A\p')$ is ambiguous; is it an ideal in $R$ or
  an ideal in $S$? But its degree two piece is the same in both cases; $I_2
  (\p'\, A\p')_2$ is simply the $\K$-vector space spanned by the $2 \times 2$
  minors of $(\p'\, A\p')$. The ideals in $R$ and $S$ generated by these
  minors are therefore equal to $I_2 (\p'\, A\p')_2 R$ and $I_2 (\p'\, A\p')_2
  S$, respectively.
\end{rem}

Since $\cR$ is our default ring, $M_f$ will always mean what definition
\ref{def:Mf} says, i.e. $M_f = \{ A \in \Mat_\K (r,r) \suchthat I_2 (\p\:
A\p)_2 \subseteq \ann_R f \}$. It is not immediately clear what the analogue
of $M'_f$ should be for a more general subring $\cS \subseteq \cR$. We will in
pro\-position \ref{prop:MfE} prove that the following definition gives us what
we want.

\begin{defn} \label{def:MfE}
  Assume $f \in \cR_d$. Let $E \in M_f$ be idempotent. Define
  % Let $E$ be any element of $\Mat_\K (r,r)$.
  \begin{displaymath}
    M_f^E = M_f \isect E \Mat_\K (r,r) E.
  \end{displaymath}
\end{defn}

Of course, $M_f^I = M_f$. Note that $E \Mat_\K (r,r) E$ is closed under
multiplication. Hence $M_f^E$ is a $\K$-algebra if $M_f$ is closed under
matrix multiplication. In any case, we note that $E \in M_f^E$, and that $E$
acts as the identity on $M_f^E$.

We want to show that if $E \p f = \p f$ then $M_f^E$ reduces to $M'_f$ (cf.
remark \ref{rem:M'f}) when we perform a suitable base change and forget about
extra variables. In remark \ref{rem:M'f} we used both $\ann_R f$ and $\ann_S
f$. In general, if $f \in \cS \subseteq \cR$ and $S \iso \cS^*$, then by
definition $\ann_S f = \{ D \in S \suchthat Df = 0 \}$. Hence
\begin{displaymath}
  \ann_S f = S \isect \ann_R f
\end{displaymath}
is always true. Recall that, if $P \in \GL_r$, then $\phi_P : \cR \to \cR$ is
the $\K$-algebra homomorphism induced by $x \mapsto P^\T x$, and $\phi_P : R
\to R$ is induced by $\p \mapsto P^{-1} \p$.

\begin{prop} \label{prop:MfE}
  Let $f \in \cR_d$, $d>0$. Suppose $E \in M_f$ is idempotent and satisfies $E
  \p f = \p f$. Let $s = \rank E$, $W = \{ v^\T x \suchthat v \in \im E \}$
  and $V = \{ v^\T \p \suchthat v \in \im E^\T \}$. Define $\cS = \K [W]^{DP}
  \subseteq \cR$ and $S = \K [V] \subseteq R$. Choose $P \in \GL_r$ such that
  \begin{displaymath}
    E' = P E P^{-1} =
    \begin{pmatrix}
      I & 0 \\
      0 & 0
    \end{pmatrix}\!.
  \end{displaymath}
  Let $\cS' = \phi_P (\cS)$, $S' = \phi_P (S)$ and $f' = \phi_P (f)$. Then
  $\cS' = \K [x_1, \dots, x_s]^{DP}$, $S' = \K [\p_1, \dots, \p_s]$ and $f'
  \in \cS'$. Let $\p' = [ \p_1, \dots, \p_s ]^\T$. Then
  \begin{displaymath}
    M_f^E \iso M'_{f'} = \bigl\{ A \in \Mat_\K (s,s)  \,\big\vert\: I_2 (\p'\:
    A\p')_2 \subseteq \ann_{S'} (f') \bigr\}.
  \end{displaymath}
\end{prop}

\begin{proof}
  We start by proving that $\phi_P (\cS) = \K [x_1, \dots, x_s]^{DP}$. We know
  that $W = \{ x^\T Eu \suchthat u \in \K^r \}$. Since $\phi_P (x^\T Eu) =
  (P^\T x)^\T Eu = x^\T PE u = x^\T E'Pu$, it follows that $\phi_P (W) = \{
  x^\T E' v \suchthat v \in \K^r \} = \langle x_1, \dots, x_s \rangle$. Thus
  $\phi_P (\cS) = \K [x_1, \dots, x_s]^{DP}$. In a similar fashion we get
  $\phi_P (V) = \{ v^\T E' \p \suchthat v \in \K^r \} = \langle \p_1, \dots,
  \p_s \rangle$, implying $\phi_P (S) = \K [\p_1, \dots, \p_s]$. Furthermore,
  $E \p f = \p f$ implies $R_{d-1} (f) \subseteq W$ by lemma \ref{lem:Epf}.
  Thus $f \in \K [R_{d-1} (f)]^{DP} \subseteq \cS$, and therefore, $f' =
  \phi_P (f) \in \cS'$.

  In order to show that $M_f^E \iso M'_{f'}$, we first prove that
  \begin{equation} \label{eq:MfE}
    M_f^E = \bigl\{ A \in E \Mat_\K (r,r) E \,\big\vert\: I_2 (E\p\: A\p)_2
    \subseteq \ann_S f \bigr\}.
  \end{equation}
  Assume that $A \in E \Mat_\K (r,r) E$. Since $A = AE$ and$(E\p)_i \in S$ for
  all $i$, it follows that $I_2 (E\p\: A\p)_2 \subseteq S$ automatically.
  Hence $I_2 (E\p\: A\p)_2 \subseteq \ann_S f$ if and only if $I_2 (E\p\:
  A\p)_2 \subseteq \ann_R f$. By lemma \ref{lem:main2} this latter statement
  holds if and only if $(A\p)(E\p)^\T (f)$ is symmetric, which is equivalent
  to $A \p\p^\T f$ being symmetric, since $E \p f = \p f$. And $A \p\p^\T f$
  is symmetric if and only if $A \in M_f$. Hence, if $A \in E \Mat_\K (r,r)
  E$, then $I_2 (E\p\: A\p)_2 \subseteq \ann_S f \iff A \in M_f$, which proves
  equation \eqref{eq:MfE}.

  % Note that $I_2^S (E\p\: A\p) = I_2 (E\p\: A\p)_2 S$ is the ideal in $S$
  % generated by the $2 \times 2$ minors of $(E\p\: A\p)$.

  Now let $M = E \Mat_\K (r,r) E$ and
  \begin{displaymath}
    M' = P M P^{-1} = E' \Mat_\K (r,r) E' = \left\{ \left.
        \begin{pmatrix}
          A & 0 \\ 0 & 0
        \end{pmatrix}
        \,\right\vert\: A \in \Mat_\K (s,s) \right\}.
  \end{displaymath}
  Applying equation \eqref{eq:MfE} to $f'$ and $E'$, we see that
  \begin{displaymath}
    M_{f'}^{E'} = \bigl\{ A \in M' \,\big\vert\: I_2 (E'\p\: A\p)_2 \subseteq
    \ann_{S'} (f') \bigr\}.
  \end{displaymath}
  Clearly,
  \begin{math}
    A \mapsto \left(
      \begin{smallmatrix}
        A & 0 \\ 0 & 0
      \end{smallmatrix}
    \right)
  \end{math}
  defines an isomorphism $M'_{f'} \to M_{f'}^{E'}$. Thus to finish the proof,
  it is enough to show that $M_{f'}^{E'} = P M_f^E P^{-1}$.

  Let $A \in M'$. Then $A \in M_{f'}^{E'}$ if and only if $A \p\p^\T (f')$ is
  symmetric. Note that $\phi_{P^{-1}} (\p\p^\T (f')) = (P\p) (P\p)^\T (f) = P
  \p\p^\T (f) P^\T$. Hence $A \p\p^\T (f')$ is symmetric if and only if
  $\phi_{P^{-1}} (P^{-1} A \p\p^\T (f') (P^{-1})^\T) = P^{-1} AP \p\p^\T f$ is
  symmetric, which is equivalent to $P^{-1} AP \in M_f^E$. Thus $M_{f'}^{E'} =
  P M_f^E P^{-1} \iso M_f^E$, and we are done.
\end{proof}

% There is an embedding $\Mat_\K (s,s) \hookrightarrow \Mat_\K (r,r)$ mapping
% $I$ to $E$, (unique up to an automorphism of $\Mat_\K (s,s)$,) and $M_f^E$
% is the image of $M'_{f'}$. Thus a natural way $M_f^E$ is just an embedded
% version of $M'_{f'}$.

Before we go on to theorem \ref{thm:regsplitE}, we need two more lemmas.

\begin{lem} \label{lem:MfE}
  Suppose $d \ge 2$ and $f \in \cR_d$. Let $E \in M_f$ be idempotent. Then
  $M_f^E = E M_f E$. If $E \p f = \p f$ and $\rank E = \dim_\K R_{d-1} (f)$,
  then $M_f^E = M_f E$ and $M_f = M_f^E \dsum \ker \gamma_f$.
\end{lem}

% Note that $\gamma_f (M_f) = \gamma_f (M_f^E) \iso M_f^E$.

\begin{proof}
  $E \p\p^\T f$ is symmetric since $E \in M_f$. If $A \in M_f$, then $A
  \p\p^\T f$ is symmetric, hence $E A \p \p^\T (f) E^\T = EAE \p\p^\T f$ is
  also symmetric. This proves $EAE \in M_f$, and therefore $EAE \in M_f^E$.
  Hence $A \mapsto EAE$ defines a $\K$-linear map $M_f \to M_f^E$. It is
  clearly surjective. Indeed, if $EAE \in M_f^E \subseteq M_f$, then $EAE
  \mapsto E^2AE^2 = EAE$. Thus $M_f^E = E M_f E$.

  % If $d \ge 3$, then $EA \p f = AE \p f = A \p f$ is a simpler proof of the
  % next statement.

  If $E \p f = \p f$, then $EA \p\p^\T f = E \p\p^\T A^\T (f) = \p\p^\T A^\T
  (f) = A \p\p^\T f$, hence $EA \p f = A \p f$ because $d \ge 2$. Since
  $R_{d-1} (f) = \{ v^\T x \suchthat v \in \im E \}$ by lemma \ref{lem:Epf},
  we have $\{ \p D f \suchthat D \in R_{d-1} \} = \im E$. It follows that
  \begin{align*}
    EA \p f = A \p f & \iff EA \p Df = A \p Df \:\forall\: D \in R_{d-1} \\
    & \iff EAEv = AEv \:\forall\: v \in \K^r \iff EAE = AE.
  \end{align*}
  Similarly, $A \p f = 0$ if and only if $AE=0$. Hence the map $M_f \to M_f^E$
  above is also given by $A \mapsto AE$. This proves that $M_f^E = M_f E$.
  Furthermore, the kernel of this map is obviously $\{ A \in M_f \suchthat
  AE=0 \} = \{ A \in M_f \suchthat A\p f=0 \} = \ker \gamma_f$. Finally, the
  composition $M_f^E \subseteq M_f \to M_f^E$ is the identity, implying $M_f =
  M_f^E \dsum \ker \gamma_f$.
\end{proof}

\begin{lem} \label{lem:MgE}
  Suppose $d \ge 2$ and $f \in \cR_d$. Let $E \in M_f$ be idempotent and $g =
  \gamma_f (E)$. Then $M_g^E = M_f^E$. If $d \ge 3$, then even $M_g E = M_f
  E$.
\end{lem}

\begin{proof}
  Since $\p g = E \p f$, we get $A \p\p^\T g = AE \p\p^\T f$. It follows that
  \begin{equation} \label{eq:Mg}
    M_g = \{ A \in \Mat_\K (r,r) \suchthat AE \in M_f \}.
  \end{equation}
  Indeed, $A \in M_g$ if and only if $A \p\p^\T g$ is symmetric. But $A
  \p\p^\T g = AE \p\p^\T f$, and $AE \p\p^\T f$ is symmetric if and only if
  $AE \in M_f$. This proves equation \eqref{eq:Mg}.

  Let $A \in M_g$. Then $AE \in M_f$, and therefore $AE = (AE)E \in M_f E$.
  Thus $M_g E \subseteq M_f E$. This implies that $M_g^E = E M_g E \subseteq E
  M_f E = M_f^E$. Conversely, let $A \in E M_f E \subseteq M_f$. Since $AE =
  A$, we have $AE \in M_f$, and therefore $A \in M_g$. Hence $A = EAE \in E
  M_g E$. This proves that $M_g^E = M_f^E$.

  Assume $d \ge 3$, and let $A \in M_f$. Since $E \in M_f$ and $M_f$ is closed
  under multiplication, it follows that $AE \in M_f$, which implies $A \in
  M_g$. This shows that $M_f \subseteq M_g$. Thus $M_f E \subseteq M_g E
  \subseteq M_f E$, and we are done.
\end{proof}

We are now in a position to prove a generalization of theorem
\ref{thm:regsplit}. This time we do not assume $\ann_R (f)_1 = 0$. More
importantly, however, is that we are able to show how $M_f$ and the
$M_{g_i}$'s are related. Recall that $E$ acts as the identity on $M_f^E$.
Therefore $\{ E_1, \dots, E_n \}$ is a complete set of idempotents in $M_f^E$
if and only if $\sum_{i=1}^n E_i = E$ and $E_i E_j = 0$ for all $i \ne j$.

\begin{thm} \label{thm:regsplitE}
  Let $d \ge 2$ and $f \in \cR_d$. Choose a matrix $E \in M_f$ such that $E \p
  f = \p f$ and $\rank E = \dim_\K R_{d-1} (f)$. Let
  \begin{align*}
    \Coid \left( M_f^E \right) & = \biggl\{ \{ E_i \}_{i=1}^n \,\bigg\vert\, 0
    \ne E_i \in M_f^E, \sum_{i=1}^n E_i = E \text{ and } E_i E_j = 0
    \:\forall\: i \ne j \, \biggr\}, \\
    \Reg (f) & = \biggl\{ \{ g_1, \dots, g_n \} \,\bigg\vert\: f = g_1 + \dots
    + g_n \text{ is a regular splitting of } f \biggr\}.
  \end{align*}
  The map $\{ E_i \}_{i=1}^n \mapsto \{ g_i = \gamma_f (E_i) \}_{i=1}^n$
  defines a bijection
  \begin{displaymath}
    \Coid \left( M_f^E \right) \to \Reg (f).
  \end{displaymath}

  Assume $d \ge 3$. Then $M_f^E$ is a commutative $\K$-algebra, and there
  exists a unique maximal regular splitting of $f$. Let $\{ E_1, \dots, E_n
  \}$ be a complete set of orthogonal idempotents in $M_f^E$, and let $g_i =
  \gamma_f (E_i)$. Then
  \begin{displaymath}
    M_{g_i}^{E_i} = M_f E_i = M_f^E E_i \quad \text{ for all $i$,} \quad
    \text{ and } \quad M_f^E = \dsum_{i=1}^n M_{g_i}^{E_i}.
  \end{displaymath}
\end{thm}

% To be complete, if should maybe include the following facts: $E_i \in
% M_{g_i}^{E_i} \subseteq M_{g_i}$ and satisfies $E_i \p g_i = \p g_i$ and
% $\rank E_i = \dim_\K R_{d-1} (g_i)$.

% There exists a unique regular splitting of $f$ of maximal length, and every
% other regular splitting is obtained by grouping some of the summands.

% If $d=2$, then $M_{g_i}^{E_i} = E_i M_f^E E_i$ for all $i$, but $M_f^E \ne
% \dsum_i E_i M_f^E E_i$ since $M_f^E$ is not commutative.

% Assume $\ann_R (f)_1 = 0$. Let $\{ E_1, \dots, E_n \}$ be a complete set of
% orthogonal idempotents in $M_f$, and let $g_i = \gamma_f (E_i)$. Since $\p
% g_i = E_i \p f$, it follows that $\{ \p D g_i \suchthat D \in R_{d-1} \} =
% \im E_i$. Now $M_{g_i}^{E_i} = M_{g_i} E_i$ by lemma \ref{lem:MfE}, and
% $M_{g_i} E_i = M_f E_i$ by lemma \ref{lem:MgE}. Hence $M_{g_i}^{E_i} = M_f
% E_i$ for all $i$, and $M_f = \dsum_{i=1}^n M_{g_i}^{E_i}$.

\begin{proof}
  We know that every regular splitting happens inside $\cS = \K [R_{d-1}
  (f)]^{DP}$ by corollary \ref{cor:basic}. Using the isomorphism of
  proposition \ref{prop:MfE}, the first statements of the theorem are
  equivalent to the corresponding statements about $M'_{f'}$ and $\Reg (f')$,
  and follows from theorem \ref{thm:regsplit}.

% Alternatively, the proof of these statements is essentially the same as the
% proof of theorem \ref{thm:regsplit}, with the following amendments. First
% note that $\{ \p D f \suchthat D \in R_{d-1} \} = \im E$. But since $E_i E =
% E_i$, it still follows that $R_{d-1} (g_i) = \{ v^\T x \suchthat v \in \im
% E_i \}$. Since this is non-zero, it follows that $g_i \ne 0$. Next, when we
% have $g_i$ and find $E_i \in M_f$ such that $\p g_i = E_i \p f$, we see that
% we may choose $E_i \in M_f^E$ because $M_f = M_f^E \dsum \ker \gamma_f$.
% Finally, $\p f = \sum_i E_i \p f$ implies $E = \sum_i E_i E = \sum_{i=1}^n
% E_i$.

  Let $d \ge 3$. It follows from proposition \ref{prop:Mf} and lemma
  \ref{lem:MfE} that $M_f^E$ is a commutative $\K$-algebra. (Or by the
  isomorphism with $M'_{f'}$.) The existence of the unique maximal regular
  splitting of $f$ then follows by proposition \ref{prop:idempot}b.

  It remains only to prove the last two statements. Let $\{ E_1, \dots, E_n
  \}$ be a complete set of orthogonal idempotents in $M_f^E$, and let $g_i =
  \gamma_f (E_i)$. Note that $\{ \p D f \suchthat D \in R_{d-1} \} = \im E$ by
  lemma \ref{lem:Epf}, and recall that $E$ is the identity in $M_f^E$. Since
  $\p g_i = E_i \p f$, it follows that $\{ \p D g_i \suchthat D \in R_{d-1} \}
  = \im (E_i E) = \im E_i$ and $M_{g_i}^{E_i} = M_{g_i} E_i$, cf. the proof of
  lemma \ref{lem:MfE}. Moreover, $M_{g_i} E_i = M_f E_i$ by lemma
  \ref{lem:MgE}, and $M_f E_i = M_f E E_i = M_f^E E_i$ by lemma \ref{lem:MfE}.
  It follows that $M_{g_i}^{E_i} = M_f E_i = M_f^E E_i$ for all $i$, and
  $M_f^E = \dsum_{i=1}^n M_f^E E_i = \dsum_{i=1}^n M_{g_i}^{E_i}$.
\end{proof}

\begin{rem}
  Note that an idempotent $E$ as in theorem \ref{thm:regsplitE} always exists.
  Given $f \in \cR_d$, let $W = R_{d-1} (f)$, and choose $W' \subseteq \cR_1$
  such that $W \dsum W' = \cR_1$. Let $E \in \Mat_\K (r,r)$ be the idempotent
  determined by
  \begin{displaymath}
    \im E = \{ v \in \K^r \suchthat v^\T x \in W \} \quad \text{ and } \quad
    \ker E = \{ v \in \K^r \suchthat v^\T x \in W' \},
  \end{displaymath}
  cf. remark \ref{rem:Escs}. Then $E \p f = \p f$ and $\rank E = \dim_\K
  R_{d-1} (f)$ by lemma \ref{lem:Epf}. Moreover, $E \p f = \p f$ implies $E
  \in M_f$. Also note that this $E$ is not unique since we have the choice of
  $W' \in \cR_1$.
\end{rem}

\begin{rem}
  One goal of this paper is to find out what the algebra $M_f$ can tell us
  about $f$. Assume that $\ann_R (f)_1 = 0$. The idempotent $E$ in theorem
  \ref{thm:regsplitE} must then be the identity matrix $I$, and therefore
  $M_f^E = M_f$. Then the first part of theorem \ref{thm:regsplitE} reduces to
  theorem \ref{thm:regsplit}, and tells us that the idempotents in $M_f$
  determines the regular splittings of $f$, and how this happens.

  Assume $d \ge 3$. The last two statements of theorem \ref{thm:regsplitE}
  have no counter part in theorem \ref{thm:regsplit}. They say that if $A \in
  M_f$, then $A_i = A E_i \in M_{g_i}^{E_i}$ and $A = \sum_{i=1}^n A_i$. Thus
  any ``information'' about $f$ contained in $M_f$ is passed on as
  ``information'' about $g_i$ contained in $M_{g_i}^{E_i}$. For example, $M_f$
  contains a nilpotent matrix if and only if (at least) one of the
  $M_{g_i}^{E_i}$ contains a nilpotent matrix.

  In other words, in order to figure out what $M_f$ can tell us about $f$, it
  should be enough to find out what $M_{g_i}^{E_i}$ can tell us about $g_i$
  for all $i$. (Proposition \ref{prop:regularM} can be used for similar
  purposes.) Hence we may assume that $M_f$ does not contain any non-trivial
  idempotents. If $\K$ contains every eigenvalue of each $A \in M_f$, then
  this implies that $M_f = \langle I \rangle \dsum M_f^{\nil}$ by proposition
  \ref{prop:eigen}. And if $\K = \bar{\K}$, then it is always so, hence modulo
  theorem \ref{thm:regsplitE} it is enough to study all $f \in \cR_d$ such
  that $M_f = \langle I \rangle \dsum M_f^{\nil}$. It is this situation we
  study in chapter \ref{chapter:limits}.
\end{rem}

Theorem \ref{thm:regsplitE} is formulated using a non-unique idempotent $E$.
We will now give an intrinsic reformulation of that theorem when $d \ge 3$.
For that purpose, we define the following $\K$-algebra.

\begin{defn} \label{def:Gf}
  Assume $d \ge 3$ and $f \in \cR_d$. Define $G_f = \gamma_f (M_f)$, and let
  \begin{displaymath}
    \star : G_f \times G_f \to G_f
  \end{displaymath}
  be the map induced by multiplication in $M_f$.
\end{defn}

Of course, we could define $G_f$ also for smaller $d$, but then we would not
get an induced multiplication. The induced map is clearly the following. For
any $g,h \in G_f$, we may choose $A,B \in M_f$ such that $g = \gamma_f (A)$
and $h = \gamma_f (B)$, and define $g \star h = \gamma_f (AB)$. We can prove
that this is well defined, and that $\star$ is a bilinear, associative and
commutative multiplication on $G_f$, like we do in proposition
\ref{prop:star}. But here we choose a different approach.

The idempotent $E \in M_f$ in theorem \ref{thm:regsplitE} satisfies $E \p f =
\p f$ and $\rank E = \dim_\K R_{d-1} (f)$. Hence $M_f = M_f^E \dsum \ker
\gamma_f$ by lemma \ref{lem:MfE}. Therefore,
\begin{displaymath}
  G_f = \gamma_f (M_f) = \gamma_f (M_f^E) \iso M_f^E.
\end{displaymath}
The map $\star$ is clearly induced by the multiplication in $M_f^E$, proving
that $\star$ is well defined and giving $G_f$ the structure of a commutative
$\K$-algebra. Note that $\star$ is independent of $E$, by its definition
\ref{def:Gf}.

Note that $f$ is the identity element of $(G_f, \star)$ since $f = \gamma_f
(I)$. We have the following immediate consequence of theorem
\ref{thm:regsplitE}.

\begin{cor} \label{cor:regG}
  Let $d \ge 3$ and $f \in \cR_d$. Then $f = \sum_{i=1}^n g_i$ is a regular
  splitting of $f$ if and only if $\{ g_1, \dots, g_n \}$ is a complete set of
  orthogonal idempotents in $G_f$. In particular, there is a unique maximal
  regular splitting. If $f = \sum_{i=1}^n g_i$ is any regular splitting, then
  $G_{g_i} = G_f \star g_i$ for all $i$, and $G_f = \dsum_{i=1}^n G_{g_i}$.
\end{cor}

\begin{exmp}
  Let $r=d=3$ and $f = x_1 x_2^\pd 2 + x_2 x_3^\pd 2 + x_3^\pd 3$. Then
  \begin{displaymath}
    \p f =
    \begin{pmatrix}
      x_2^\pd 2 \\ x_1 x_2 + x_3^\pd 2 \\ x_2 x_3 + x_3^\pd 2
    \end{pmatrix}
    \quad \text{ and } \quad
    \p\p^\T f =
    \begin{pmatrix}
      0 & x_2 & 0 \\
      x_2 & x_1 & x_3 \\
      0 & x_3 & x_2 + x_3
    \end{pmatrix}\!.
  \end{displaymath}
  It follows that $\ann_R (f)_1 = 0$ and $\ann_R (f)_2 = \langle \p_1^2, \p_1
  \p_3, \p_1 \p_2 + \p_2 \p_3 - \p_3^2 \rangle$. Thus
  \begin{displaymath}
    I_2
    \begin{pmatrix}
      \p_1 & \p_2 & \p_3 \\
      0 & \p_3 & \p_1 + \p_3
    \end{pmatrix}
    \subseteq \ann_R f.
  \end{displaymath}
  It follows that
  \begin{displaymath}
    A =
    \begin{pmatrix}
      0 & 0 & 0 \\
      0 & 0 & 1 \\
      1 & 0 & 1
    \end{pmatrix}
    \in M_f.
  \end{displaymath}
  We note that $\det (\lambda I - A) = \lambda^2 (\lambda-1)$. Since $A$ has
  both 0 and 1 as eigenvalues, $A$ is neither invertible nor nilpotent. Hence
  there must exists a non-trivial idempotent in $M_f$! Indeed, we know that
  \begin{displaymath}
    A^2 =
    \begin{pmatrix}
      0 & 0 & 0 \\
      1 & 0 & 1 \\
      1 & 0 & 1
    \end{pmatrix}
    \in M_f,
  \end{displaymath}
  and we see that $A^3 = A^2$. Thus $E = A^2$ is such an idempotent.

  So far we have shown that $M_f \supseteq \K [A] = \langle I, A, A^2
  \rangle$. To prove equality, we show that $\ann_R f$ has exactly two
  generators of degree 3. Since $R/\ann_R f$ is Gorenstein of codimension 3,
  the structure theorem of Buchsbaum-Eisenbud \cite{BE} applies. Because we
  already know that $\ann_R f$ has three generators of degree 2 and at least
  two generators of degree 3, it follows easily that it cannot have more
  generators. Hence
  \begin{displaymath}
    \ann_R f = \bigl( \p_1^2, \p_1 \p_3, \p_1 \p_2 + \p_2 \p_3 - \p_3^2,
    \p_2^3, \p_2^2 \p_3 \bigr),
  \end{displaymath}
  which are the five Pfaffians of
  \begin{displaymath}
    \begin{pmatrix}
      0 & 0 & \p_1 & \p_2 & \p_3 \\
      0 & 0 & 0 & \p_3 & \p_1 + \p_3 \\
      -\p_1 & 0 & 0 & 0 & \p_2^2 \\
      -\p_2 & -\p_3 & 0 & 0 & 0 \\
      -\p_3 & -\p_1 -\p_3 & -\p_2^2 & 0 & 0
    \end{pmatrix}\!.
  \end{displaymath}
  Thus $M_f = \langle I, A, A^2 \rangle$, and $E=A^2$ is an idempotent of rank
  1. We note that
  \begin{displaymath}
    M_f \cdot E = \langle E \rangle \quad \text{ and } \quad M_f \cdot (I-E) =
    \langle I-E, A-A^2 \rangle.
  \end{displaymath}
  Since $A-A^2$ obviously is nilpotent, $M_f$ cannot contain another
  idempotent (in addition to $I$, $E$ and $I-E$). Let $g$ be the additive
  component of $f$ satisfying $\p g = E \p f$. Since
  \begin{displaymath}
    E \p f =
    \begin{pmatrix}
      0 & 0 & 0 \\
      1 & 0 & 1 \\
      1 & 0 & 1
    \end{pmatrix}
    \begin{pmatrix}
      x_2^\pd 2 \\ x_1 x_2 + x_3^\pd 2 \\ x_2 x_3 + x_3^\pd 2
    \end{pmatrix}
    =
    \begin{pmatrix}
      0 \\ (x_2+x_3)^\pd 2 \\ (x_2+x_3)^\pd 2
    \end{pmatrix},
  \end{displaymath}
  it follows that
  \begin{displaymath}
    g = (x_2+x_3)^\pd 3 \in \K [x_2+x_3]^{DP}.
  \end{displaymath}
  The other additive component is therefore
  \begin{displaymath}
    h = f-g = (x_1-x_3) x_2^\pd 2 - x_2^\pd 3 \in \K [x_1-x_3, x_2]^{DP}.
  \end{displaymath}
  This verifies that $f = g+h$ is a regular splitting of $f$, as promised by
  theorem \ref{thm:regsplit}. Furthermore, $M_g^E = M_f E$ and $M_h^{I-E} =
  M_f (I-E)$. Since $M_h^{I-E}$ contains a nilpotent matrix, we will in
  chapter 4 see that $h$ has a degenerate splitting.

  We also see that $G_f = \langle f,g, x_2^\pd 3 \rangle = \langle (x_1-x_3)
  x_2^\pd 2, x_2^\pd 3, (x_2+x_3)^\pd 3 \rangle$. And we note that $f \sim x_1
  x_2^\pd 2 + x_2^\pd 3 + x_3^\pd 3$, and $f \sim x_1 x_2^\pd 2 + x_3^\pd 3$
  as long as $\chr \K \ne 3$.
\end{exmp}

In remark \ref{rem:XM} we claimed that results concerning $M_f$ often
corresponds to results about $I(M)$. In this section we have seen how
idempotents in $M_f$ are related to regular splittings of $f$. We end this
section with a result showing how $I(M)$ and $X(M)$ ``splits'' if $M$ contains
a complete set of orthogonal idempotents. Recall that
\begin{displaymath}
  I(M) = \sum_{A \in M} I_2 (\p\: A\p) \quad \text{ and } \quad X(M) = \bigl\{
  f \in \cR \,\big\vert\, \ann_R f \supseteq I(M) \bigr\}.
\end{displaymath}

% We have already defined $I_R (M) = \sum_{A \in M} I_2 (\p\: A\p)$ and $X_\cR
% (M) = \{ f \in \cR \suchthat \ann_R f \supseteq I_R (M) \}$. We want to
% extend these to subrings. There is a unique inclusion $\dsum_{i=1}^n \Mat_\K
% (s_i, s_i) \hookrightarrow \Mat_\K (r,r)$ mapping $I \in \Mat_\K (s_i, s_i)$
% to $E_i$. In fact, $M$ lies in the image of this map, because $\{ E_i \}$ is
% a coid in $M$. Let $M'_i \subseteq \Mat_\K (s_i, s_i)$ be the preimage of
% $M_i$. Then both $I_{S_i} (M'_i)$ and $X_{\cS_i} (M'_i)$ are defined. And it
% follows that $I_R (M) = \bigl( \sum_{i<j} V_i V_j + \sum_i I_{S_i} (M'_i)
% \bigr) R$ and $X_\cR (M) = \dsum_{i=1}^n X_{\cS_i} (M'_i)$. Also $\bigl( R/
% I_R (M) \bigr)_d = \dsum_{i=1}^n \bigl( S_i / I_{S_i} (M'_i) \bigr)_d$ for
% all $d>0$, since $\bigl( R/ \sum_{i<j} V_i V_j R \bigr)_d = \dsum_{i=1}^n
% (S_i)_d$.

\begin{prop} \label{prop:regularM}
  Let $M \subseteq \Mat_\K (r,r)$ be a commutative subalgebra containing the
  identity matrix $I$. Let $\{ E_1, \dots, E_n \}$ be a complete set of
  orthogonal idempotents in $M$. For every $i$, let $M_i = M E_i$, $V_i = \{
  v^\T \p \suchthat v \in \im E_i^\T \} \subseteq R_1$, $S_i = \K [V_i]$ and
  $\cS_i = \K [ \{ v^\T x \suchthat v \in \im E_i \} ]^{DP} \iso S_i^*$.
  Define $I_{S_i} (M) = S_i \isect I(M)$ and $X_{\cS_i} (M) = \cS_i \isect
  X(M)$. Then
  \begin{enumerate}
    \setlength{\itemsep}{3pt}
    \setlength{\parskip}{0pt}
    \renewcommand{\theenumi}{\alph{enumi}}
    \renewcommand{\labelenumi}{\normalfont(\theenumi)}
  \item $I_R (M) = \bigl( \smash{\sum_{i<j}} R V_i V_j \bigr) \dsum \bigl(
    \dsum_{i=1}^n I_{S_i} (M_i) \bigr)$,
  \item $\bigl( R/ I_R (M) \smash{\bigr)}_d = \dsum_{i=1}^n \bigl( S_i /
    I_{S_i} (M_i) \smash{\bigr)}_d$ for all $d>0$, and
  \item $X_\cR (M)_d = \dsum_{i=1}^n X_{\cS_i} (M_i)_d$ for all $d>0$.
  \end{enumerate}
\end{prop}

% This cannot be made into an equivalence, since $I_R (M') = I_R (M)$ does not
% imply $M'=M$.

\begin{proof}
  Note that $R_1 = \dsum_{i=1}^n V_i$ by proposition \ref{prop:eigen}d. This
  implies
  \begin{equation*}
    R_d = \Bigl( \sum_{i<j} R_{d-2} V_i V_j \Bigr) \dsum \Bigl( \dsum_{i=1}^n
    V_i^d \Bigr)
  \end{equation*}
  for all $d \ge 1$. Since $V_i^d = (S_i)_d$, the degree $d$ part of $S_i$, we
  get
  \begin{equation*}
    \Bigl( R \Big/ \sum_{i<j} R V_i V_j \Bigr)_d = \dsum_{i=1}^n (S_i)_d
  \end{equation*}
  for all $d>0$. Thus (b) follows immediately from (a).

  % In particular, $R_2 = \dsum_{i \le j} V_i V_j$, and in general
  % $R_d = \dsum_{i_1 \le \dots \le i_d} V_{i_1} \dots V_{i_d}$.

  Since $M = \dsum_{i=1}^n M_i$ (proposition \ref{prop:eigen}c), it follows by
  definition that
  \begin{equation} \label{eq:IM1}
    I_R (M) = I(M) = \sum_{A \in M} I_2 (\p\: A\p) = \sum_{i=1}^n \sum_{A \in
    M_i} I_2 (\p\: A\p).
  \end{equation}
  Fix $i$, and let $A \in M_i$. Putting $(A,B) = (A,E_i)$ into equation
  \eqref{eq:AB} proves that
  \begin{equation*}
    I_2 (\p\: A\p) \subseteq I_2 (\p\: E_i\p) + I_2 (E_i\p\: A\p),
  \end{equation*}
  and putting $(A,B) = (E_i,A)$ gives
  \begin{equation*}
    I_2 (E_i\p\: A\p) \subseteq I_2 (\p\: A\p).
  \end{equation*}
  Since $E_i \in M_i$, this shows that
  \begin{equation} \label{eq:IM2}
    \sum_{A \in M_i} I_2 (\p\: A\p) = I_2 (\p\: E_i\p) + \sum_{A \in M_i} I_2
    (E_i\p\: A\p).
  \end{equation}

  Note that $(E_i\p)_k \in V_i$ and $((I-E_i)\p)_k \in \sum_{j \ne i} V_j$ for
  all $k$. Hence the minors of $(\p\: E_i\p)$ satisfy
  \begin{equation} \label{eq:IM3}
    \begin{vmatrix}
      \p_k & (E_i\p)_k \\
      \p_l & (E_i\p)_l
    \end{vmatrix}
    =
    \begin{vmatrix}
      ((I-E_i)\p)_k & (E_i\p)_k \\
      ((I-E_i)\p)_l & (E_i\p)_l
    \end{vmatrix}
    % = ((I-E_i)\p)_k (E_i\p)_l - ((I-E_i)\p)_l (E_i\p)_k
    \in \sum_{j \ne i} V_i V_j.
  \end{equation}
  For all $u,v \in \K^r$ and $j \ne i$ we have (cf. equation \eqref{eq:UV})
  \begin{equation*}
    \sum_{k,l=1}^r (E_j^\T u)_k v_l
    \begin{vmatrix}
      \p_k & (E_i \p)_k \\
      \p_l & (E_i \p)_l
    \end{vmatrix}
    =
    \begin{vmatrix}
      u^\T E_j \p & u^\T E_j E_i \p \\
      v^\T \p & v^\T E_i \p
    \end{vmatrix}
    = (u^\T E_j \p) \cdot (v^\T E_i \p)
  \end{equation*}
  because $E_j E_i = 0$. Since $\{ v^\T E_i \p \suchthat v \in \K^r \} = V_i$,
  this means that $I_2 (\p\: E_i\p)$ contains every product $V_i V_j$, $j \ne
  i$. Hence $I_2 (\p\: E_i\p) = \sum_{j \ne i} R V_i V_j$ for all $i$ by
  equation \eqref{eq:IM3}. Therefore,
  \begin{equation} \label{eq:IM4}
    \sum_{i=1}^n I_2 (\p\: E_i\p) = \sum_{i<j} R V_i V_j.
  \end{equation}

% Can this be confusing? In one we $i$ is fixed and we take the sum over all
% $j$ such that $j \ne i$, in the other no-one is fixed and we take the sum
% over all $i$ and $j$ such that $i < j$.

% $\sum_{i=1}^n I_2 (\p\: E_i\p) = \sum_{j \ne i} I_2 (E_j\p\: E_i\p)$.

  Combining equations \eqref{eq:IM1}, \eqref{eq:IM2} and \eqref{eq:IM4}, we
  have proven so far that
  \begin{equation*}
    I(M) = \sum_{i<j} R V_i V_j + \sum_{i=1}^n \sum_{A \in M_i} I_2 (E_i\p\:
    A\p).
  \end{equation*}
  If $A \in M_i$, then $A\p = AE_i\p$, and therefore $I_2 (E_i\p\: A\p)_2
  \subseteq V_i^2 \subseteq S_i$. Hence
  \begin{equation} \label{eq:IM5}
    I(M) = \biggl( \,\sum_{i<j} R V_i V_j \biggr) \dsum \biggl( \dsum_{i=1}^n
    \sum_{A \in M_i} I_2 (E_i\p\: A\p)_2 S_i \biggr),
  \end{equation}
  a direct sum of graded $\K$-vector spaces. What we have proven also shows
  that
  \begin{equation} \label{eq:IM6}
    I(M_i) = \sum_{A \in M_i} I_2 (\p\: A\p) = \biggl( \,\sum_{j \ne i} R V_i
    V_j \biggr) \dsum \biggl( \sum_{A \in M_i} I_2 (E_i\p\: A\p)_2 S_i \biggr)
  \end{equation}
  for all $i$. It follows that
  \begin{equation} \label{eq:IM7}
    I_{S_i} (M_i) = S_i \isect I (M_i) = S_i \isect I(M) = \sum_{A \in M_i}
    I_2 (E_i\p\: A\p)_2 S_i.
  \end{equation}
  With equation \eqref{eq:IM5} this proves (a).

  To prove (c), note for any $i$ and $f \in \cS_i$ that $\ann_R f = (\sum_{j
    \ne i} V_j) \dsum \ann_{S_i} f$. It follows from equations \eqref{eq:IM5},
  \eqref{eq:IM6} and \eqref{eq:IM7} that
  \begin{align*}
    X_{\cS_i} (M_i) & = \bigl\{ f \in \cS_i \,\big\vert\: \ann_R f \supseteq I
    (M_i) \bigr\} \\
    & = \bigl\{ f \in \cS_i \,\big\vert\: \ann_{S_i} f \supseteq I_{S_i} (M_i)
    \bigr\} = X_{\cS_i} (M) \subseteq X(M).
  \end{align*}
  Since $\cS_i \isect \cS_j = \K$ for $i \ne j$, it follows that
  $\dsum_{i=1}^n X_{\cS_i} (M_i)_d \subseteq X(M)_d$ for all $d>0$. To prove
  equality it is enough to show that their dimensions are equal. And this
  follows from (b), since $X_{\cS_i} (M_i)_d = \{ f \in (\cS_i)_d \suchthat Df
  = 0 \:\forall\: D \in I_{S_i} (M_i)_d \}$ (by lemma \ref{lem:XM}d) implies
  $\dim_\K X_{\cS_i} (M_i)_d = \dim_\K \bigl( S_i / I_{S_i} (M_i)
  \smash{\bigr)}_d$.
\end{proof}

% In fact, for any $A \in M_f^E$, we have $I_2 (\p\: A\p) = I_2 (E\p\: A\p) +
% I_2 ((I-E)\p\: A\p)$, and $I_2 ((I-E)\p\: A\p) \subseteq I_2 ((I-E)\p\: E\p)
% = I_2 (\p\: E\p)$. This is easily seen if we transform $E$ to standard form.
% To prove it in general, one may use equation \eqref{eq:AB}. Putting $(A,B)
% =(E,A)$ proves $I_2 (E\p\: A\p) \subseteq I_2 (\p\: A\p)$. Putting $(A,B) =
% (I-E,A)$ proves $I_2 ((I-E)\p\: A\p) \subseteq I_2 (\p\: A\p)$. Hence $I_2
% (\p\: A\p) = I_2 (E\p\: A\p) + I_2 ((I-E)\p\: A\p)$ since the other
% inclusion follows by linearity. Finally, putting $(A,B) = (A,I-E)$ proves
% $I_2 ((I-E)\p\: A\p) \subseteq I_2 (\p\: (I-E)\p) = I_2 (\p\: E\p)$.

\begin{rem}
  We can give a direct proof of the other inclusion in part (c). By
  definition, $f \in X(M)$ if and only if $M \subseteq M_f$. Let $f \in
  X(M)_d$. Since $\{E_i\} \subseteq M \subseteq M_f$, there exists $g_i \in
  \cS_i$ such that $f = \sum_{i=1}^n g_i$ is a regular splitting by theorem
  \ref{thm:regsplit} ($d=1$ is trivial). Let $D \in I_{S_i} (M_i)$. Then $D
  (g_j) = 0$ for all $j \ne i$ since $D \in (V_i)$, and $D(f) = 0$ since $D
  \in I(M)$. Hence $D (g_i) = 0$. This proves that $I_{S_i} (M_i) \subseteq
  \ann_{S_i} g_i$, i.e. $g_i \in X_{\cS_i} (M_i)_d$ for all $i$.
\end{rem}

\section{Minimal resolutions} \label{section:minimal}

Now that we know how to find all regular splittings of a form $f \in \cR_d$,
we turn to consequences for the graded Artinian Gorenstein quotient $R/ \ann_R
f$. In this section we obtain a minimal free resolution of $R/ \ann_R f$ when
$f$ splits regularly. This allows us to compute the (shifted) graded Betti
numbers of $R/ \ann_R f$.

Fix $n \ge 1$, and let $W_1, \dots, W_n \subseteq \cR_1$ satisfy $\cR_1 =
\dsum_{i=1}^n W_i$. For all $i$ define $\cS^i = \K [W_i]^{DP}$. Note that
$\cR_1 = \dsum_{i=1}^n W_i$ implies $\cR = \cS^1 \otimes_\K \dots \otimes_\K
\cS^n$. For each $i$, let $V_i = (\sum_{j \ne i} W_i)^\perp \subseteq R_1$
and $S^i = \K [V_i] \iso (\cS^i)^*$. Then $R_1 = \dsum_{i=1}^n V_i$, and
therefore $R = S^1 \otimes_\K \dots \otimes_\K S^n$.

\begin{rem}
  Let $s_i = \dim_\K W_i = \dim_\K V_i$, and note that $\sum_{i=1}^n s_i = r$.
  Let
  \begin{displaymath}
    \J_i = \biggl\{ j \in \Z \,\bigg\vert\: \sum_{k<i} s_k < j \le \sum_{k \le
      i} s_k \biggr\}.
  \end{displaymath}
  for all $i$. There is a base change (that is, a homogeneous change of
  variables) of $\cR$ such that $\cS^i = \K [ \{ x_j \suchthat j \in \J_i \}
  ]^{DP}$ for all $i$ (cf. remark \ref{rem:regsplit}). This implies for all
  $i$ that $S^i = \K [ \{ \p_k \suchthat j \in \J_i \} ]$. Note that the
%% CHANGED 2013-07-13: upper index from $r$ to $n$ two places
  subspaces $\{ W_i \}_{i=1}^n$, or equivalently $\{ V_i \}_{i=1}^n$,
  determine and is determined by a unique set of orthogonal idempotents $\{
  E_i \}_{i=1}^n \subseteq \Mat_\K (r,r)$, cf. remark \ref{rem:Escs}. Thus the
  ``rectifying'' base change above corresponds to a simultaneous
%% CHANGED 2013-07-13: upper index from $r$ to $n$
  diagonalization of $\{ E_i \}_{i=1}^n$ as in remark \ref{rem:diagEi}. We
  will not assume that this base change has been made when we state and prove
  our results, but some claims may be easier to understand with this in mind.
\end{rem}

%% CHANGED 2013-07-13: upper index from $r$ to $n$
Let $f = \sum_{i=1}^n g_i$ be a regular splitting with $g_i \in \cS^i_d$, $g_i
\ne 0$, $d>0$. The following result is fundamental to this section, comparing
the ideals $\ann_R (f)$, $\ann_R (g_i)$ and $\ann_{S^i} (g_i)$.

% Let $d>0$ and $f \in \cR_d$. Let $\cR = \cS^1 \otimes_\K \dots \otimes_\K
% \cS^n$ and $R = S^1 \otimes_\K \dots \otimes_\K S^n$ such that $D(g) = 0$
% for all $g \in \cS^j$ and $D \in S^i_e$ with $i \ne j$ and $e>0$. Assume
% that $f = \sum_{i=1}^n g_i$ is a regular splitting of $f$ with $g_i \in
% \cS^i_d$, $g_i \ne 0$.

\begin{lem} \label{lem:annfg}
  With the notation above, the following statements are true.
  \begin{enumerate}
    \setlength{\itemsep}{2pt}
    \setlength{\parskip}{0pt}
    \renewcommand{\theenumi}{\alph{enumi}}
    \renewcommand{\labelenumi}{\normalfont(\theenumi)}

  \item For every $i$ we have $\ann_{S^i} (g_i) = S^i \isect \ann_R (g_i)$
    and
    \begin{enumerate}
      \setlength{\itemsep}{1pt}
      \setlength{\parskip}{0pt}
      \renewcommand{\theenumii}{\roman{enumii}}
      \renewcommand{\labelenumi}{\normalfont(\theenumi)}
    \item $\ann_R (g_i) = \bigl( \sum_{j \ne i} S^j_1 \bigr) \dsum \ann_{S^i}
      (g_i)$ as graded $\K$-vector spaces,
    \item $\ann_R (g_i) = \bigl( \sum_{j \ne i} S^j_1 \bigr) + R \ann_{S^i}
      (g_i)$ as ideals in $R$, and
    \item $R/ \ann_R (g_i) \iso S^i/ \ann_{S^i} (g_i)$.
    \end{enumerate}
  \item There exist nonzero $D_i \in S^i_d$, $i = 1, \dots, n$, such that
    \begin{displaymath}
      \ann_R (f) = \isect_{i=1}^n \ann_R (g_i) + (D_2-D_1, \dots, D_n-D_1).
    \end{displaymath}

  \item We may express $\isect_{i=1}^n \ann_R (g_i)$ as a direct sum of graded
    $\K$-vector spaces;
    \begin{displaymath}
      \isect_{i=1}^n \ann_R (g_i) = \biggl( \sum_{i<j} R S^i_1 S^j_1 \biggr)
      \dsum \biggl( \dsum_{i=1}^n \ann_{S^i} (g_i) \biggr),
    \end{displaymath}

  \item or as a sum of ideals in $R$;
    \begin{displaymath}
      \isect_{i=1}^n \ann_R (g_i) = \sum_{i<j} R S^i_1 S^j_1 + \sum_{i=1}^n
      R \ann_{S^i} (g_i).
    \end{displaymath}

  \item The Hilbert function $H$ of $R/ \ann_R (f)$ satisfies
    \begin{displaymath}
      H \bigl( R/ \ann_R (f) \bigr) = \sum_{i=1}^n H \bigl( S^i/ \ann_{S^i}
      (g_i) \bigr) - (n-1) \bigl( \delta_0 + \delta_d \bigr),
    \end{displaymath}
    where $\delta_e$ is $1$ in degree $e$ and zero elsewhere.
  \end{enumerate}
\end{lem}

\begin{proof}
  By definition, $\ann_{S^i} (g_i) = \{ D \in S^i \suchthat D (g_i) = 0 \}$,
  which clearly equals $S^i \isect \ann_R (g_i)$. By construction, $D (g_i) =
  0$ for all $D \in S^j_1$, $j \ne i$. Hence $\bigl( \sum_{j \ne i} S^j_1
  \bigr) \subseteq \ann_R (g_i)$. Since $R/ \bigl( \sum_{j \ne i} S^j_1 \bigr)
  = S^i$, we get
  \begin{displaymath}
    \ann_R (g_i) = \Bigl( \sum_{j \ne i} S^j_1 \Bigr) \dsum \ann_{S^i} (g_i)
  \end{displaymath}
  as graded $\K$-vector subspaces of $R$. The rest of (a) follows immediately.

  Consider the regular splitting $f = \sum_{i=1}^n g_i$. By lemma
  \ref{lem:basic} we have
  \begin{displaymath}
    \ann_R (f)_e = \isect_{i=1}^n \ann_R (g_i)_e \text{ for all } e<d.
  \end{displaymath}
  Thus the ideals $\ann_R (f)$ and $\isect_{i=1}^n \ann_R (g_i)$ are equal in
  every degree $e \ne d$. In degree $d$ the right-hand side has codimension
  $n$ (since the $g_i$ are linearly independent), hence $\ann_R(f)$ must have
  $n-1$ extra generators of degree $d$. If we choose $D_i \in S^i_d$ such that
  $D_1 (g_1) = \dots = D_n (g_n) \ne 0$, then clearly
  \begin{displaymath}
    \ann_R (f) = \isect_{i=1}^n \ann_R (g_i) + (D_2-D_1, \dots, D_n-D_1).
  \end{displaymath}

  By (a) we have $\sum_{i<j} R S^i_1 S^j_1 \subseteq \ann_R (g_k)$ for all
  $k$. Note that
  \begin{displaymath}
    R_e = \biggl( \sum_{i<j} R_{e-2} S^i_1 S^j_1 \biggr) \dsum \biggl(
    \dsum_{i=1}^n S^i_e \biggr) \text{ for all } e>0.
  \end{displaymath}
  Because $\bigl( \isect_{i=1}^n \ann_R (g_i) \bigr) \isect S^j = \ann_{S^j}
  (g_j)$, this implies both (c) and (d). Combining (b) and (c), it follows
  that $(R/ \ann_R f)_e = \dsum_{i=1}^n (S^i/ \ann_{S^i} g_i)_e$ for all $e
  \ne 0,d$, proving (e).
\end{proof}

% In particular, if $d \ge 2$, then $\ann_R (f)_1 = 0$ if and only if
% $\ann_{S^i} (g_i)_1 = 0$ for all $i$.

Most of the time in this section we will assume $n=2$. This makes it easier to
state and prove our results. Let $\cS = \cS^1$ and $\cT = \cS^2$. (Of course,
we may think of $\cT$ as $\cT = \cS^2 \otimes_\K \dots \otimes_\K \cS^n$,
reaching $n>2$ by induction.) Similarly, let $S = S^1$ and $T = S^2$, and $s =
s_1$ and $t = s_2 = r-s$. Hence $\cR = \cS \otimes_\K \cT$ and $R = S
\otimes_\K T$. We will often compare ideals of $R$, $S$ and $T$, and some
words are in order.

% When $S,T \subseteq R$ and we write $R = S \otimes_\K T$, we do not just
% mean $R \iso S \otimes_\K T$ abstractly. We mean that the natural map $S
% \otimes_\K T \to R$ given by $(s,t) \mapsto st$ is an isomorphism. Cf. $V =
% \dsum_i V_i$. I think this is equivalent to $R = S \cdot T$ and $S \isect T
% = \K$ since they are $\K$-algebras.

Given a homogeneous ideal $I \subseteq S$, the inclusion $S \subseteq R$ makes
$I$ into a graded $\K$-vector subspace of $R$. If $J \subseteq T$ is another
homogeneous ideal, then $IJ$ is the $\K$-vector subspace of $R$ spanned by all
products $ij$ with $i \in I$ and $j \in J$. Since $IJ$ automatically is closed
under multiplication from $R$, it is equal to the ideal in $R$ generated by
all products $ij$. In particular, $IT$ is simply the ideal in $R$ generated by
$I$. There are many ways to think of and write this ideal, including
\begin{displaymath}
  (I) = R \cdot I = I \otimes_S R = I \otimes_S (S \otimes_\K T) = I
  \otimes_\K T = IT.
\end{displaymath}
Similarly, $IT \cdot SJ = (I \otimes_S R) \otimes_R (R \otimes_T J) = I
\otimes_\K J = IJ = (IJ)$. We have used here a property of tensor products
often called \emph{base change}, cf. \cite[proposition A2.1]{Eis}. Note that
$IT \isect SJ = IT \cdot SJ = IJ$. It follows that
\begin{equation} \label{eq:IJ}
  I_1 J_1 \isect I_2 J_2 = (I_1 \isect I_2)(J_1 \isect J_2)
\end{equation}
for all homogeneous ideals $I_1, I_2 \subseteq S$ and $J_1, J_2 \subseteq T$.

% \begin{lem} %\label{lem:isect}
%   Given $\K$-vector subspaces $I_1, I_2 \subseteq S$ and $J_1, J_2 \subseteq
%   T$, we have
%   \begin{displaymath}
%     I_1 J_1 \isect I_2 J_2 = (I_1 \isect I_2)(J_1 \isect J_2).
%   \end{displaymath}
% \end{lem}
%
% % The proof basically exhibits explicit bases.
%
% \begin{proof}
%   We may assume that $S = I_1 + I_2$ and $T = J_1 + J_2$. Let $U_1 = I_1
%   \isect I_2$. Choose $U_2, U_3 \subseteq S$ such that $I_1 = U_1 \dsum U_2$
%   and $I_2 = U_1 \dsum U_3$. Then $S = \dsum_i U_i$. Similarly, let $V_1 =
%   J_1 \isect J_2$, and choose $V_2, V_3 \subseteq T$ such that $J_1 = V_1
%   \dsum V_2$ and $J_2 = V_1 \dsum V_3$. Thus $T = \dsum_i V_i$. Clearly,
%   $I_1 J_1 = U_1 V_1 \dsum U_1 V_2 \dsum U_2 V_1 \dsum U_2 V_2$. Similarly
%   for $I_2 J_2$, and it follows that $I_1 J_1 \isect I_2 J_2 = U_1 V_1$.
% \end{proof}

% Normally, when we tensor $0 \to I \to S \to S/I \to 0$ with an $S$-module
% $M$, we only get a right exact sequence $M \otimes_S I \to M \to M \otimes_S
% S/I \to 0$. The image of $M \otimes_S I$ in $M$ is clearly the submodule
% $IM$, hence $M \otimes_S S/I \iso M/IM$. But when we tensor with $T$ over
% $\K$, we even get $I \otimes_\K T = IT$, since $T$ is free over $\K$.

Fix $d \ge 1$, and let $g \in \cS_d$ and $h \in \cT_d$. We want to point out
what lemma \ref{lem:annfg} says in this simpler situation. Note that the ideal
$\ann_S (g)$ in $S$ generates the ideal $T \ann_S (g)$ in $R$. Let
\begin{displaymath}
  \m_S = (S_1) \subseteq S \quad \text{ and } \quad \m_T = (T_1) \subseteq T
\end{displaymath}
be the maximal homogeneous ideals in $S$ and $T$, respectively. Since $T =
\m_T \dsum \K$, we get $R = S \m_T \dsum S$. Lemma \ref{lem:annfg} tells us
that $\ann_R (g) = S \m_T \dsum \ann_S (g)$ and $\ann_R (g) = S \m_T + T
\ann_S (g)$. Furthermore,
\begin{equation} \label{eq:gh}
  \ann_R (g) \isect \ann_R (h) = \m_S \m_T + T \ann_S (g) + S \ann_T (h)
\end{equation}
as ideals in $R$, and there exist $D \in S_d$ and $E \in T_d$ such that
\begin{equation} \label{eq:fgh}
  \ann_R (f) = \ann_R (g) \isect \ann_R (h) + (D-E).
\end{equation}
We will use these equations to calculate the minimal resolution of $R/ \ann_R
(f)$. They involve products of ideals, and we start with the following lemma.

\begin{lem} \label{lem:IJ}
  Given homogeneous ideals $I \subseteq S$ and $J \subseteq T$, let $\F$ and
  $\G$ be their resolutions
  \begin{align*}
    \F & : 0 \to F_s \xrightarrow{\varphi_s} \dots \xrightarrow{\varphi_2} F_1
    \xrightarrow{\varphi_1} I \to 0, \\
    \G & : 0 \to G_t \xrightarrow{\psi_t} \dots \xrightarrow{\psi_2} G_1
    \xrightarrow{\psi_1} J \to 0,
  \end{align*}
  where the $F_i$'s are free $S$-modules and the $G_i$'s are free $T$-modules.
  Then the tensor complex
  \begin{displaymath}
    \F \otimes_\K \G : 0 \to H_{s+t-1} \xrightarrow{\eta_{s+t-1}} \dots
    \xrightarrow{\eta_2} H_1 \xrightarrow{\eta_1} IJ \to 0
  \end{displaymath}
  is exact, hence a free resolution of $IJ$ in $R = S \otimes_\K T$, and
  minimal if both $\F$ and $\G$ are minimal.
\end{lem}

The definition of the tensor complex can be found in \cite[section 17.3]{Eis}.
Its construction gives $H_i = \dsum_{j=1}^{i} F_j \otimes_\K G_{i+1-j}$ for
all $i \ge 1$. Note that this is a free $R$-module. The maps $\eta_i : H_i \to
H_{i-1}$ for $i>1$ are given by
\begin{displaymath} % alternative: C=-25pt
  \xymatrix@R=20pt@C=-20pt{
    F_1 \otimes_\K G_i \ar[dr]_(.4){\id_{F_1} \otimes \psi_i} & \oplus & F_2
    \otimes_\K G_{i-1} \ar[dl]^(.4){\varphi_2 \otimes \id_{G_{i-1}}} &
    \qquad \oplus \quad & \quad \dots \quad & \quad \oplus \qquad & F_{i-1}
    \otimes_\K G_2 \ar[dr]_(.4){(-1)^i \id_{F_{i-1}} \otimes \psi_2} &
    \oplus & F_i \otimes_\K G_1 \ar[dl]^(.4){\varphi_i \otimes \id_{G_1}} \\
    & F_1 \otimes_\K G_{i-1} & \oplus && \dots && \oplus & F_{i-1}
    \otimes_\K G_1 &
  }
\end{displaymath}
that is, $\eta_i = \dsum_{j=1}^{i-1} \left( \varphi_{j+1} \otimes
  \id_{G_{i-j}} - (-1)^j \id_{F_j} \otimes\: \psi_{i-j+1} \right)$, and
$\eta_1 = \varphi_1 \otimes \psi_1$.

% For $i>1$ the map $\eta_i : H_i \to H_{i-1}$ is $\eta_i = \dsum_{j=1}^{i-1}
% \left( \varphi_{j+1,i-j} - (-1)^j \psi_{j,i-j+1} \right)$, a direct sum of
% maps into each direct summand $F_j \otimes_\K G_{i-j}$ of its codomain
% $H_{i-1}$, where $\varphi_{ij} = \varphi_i \otimes \id_{G_j}$ and $\psi_{ij}
% = \id_{F_i} \otimes\: \psi_j$.

\begin{proof}[Proof of lemma \ref{lem:IJ}:]
  The complex is exact since we get it by tensoring over $\K$, and $I$ and $J$
  are free over $\K$, hence flat. It is trivially minimal when $\F$ and $\G$
  are minimal by looking at the maps $\eta_i$.
\end{proof}

Note that $\F \otimes_S R = \F \otimes_\K T$ is a resolution of $I \otimes_S R
= IT$, the ideal in $R$ generated by $I$. Similarly, $R \otimes_T \G$ is a
resolution of $SJ$. Furthermore, $(\F \otimes_S R) \otimes_R (R \otimes_T \G)
= \F \otimes_\K \G$.

\begin{exmp} \label{ex:MN}
  Let
  \begin{align*}
    \M & : 0 \to M_s \to \dots \to M_1 \to \m_S \to 0, \\
    \N & : 0 \to N_t \to \dots \to N_1 \to \m_T \to 0
  \end{align*}
  be the Koszul resolutions of $\m_S \subseteq S$ and $\m_T \subseteq T$,
  respectively. We know that $M_k = \binom{s}{k} S(-k)$ and $N_k =
  \binom{t}{k} T(-k)$ for all $k$. If we apply lemma \ref{lem:IJ} to $I =
  \m_S$ and $J = \m_T$, we get a graded minimal free resolution
  \begin{displaymath}
    \MN = \M \otimes_\K \N : 0 \to MN_{s+t-1} \to \dots \to MN_1 \to \m_S \m_T
    \to 0
  \end{displaymath}
  of $\m_S \m_T \subseteq R = S \otimes_\K T$. Here $MN_k = \dsum_{i=1}^k M_i
  \otimes_\K N_{k+1-i}$ for all $k>0$. Hence $MN_k = \nu_k R(-k-1)$ where
  \begin{displaymath}
    \nu_k = \sum_{i=1}^k \binom{s}{i} \binom{t}{k+1-i} = \binom{s+t}{k+1} -
    \binom{s}{k+1} - \binom{t}{k+1}.
  \end{displaymath}
  This agrees with the Eagon-Northcott resolution of
  \begin{displaymath}
    I_2
    \begin{pmatrix}
      \p_1 & \dots & \p_s & 0 & \dots & 0 \\
      0 & \dots & 0 & \p_{s+1} & \dots & \p_{s+t}
    \end{pmatrix}.
  \end{displaymath}
\end{exmp}

We chose to write lemma \ref{lem:IJ} in terms of ideals $I \subseteq S$ and $J
\subseteq T$ because this is how we will use it most of the time. Of course,
the result is true more generally. Indeed, if $\F$ and $\G$ are resolutions of
an $S$-module $M$ and a $T$-module $N$, respectively, then the tensor complex
$\F \otimes_\K \G$ is a resolution of $M \otimes_\K N$, with the same proof.
We will use this is in the next lemma.

\begin{lem} \label{lem:betti}
  Let $I \subseteq S$ be a homogeneous ideal, and let $I' = S \m_T + IT
  \subseteq R$. Denote the shifted graded Betti numbers of $S/I$ and $R/I'$ by
  $\hat{\beta}^I_{ij}$ and $\hat{\beta}^{I'}_{ij}$, respectively. Then for all
  $j,k \ge 0$, we have
  \begin{displaymath}
    \hat{\beta}^{I'}_{kj} = \sum_{i=0}^k \binom{t}{k-i} \hat{\beta}^I_{ij}.
  \end{displaymath}
\end{lem}

% Note that $R/ S \m_T \iso S$, hence $R/I' \iso S/I$.

% More generally, the shifted graded Betti numbers of $IT + SJ$ are
% \begin{displaymath}
%   \hat{\beta}_{kj} = \sum_{i=0}^k \sum_{l=0}^j \hat{\beta}^I_{i,j-l}
%   \hat{\beta}^J_{k-i,l}.
% \end{displaymath}

\begin{proof}
  The proof rests upon the following observation. If $I \subseteq S$ and $J
  \subseteq T$ are ideals, then $S/I \otimes_\K T/J \iso R/ (IT + SJ)$.
  Indeed,
  \begin{align*}
    S/I \otimes_\K T/J & = S/I \otimes_S (S \otimes_\K T/J) = S/I \otimes_\K
    R/SJ \\
    & = (S/I \otimes_S R) \otimes_R R/SJ = R/IT \otimes_R R/SJ = R/(IT + SJ).
  \end{align*}
  It follows that we may compute a resolution of $R/(IT+SJ)$ as the tensor
  complex of the resolutions of $S/I$ and $T/J$. We do this with $J = \m_T$.

  Let $\F$ and $\N$ be the graded minimal free resolutions of $S/I$ and
  $T/\m_T$, respectively, cf. example \ref{ex:MN}. That is,
  \begin{align*}
    \F & : 0 \to F_s \to \dots \to F_1 \to F_0 \to S/I \to 0, \\
    \N & : 0 \to N_t \to \dots \to N_1 \to N_0 \to T/\m_T \to 0,
  \end{align*}
  with $F_i = \dsum_{j \ge 0} \hat{\beta}^I_{ij} \, S(-i-j)$ and $N_i =
  \tbinom{t}{i} \, T(-i)$ for all $i \ge 0$.

  The tensor complex $\F \otimes_\K \N$ gives a graded minimal free resolution
  \begin{displaymath}
    \HH : 0 \to H_{s+t} \to \dots \to H_1 \to H_0
  \end{displaymath}
  of $R/(S \m_T + IT) = R/I'$, where for all $k \ge 0$ we have
  \begin{displaymath}
    H_k = \dsum_{i=0}^k F_i \otimes_\K N_{k-i} = \dsum_{i=0}^k \dsum_{j \ge 0}
    \binom{t}{k-i} \hat{\beta}^I_{ij} \, R(-k-j).
  \end{displaymath}
  The result follows by reading off the Betti numbers from this equation.
\end{proof}

% Alternatively, we could also use the exact sequence
% \begin{displaymath}
%   0 \to IJ \to IT \dsum SJ \to IT + SJ \to 0,
% \end{displaymath}
% and proceed as in proposition \ref{prop:IJ}, step 1.

Since $\ann_R (g) = S \m_T + T \ann_S (g)$, we may use this lemma to compare
the (shifted) graded Betti numbers of $R/\ann_R g$ and $S/\ann_S g$. In the
next two results we use the short exact sequence
\begin{displaymath}
  0 \to I \isect J \to I \dsum J \to I+J \to 0
\end{displaymath}
and the mapping cone construction (cf. \cite[appendix A3.12]{Eis}) several
times.

\begin{prop} \label{prop:IJ}
  Let $I \subseteq S$ and $J \subseteq T$ be homogeneous ideals, and let
  $\m_S$ and $\m_T$ be the maximal homogeneous ideals in $S$ and $T$,
  respectively. Assume that $I_1 = J_1 = 0$. Let $\F$ and $\G$ be graded
  minimal free resolutions
  \begin{align*}
    \F & : 0 \to F_s \xrightarrow{\varphi_s} \dots \xrightarrow{\varphi_2} F_1
    \xrightarrow{\varphi_1} I \to 0, \\
    \G & : 0 \to G_t \xrightarrow{\psi_t} \dots \xrightarrow{\psi_2} G_1
    \xrightarrow{\psi_1} J \to 0.
  \end{align*}
  Denote the shifted graded Betti numbers of $S/I$ and $T/J$ by
  $\hat{\beta}^I_{ij}$ and $\hat{\beta}^J_{ij}$. Then $\m_S \m_T + IT + SJ
  \subseteq R = S \otimes_\K T$ has a graded minimal free resolution
  \begin{displaymath}
    \HH : 0 \to H_r \to \dots \to H_1 \to \m_S \m_T + IT + SJ \to 0
  \end{displaymath}
  where $r = s+t$ and
  \begin{displaymath}
    H_k = \nu_k \, R(-k-1) \dsum \left( \dsum_{j \ge 0} \sum_{i=1}^k
      \left( \binom{t}{k-i} \hat{\beta}^I_{ij} + \binom{s}{k-i}
        \hat{\beta}^J_{ij} \right) R(-k-j) \right)
  \end{displaymath}
  for all $k>0$. Here $\nu_k = \binom{r}{k+1} - \binom{s}{k+1} -
  \binom{t}{k+1}$.
\end{prop}

\begin{proof}
  Remember, by definition of the shifted graded Betti numbers, we have
  \begin{displaymath}
    F_i = \dsum_{j \ge 0} \hat{\beta}^I_{ij} \: S(-i-j) \quad \text{ and }
    \quad G_i = \dsum_{j \ge 0} \hat{\beta}^J_{ij} \: T(-i-j)
  \end{displaymath}
  for every $i$. We will construct the minimal resolution in two similar
  steps.

% \medskip
  \noindent\textit{Step 1.} Note that $IT \isect \m_S \m_T = (I \isect
  \m_S) (T \isect \m_T) = I \m_T$ by equation \eqref{eq:IJ}. This gives us a
  short exact sequence
  \begin{equation} \label{eq:seq1}
    0 \to I \m_T \to IT \dsum \m_S \m_T \to \m_S \m_T + IT \to 0.
  \end{equation}

  Let $\M$ and $\N$ be the Koszul resolutions of $\m_S \subseteq S$ and $\m_T
  \subseteq T$, respectively, as in example \ref{ex:MN}. By lemma \ref{lem:IJ}
  we have four minimal resolutions;
  \begin{align*}
    \F' = \F \otimes_\K T & : 0 \to F'_s \xrightarrow{\varphi_s} \dots
    \xrightarrow{\varphi_2} F'_1 \xrightarrow{\varphi_1} IT \to 0, \\
    \G' = S \otimes_\K \G & : 0 \to G'_t \xrightarrow{\psi_t} \dots
    \xrightarrow{\psi_2} G'_1 \xrightarrow{\psi_1} SJ \to 0, \\
    \F'' = \F \otimes_\K \N & : 0 \to F''_{s+t-1} \xrightarrow{\zeta_{s+t-1}}
    \dots \xrightarrow{\zeta_2} F''_1 \xrightarrow{\zeta_1} I \m_T \to 0, \\
    \G'' = \M \otimes_\K \G & : 0 \to G''_{s+t-1} \xrightarrow{\xi_{s+t-1}}
    \dots \xrightarrow{\xi_2} G''_1 \xrightarrow{\xi_1} \m_S J \to 0.
  \end{align*}
  The free modules in the first resolution are $F'_i = F_i \otimes_\K T =
  \dsum_{j \ge 0} \hat{\beta}^I_{ij} R(-i-j)$, and we identify the map
  $\varphi_i \otimes \id_T$ with $\varphi_i$ since they are given by the same
  matrix. Similarly, for the second resolution, we have $G'_i = S \otimes_\K
  G_i = \dsum_{j \ge 0} \hat{\beta}^J_{ij} R(-i-j)$. The modules in the third
  and fourth resolution satisfy
  \begin{align*}
    F''_{k-1} & = \dsum_{i=1}^{k-1} F_i \otimes_\K N_{k-i} \\
    & = \dsum_{i=1}^{k-1} \left( \Bigl( \dsum_{j \ge 0} \hat{\beta}^I_{ij} \,
    S(-i-j) \Bigr) \otimes_\K \binom{t}{k-i} \, T(-k+i) \right) \\
    & = \dsum_{j \ge 0} \biggl( \: \sum_{i=1}^{k-1} \binom{t}{k-i}
    \hat{\beta}^I_{ij} \biggr) \, R(-k-j),
  \end{align*}
  and similarly, $G''_{k-1} = \dsum_{j \ge 0} \left( \sum_{i=1}^{k-1}
    \tbinom{s}{k-i} \hat{\beta}^J_{ij} \right) R(-k-j)$.

  By tensoring the exact sequence $0 \to \m_T \to T \to T/\m_T \to 0$ with
  $I$, we get a short exact sequence
  \begin{displaymath}
    0 \to I \m_T \to I T \to I \otimes_\K T/\m_T \to 0.
  \end{displaymath}
  We need to lift the inclusion $I \m_T \subseteq IT$ to a map of complexes
  $\F'' \to \F'$. This is easily achieved by defining the map $F''_i \to F'_i
  = F_i \otimes_\K T$ to be $\id_{F_i} \otimes \psi_1$ on the summand $F_i
  \otimes_\K N_1$, and zero on all other direct summands of $F''_i$. The
  mapping cone construction now gives a resolution $\cdots \to F'_3 \dsum
  F''_2 \to F'_2 \dsum F''_1 \to F'_1$ of $I \otimes_\K T/\m_T$ that actually
  equals the tensor complex associated to $I \otimes_\K T/\m_T$ (similar to
  lemma \ref{lem:IJ}). It is obviously minimal by looking at the maps.

  Next we lift the inclusion $I \m_T \subseteq \m_S \m_T$ to a map of
  complexes $\F'' \to \MN$. By looking at the degrees of these maps, we see
  that they must be minimal when $I_1 = 0$, that is, when $I$ has no linear
  generators. Indeed, one such lift is
  \begin{displaymath}
    \bar{\pi}_i = \dsum_{j=1}^i \pi_j \otimes \id : \dsum_{j=1}^i F_j
    \otimes_\K N_{i+1-j} \to \dsum_{j=1}^i M_j \otimes_\K N_{i+1-j},
  \end{displaymath}
  where $\pi$ is a lift of $I \subseteq \m_S$ to a map of complexes $\F \to
  \M$.

  Thus we can lift the map $I \m_T \hookrightarrow IT \dsum \m_S \m_T$, $z
  \mapsto (z,-z)$, in the exact sequence \eqref{eq:seq1} to a map $(\id
  \otimes \psi_1) \dsum (-\bar{\pi})$ of complexes $\F'' \to \F' \dsum \MN$.
  The mapping cone construction now gives a minimal free resolution
  \begin{displaymath}
    \HH' : 0 \to H'_{s+t} \to \dots \to H'_1
  \end{displaymath}
  of $\m_S \m_T + IT$, where
  \begin{displaymath}
    H'_k = MN_k \dsum F'_k \dsum F''_{k-1} = \nu_k \, R(-k-1) \dsum
    \left( \dsum_{j \ge 0} \sum_{i=1}^k \binom{t}{k-i} \hat{\beta}^I_{ij} \:
      R(-k-j) \right)
  \end{displaymath}
  for all $k \ge 1$. This concludes the first step.

  \medskip\noindent\textit{Step 2.} We notice that $\m_S \m_T + IT \subseteq
  \m_S T$, and therefore
  \begin{displaymath}
    \m_S J \subseteq (\m_S \m_T + IT) \isect SJ \subseteq \m_S T \isect SJ =
    \m_S J.
  \end{displaymath}
  Hence $(\m_S \m_T + IT) \isect SJ = \m_S J$, and we have a short exact
  sequence
  \begin{equation} \label{eq:seq2}
    0 \to \m_S J \to (\m_S \m_T + IT) \dsum SJ \to \m_S \m_T + IT + SJ \to 0.
  \end{equation}

  We now proceed as in the first step, getting a lift of the inclusion $\m_S J
  \subseteq SJ$ to a map of complexes $\G'' \to \G'$. To lift the inclusion
  $\m_S J \subseteq \m_S \m_T + TI$ to a map of complexes $\G'' \to \HH'$, we
  take the lift of $\m_S J \subseteq \m_S \m_T$ to $\G'' \to \MN$, as in step
  one, and extend it by zero, since $H'_k = MN_k \dsum F'_k \dsum F''_{k-1}$
  for all $k \ge 1$. And then the mapping cone construction produces a free
  resolution
  \begin{displaymath}
    \HH : 0 \to H_r \to \dots \to H_1 \to \m_S \m_T + IT + SJ \to 0,
  \end{displaymath}
  which is minimal since all maps are minimal. Here $H_k = H'_k \dsum G'_k
  \dsum G''_{k-1}$ is for all $k>0$ equal to
  \begin{displaymath}
    H_k = \nu_k \, R(-k-1) \dsum \left( \dsum_{j \ge 0} \sum_{i=1}^k
      \left( \binom{t}{k-i} \hat{\beta}^I_{ij} + \binom{s}{k-i}
        \hat{\beta}^J_{ij} \right) R(-k-j) \right). \qedhere
  \end{displaymath}
\end{proof}

% As we see from the proof, the conditions $I_1 = J_1 = 0$ are only used to
% make the final resolution minimal. Without them, the resolution still
% exists, but will be non-minimal.

\begin{rem}
  Because $\ann_R (g) \isect \ann_R (h) = \m_S \m_T + T \ann_S (g) + S \ann_T
  (h)$, we will use proposition \ref{prop:IJ} with $I = \ann_S (g)$ and $J =
  \ann_T (h)$ when we calculate the resolution of $\ann_R (f) = \ann_R (g)
  \isect \ann_R (h) + (D-E)$. There is another way to find the resolution of
  $\ann_R (g) \isect \ann_R (h)$, using the sequence
  \begin{displaymath}
    0 \to \ann_R (g) \isect \ann_R (h) \to \ann_R (g) \dsum \ann_R (h) \to
    \m_R \to 0.
  \end{displaymath}

  This is a short exact sequence, and we know the minimal resolutions of the
  middle and right-hand side modules. Since the quotients are Artinian, these
  resolutions all have the ``right'' length. Hence we may dualize the
  sequence, use the mapping cone to construct a resolution of $\Ext_R^{r-1}
  \bigl( \ann_R (g) \isect \ann_R (h), R \bigr)$, and dualize back. Compared
  to the proof of proposition \ref{prop:IJ}, this is done in one step, but the
  resulting resolution is not minimal. Thus more work is needed to find the
  cancelations, and in the end the result is obviously the same.
\end{rem}

% The cancelations should be those coming from $S \m_T \dsum T \m_S \to \m_R$.
% At least, this seems to give us what we must get, since $\m_R = \m_S \m_T
% \dsum \m_S \dsum \m_T$.

We are now ready to find the minimal resolution of $R/ \ann_R f$. Note that we
here use the convention that $\binom{a}{b} = 0$ for all $b<0$ and all $b>a$.

\begin{thm} \label{thm:fgh}
  Let $g \in \cS_d$ and $h \in \cT_d$ for some $d \ge 2$. Let $f = g+h \in
  \cR_d$, and assume that $\ann_S (g)_1 = \ann_T (h)_1 = 0$. Let $\F$ and $\G$
  be graded minimal free resolutions of $\ann_S g \subseteq S$ and $\ann_T h
  \subseteq T$,
  \begin{align*}
    \F & : 0 \to F_s \xrightarrow{\varphi_s} \dots \xrightarrow{\varphi_2}
    F_1 \xrightarrow{\varphi_1} \ann_S g \to 0, \\
    \G & : 0 \to G_t \xrightarrow{\psi_t} \dots \xrightarrow{\psi_2} G_1
    \xrightarrow{\psi_1} \ann_T h \to 0.
  \end{align*}
  Denote the shifted graded Betti numbers of $S/\ann_S g$ and $T/\ann_T h$ by
  $\hat{\beta}^g_{ij}$ and $\hat{\beta}^h_{ij}$, respectively. That is,
  \begin{displaymath}
    F_i = \dsum_{j=0}^d \hat{\beta}^g_{ij} \: S(-i-j) \quad \text{ and } \quad
    G_i = \dsum_{j=0}^d \hat{\beta}^h_{ij} \: T(-i-j)
  \end{displaymath}
  for every $i$. Then $\ann_R f \subseteq R = S \otimes_\K T$ has a graded
  minimal free resolution
  \begin{displaymath}
    \HH : 0 \to H_r \to \dots \to H_1 \to \ann_R f \to 0
  \end{displaymath}
  with $H_r = R(-r-d)$ and
  \begin{align*}
    H_k & = \nu_k \, R(-k-1) \dsum \nu_{r-k} \, R(-d-k+1) \\
    & \qquad \dsum \left( \dsum_{j=1}^{d-1} \biggl( \sum_{i=1}^{s-1}
      \binom{r-s}{k-i} \hat{\beta}^g_{ij} + \sum_{i=1}^{t-1} \binom{r-t}{k-i}
      \hat{\beta}^h_{ij} \biggr) R(-k-j) \right)
  \end{align*}
  for all $0<k<r$. Here $r = s+t$ and $\nu_k = \binom{r}{k+1} - \binom{s}{k+1}
  - \binom{t}{k+1}$.
\end{thm}

\begin{proof}
  Since $\ann_R g \isect \ann_R h = \m_S \m_T + T \ann_S g + S \ann_T h$ by
  equation \eqref{eq:gh} (or lemma \ref{lem:annfg}d), we may apply proposition
  \ref{prop:IJ}. This gives us a graded minimal free resolution
  \begin{displaymath}
    \HH' : 0 \to H'_r \to \dots \to H'_1 \to \ann_R g \isect \ann_R h \to 0
  \end{displaymath}
  with
  \begin{displaymath}
    H'_k = \nu_k \, R(-k-1) \dsum \left( \dsum_{j=0}^d \sum_{i=1}^k
      \left( \binom{t}{k-i} \hat{\beta}^g_{ij} + \binom{s}{k-i}
        \hat{\beta}^h_{ij} \right) R(-k-j) \right)\!.
  \end{displaymath}

  % Clearly, $D-E \notin \ann_R g \isect \ann_R h$ and $(\ann_R g \isect
  % \ann_R h)_e = R_e$ for all $e>d$, implying $(\ann_R g \isect \ann_R h)
  % \isect (D-E) = (D-E) \m_R$.

  By lemma \ref{lem:annfg}b, we may choose $D \in S_d$ and $E \in T_d$ such
  that
  \begin{displaymath}
    \ann_R f = (\ann_R g \isect \ann_R h) + (D-E).
  \end{displaymath}
% $D(g) = E(h) \ne 0$.
  Since $(\ann_R g \isect \ann_R h) \isect (D-E) = (D-E) \m_R$, we have a
  short exact sequence
  \begin{equation} \label{eq:seq3}
    0 \to (D-E) \m_R \to (\ann_R g \isect \ann_R h) \dsum (D-E) \to \ann_R f
    \to 0.
  \end{equation}
  Evidently, $(D-E) \m_R$ has a Koszul type resolution with \te{k} free module
  $M_k = \binom{r}{k} \, R(-d-k)$. Thus by the mapping cone construction we
  have a resolution
  \begin{displaymath}
    \HH'' : 0 \to H''_{r+1} \to \dots \to H''_1
  \end{displaymath}
  of $\ann_R f$, with
  \begin{equation} \label{eq:H}
    \begin{aligned}
      H''_k & = M_{k-1} \dsum H'_k
      = \binom{r}{k-1} \, R(-d-k+1) \dsum \nu_k \, R(-k-1) \\
      & \quad \dsum \left( \dsum_{j=0}^d \sum_{i=1}^k \left(
          \binom{t}{k-i} \hat{\beta}^g_{ij} +
          \binom{s}{k-i} \hat{\beta}^h_{ij} \right) R(-k-j) \right)\!.
    \end{aligned}
  \end{equation}

  Since $R/ \ann_R f$ is Gorenstein, its minimal resolution is self-dual. We
  now use this to find terms in $\HH''$ that must be canceled. When we dualize
  $\HH''$ (using $M^\vee = \Hom_R (M,R)$), we get a resolution whose \te{k}
  term is
  \begin{equation} \label{eq:H*}
    \begin{aligned}
      (H''_{r-k}&)^\vee \otimes_\K \K(-d-r) =
      \nu_{r-k} \, R(-d-k+1) \dsum \binom{r}{k+1} \, R(-k-1) \\
      & \dsum \left( \dsum_{j=0}^d \biggl( \sum_{i=k-t}^{s-1} \binom{t}{k-i}
        \hat{\beta}^g_{ij} + \sum_{i=k-s}^{t-1} \binom{s}{k-i}
        \hat{\beta}^h_{ij} \biggr) R(-k-j) \right)\!.
    \end{aligned}
  \end{equation}
  Here we have used $\hat{\beta}^g_{s-i,d-j} = \hat{\beta}^g_{ij}$ and
  $\hat{\beta}^h_{t-i,d-j} = \hat{\beta}^h_{ij}$, which follow from the
  symmetry of the resolutions $\F$ and $\G$.

  Since $\ann_S (g)_1 = 0$, we know that $\hat{\beta}^g_{sd} =
  \hat{\beta}^g_{00} = 1$, but otherwise the ``rim'' of the Betti diagram is
  zero, i.e. $\hat{\beta}^g_{ij} = 0$ for $i=0$, $j \ne 0$, for $j=0$, $i \ne
  0$, for $i=s$, $j \ne d$, and for $j = d$, $i \ne s$. Similar statements
  hold for $\hat{\beta}^h_{ij}$. Putting this into equations \eqref{eq:H} and
  \eqref{eq:H*}, we see that the first has no terms with twist $(-k)$, whereas
  the second has $[ \tbinom{t}{k} + \tbinom{s}{k} ] R(-k)$. Thus we see that
  at least a summand
  \begin{displaymath}
    \rho = \bigl[ \tbinom{t}{k-s} + \tbinom{s}{k-t} \bigr] R(-d-k)
  \end{displaymath}
  must be canceled from every $H''_k$. By looking at the expression for
  $H''_k$, we see that its summand with twist equal to $(-d-k)$, is exactly
  $\rho$.

  By the construction, the only part of the map $H''_{k+1} \to H''_k$ that can
  possibly be non-minimal, is the map from the direct summand $M_k =
  \tbinom{r}{k} R(-d-k)$ of $H''_{k+1}$ to the summand $\rho$ of $H''_k$. By
  the previous paragraph, all of $\rho$ must cancel. But $\rho$ is mapped into
  $H''_{k-1}$ by a map that we know is minimal, hence it must cancel against
  $M_k$. When we have done so for all $k$, every resulting map is minimal. So
  we are left with a graded free resolution that must be minimal. Since
  $\tbinom{r}{k} - \tbinom{t}{k-s} - \tbinom{s}{k-t} = \nu_{r-k-1}$, we see
  that this resolution is $\HH : 0 \to H_r \to \dots \to H_1 \to \ann_R f \to
  0$ with $H_r = R(-d-r)$ and
  \begin{align*}
    H_k & = \nu_k \, R(-k-1) \dsum \nu_{r-k} \, R(-d-k+1) \\
    & \qquad \dsum \left( \dsum_{j=1}^{d-1} \biggl( \sum_{i=1}^{s-1}
      \binom{r-s}{k-i} \hat{\beta}^g_{ij} + \sum_{i=1}^{t-1} \binom{r-t}{k-i}
      \hat{\beta}^h_{ij} \biggr) R(-k-j) \right)
  \end{align*}
  for all $0<k<r$.
\end{proof}

% Faktisk er formelen for $H_k$ gyldig for alle $k \ne -1,0,r,r+1$. For disse
% verdiene gir formelen $H_{-1} = -R$, $H_0 = H_r = 0$ og $H_{r+1} = R(-r-d)$,
% mens de riktige uttrykkene er $H_{-1} = H_{r+1} = 0$, $H_0 = R$ og $H_r =
% R(-r-d)$. Dette kan ikke være tilfeldig, men jeg vet ikke hva det betyr.

\begin{rem}
  If we compare theorem \ref{thm:fgh} in the case $(s,t) = (3,1)$ with the
  resolution obtained by Iarrobino and Srinivasan in \cite[theorem 3.9]{IS},
  we see that they agree. Our methods are, however, very different.
\end{rem}

% in \emph{Artinian Gorenstein algebras of embedding dimension four:
%   Components of $PGor(H)$ for $H=(1,4,7,..., 1)$}.

As a consequence we can compute the graded Betti numbers of $R/ \ann_R f$.

\begin{thm} \label{thm:betti}
  Let $d \ge 2$ and $f, g_1, \dots, g_n \in \cR_d$. Suppose $f = g_1 + \dots +
  g_n$ is a regular splitting of $f$. Let $s_i = \dim_\K R_{d-1} (g_i)$ for
  every $i$. Let $s = \sum_{i=1}^n s_i$, and define
  \begin{displaymath}
    \nu_{nk} = (n-1) \binom{r}{k+1} + \binom{r-s}{k+1} - \sum_{i=1}^n
    \binom{r-s_i}{k+1}.
  \end{displaymath}
  Denote by $\smash{\hat{\beta}^f_{kj}}$ and $\smash{\hat{\beta}^{g_i}_{kj}}$
  the shifted graded Betti numbers of $R/ \ann_R (f)$ and $R/ \ann_R (g_i)$,
  respectively. Then
  \begin{equation} \label{eq:betti}
    \hat{\beta}^f_{kj} = \sum_{i=1}^n \hat{\beta}^{g_i}_{kj} + \nu_{nk}
    \delta_{1j} + \nu_{n,r-k} \delta_{d-1,j}
  \end{equation}
  for all $0<j<d$ and all $k \in \Z$. Here the symbol $\delta_{ij}$ is defined
  by $\delta_{ii} = 1$ for all $i$, and $\delta_{ij} = 0$ for all $i \ne j$.
\end{thm}

\begin{proof}
  Since $\hat{\beta}^f_{kj} = v_{nk} = 0$ for all $k \ge r$ and all $k \le 0$,
  it is enough to prove \eqref{eq:betti} for $0<k<r$. Let $\cS = \K [R_{d-1}
  (f)]^{DP}$ and $\cS_i = \K [ R_{d-1} (g_i) ]^{DP}$. Recall that $f \in \cS$
  and $g_i \in \cS_i$. It follows from the definition of a regular splitting
  that $R_{d-1} (f) = \dsum_{i=1}^n R_{d-1} (g_i)$, and therefore $\cS = \cS_1
  \otimes_\K \dots \otimes_\K \cS_n \subseteq \cR$, cf. remark
  \ref{rem:basic}. In particular, $s = \sum_{i=1}^n s_i = \dim_\K R_{d-1} (f)
  \le r$.

  Choose $V \subseteq R_1$ such that $R_1 = R_{d-1} (f)^\perp \dsum V$, and
  let $S = \K [V]$. Then $S \iso \cS^*$, cf. remark \ref{rem:scs}. Denote the
  shifted graded Betti numbers of $S/ \ann_S (f)$ by
  $\smash{\hat{\beta}^{S/f}_{kj}}$. It follows from lemma \ref{lem:betti} that
  \begin{equation} \label{eq:S/f}
    \hat{\beta}^f_{kj} = \sum_{i=1}^{s-1} \binom{r-s}{k-i}
    \hat{\beta}^{S/f}_{ij} + \binom{r-s}{k} \delta_{0j} + \binom{r-s}{k-s}
    \delta_{dj}
  \end{equation}
  for all $j,k \ge 0$. Note that $\ann_S (f)_1 = 0$.

  For every $i$ let $V_i = \bigl( \sum_{j \ne i} R_{d-1} (g_j)
  \smash{\bigr)}^\perp \isect V \subseteq R_1$ and $S_i = \K [V_i]$. Then $V =
  \dsum_{i=1}^n V_i$, and therefore $S = S_1 \otimes_\K \dots \otimes_\K S_n
  \subseteq R$. Furthermore, $S_i \iso \cS_i^*$ for all $i$, and $\ann_S (f)_1
  = \dsum_{i=1}^n \ann_{S_i} (g_i)_1$ by lemma \ref{lem:annfg}. Thus $\ann_R
  (f)_1 = 0$ is equivalent to $\ann_{S_i} (g_i)_1 = 0$ for all $i$.

  Denote the shifted graded Betti numbers of $S_i/ \ann_{S_i} (g_i)$ by
  $\hat{\beta}^{S_i/g_i}_{kj}$. If we apply equation \eqref{eq:S/f} to $g_i$,
  we get
  \begin{equation} \label{eq:S/g}
    \hat{\beta}^{g_i}_{kj} = \sum_{l=1}^{s_i-1} \binom{r-s_i}{k-l}
    \hat{\beta}^{S_i/g_i}_{lj}
  \end{equation}
  for all $k \ge 0$ and all $0<j<d$. To prove the theorem we first show that
  \begin{equation} \label{eq:betti2}
    \hat{\beta}^{S/f}_{kj} = \sum_{i=1}^n \sum_{l=1}^{s_i-1}
    \binom{s-s_i}{k-l} \hat{\beta}^{S_i/g_i}_{lj} + \nu_{nk} \delta_{1j} +
    \nu_{n,s-k} \delta_{d-1,j}.
  \end{equation}
  for all $0<j<d$ and $0<k<r$.

  Note that $\nu_{1k} = 0$ for all $k$, since $n=1$ implies $s=s_1$. Thus
  equation \eqref{eq:betti2} is trivially fulfilled for $n=1$. We proceed by
  induction on $n$.

  Assume \eqref{eq:betti2} holds for $h = g_1 + \dots + g_{n-1}$. Let $T = S_1
  \otimes_\K \dots \otimes_\K S_{n-1}$, which is a polynomial ring in $t =
  \sum_{i=1}^{n-1} s_i$ variables. Since $f = h + g_n$ and $\ann_T (h)_1 =
  \ann_{S_n} (g_n)_1 = 0$, we may use theorem \ref{thm:fgh} to find the
  minimal resolution of $S/ \ann_S f$. We see that its graded Betti numbers
  are given by
  \begin{displaymath}
    \hat{\beta}^{S/f}_{kj} = \sum_{c=1}^{t-1} \binom{s-t}{k-c}
    \hat{\beta}^{T/h}_{cj} + \sum_{l=1}^{s_n-1} \binom{s-s_n}{k-l}
    \hat{\beta}^{S_n/g_n}_{lj} + \nu_{2k} \delta_{1j} + \nu_{2,s-k}
    \delta_{d-1,j}
  \end{displaymath}
  for all $0<k<s$ and $0<j<d$. Since by induction
  \begin{displaymath}
    \hat{\beta}^{T/h}_{cj} = \sum_{i=1}^{n-1} \sum_{l=1}^{s_i-1}
    \binom{t-s_i}{c-l} \hat{\beta}^{S_i/g_i}_{lj} + \nu_{n-1,c} \delta_{1j} +
    \nu_{n-1,t-c} \delta_{d-1,j},
  \end{displaymath}
  the proof of equation \eqref{eq:betti2} reduces to the following three
  binomial identities.
  \begin{enumerate}
    \setlength{\itemsep}{2pt}
    \setlength{\parskip}{0pt}
    \renewcommand{\labelenumi}{\normalfont(\theenumi)}
  \item $\displaystyle\sum_{c=1}^{t-1} \binom{s-t}{k-c} \binom{t-s_i}{c-l} =
    \binom{s-s_i}{k-l}$
  \item $\displaystyle\sum_{c=1}^{t-1} \binom{s-t}{k-c} \nu_{n-1,c} +
    \nu_{2k} = \nu_{nk}$
  \item $\displaystyle\sum_{c=1}^{t-1} \binom{s-t}{k-c} \nu_{n-1,t-c} +
    \nu_{2,s-k} = \nu_{n,s-k}$
  \end{enumerate}
  They all follow from the well known formula $\sum_{i \in \Z} \binom{a}{i}
  \binom{b}{k-i} = \binom{a+b}{k}$.

  The first follows immediately since we may extend the summation to $c \in
  \Z$ because $1 \le l < s_i$. In the second we note that
  \begin{displaymath}
    \nu_{n-1,c} = (n-2) \binom{t}{c+1} + \binom{0}{c+1} - \sum_{i=1}^{n-1}
    \binom{t-s_i}{c+1}.
  \end{displaymath}
  Note that $\nu_{n-1,c} = 0$ for all $c \ge t$ and all $c \le 0$, even $c=-1$
  since $\tbinom 00 = 1$. Hence we can extend the summation in equation (2) to
  all $c \in \Z$, implying
  \begin{displaymath}
    \sum_{c=1}^{t-1} \binom{s-t}{k-c} \nu_{n-1,c} = (n-2) \binom{s}{k+1} +
    \binom{s-t}{k+1} - \sum_{i=1}^{n-1} \binom{s-s_i}{k+1}.
  \end{displaymath}
  Since
  \begin{displaymath}
    \nu_{2k} = \binom{s}{k+1} + \binom{0}{k+1} - \binom{s-t}{k+1} -
    \binom{s-s_n}{k+1},
  \end{displaymath}
  equation (2) follows easily. Finally, the third equation equals the second
  by letting $(c,k) \mapsto (t-c,s-k)$, finishing the proof of equation
  \eqref{eq:betti2}.

  The theorem now follows by combining equations \eqref{eq:S/f},
  \eqref{eq:S/g} and \eqref{eq:betti2}. Also here the proof reduces to three
  binomial identities, and their proofs are similar to equation (1) above.
\end{proof}

\begin{rem}
  We may express $\hat{\beta}^f_{kj}$ in terms of
  $\hat{\beta}^{S_i/g_i}_{lj}$, the shifted graded Betti numbers of $S_i /
  \ann_{S_i} (g_i)$. From the proof of theorem \ref{thm:betti}, we see that
  \begin{equation*}
    \hat{\beta}^f_{kj} = \sum_{i=1}^n \sum_{l=1}^{s_i-1} \binom{r-s_i}{k-l}
    \hat{\beta}^{S_i/g_i}_{lj} + \nu_{nk} \delta_{1j} + \nu_{n,r-k}
    \delta_{d-1,j}.
  \end{equation*}
\end{rem}

\begin{rem}
  For any $f \in \cR_d$ we may arrange the shifted graded Betti numbers
  $\hat{\beta}_{ij}$ of $R/\ann_R f$ into the following $(d+1) \times (r+1)$
  box.
  \begin{displaymath}
    \begin{array}{|ccccc|} \hline
      \rule{0pt}{12pt}1 & \hat{\beta}_{r-1,d} &\dots& \hat{\beta}_{1d} & 0 \\
      0 & \hat{\beta}_{r-1,d-1} & \dots & \hat{\beta}_{1,d-1} & 0 \\
      \vdots & \vdots && \vdots & \vdots \\
      0 & \hat{\beta}_{r-1,1} & \dots & \hat{\beta}_{11} & 0 \\
      \rule[-5pt]{0pt}{5pt}0 & \hat{\beta}_{r-1,0} & \dots & \hat{\beta}_{10}
      & 1 \\ \hline
    \end{array}
  \end{displaymath}
  We call this the Betti diagram of $R/\ann_R f$. The Betti numbers are all
  zero outside this box, i.e. $\hat{\beta}_{ij} = 0$ for $i<0$, for $j<0$, for
  $i>r$, and for $j>d$. Thus the socle degree $d$ is equal to the
  Castelnuovo-Mumford regularity of $R/\ann_R f$. In addition,
  $\hat{\beta}_{ij}$ will always be zero for $i=0$, $j>0$ and for $i=r$,
  $j<d$, and $\hat{\beta}_{00} = \hat{\beta}_{rd} = 1$, as indicated.

  The values of $\hat{\beta}_{ij}$ when $j=0$ or $j=d$ are easily determined
  by equation \eqref{eq:S/f}. Since $\ann_S (f)_1 = 0$, it follows that
  \begin{displaymath}
    \hat{\beta}_{i0} = \binom{r-s}{i} \quad \text{ and } \quad
    \hat{\beta}_{id} = \binom{r-s}{i-s}
  \end{displaymath}
  for all $i$. In particular, if $\ann_R (f)_1 = 0$, then they are all zero
  (except $\hat{\beta}_{00} = \hat{\beta}_{rd} = 1$).

  The ``inner'' rectangle of the Betti diagram, that is, $\hat{\beta}_{ij}$
  with $0<i<r$ and $0<j<d$, is determined by theorem \ref{thm:betti}. We note
  that it is simply the sum of the ``inner'' rectangles of the Betti diagrams
  of $R/\ann_R (g_i)$, except an addition to the rows with $j=1$ and $j=d-1$.
\end{rem}

% Let $R_1 = \dsum_{i=1}^n V_i$ and $I = \sum_{i<j} (V_i V_j)$. $R/I$ should
% have a minimal free resolution of the form $0 \to F_{r-1} \to \dots \to F_1
% \to R$ with $F_k = \nu_{nk} R(-k-1)$ for all $k>0$.

\section{The parameter space}  \label{section:pgor}

% This section is inspired by \cite{IS}. We assume $\K = \bar{\K}$. I guess
% most of it would still be true as long as $\K$ is infinite.

The closed points of the quasi-affine scheme $\bGor (r,H)$ parameterize every
$f \in \cR_d$ such that the Hilbert function of $R/\ann_R f$ equals $H$. We
will in this section define some ``splitting subfamilies'' of $\bGor (r,H)$,
and discuss some of their properties. We assume here that $\K$ is an
algebraically closed field. We start by defining $\bGor (r,H)$, cf.
\cite[definition 1.10]{IK}.

Let
\begin{equation} \label{eq:A}
  A = \biggl\{ \alpha = (\alpha_1, \dots, \alpha_r) \in \Z^r \,\bigg\vert\:
  \alpha_i \ge 0 \text{ for all $i$ and } \sum_{i=1}^r \alpha_i = d \biggr\},
\end{equation}
and note that $|A| = \tbinom{r+d-1}{d} = \dim_\K \cR_d$. We consider $\A =
\K[\{z_\alpha \suchthat \alpha \in A\}]$, which is a polynomial ring in $|A|$
variables, to be the coordinate ring of $\bA (\cR_d)$. We think of
\begin{displaymath}
  F = \sum_{\alpha \in A} z_\alpha x^\pd \alpha \in \A \otimes_\K \cR_d
\end{displaymath}
as the generic element of $\cR_d$. The action of $R$ on $\cR$ extend by
$\A$-linearity to an action on $\A \otimes_\K \cR$. In particular, if $D \in
R_d$, then $D(F) = \sum_{\alpha \in A} z_\alpha D (x^\pd \alpha)$ is an
element of $\A_1$.

For any $0 \le e \le d$, fix bases $\D = \{ D_1, \dots, D_M \}$ and $\E = \{
E_1, \dots, E_N \}$ for $R_{d-e}$ and $R_e$, respectively. Let $D = [ D_1,
\dots, D_M ]^\T$ and $E = [ E_1, \dots, E_N ]^\T$, and define $\Cat^d_e =
DE^\T$. It is customary to require that $\D$ and $\E$ are the standard bases
$\{ \p^\alpha \}$ ordered lexicographically, and to call $\Cat^d_e$ the
``catalecticant'' matrix of this size. Note that the \te{(i,j)} entry of
$\Cat^d_e (F)$ is
\begin{displaymath}
  \bigl( \Cat^d_e (F) \bigr)_{ij} = D_i E_j (F) = \sum_{\alpha \in A} z_\alpha
  D_i E_j \bigl( x^\pd \alpha \bigr) \in \A_1.
\end{displaymath}

If $f \in \cR_d$, then $\Cat^d_e (f)$ is a matrix representation of the map
$R_e \to \cR_{d-e}$ given by $D \mapsto D(f)$. Hence
\begin{displaymath}
  \dim_\K (R/\ann f)_e = \rank \Cat^d_e (f) = \dim_\K (R/\ann f)_{d-e}
\end{displaymath}
by lemma \ref{lem:ann}. Therefore the $k \times k$ minors of $\Cat^d_e (F)$
cut out the subset
\begin{displaymath}
  \bigl\{ f \in \cR_d \,\big\vert\: \dim_\K (R/\ann f)_e < k \bigr\}
  \subseteq \bA (\cR_d).
\end{displaymath}

\begin{defn} \label{def:PGor}
  Let $H = (h_0, \dots, h_d)$ be a symmetric sequence of positive integers
  (i.e. $h_{d-i} = h_i$ for all $i$) such that $h_0 = 1$ and $h_1 \le r$. We
  define $\bGor_{\le} (r,H)$ to be the affine subscheme of $\bA (\cR_d)$
  defined by the ideal
  \begin{displaymath}
    I_H = \sum_{e=1}^{d-1} I_{h_e+1} \bigl( \Cat^d_e (F) \bigr).
  \end{displaymath}
  We let $\bGor (r,H)$ be the open subscheme of $\bGor_\le (r,H)$ where some
  $h_e \times h_e$ minor is nonzero for each $e$. We denote by $\Gor(r,H)$ the
  corresponding reduced scheme, which is then the quasi-affine algebraic set
  parameterizing all $f \in \cR_d$ such that $H(R/\ann f) = H$. Furthermore,
  let $\bPGor (r,H)$ and $\PGor (r,H)$ be the projectivizations of $\bGor
  (r,H)$ and $\Gor (r,H)$, respectively. By virtue of the Macaulay duality
  (cf. lemma \ref{lem:mac}), $\PGor(r,H)$ parameterizes the graded Artinian
  Gorenstein quotients $R/I$ with Hilbert function $H$.
\end{defn}

% (See definition 1.10 in \emph{Power Sums, Gorenstein Algebras, and
%   Determinantal Loci} by Iarrobino and Kanev, \cite{IK}.)

% As usual, $I_k (M)$ is the ideal generated by the $k \times k$ minors of the
% matrix $M$. To require one particular minor to be nonzero, is a localization
% of the ring. To require one of many is glueing these to a bigger open.

% $\bGor (r,H)$ can be covered by open neighbourhoods whose coordinate rings
% are localizations of $\A / I_H$.

We are now ready to define a set of $f \in \Gor (r,H)$ that split. This subset
will depend on the Hilbert function of every additive component of $f$. Recall
that if $f = \sum_{i=1}^n g_i$ is a regular splitting of $f$, then by lemma
\ref{lem:annfg} (a and e)
\begin{displaymath}
  H \bigl( R/ \ann_R f \bigr) = \sum_{i=1}^n H \bigl( R/ \ann_R g_i \bigr) -
  (n-1) \bigl( \delta_0 + \delta_d \bigr).
\end{displaymath}

\begin{defn} \label{def:Split}
  Let $r \ge 1$, $d \ge 2$ and $n \ge 1$. For each $i = 1, \dots, n$, suppose
  $H_i = (h_{i0}, \dots, h_{id})$ is a symmetric sequence of positive integers
  such that $h_{i0} = 1$ and $\sum_{i=1}^n h_{i1} \le r$. Let $\underline{H} =
  (H_1, \dots, H_n)$ and $H = \sum_{i=1}^n H_i - (n-1) (\delta_0 + \delta_d)$,
  i.e $H = (h_0, \dots, h_d)$ where $h_0 = h_d = 1$ and $h_j = \sum_{i=1}^n
  h_{ij}$ for all $0<j<d$. Define
  \begin{displaymath}
    \Split (r,\underline{H}) = \Split (r,d,n, \underline{H}) \subseteq \Gor
    (r,H)
  \end{displaymath}
  to be the subset parameterizing all $f \in \cR_d$ with the following
  property: There exist a regular splitting $f = \sum_{i=1}^n g_i$ such that
  $H ( R/ \ann_R g_i) = H_i$ for all $i$. Let $\PSplit (r,\underline{H})
  \subseteq \PGor (r,H)$ be the projectivization of $\Split
  (r,\underline{H})$.
\end{defn}

% Old definition: Let $r \ge 1$, $d \ge 2$ and $n \ge 1$. For each $i = 1,
% \dots, n$, let $H_i = (h_{i0}, \dots, h_{id})$ be symmetric, with $h_{i0} =
% 1$ and $\sum_{i=1}^n h_{i1} = r$. Let $\underline{H} = (H_1, \dots, H_n)$,
% and define $H = (h_0, \dots, h_d)$ by $h_j = \sum_{i=1}^n h_{ij}$ for all
% $0<j<d$ and $h_0 = h_d = 1$. Let $\J_i = \{ j \in \Z \suchthat \sum_{k<i}
% h_{k1} < j \le \sum_{k \le i} h_{k1} \}$ for all $i = 1, \dots, n$. Then
% $|\J_i| = h_{i1}$ and $\{ \J_i \}_{i=1}^n$ is a partition of $\{ 1, \dots, r
% \}$. Define $S^i = \K[ \{ \p_j \suchthat j \in \J_i \} ]$. Then $R = S^1
% \otimes_\K \dots \otimes_\K S^n$, and we identify $S^i$ with its image in
% $R$. Define $\PSplit (\underline{H}) = \PSplit (r,d,n, \underline{H})
% \subseteq \PGor(H)$ to be the subset parameterizing all $R/I \in \PGor(H)$
% that after a base change satisfy $S^i_1 S^j_1 \subseteq I_2$ for all $i \ne
% j$ and $H(S^i/I \isect S^i) = H_i$ for all $i$.

% Let $\underline{s} = (s_1, \dots, s_n)$ be an $n$-tuple of positive integers
% satisfying $\sum_{i=1}^n s_i = s \le r$, and assume $h_{i1} = s_i$. Or let
% $s_i = h_{i1}$ and $\underline{s} = (s_1, \dots, s_n)$, thus $\sum_{i=1}^n
% s_i = h_1$. I guess the full set of parameters would be $\PSplit (r,d,n,
% \underline{s}, \underline{H})$, but they are all encoded in $\underline{H}$.

Obviously, $\Split (r,\underline{H})$ reduces to $\Gor (r,H)$ if $n=1$.
$\Split (r,\underline{H})$ is always a constructible subset of $\Gor (r,H)$,
since it is the image of the morphism $\rho$, see lemma \ref{lem:Split}. Note
that every linear map $\K^s \to \K^r$, that is, every matrix $C \in \Mat_\K
(r,s)$, induces a homomorphism of $\K$-algebras $\K [x_1, \dots, x_s]^{DP} \to
\cR$, determined by $[x_1, \dots, x_s] \mapsto [x_1, \dots, x_r] C$, that we
denote $\phi_C$.

\begin{lem} \label{lem:Split}
  Let $\underline{H} = (H_1, \dots, H_n)$ be an $n$-tuple of symmetric
  $h$-vectors $H_i = (h_{i0}, \dots, h_{id})$ such that $h_{i0} = 1$ for all
  $i$, and $\sum_{i=1}^n h_{i1} \le r$. Let $s_i = h_{i1}$, $\underline{s} =
  (s_1, \dots, s_n)$ and $H = \sum_{i=1}^n H_i - (n-1) (\delta_0 + \delta_d)$,
  where $\delta_e$ is $1$ in degree $e$ and zero elsewhere. Define
% \begin{displaymath}
%   \Phi_{\underline{s}} = \biggl\{ (\phi_1, \dots, \phi_n) \,\bigg\vert\:
%   \phi_i : \K^{s_i} \hookrightarrow \cR_1 \text{ and } \im \phi_i \isect
%   \sum_{j \ne i} \im \phi_j = 0 \text{ for all } i \biggr\}.
% \end{displaymath}
  \begin{displaymath}
    \Phi_{\underline{s}} = \biggl\{ (\phi_{C_1}, \dots, \phi_{C_n})
    \,\bigg\vert\: C_i \in \Mat_\K (r,s_i) \text{ and } \dim_\K \sum_{i=1}^n
    \im C_i = \sum_{i=1}^n s_i \biggr\}.
  \end{displaymath}
  Then $\Split (r, \underline{H})$ is the image of the morphism
  \begin{align*}
    \rho : \Phi_{\underline{s}} \times \prod_{i=1}^n \Gor (s_i, H_i) & \to
    \Gor (r,H), \\
    \bigl( (\phi_{C_1}, \dots, \phi_{C_n}), (g_1, \dots, g_n) \bigr) & \mapsto
    \sum_{i=1}^n \phi_{C_i} (g_i).
  \end{align*}
  Furthermore, the fiber over any closed point has dimension $\sum_{i=1}^n
  s_i^2$.
\end{lem}

\begin{proof}
  The first part is clear from definition \ref{def:Split}. Note that the
  condition $\dim_\K \sum_{i=1}^n \im C_i = \sum_{i=1}^n s_i$ in the
  definition of $\Phi_{\underline{s}}$ is equivalent to $\rank C_i = s_i$ and
  $\im C_i \isect \sum_{j \ne i} \im C_j = 0$ for all $i$.

  To find the dimension of the fibers, we will start by describing a group
  that acts on $\Phi_{\underline{s}} \times \sprod_{i=1}^n \Gor (s_i, H_i)$ in
  such a way that the morphism $\rho$ is constant on the orbits of the group
  action.

  First, let the group $\sprod_{i=1}^n \GL_{s_i}$ act on $\Phi_{\underline{s}}
  \times \sprod_{i=1}^n \Gor (s_i, H_i)$ by
  \begin{displaymath}
    \bigl( P_i \smash{\bigr)}_{i=1}^n \times \Bigl( \bigl( \phi_{C_i}
    \smash{\bigr)}_{i=1}^n, \bigl( g_i \smash{\bigr)}_{i=1}^n \Bigr) \mapsto
    \Bigl( \bigl( \phi_{C_i P_i^{-1}} \smash{\bigr)}_{i=1}^n, \bigl(
    \phi_{P_i} (g_i) \smash{\bigr)}_{i=1}^n \Bigr).
  \end{displaymath}
  Obviously, $\phi_{C_i P_i^{-1}} = \phi_{C_i} \circ \phi_{P_i^{-1}}$, and
  therefore, $\bigl( \phi_{C_i P_i^{-1}} \bigr) (\phi_{P_i} g_i) = \phi_{C_i}
  (g_i)$.

% Recall that $\phi_P$ is defined on $\cR$ by $x \mapsto P^\T x$ and on $R$ by
% $\p \mapsto P^{-1} \p$. This implies that $\phi_P \circ \phi_Q = \phi_{PQ}$,
% and furthermore, $P (Q (\phi, g)) = (PQ) (\phi, g)$. Thus the above morphism
% is indeed a group action.

  Second, let $\SG_n$ denote the symmetric group on $n$ symbols. A permutation
  $\sigma \in \SG_n$ acts on the $n$-tuple $\underline{H} = (H_1, \dots, H_n)$
  by permuting its coordinates, i.e., $\sigma (\underline{H}) = \bigl(
  H_{\sigma^{-1} (1)}, \dots, H_{\sigma^{-1} (n)} \bigr)$. Let
  $G_{\underline{H}}$ be the subgroup of $\SG_n$ defined by
  \begin{displaymath}
    G_{\underline{H}} = \{ \sigma \in \SG_n \suchthat \sigma (\underline{H}) =
    \underline{H} \}.
  \end{displaymath}

  Note that $G_{\underline{H}}$ is a product of symmetric groups. Indeed, let
  $k$ be the number of distinct elements of $\{ H_1, \dots, H_n \}$. Call
  these elements $H'_1, \dots, H'_k$, and let $n_i \ge 1$ be the number of $j$
  such that $H_j = H'_i$. Then $\sum_{i=1}^k n_i = n$, and
  \begin{displaymath}
    G_{\underline{H}} \iso \SG_{n_1} \times \dots \times \SG_{n_k}.
  \end{displaymath}

  The group $G_{\underline{H}}$ acts on $\Phi_{\underline{s}} \times
  \sprod_{i=1}^n \Gor (s_i, H_i)$ by
  \begin{displaymath}
    \sigma \times \Bigl( \bigl( \phi_{C_i} \smash{\bigr)}_{i=1}^n, \bigl( g_i
    \smash{\bigr)}_{i=1}^n \Bigr) \mapsto \Bigl( \bigl( \phi_{C_{\sigma^{-1}
        (i)}} \smash{\bigr)}_{i=1}^n, \bigl( g_{\sigma^{-1} (i)}
    \smash{\bigr)}_{i=1}^n \Bigr).
  \end{displaymath}
  Indeed, since any $\sigma \in G_{\underline{H}}$ fixes $\underline{H}$, we
  have $H_{\sigma^{-1} (i)} = H_i$, and in particular $s_{\sigma^{-1} (i)} =
  s_i$ since $s_i = h_{i1}$. Thus $C_{\sigma^{-1} (i)} \in \Mat_\K (r,s_i)$
  and $g_{\sigma^{-1} (i)} \in \Gor (s_i, H_i)$. Clearly, $\sum_{i=1}^n
  \phi_{C_{\smash{\sigma^{-1} (i)}}} \bigl( g_{\sigma^{-1} (i)} \bigr) =
  \sum_{i=1}^n \phi_{C_i} (g_i)$. Thus the morphism $\rho$ is constant on the
  orbits of also this group action.

  Suppose $f \in \im \rho$. By theorem \ref{thm:regsplitE} $f$ has a unique
  maximal regular splitting $f = \sum_{i=1}^m f'_i$, and every other regular
  splitting is obtained by grouping some of the summands. Evidently, since $f
  \in \Split (r,d,n, \underline{H})$, there is at least one way to group the
  summands such that $f = \sum_{i=1}^n f_i$ is a regular splitting and $H
  \bigl( R/ \ann_R (f_i) \bigr) = H_i$ for all $i$, and there are only
  finitely many such ``groupings''. If $f = \sum_{i=1}^n f_i$ is any such
  expression, then clearly there exists $\bigl( (\phi_{C_i})_{i=1}^n,
  (g_i)_{i=1}^n \bigr) \in \Phi_{\underline{s}} \times \sprod_{i=1}^n \Gor
  (s_i, H_i)$ such that $f_i = \phi_{C_i} (g_i)$ for all $i$.

  Now, if $\bigl( (\phi_{C_i})_{i=1}^n, (g_i)_{i=1}^n \bigr) \in \rho^{-1}
  (f)$ is any element of the fiber over $f$, then the expression $f =
  \sum_{i=1}^n \phi_{C_i} (g_i)$ is one of those finitely many groupings.
  Assume $\bigl( (\phi_{C'_i})_{i=1}^n, (g'_i)_{i=1}^n \bigr)$ is another
  element of the fiber such that the expression \mbox{$f = \sum_{i=1}^n
    \phi_{C'_i} (g'_i)$} corresponds to the same grouping. Since
  \begin{displaymath}
    H \bigl( R/ \ann_R (\phi_{C_i} (g_i)) \bigr) = H_i = H \bigl( R/ \ann_R
    (\phi_{C'_i} (g'_i)) \bigr),
  \end{displaymath}
  there exists $\sigma \in G_{\underline{H}}$ such that $\phi_{C'_i} (g'_i) =
  \phi_{C_{\smash{\sigma^{-1} (i)}}} \bigl( g_{\sigma^{-1} (i)} \bigr)$ for
  all $i$. By composing with $\sigma$, we may assume $\phi_{C'_i} (g'_i) =
  \phi_{C_i} (g_i)$ for all $i$. Note that $\p (\phi_{C_i} (g_i)) = C_i
  \phi_{C_i} (\p g)$ and $R_{d-1} (\p g) = \K^{s_i}$. It follows that $R_{d-1}
  \p (\phi_{C_i} (g_i)) = \im C_i$, and therefore $\im C'_i = \im C_i$. Thus
  there exists $P_i \in \GL_{s_i}$ such that $C'_i = C_i P_i^{-1}$ for all
  $i$. Moreover, $\phi_{C_i} (g_i) = \phi_{C'_i} (g'_i) = \phi_{C_i}
  (\phi_{P_i^{-1}} g'_i)$ implies $g'_i = \phi_{P_i} (g_i)$ since $\phi_{C_i}$
  is injective. This proves that $\bigl( (\phi_{C'_i})_{i=1}^n, (g'_i)_{i=1}^n
  \bigr)$ and $\bigl( (\phi_{C_i})_{i=1}^n, (g_i)_{i=1}^n \bigr)$ are in the
  same orbit.

  We have shown that the fiber $\rho^{-1} (f)$ over $f$ is of a finite union
  of $( G_{\underline{H}} \times \sprod_{i=1}^n \GL_{s_i} )$-orbits; one orbit
  for each grouping $f = \sum_{i=1}^n f_i$ of the maximal splitting of $f$
  such that $H \bigl( R/ \ann_R (f_i) \bigr) = H_i$. By considering how the
  group acts on $\Phi_{\underline{s}}$, we see that different group elements
  give different elements in the orbit. It follows that the dimension of any
  fiber equals $\dim (\sprod_{i=1}^n \GL_{s_i}) = \sum_{i=1}^n s_i^2$.
\end{proof}

\begin{exmp}
  Let $n=2$. The fiber over $f = x_1^\pd d + x_2^\pd d \in \Split (r,d,2,
  \underline{H})$ is a single orbit. However, the fiber over $f = x_1^\pd d +
  x_2^\pd d + x_3^\pd d \in \Split (r,d,2, \underline{H})$ consists of three
  orbits, one for each of the expressions $f = x_i^\pd d + \sum_{j \ne i}
  x_j^\pd d$.
\end{exmp}

% \begin{defn} \label{def:Ws}
%   Define $\W_{\underline{s}}$ to be the open subscheme of $\sprod_{i=1}^n
%   \Grass (s_i,r)$ parameterizing all $n$-tuples $W = (W_1, \dots, W_n)$ of
%   subspaces $W_i \subseteq \cR_1$ such that $\dim_\K W_i = s_i$ and $W_i
%   \isect \sum_{j \ne i} W_j = 0$ for all $i$.
% \end{defn}

% Define $\W'_{\underline{H}} = \W_{\underline{s}} / G_{\underline{H}}$, if
% this orbit space exists. (It does at least as long as $\chr \K$ does not
% divide the order of $G_{\underline{H}}$.)

\begin{rem}
  We have seen that $\rho$ is constant on the orbits of the action of
  $G_{\underline{H}} \times \sprod_{i=1}^n \GL_{s_i}$. If the geometric
  quotient exists, we get an induced map
  \begin{displaymath}
    \biggl( \Phi_{\underline{s}} \times \prod_{i=1}^n \Gor (s_i, H_i) \biggr)
    \bigg/ \biggl( G_{\underline{H}} \times \prod_{i=1}^n \GL_{s_i} \biggr)
    \to \Gor (r,H).
  \end{displaymath}
  
  Let $U_i \subseteq \Gor (s_i, H_i)$ parameterize all $g \in \K [x_1, \dots,
  x_{s_i}]^{DP}$ that do not have any non-trivial regular splitting, and $U
  \subseteq \Split (r,n, \underline{H})$ be those $f \in \cR_d$ where $f =
  \sum_{i=1}^n g_i$ is a maximal splitting. The morphism above restricts to a
  map $\bigl( \Phi_{\underline{s}} \times \sprod_{i=1}^n U_i \bigr) \big/
  \bigl( G_{\underline{H}} \times \sprod_{i=1}^n \GL_{s_i} \bigr) \to U$. By
  the proof of lemma \ref{lem:Split}, this is a bijection.
\end{rem}

% Its inverse is essentially given by $f = \sum_{i=1}^n g_i \mapsto (R_{d-1}
% (g_i), g_i)_{i=1}^n$.

\begin{rem}
  We would like to identify $\W_{\underline{s}} = \Phi_{\underline{s}} /
  \sprod_{i=1}^n \GL_{s_i}$. Let $\Grass (s_i,r)$ be the Grassmannian that
  parameterizes $s_i$-dimensional $\K$-vector subspaces of $\cR_1 \iso \K^r$.
  We may think of $\Grass (s_i,r)$ as the set of equivalence classes of
  injective, linear maps $\K^{s_i} \hookrightarrow \cR_1$, two maps being
  equivalent if they have the same image. It follows that $\W_{\underline{s}}$
  is the open subscheme of $\sprod_{i=1}^n \Grass (s_i,r)$ parameterizing all
  $n$-tuples $W = (W_1, \dots, W_n)$ of subspaces $W_i \subseteq \cR_1$ such
  that $\dim_\K W_i = s_i$ and $W_i \isect \sum_{j \ne i} W_j = 0$ for all
  $i$.
\end{rem}

\begin{rem}
  For completeness, we want to describe the corresponding map of structure
  sheafs, $\rho^\# : \cO_{\Gor (r,H)} \to \rho_* \cO_{\Phi_{\underline{s}}
    \times \sprod_{i=1}^n \Gor (s_i, H_i)}$.

  \smallskip
  For each $i$, let $(c_{ijk})$ be the entries of $C_i \in \Mat_\K (r,s_i)$,
  i.e.
  \begin{displaymath}
    C_i =
    \begin{pmatrix}
      c_{i11} & \dots & c_{i1s_i} \\
      \vdots && \vdots \\
      c_{ir1} & \dots & c_{irs_i}
    \end{pmatrix}\!.
  \end{displaymath}
  Since $\sum_{i=1}^n \im C_i = \im [C_1, \dots, C_n]$, it follows that
  $\Phi_{\underline{s}}$ is isomorphic to the set of $r \times (\sum_i
  s_i)$-matrices of maximal rank. Let $Y$ be the coordinate ring of $\Mat_\K
  (r, \sum_i s_i)$. We choose to write $Y$ as
  \begin{displaymath}
    Y = \mathop{\rule[-2.3pt]{0pt}{9pt}}_{i=1}^n \mspace{-22mu} \otimes_\K \;
    \K \bigl[ \bigl\{ y_{ijk} \,\big\vert\: 1 \le j \le r \text{ and } 1 \le k
    \le s_i \bigr\} \bigr].
  \end{displaymath}

  Let $\cS^i = \K [x_1, \dots, x_{s_i}]^{DP}$ and $S^i = \K [\p_1, \dots,
  \p_{s_i}]$. By definition, $\Gor (s_i, H_i)$ parametrizes all $g_i \in
  \cS^i_d$ such that the Hilbert function of $S^i / \ann_{S^i} (g_i)$ is
  $H_i$. The coordinate ring of $\bA (\cS^i_d)$ is $\A_i = \K[\{z_{i\gamma}
  \suchthat \gamma \in A_i\}]$, where
  \begin{displaymath}
    A_i = \biggl\{ \gamma = (\gamma_1, \dots, \gamma_{s_i}) \in \Z^{s_i}
    \,\bigg\vert\: \gamma_k \ge 0 \text{ for all $k$ and } \sum_{k=1}^{s_i}
    \gamma_k = d \biggr\}.
  \end{displaymath}
  $\bGor_\le (s_i, H_i)$ is the affine subscheme of $\bA (\cS^i_d)$ whose
  coordinate ring is $\A_i / I_{H_i}$, cf. definition \ref{def:PGor}. Any $g_i
  \in \cS^i_d$ can be written as
  \begin{displaymath}
    g_i = \sum_{\gamma \in A_i} a_{i\gamma} \prod_{k=1}^{s_i}
    x_k^{(\gamma_k)}.
  \end{displaymath}
  It follows that
  \begin{displaymath}
    \sum_{i=1}^n \phi_{C_i} (g_i) = \sum_{i=1}^n \sum_{\gamma \in A_i}
    a_{i\gamma} \prod_{k=1}^{s_i} \Bigl( \sum_{j=1}^r c_{ijk} x_j \Bigr)^\pd
    {\gamma_k}.
  \end{displaymath}
  When we expand this, we see that for any $\alpha = (\alpha_1, \dots,
  \alpha_r) \in A$ (cf. equation \eqref{eq:A}) the coefficient in front of
  $x^\pd \alpha = \sprod_{j=1}^r x_j^\pd {\alpha_j}$ is
  \begin{equation*}
    \sum_{i=1}^n \sum_{\gamma \in A_i} a_{i\gamma} \cdot \mspace{-20mu}
    \sum_{\substack{\{\beta_{jk} \ge 0\}\\[1pt] \sum_{j=1}^r \beta_{jk} =
        \gamma_k\\[1pt] \sum_{k=1}^{s_i} \beta_{jk} = \alpha_j}} \prod_{j=1}^r
    \: \biggl[ \binom {\alpha_j} {\beta_{j1}, \dots, \beta_{js_i}}
    \prod_{k=1}^{s_i} c_{ijk}^{\beta_{jk}} \biggr].
  \end{equation*}
  The multinomial
  \begin{displaymath}
    \binom {\alpha_j} {\beta_{j1}, \dots, \beta_{js_i}} = \frac {\alpha_j!}
    {\beta_{j1}! \dotsm \beta_{js_i}!}
  \end{displaymath}
  appears as a result of how the multiplication in $\cR$ is defined.

  The coordinate ring of $\bA (\cR_d)$ is $\A = \K [\{z_\alpha \suchthat
  \alpha \in A\}]$. Let
  \begin{displaymath}
    \A \to Y \otimes_\K \A_1 \otimes_\K \dots \otimes_\K \A_n
  \end{displaymath}
  be the $\K$-algebra homomorphism induced by
  \begin{equation*}
    z_\alpha \mapsto \sum_{i=1}^n \sum_{\gamma \in A_i} z_{i\gamma} \cdot
    \mspace{-20mu} \sum_{\substack{\{\beta_{jk} \ge 0\}\\[1pt] \sum_{j=1}^r
        \beta_{jk} = \gamma_k\\[1pt] \sum_{k=1}^{s_i} \beta_{jk} = \alpha_j}}
    \prod_{j=1}^r \: \biggl[ \binom {\alpha_j} {\beta_{j1}, \dots,
      \beta_{js_i}} \prod_{k=1}^{s_i} y_{ijk}^{\beta_{jk}} \biggr]
  \end{equation*}
  for all $\alpha \in A$. This implies that $F = \sum_{\alpha \in A} z_\alpha
  x^\pd \alpha \in \A \otimes_\K \cR_d$ is mapped to $\sum_{i=1}^n \phi_i
  (F_i)$, where $F_i = \sum_{\gamma \in A_i} z_{i\gamma} x^\pd \gamma \in \A_i
  \otimes_\K \cS^i_d$ and
  \begin{displaymath}
    \phi_i :
    \begin{pmatrix}
      x_1 \\ \vdots \\ x_{s_i}
    \end{pmatrix}
    \mapsto
    \begin{pmatrix}
      y_{i11} & \dots & y_{ir1} \\
      \vdots && \vdots \\
      y_{i1s_i} & \dots & y_{irs_i}
    \end{pmatrix}
    \begin{pmatrix}
      x_1 \\ \vdots \\ x_r
    \end{pmatrix}\!.
  \end{displaymath}
  Hence $\Cat^d_e (F) \mapsto \sum_{i=1}^n \Cat^d_e (\phi_i (F_i)) =
  \sum_{i=1}^n P_i \Cat^d_e (F_i) P'_i$ for suitable matrices $P_i$ and $P'_i$
  with entries in $Y$. Since every $(h_{ie}+1) \times (h_{ie}+1)$-minor of
  $\Cat^d_e (F_i)$ is zero in $\A_i / I_{H_i}$, it follows that every $(h_e+1)
  \times (h_e+1)$-minor of $\Cat^d_e (F)$ maps to zero in $Y \otimes_\K \A_1 /
  I_{H_1} \otimes_\K \dots \otimes_\K \A_n / I_{H_n}$. This induces a map
  \begin{displaymath}
    \A / I_H \to Y \otimes_\K \A_1 / I_{H_1} \otimes_\K \dots \otimes_\K \A_n
    / I_{H_n}.
  \end{displaymath}

% Note that $\phi_i = \phi_{Y_i} : x \mapsto Y_i^\T x$. Hence $\Cat^d_e
% (F) \mapsto \sum_{i=1}^n DE^\T (\phi_{Y_i} (F_i)) = P_{Y_i,D} DE^\T (F_i)
% P_{Y_i,E}^\T$. If we consider $F_i$ as an element of $\A_i \otimes_\K
% \cR_d$, then $\Cat^d_e (F_i) = DE^\T (F_i)$. However, $F_i \in \A_i
% \otimes_\K \cS^i_d$, but still $DE^\T (F_i) = Q_i \Cat^d_e (F_i) Q'_i$. In
% fact, we may choose the bases $D$ and $E$ such that
% \begin{math}
%   DE^\T (F_i) = \bigl(
%     \begin{smallmatrix}
%       \Cat^d_e (F_i) & 0 \\ 0 & 0
%     \end{smallmatrix}
%   \bigr).
% \end{math}

  This ringhomomorphism is equivalent to a morphism of affine schemes;
  \begin{displaymath}
    \psi : \Mat_\K (r, {\textstyle\sum_i s_i}) \times \prod_{i=1}^n \bGor_\le
    (s_i, H_i) \to \bGor_\le (r,H).
  \end{displaymath}
  Let $f = \sum_{i=1}^n \phi_{C_i} (g_i) \in \im \psi \isect \Gor (r,H)$.
  Since $R_{d-1} \p (\phi_{C_i} (g_i)) = \im C_i$, it follows that
  \begin{displaymath}
    \im \Cat^d_{d-1} \biggl( \sum_{i=1}^n \phi_{C_i} (g_i) \biggr) = R_{d-1}
    \p \biggl( \sum_{i=1}^n \phi_{C_i} (g_i) \biggr) \subseteq \sum_{i=1}^n
    \im C_i.
  \end{displaymath}
  Hence $\rank \Cat^d_{d-1} (f) = h_1 = \sum_{i=1}^n s_i$ implies that
  $\dim_\K \sum_{i=1}^n \im C_i = \sum_{i=1}^n s_i$. Thus
  \begin{displaymath}
    \psi^{-1} \bigl( \Gor (r,H) \bigr) = \Phi_{\underline{s}} \times
    \prod_{i=1}^n \Gor (s_i, H_i).
  \end{displaymath}
  Since $\bGor (r,H)$ is an open subscheme of $\bGor_\le (r,H)$, it follows
  that $(\psi,\psi^\#)$ restricts to $(\rho,\rho^\#)$.
\end{rem}

% Because: $\psi^\# : \cO_{\bGor_\le (r,H)} \to \psi_* \cO_{\Mat_\K (r, \sum_i
%   s_i) \times \sprod_{i=1}^n \bGor_\le (s_i, H_i)}$ restricts to $\psi^\# :
% \cO_{\Gor (r,H)} \to \psi_* \cO_{\Phi_{\underline{s}} \times \sprod_{i=1}^n
%   \Gor (s_i, H_i)}$, that is, $\psi^\# = \rho^\#$.

The next lemma rewrites the definition of $\Split (r,\underline{H})$ so that
it gives conditions on the ideal $I = \ann_R f$ instead of conditions on $f$
directly.

\begin{lem} \label{lem:PSplit}
  $\PSplit (r,\underline{H})$ parameterizes all $R/I \in \PGor (r,H)$ that
  have the following properties: There exist subspaces $V_1, \dots, V_n
  \subseteq R_1$ with $\dim_\K V_i = h_{i1}$ such that $R_1 = I_1 \dsum \bigl(
  \dsum_{i=1}^n V_i \bigr)$ and $V_i V_j \subseteq I_2$ for all $i \ne j$.
  Furthermore, $S^i/I \isect S^i \in \PGor (h_{i1},H_i)$ for all $i$, where
  $S^i = \K [V_i] \subseteq R$.
\end{lem}

\begin{proof}
  Pick $f \in \Split (r,\underline{H})$ such that $I = \ann_R f$. By
  definition \ref{def:Split} there exists a regular splitting $f =
  \sum_{i=1}^n g_i$ such that $H ( R/ \ann_R g_i) = H_i$ for all $i$, and $g_i
  \in \cS = \K [R_{d-1} (f)]^{DP}$ by corollary \ref{cor:basic}. Choose $V
  \subseteq R_1$ such that $R_1 = I_1 \dsum V$, and let $S = \K [V] \iso
  \cS^*$. By lemma~\ref{lem:annfg}(ai) we get $\ann_R f = (I_1) \dsum \ann_S
  f$. For all $i$ let $W_i = R_{d-1} (g_i) \subseteq \cR_1$ and define $V_i =
  ( \sum_{j \ne i} W_j )^\perp \isect S \subseteq V$.
  
  Note that $\dim_\K W_i = \dim_\K (R/ \ann_R g_i)_1 = h_{i1}$. Since $\cS_1 =
  \dsum_{i=1}^n W_i$, it follows that $S_1 = V = \dsum_{i=1}^n V_i$. Therefore
  $V_i \iso W_i^*$, and $\dim_\K V_i = h_{i1}$. Let $S^i = \K [V_i]$. By
  lemma~\ref{lem:annfg} (b and c) there exist nonzero $D_i \in S^i_d$ such
  that
  \begin{equation*}
    \ann_S f = \biggl( \sum_{i<j} S V_i V_j \biggr) \dsum \biggl(
    \dsum_{i=1}^n \ann_{S^i} (g_i) \biggr) + (D_2 - D_1, \dots, D_n - D_1).
  \end{equation*}
  It follows that $\ann_{S^i} (g_i) = \ann_S (f) \isect S^i = I \isect S^i$.
  Therefore,
  \begin{equation} \label{eq:I}
    I = (I_1) \dsum \biggl( \sum_{i<j} S V_i V_j \biggr) \dsum \biggl(
    \dsum_{i=1}^n (I \isect S^i) \biggr) + (D_2 - D_1, \dots, D_n - D_1).
  \end{equation}
  In particular, $V_i V_j \subseteq I_2$ for all $i \ne j$. This proves all
  the properties listed in lemma~\ref{lem:PSplit}. The opposite implication
  follows from equation~\eqref{eq:I}.
\end{proof}

\begin{rem}
  Note that the existence of the $D_i$'s in equation~\eqref{eq:I} implies that
  the map $I \mapsto (I_1, \{ V_i \}, \{ I \isect S^i \})$ is not 1-to-1. This
  is easily understood if we translate to polynomials. Since annihilator
  ideals determine polynomial only up to a nonzero scalar, it follows that the
  fiber over $\{ I \isect S^i = \ann_{S^i} (g_i) \}$ are all $I = \ann_R (f)$
  such that $f = \sum_{i=1}^n c_i g_i$ and $c_i \ne 0$ for all $i$.
\end{rem}

If $R/I \in \bPGor (r,H)$, we denote by $\cT_{R/I}$ the tangent space to
$\bGor (r,H)$ (the affine cone over $\bPGor (r,H)$) at a point corresponding
to $R/I$. Recall that $\PSplit (r, \underline{H})$ parametrizes all $R/\ann_R
f$ such that $f \in \cR_d$ and there exist a regular splitting $f =
\sum_{i=1}^n g_i$ such that $H ( R/ \ann_R g_i) = H_i$ for all $i$, cf.
definition \ref{def:Split}.

% For $I \in \PSplit (n, \underline{s}, \underline{H})$, let us denote by
% $\cT_I$ the tangent space to the affine cone over $\PGor(H)$ at a point
% corresponding to $R/I$. Similarly, for each $i$, let $\cT_{J_i}$ denote the
% tangent space to the affine cone over $\PGor (H(S^i/J_i))$ at a point
% corresponding to $S^i/J_i$.

\begin{thm} \label{thm:gor}
  Assume $\K = \bar{\K}$. Let $r \ge 1$, $d \ge 4$ and $n \ge 1$. Let
  $\underline{H} = (H_1, \dots, H_n)$ be an $n$-tuple of symmetric $h$-vectors
  $H_i = (h_{i0}, \dots, h_{id})$ such that $\sum_{i=1}^n h_{i1} \le r$ and
  $h_{i0} = 1$ for all $i$. Let $s_i = h_{i1} \ge 1$ and $H = \sum_{i=1}^n H_i
  - (n-1) (\delta_0 + \delta_d)$ where $\delta_e$ is $1$ in degree $e$ and
  zero elsewhere.
  \begin{enumerate}
    \setlength{\itemsep}{2pt}
    \setlength{\parskip}{0pt}
    \renewcommand{\theenumi}{\alph{enumi}}
    \renewcommand{\labelenumi}{\normalfont(\theenumi)}

  \item The dimension of $\PSplit (r, \underline{H}) \subseteq \PGor (r,H)
    \subseteq \bP (\cR_d)$ is
    \begin{displaymath}
      \dim \PSplit (r,\underline{H}) = n - 1 + \sum_{i=1}^n \dim \PGor
      (s_i,H_i) + \sum_{i=1}^n s_i (r-s_i).
    \end{displaymath}

  \item $\PSplit (r,\underline{H})$ is irreducible if $\PGor (s_i,H_i)$ is
    irreducible for all $i$.

  \end{enumerate}
    
  \noindent Let $R/I \in \PSplit (r,\underline{H})$. Choose $V_1, \dots V_n
  \subseteq R_1$ such that $\dim_\K V_i = s_i$ for all $i$, $R_1 = I_1 \dsum
  \bigl( \dsum_{i=1}^n V_i \bigr)$ and $V_i V_j \subseteq I_2$ for all $i \ne
  j$, cf. lemma \ref{lem:PSplit}. Let $S^i = \K [V_i]$ and $J_i = I \isect S^i
  \in \PGor (s_i,H_i)$. For each $i$, let $\beta^i_{\smash{1j}}$ be the
  minimal number of generators of degree $j$ of $J_i$ (as an ideal in $S^i$).

  \begin{enumerate}
    \setlength{\itemsep}{2pt}
    \setlength{\parskip}{0pt}
    \renewcommand{\theenumi}{\alph{enumi}}
    \renewcommand{\labelenumi}{\normalfont(\theenumi)}
    \setcounter{enumi}{2}
    
  \item The dimension of the tangent space to the affine cone over $\bPGor
    (r,H)$ at a point corresponding to $R/I$ is
    \begin{displaymath}
      \dim_\K \cT_{R/I} = \sum_{i=1}^n \dim_\K \cT_{S^i/J_i} + \sum_{i=1}^n
      s_i (r-s_i) + \sum_{i=1}^n \sum_{j \ne i} s_j \, \beta^i_{1,d-1}.
    \end{displaymath}

%   \begin{displaymath}
%     \dim_\K \cT_I = \sum_{i=1}^n \dim_\K \cT_{J_i} + \sum_{i=1}^n s_i
%     (r-s_i) + \sum_{i=1}^n (s-s_i) \beta^i_{1,d-1}.
%   \end{displaymath}
    
  \item Assume in addition for all $i$ that $S^i/J_i$ is a smooth point of
    $\bPGor (s_i,H_i)$ and $\beta^i_{1,d-1} = 0$. Then $R/I$ is a smooth point
    of $\bPGor (r,H)$. Moreover, $R/I$ is contained in a unique irreducible
    component of the closure $\overline{\PSplit (r,\underline{H})}$. This
    component is also an irreducible component of $\bPGor (r,H)$.

  \end{enumerate}
  In particular, if $\bPGor (s_i,H_i)$ is irreducible and generically smooth
  for all $i$, and $\beta_{1,d-1} (J_i) = 0$ for general $S^i/J_i \in \bPGor
  (s_i,H_i)$, then the closure $\overline{\PSplit (r,\underline{H})}$ is an
  irreducible component of $\bPGor (r,H)$, and $\bPGor (r,H)$ is smooth in
  some non-empty open subset of $\PSplit (r,\underline{H})$.
\end{thm}

This is a generalization of \cite[theorem 3.11]{IS}.

\begin{proof}
  (a) follows from lemma \ref{lem:Split}, since the lemma implies that
  \begin{displaymath}
    \dim \Split (r, \underline{H}) = \sum_{i=1}^n \dim \Gor (s_i, H_i) +
    \sum_{i=1}^n r s_i - \sum_{i=1}^n s_i^2.
  \end{displaymath}

  Alternatively, we can count dimensions using equation \eqref{eq:I}, just
  note that the $V_i$'s are determined only modulo $I_1$. Let $s = \dim_\K
  (R/I)_1 = \sum_{i=1}^n s_i$. Then we get $s(r-s)$ for the choice of $I_1
  \subseteq R_1$, $s_i (s-s_i)$ for the choice on $V_i$ (modulo $I_1$), $\dim
  \PGor (s_i,H_i)$ for the choice of $I \isect S^i \subseteq S^i$, and finally
  $n-1$ for the choice of $D_2-D_1, \dots, D_n-D_1 \in R_d$. Adding these
  together proves (a).
  
  (b) follows immediately from lemma \ref{lem:Split}.
  
  To prove (c), we use theorem 3.9 in \cite{IK} (see also remarks 3.10 and 4.3
  in the same book), which tells us that $\dim_\K \cT_{R/I} = \dim_\K
  (R/I^2)_d$. Note that $H (S^i / J_i) = H_i$ for all $i$ by definition of
  $\PSplit (r,\underline{H})$.
  
  Assume first that $I_1 = 0$. Note that this implies $R_1 = \dsum_{i=1}^n
  V_i$, and therefore $R = S^1 \otimes_\K \dots \otimes_\K S^n$ and $r =
  \sum_{i=1}^n s_i$. By equation \eqref{eq:I} we have
  \begin{displaymath}
    I_e = \biggl( \sum_{i<j} R_{e-2} S^i_1 S^j_1 \biggr) \dsum \biggl(
    \dsum_{i=1}^n J_{i,e} \biggr)
  \end{displaymath}
  as a direct sum of $\K$-vector subspaces of $R_e$ for all degrees $e<d$. In
%% CHANGED 2013-07-13: nulltallet i $J_{i,1} = 0$ manglet
  particular, $I_1 = 0$ is equivalent to $J_{i,1} = 0$ for all $i$.
  
% Moreover, $\beta^I_{1,d-1} = \sum_{i=1}^n \beta^i_{1.d-1}$ since $d \ge 4$.
  
% Note that $R_e = ( \sum_{i<j} R_{e-2} S^i_1 S^j_1 ) \dsum ( \dsum_{i=1}^n
% S^i_e )$ for all $e>0$.
  
  Let $S = S^1 \otimes_\K \dots \otimes_\K S^{n-1}$, $J_S = I \isect S$ and $s
  = \sum_{i=1}^{n-1} s_i$, and let $T = S^n$, $J_T = I \isect T$ and $t =
  s_n$. Then $I_e = R_{e-2} S_1 T_1 \dsum J_{S,e} \dsum J_{T,e}$ for all
  $e<d$. It follows for all $2 \le e \le d-2$ that
  \begin{align*}
    I_e \cdot I_{d-e} & = R_{d-4} S_2 T_2 \dsum J_{S,e} \cdot J_{S,d-e} \dsum
    J_{T,e} \cdot J_{T,d-e} \\
    & \quad \dsum T_1 ( S_{d-e-1} J_{S,e} + S_{e-1} J_{S,d-e} ) \dsum S_1 (
    T_{d-e-1} J_{T,e} + T_{e-1} J_{T,d-e} ).
  \end{align*}
  Since $I_1 = 0$ implies $J_{S,1} = J_{T,1} = 0$, and $\sum_{e=2}^{d-2}
  S_{d-e-1} J_{S,e} = S_1 J_{S,d-2}$, we get
  \begin{align*}
    (I^2)_d = \sum_{e=2}^{d-2} I_e \cdot I_{d-e} = R_{d-4} S_2 T_2 \dsum
    (J_S^2)_d \dsum (J_T^2)_d \dsum S_1 T_1 J_{S,d-2} \dsum S_1 T_1 J_{T,d-2}.
  \end{align*}
  Because $R_d = S_d \dsum T_1 S_{d-1} \dsum R_{d-4} S_2 T_2 \dsum S_1 T_{d-1}
  \dsum T_d$, it follows that
  \begin{displaymath}
    \bigl( R/I^2 \smash{\bigr)}_d = \bigl( S/J_S^2 \smash{\bigr)}_d \dsum
    \bigl( T/J_T^2 \smash{\bigr)}_d \dsum T_1 \bigl( S_{d-1} / S_1 J_{S,d-2}
    \bigr) \dsum S_1 \bigl( T_{d-1} / T_1 J_{T,d-2} \bigr).
  \end{displaymath}

  % This is correct even for $d \le 3$ since $J_{S,1} = J_{T,1} = 0$.

  % If $D_1, \dots, D_k$ are those of degree $d-1$ in a minimal set of
  % generators for $J_S$, then $k = \beta^{J_S}_{1,d-1}$ and $J_{S,d-1} = S_1
  % J_{S,d-2} \dsum \langle D_1, \dots, D_k \rangle$.
  
  To find the dimension of $(R/I^2)_d$, we need the dimension of $S_{d-1} /
  S_1 J_{S,d-2}$. We note that $S_{d-1} / S_1 J_{S,d-2} \iso S_{d-1} /
  J_{S,d-1} \dsum J_{S,d-1} / S_1 J_{S,d-2}$ as $\K$-vector spaces. And
  furthermore, $\dim_\K S_{d-1} / J_{S,d-1} = \dim_\K ( S/J_S )_{d-1} =
  \dim_\K ( S/J_S )_1 = s$ and $\dim_\K (J_{S,d-1} / S_1 J_{S,d-2}) =
  \beta^{J_S}_{1,d-1}$. Thus
  \begin{displaymath}
    \dim_\K T_1 \bigl( S_{d-1} / S_1 J_{S,d-2} \bigr) = t \bigl( s +
    \beta^{J_S}_{1,d-1} \bigr),
  \end{displaymath}
  and similarly $\dim_\K S_1 \bigl( T_{d-1} / T_1 J_{T,d-2} \bigr) = s \bigl(
  t + \beta^{J_T}_{1,d-1} \bigr)$. Therefore,
  \begin{displaymath}
    \dim_\K \bigl( R/I^2 \smash{\bigr)}_d = \dim_\K \bigl( S/J_S^2
    \smash{\bigr)}_d + \dim_\K \bigl( T/J_T^2 \smash{\bigr)}_d + 2st + t
    \beta^{J_S}_{1,d-1} + s \beta^{J_T}_{1,d-1}.
  \end{displaymath}
  Note that $\beta^{J_S}_{1,d-1} = \sum_{i=1}^{n-1} \beta^i_{1,d-1}$ since $d
  \ge 4$. Induction on $n$ now gives
  \begin{equation} \label{eq:I2} \tag{$*$}
    \dim_\K \bigl( R/I^2 \smash{\bigr)}_d = \sum_{i=1}^n \dim_\K \bigl(
    S^i/J_{S^i}^2 \smash{\bigr)}_d + \sum_{i=1}^n s_i (r-s_i) + \sum_{i=1}^n
    (r-s_i) \beta^i_{1,d-1}.
  \end{equation}
  
  Next we no longer assume $I_1 = 0$. Let $V = \dsum_{i=1}^n V_i$, $S = \K
  [V]$, $J = I \isect S$ and $s = \sum_{i=1}^n s_i \le r$. Let $T = \K [I_1]$
  so that $R = S \otimes_\K T$. Since $I_e = R_{e-1} T_1 \dsum J_e$ for all
  $e$, it follows that $(I^2)_d = R_{d-2} T_2 \dsum T_1 J_{d-1} \dsum
  (J^2)_d$. This implies that $\dim_\K (R/I^2)_d = \dim_\K (S/J^2)_d +
  s(r-s)$. Since $J_1 = 0$, we can find $\dim_\K (S/J^2)_d$ by using
  \eqref{eq:I2} (with $r$ replaced by $s$). Doing this proves (c).
  
  To prove (d), we use the morphism $\rho : \Phi_{\underline{s}} \times
  \sprod_{i=1}^n \Gor (s_i, H_i) \to \Gor (r,H)$ from lemma \ref{lem:Split}.
  For each $i$ let $X_i$ be the unique irreducible component of $\Gor (s_i,
  H_i)$ containing $S^i/J_i$. It is indeed unique since $S^i/J_i$ is a smooth
  point on $\PGor (s_i, H_i)$. Let $\rho' : \Phi_{\underline{s}} \times
  \sprod_{i=1}^n X_i \to \Gor (r,H)$ be the restriction of $\rho$, and let
  $\overline{\im \rho'}$ be the closure of $\im \rho'$ in $\Gor (r,H)$. Note
  that $\overline{\im \rho'}$ is irreducible. It is well known that the fiber
  $(\rho')^{-1} (R/I)$ must have dimension
  \begin{displaymath}
    \ge \dim \biggl( \Phi_{\underline{s}} \times \prod_{i=1}^n X_i \biggr) -
    \dim \overline{\im \rho'}.
  \end{displaymath}
  Furthermore, $\dim (\rho')^{-1} (R/I) \le \dim \rho^{-1} (R/I) =
  \sum_{i=1}^n s_i^2$ by lemma \ref{lem:Split}. Note that $\dim X_i = \dim_\K
  \cT_{S^i/J_i}$ since $S^i / J_i$ is a smooth point on $\PGor (s_i,H_i)$.
  Since $\beta^i_{1,d-1} = 0$, it follows from (c) that the dimension of $\Gor
  (r,H)$ at $R/I$ is
  \begin{align*}
    \dim_{R/I} \Gor (r,H) & \ge \dim \overline{\im \rho'} \\
    & \ge \dim \biggl( \Phi_{\underline{s}} \times
    \prod_{i=1}^n X_i \biggr) - \sum_{i=1}^n s_i^2 \\
    & = \sum_{i=1}^n \dim_\K \cT_{S^i/J_i} + \sum_{i=1}^n s_i (r-s_i) \\
    & = \dim_\K \cT_{R/I} \ge \dim_{R/I} \Gor (r,H)
  \end{align*}
  Hence $\dim_\K \cT_{R/I} = \dim_{R/I} \Gor (r,H) = \dim \overline{\im
    \rho'}$. Thus $R/I$ is a smooth point on $\PGor (r,H)$, and is therefore
  contained in a unique irreducible component $X$ of $\PGor (r,H)$. Since
  $\dim X = \dim_{R/I} \Gor (r,H) = \dim \overline{\im \rho'}$, it follows
  that only one component of $\overline{\Split (r,\underline{H})}$ contains
  $R/I$, namely $\overline{\im \rho'}$.
  
  The final statement follows easily.
\end{proof}

% If $\overline{\Split (r,\underline{H})}$ had another component containing
% $R/I$, then this closed, irreducible set must be contained in an irreducible
% component of $\Gor (r,H)$ containing $R/I$, of which there is only one.

\begin{rem}
  We assume in this remark that $d=3$.  We see from the proof of theorem
  \ref{thm:gor} that the dimension formula in (a) is valid also in this case.
  But the formula in (b) is no longer true in general. We need an additional
  correction term on the right-hand side. It is not difficult to show that
  this correction term is
  \begin{math}
    \sum_{i<j<k} s_i s_j s_k.
  \end{math}
  Note that if $d=3$ then $\beta^i_{1,d-1} = \tbinom{s_i}{2}$ for all $i$. It
  follows that the tangent space dimension when $d=3$ is
  \begin{displaymath}
    \dim_\K \cT_{R/I} = \sum_{i=1}^n \dim_\K \cT_{S^i/J_i} + \sum_{i=1}^n s_i
    (r-s_i) + \binom{s}{3} - \sum_{i=1}^n \binom{s_i}{3}.
  \end{displaymath}
  Thus $\dim_\K \cT_{R/I} > \dim \PSplit (r,\underline{H})$ when $n \ge 2$,
  except $n=2$ and $s_1 = s_2 = 1$.
\end{rem}

\begin{rem}
  Let $\hat{\beta}_{ij}$ be the shifted graded Betti numbers of $R/ \ann_R f$.
  The Hilbert function of $R/ \ann_R f$ for a general $f \in \cR_d$ is equal
  to
  \begin{displaymath}
    H_{d,r} (e) = \min (\dim_\K R_e, \dim_\K R_{d-e})
  \end{displaymath}
  by \cite[Proposition 3.12]{IK}. This is equivalent to $\ann_R (f)_e = 0$ for
  all $e \le d/2$, that is, $\hat{\beta}_{1j} = 0$ for all $j \le d/2 - 1$.
  It follows that $\hat{\beta}_{ij} = 0$ for all $i>0$ and $j \le d/2 - 1$.
  Recall that $\hat{\beta}_{ij} = \hat{\beta}_{r-i,d-j}$ since the minimal
  resolution of $R/ \ann_R f$ is symmetric, hence $\hat{\beta}_{ij} = 0$ for
  all $i<r$ and $j \ge d - (d/2 - 1) = d/2 + 1$. This shows that, if $d = 2m$,
  then $\hat{\beta}_{ij} = 0$ for all $j \ne m$, and if $d = 2m+1$, then
  $\hat{\beta}_{ij} = 0$ for all $j \ne m, m+1$, except $\hat{\beta}_{00} =
  \hat{\beta}_{rd} = 1$. Therefore, when $d \ge 6$, it follows that
  $\beta_{1,d-1} = \hat{\beta}_{1,d-2} = 0$ for a general $f \in \cR_d$.
  
  It is known that $\PGor (r,H)$ is smooth and irreducible for $r \le 3$.
  (For $r=3$ see \cite{Dies} and \cite{Kle}.) It is also known to be
  generically smooth in some cases with $r>3$, see \cite{IK}. Hence we can use
  theorem \ref{thm:gor} to produce irreducible, generically smooth components
  of $\PGor (r,H)$ for suitable $H$ when $d \ge 6$.
\end{rem}

% Thus when $r \le 3$, $\overline{\PSplit (\underline{H})}$ must be all of
% $\PGor (H)$ for the corresponding $H$ ! I guess I can show this directly.

%% \input{limits}
%% Chapter 4: Degenerated splittings
\chapter{Degenerate splittings}\label{chapter:limits}

In chapter \ref{chapter:regular} we proved that if $A \in M_f$ is idempotent,
then the polynomial $g$ satisfying $\p g = A \p f$ is an additive component of
$f$. In this chapter we will study what happens when $A$ is nilpotent. The
idea is to ``deform'' the situation so that $f,g \in \cR_d$ becomes $f_t, g_t
\in \cR_d [t_1, \dots, t_n]$ and $A$ becomes an idempotent $A_t \in \Mat_{\K
  [t_1, \dots, t_n]} (r,r)$, preserving the relation $\p g_t = A_t \p f_t$.

% Our investigations in this chapter were guided by the following question.

% \bigskip\noindent\hfill\parbox[c]{11cm}{%
%   Given $f \in \cR_d$, $d \ge 3$, is it possible to find $f_t \in \cR_d
%   [t_1, \dots, t_n]$ such that $f_0 = f$ and $f_t$ splits regularly $\dim_k
%   M_f - 1$ times over $\K (t_1, \dots, t_n)$?%
% }\hfill(Q)\bigskip

% The next two sections will deal with cases where we can give a positive
% answer and where we can produce counter examples, respectively.

% \begin{enumerate}
%   \setlength{\itemsep}{0pt}
%   \setlength{\parskip}{0pt}
% \item[(Q)] Given $f \in \cR_d$, $d \ge 3$, is it possible to find $f_t \in
%   \cR_d [t_1, \dots, t_n]$ such that $f_0 = f$ and $f_t$ splits regularly
%   $\dim_k M_f - 1$ times over $\K (t_1, \dots, t_n)$?
% \end{enumerate}

Our investigations in this chapter were guided by the following question.
\begin{que} \label{que:Q}
  Given $f \in \cR_d$, $d \ge 3$, is it possible to find $f_t \in \cR_d [t_1,
  \dots, t_n]$ such that $f_0 = f$ and $f_t$ splits regularly $\dim_k M_f - 1$
  times over $\K (t_1, \dots, t_n)$?
\end{que}

Sections \ref{section:positive} and \ref{section:negative} deal with cases
where we can give a positive answer to this question, and cases in which we
can produce counter examples, respectively. The motivation behind the question
is that $\dim_k M_f - 1$ is an upper bound for the number of times that $f_t$
can split when we require $f_0 = f$, see lemma \ref{lem:upbound} below. There
is also a flatness condition we would like $f_t$ to satisfy, but we will
ignore that in this paper, cf. remark \ref{rem:flatness}.

Note that $\dim_k M_f - 1 = r \beta_{11} + \beta_{1d}$ by lemma \ref{lem:Mf}.
Since $f_t$ can split at most $r-1$ times (that is, have at most $r$ additive
components), we see that question \ref{que:Q} automatically has a negative
answer if $\beta_{11} > 0$, i.e. if $\ann_R (f)_1 \ne 0$.

% Allowing $\beta_{11} > 0$ and changing the question to ``$f_t$ splits
% regularly $\beta_{1d}$ times over $\K (t_1, \dots, t_n)$'' is not any
% better, see example \ref{ex:b11} below.

Recall that by corollary \ref{cor:basic} the ``regular splitting properties''
of $f$ does not change if we add dummy variables since any regular splitting
must happen inside the subring $\K [R_{d-1} (f)]^{DP} \subseteq \cR$. It is
not so for degenerate splittings, as seen in example \ref{ex:b11} below. For
this reason most $f$ we consider in this chapter will satisfy $\ann_R (f)_1 =
0$. Note that this implies that $\dim_k M_f - 1 = \beta_{1d}$.

We will now prove that the number $\dim_\K M_f - 1$ in question \ref{que:Q} is
an upper bound. Recall that by theorem \ref{thm:regsplitE} the regular
splittings of $f_t$ inside $\cR_d \otimes_\K \K (t_1, \dots, t_n) = \cR_d
(t_1, \dots, t_n)$ are determined by the idempotents in
\begin{displaymath}
  M_{f_t} = \{ A \in \Mat_{\K (t_1, \dots, t_n)} (r,r) \suchthat I_2 (\p\:
  A\p) \subseteq \ann_{R (t_1, \dots, t_n)} f_t \}.
\end{displaymath}

\begin{lem} \label{lem:upbound}
  Let $f_t \in \cR_d [t_1, \dots, t_n]$. Then $\dim_{\K (t_1, \dots, t_n)}
  M_{f_t} \le \dim_\K M_{f_0}$. In particular, if $f_t$ splits regularly $m$
  times, then $m \le \dim_\K M_{f_0} - 1$.
\end{lem}

% In a certain sense this says that all degenerate splitting properties of $f$
% are contained in $M_f$, as we have defined degenerate. All regular are
% obviously contained, thus $M_f$ contains everything we need.

\begin{proof}
  First assume that $n=1$. Then $f_t = \sum_{k \ge 0} t^k f_k$ for some $f_k
  \in \cR_d$. Let $A_1, \dots, A_m \in \Mat_{\K (t)} (r,r)$ form a basis for
  $M_{f_t}$ as a $\K (t)$-vector space. We may multiply by denominators and
  assume $A_i \in \Mat_{\K [t]} (r,r)$ for all $i$. Write $A_i =
  \sum_{k=0}^{a_i} t^k A_{ik}$ with $A_{ik} \in \Mat_\K (r,r)$. Assume that
  $A_{10}, \dots, A_{m0}$ are linearly dependent, say $\sum_{i=0}^m c_i A_{i0}
  = 0$ where $c_i \in \K$, not all zero. Choose $j$ such that $a_j = \max \{
  a_i \suchthat c_i \ne 0 \}$, and replace $A_j$ with $(c_j t)^{-1}
  \sum_{i=0}^m c_i A_i$. The new $A_i$'s still form a $\K (t)$-basis for
  $M_{f_t}$, and the degree of $A_j$ as a polynomial in $t$ has decreased.
  Continuing this process, we arrive at a basis $\{ A_i \}$ such that $A_{10},
  \dots, A_{m0}$ are linearly independent.

  For every $i$, since $A_i \in M_{f_t}$, there exists a polynomial $g_i \in
  \cR_d (t)$ such that $\p g_i = A_i \p f_t$. And because $A_i \in \Mat_{\K
    [t]} (r,r)$ it follows that $g_i \in \cR_d [t]$. Thus $g_i = \sum_{k \ge
    0} t^k g_{ik}$ for suitable $g_{ik} \in \cR_d$. It follows that
  \begin{displaymath}
    \sum_{k \ge 0} t^k \p g_{ik} = \p g_i = A_i \p f_t = \sum_{j,k \ge 0}
    t^{j+k} A_{ij} \p f_k.
  \end{displaymath}
  In particular, $\p g_{i0} = A_{i0} \p f_0$, implying $A_{i0} \in M_{f_0}$
  for all $i$. Since $\{ A_{i0} \}$ are linearly independent, it follows that
  $\dim_\K M_{f_0} \ge \dim_{\K (t)} M_{f_t}$.

  For general $n \ge 1$, let $\K' = \K (t_1, \dots, t_{n-1})$. There exist
  $f'_k \in \cR_d [t_1, \dots, t_{n-1}]$ such that $f_t = \sum_{k \ge 0} t_n^k
  f'_k$, and the above argument shows that $\dim_{\K'} M_{f'_0} \ge \dim_{\K'
    (t_n)} M_{f_t}$. Induction on $n$ proves that $\dim_\K M_{f_0} \ge
  \dim_{\K (t_1, \dots, t_n)} M_{f_t}$.

  If $f_t$ splits regularly $m$ times, then $M_{f_t}$ contains $m+1$
  orthogonal idem\-potents, hence $\dim_\K M_{f_0} \ge \dim_{\K (t_1, \dots,
    t_n)} M_{f_t} \ge m+1$.
\end{proof}

\begin{exmp} \label{ex:b11}
  Let $d \ge 4$ and $f = x_1^\pd {d-2} x_2^\pd 2 \in \cR = \K [x_1,
  x_2]^{DP}$. With $R = \K [\p_1, \p_2]$ we get $\ann_R f = (\p_2^3,
  \p_1^{d-1})$ and $M_f = \langle I \rangle$, hence $f$ cannot be a
  specialization of an $f_t \in \cR_d [t_1, \dots, t_n]$ that splits. But it
  is easy to find $f_t \in \K [t] [x_1, x_2, x_3]^{DP}$ such that $f_0 = f$
  and $f_t$ splits! Indeed, one such choice is
  \begin{displaymath}
    f_t = t^{-3} [ t (x_1 + t x_2 + t^3 x_3)^\pd{d} - (x_1 + t^2 x_2)^\pd{d} +
    (1-t) x_1^\pd{d} ] \equiv f \text{ mod } (t).
  \end{displaymath}
  Note that even this is in concordance with lemma \ref{lem:upbound}.
\end{exmp}

% Such an $f_t$ will not give a flat deformation, since $\ann_{S_t} (f_t)_1
% \ne \ann_S (f)_1$, $S = \K [\p_1, \p_2, \p_3]$.

\begin{rem} \label{rem:flatness}
  Let $f \in \cR_d$. When we look for $f_t \in \cR_d [t_1, \dots, t_n]$ such
  that $f_0 = f$, there are several properties we would like $f_t$ to have.
  Our main concern in this chapter is that we want $f_t$ to split regularly
  $\dim_\K M_f - 1$ times over $\K (t_1, \dots, t_n)$, giving a positive
  answer to question \ref{que:Q}. But in addition, we would like $R (t_1,
  \dots, t_n) (f_t) \iso R (t_1, \dots, t_n) / \ann_{R (t_1, \dots, t_n)}
  (f_t)$ and $R (f) \iso R/ \ann_R f$ to have equal Hilbert functions, for the
  following reason.

  % From a deformational point of view, there is another condition we want
  % $f_t$ to satisfy.

  Let $\K_t = \K [t_1, \dots, t_n]$, $\cR_t = \cR \otimes_\K \K_t$ and $R_t =
  R \otimes_\K \K_t$. An $f_t \in \cR_d \otimes_\K \K_t$ determines a family
  $\K_t \to \cR_t / R_t (f_t)$. Let $C_t = R_t (f_t) = R (f_t) \otimes_\K \K_t
  \subseteq \cR_t$. It is easy to show that $\cR / C_0 = \cR_t / C_t
  \otimes_{\K_t} \K_t / (t_1, \dots, t_n) = \cR / R (f_0)$, thus $R (f_0)$ is
  a specialization of the family. We would like this family to be flat, at
  least in an open neighbourhood of the origin. This simply means that the
  generic fiber $R (t_1, \dots, t_n) (f_t)$ has the same Hilbert function as
  $R (f_0)$. (The condition that $f_t$ should have a regular splitting of
  length $\dim_k M_f$ inside $\cR_d (t_1, \dots, t_n)$, is also a statement
  about the generic fiber.)

% Since the base $\K_t$ is reduced.
  
  Note that, although the family $\K_t \to R_t / J_t$ where $J_t = \ann_{R_t}
  (f_t)$ is maybe more natural to consider, it is also more problematic, since
  $f_t \mapsto R_t / J_t \mapsto R / J_0$ does not generally commute with med
  $f_t \mapsto f_0 \mapsto R / \ann_R (f_0)$. In general we only have an
  inclusion $J_0 \subseteq \ann_R (f_0)$. If $f \ne 0$, then $(J_0)_d = \ann_R
  (f_0)_d$, and since $\ann_R (f_0)$ is determined by its degree $d$ piece by
  lemma \ref{lem:ann}a, it follows that $\ann_R (f_0) = \sat_{\le d} J_0 =
  \dsum_{e=0}^d \{ D \in R_e \suchthat R_{d-e} \cdot D \subseteq J_0 \} +
  (R_{d+1})$.

  % Cf. limit ideals, f.ex. \cite[Lemma 5.12]{IK}.

  % If we use only one deformational parameter ($n=1$), then this second
  % family is generically flat, since $\ann_{R_t} (f_t) = R_t \isect \ann_{R
  %   (t_1, \dots, t_n)} (f_t)$.
  
  Of course we would like $R(f)$ to be a specialization of a flat, splitting
  family, but in this chapter we study question \ref{que:Q} without the
  additional flatness requirement. Note that we do not know of any example in
  which question \ref{que:Q} has a positive answer, but would have had a
  negative answer if we had required $H(R (t_1, \dots, t_n) (f_t)) = H( R(f)
  )$.
\end{rem}

\section{Positive results} \label{section:positive}

In this section we consider some cases where we are able to prove that
question \ref{que:Q} has a positive answer. We start with a result that
effectively ``deforms'' a relation $\p g = A \p f$ with $A$ nilpotent to a
relation $\p g_t = A_t \p f_t$ with $A_t$ idempotent. The proof is an explicit
construction of $f_t$ using the nilpotent matrix $A \in M_f$ as input data.
This will later allow us to answer question \ref{que:Q} positively when $r \le
4$.

Suppose $A$ is nilpotent, i.e. $A^k = 0$ for $k \gg 0$. The \emph{index} of
$A$ is defined by
\begin{displaymath}
  \indx (A) = \min \{ k \ge 1 \suchthat A^k = 0 \}.
\end{displaymath}

Let $A$ be a nilpotent matrix of index $n+1$, i.e., $A^{n+1} = 0$ and $A^n \ne
0$. Then $A^0 = I, A, A^2, \dots, A^n$ are linearly independent. To see why,
assume there is a non-zero relation $\sum_{k=0}^n c_k A^k = 0$, and let $i =
\min \{ k \suchthat c_k \ne 0 \} \le n$. Multiplying the relation by $A^{n-i}$
implies that $c_i A^n = 0$, which is a contradiction.

\begin{thm} \label{thm:onematrix}
  Let $d \ge 3$ and $f \in \cR_d$. Assume that $M_f$ contains a non-zero
  nilpotent matrix $A \in \Mat_\K (r,r)$, and let $n = \indx (A) - 1 \ge 1$.
  Then $f$ is a specialization of some $f_t \in \cR_d [t_1, \dots, t_n]$ that
  splits regularly $n$ times inside $\cR_d (t_1, \dots, t_n)$.
\end{thm}

\begin{proof}
  Since $M_f$ is closed under multiplication by proposition \ref{prop:Mf}, it
  contains $\K [A] = \langle I, A, \dots, A^n \rangle$, the $\K$-algebra
  generated by $A$.

  Choose an idempotent $E \in \Mat_\K (r,r)$ such that $\ker E = \ker A^n$.
%% CHANGED 2013-07-13: $U = \ker A$ -> $U = \ker A^n$
  (I.e. let $U = \ker A^n$ and choose $W$ such that and $U \isect W = 0$ and $U
  + W = \K^r$. Then let $E$ represent the linear map that acts as the identity
  on $W$ and takes $U$ to 0.) This implies that $A^n E = A^n$ and that there
  exists a matrix $Q \in \Mat_\K (r,r)$ such that $E = Q A^n$. Note that $EA =
  0$. Define
  \begin{displaymath}
    A_t = A + tE.
  \end{displaymath}
  Then $A_t^n = A^n + t A^{n-1} E + \dots + t^n E$, and
  \begin{displaymath}
    A_t^{n+1} = A^{n+1} + t A^n E + \dots + t^{n+1} E = t A_t^n.
  \end{displaymath}
  It follows that $(A_t^n)^2 = t^n A_t^n$, hence $t^{-n} A_t^n$ is idempotent.
  Now define
  \begin{displaymath}
    P = I + \sum_{k=1}^n t^k A^{n-k} Q.
  \end{displaymath}

  $P$ is chosen so that $A_t^n = P A^n$. Since $\det P \equiv 1 \pmod t$, $P$
  is an invertible element of $\Mat_{\K(t)} (r,r)$. Let $\phi_P$ be the
  homomorphism defined by $x \mapsto P^\T x$ on $\cR$ and by $\p \mapsto
  P^{-1} \p$ on $R$, as usual. Recall that for all $g \in \cR$ and $D \in R$
  we have $\phi_P (Dg) = \phi_P (D) \phi_P (g)$. Also note that $(PA^n)^2 =
  t^n PA^n$ implies $A^n P A^n = t^n A^n$.

% $P = I + t A^{n-1} Q + t^2 A^{n-2} Q + \dots + t^n Q$ would also work.

  Since $A^n \in M_f$, there exists a polynomial $g \in \cR_d$ such that $\p g
  = A^n \p f$. Let $g_t = \phi_P (g) = \sum_{k \ge 0} t^k g_k \in \cR_d [t]$,
  and define
  \begin{displaymath}
    f_t = f + t^{-n} \biggl( g_t - \sum_{k=0}^n t^k g_k \biggr) = f +
    \sum_{k>0} t^k g_{n+k} \in \cR_d [t].
  \end{displaymath}

  We want to prove that $A_t \in M_{f_t}$. We start by calculating $\p g_t$.
  \begin{equation} \label{eq:pgt}
    \p g_t = \p \phi_P (g) = P \phi_P (\p g) = P \phi_P (A^n \p f) = A_t^n
    \phi_P (\p f)
  \end{equation}
  Multiplying \eqref{eq:pgt} by $A^n$, and using $A^n P A^n = t^n A^n$, gives
  $A^n \p g_t = t^n \phi_P (\p g)$. Since the entries of $\p g$ and $\phi_P
  (\p g)$ are in $\cR [t]$, this implies that $A^n \p g_i = 0$ for all $i <
  n$, and $A^n \p g_n = \p g = A^n \p f$. In particular, $E \p g_n = Q A^n \p
  g_n = E \p f$.

  When we multiply \eqref{eq:pgt} by $A_t$, the result is $A_t \p g_t = t \p
  g_t$. As polynomials in $t$ this equals $(A + tE) (\sum_{i \ge 0} t^i \p
  g_i) = t (\sum_{i \ge 0} t^i \p g_i)$, and implies that
  \begin{displaymath}
    A \p g_i + E \p g_{i-1} = \p g_{i-1} \text{ for all } i \ge 0.
  \end{displaymath}
  (Actually, this implies that $A \p g_i = \p g_{i-1}$ for all $0 \le i \le
  n$, since $E = Q A^n$ and we have already proven that $A^n \p g_{i-1} = 0$
  for $i \le n$.) Also, since $A \in M_f$, there exists $h \in \cR_d$ such
  that $\p h = A \p f$.

% $A^n \p g_i = 0$ for all $i < n$ can also be proven by induction on $A \p
% g_i + E \p g_{i-1} = \p g_{i-1}$ for all $i \ge 0$: $A^i \p g_{i-1} = 0$ for
% $i<n$ implies $A^n \p g_{i-1} = 0$, hence $A^{i+1} \p g_i = A^i \p g_{i-1} -
% Q A^n \p g_{i-1} = 0$.

  Putting all this together, we get
  \begin{align*}
    A_t \p f_t & = \bigl( A + tE \bigr) \biggl( \p f + \sum_{k>0} t^k \p
    g_{n+k} \biggr) \\
    & = A \p f + t E \p f + \sum_{k>0} t^k A \p g_{n+k} + \sum_{k>0} t^{k+1} E
    \p g_{n+k} \\
    & = \p h + \sum_{k>0} t^k \bigl( A \p g_{n+k} + E \p g_{n+k-1} \bigr) \\
    & = \p h + \sum_{k>0} t^k \p g_{n-1+k} = \p \bigl( h + t g_n + t^2 g_{n+1}
    + \dots \bigr).
  \end{align*}
  This proves that $A_t \in M_{f_t}$. And since $M_{f_t}$ is closed under
  multiplication, it follows that $\K [A_t] = \langle I, A_t, \dots, A_t^n
  \rangle \subseteq M_{f_t}$.

  Since $E' = I - t^{-n} A_t^n$ is idempotent, we may apply theorem
  \ref{thm:regsplitE}. It tells us that $f_t$ has a regular splitting with two
  additive components, $t^{-n} g_t$ and $f' = t^{-n} ( t^n f - g_0 - t g_1 -
  \dots - t^n g_n )$, and furthermore that
  \begin{displaymath}
    \K [A_t] \cdot E' = \langle E', A_t E', \dots, A_t^{n-1} E' \rangle
    \subseteq M_{f'}^{E'}.
  \end{displaymath}
  Hence we may repeat our procedure on $f'$. By induction on $n$, we arrive at
  some $f_{\underline{t}} \in \cR_d [t_1, \dots, t_n]$ such that $f_0 = f$ and
  $f_{\underline{t}}$ splits regularly $n$ times.
\end{proof}

\begin{rem}
  The choice of $E$ in the proof of theorem \ref{thm:onematrix} boils down to
  choosing $Q \in \Mat_\K (r,r)$ such that $A^n Q A^n = A^n$, and then letting
  $E = Q A^n$. This then implies $\ker E = \ker A^n$ and that $E$ is
  idempotent. We note that $Q$ is certainly not unique. If $A^n$ is in Jordan
%% CHANGED 2013-07-13: $Q = A^\T$ -> $Q = (A^n)^\T$
  normal form, then we may let $Q = (A^n)^\T$. This is what we will do in most
  explicit cases.
\end{rem}

\begin{cor} \label{cor:onematrix}
  Suppose $\K = \bar\K$ and $d \ge 3$. Let $f \in \cR_d$. Assume that $\ann_R
  (f)_1 = 0$, and let $\beta_{1j}$ be the minimal number of generators of
  $\ann_R f$ of degree $j$. Then $f$ has a regular or degenerate splitting if
  and only if $\beta_{1d} > 0$.
\end{cor}

% More precisely, $\beta_{1d} > 0$ if and only if $f$ has a non-trivial
% regular or degenerate splitting, i.e. of length $\ge 2$.

\begin{proof}
  Since $\beta_{11} = 0$, we have $\dim_\K M_f - 1 = \beta_{1d}$. Thus
  $\beta_{1d} > 0$ if and only if $M_f$ contains a matrix $A \notin \langle I
  \rangle$. Since $\K = \bar\K$, we may assume that $A$ is either idempotent
  or nilpotent. It follows from theorem \ref{thm:regsplitE} that $M_f$
  contains a non-trivial idempotent if and only if $f$ splits regularly. By
  theorem \ref{thm:onematrix}, if $A \in M_f$ is non-zero and nilpotent, then
  $f$ has a degenerate splitting. Finally, if $f$ has a degenerate splitting,
  then $\dim_\K M_f - 1 \ge 1$ by lemma \ref{lem:upbound}.
\end{proof}

% With a more general definition of degenerate splittings, this is true over
% any field.

Let $f \in \cR_d$ with $d \ge 3$. If $M_f$ is generated by one matrix, then
theorem \ref{thm:onematrix} answers question \ref{que:Q} affirmatively, that
is, we can find $f_t \in \cR_d [t_1, \dots, t_n]$ such that $f_0 = f$ and
$f_t$ splits regularly $\dim_\K M_f - 1$ times over $\K (t_1, \dots, t_n)$.
This is the best we can hope for by lemma \ref{lem:upbound}, and our next
theorem proves that this is always possible when $r \le 4$. But first we need
some facts about matrices.

\begin{lem} \label{lem:ranks}
  Given matrices $A, B \in \Mat_\K (r,r)$ the following are true.
  \begin{enumerate}
    \setlength{\itemsep}{2pt}
    \setlength{\parskip}{0pt}
    \renewcommand{\theenumi}{\alph{enumi}}
    \renewcommand{\labelenumi}{\normalfont(\theenumi)}
  \item $\rank A + \rank B - r \le \rank (AB) \le \min (\rank A, \rank B)$.
  \item If $AB = BA$, $A \ne 0$ and $B$ is nilpotent, then $\rank (AB) < \rank
    A$.
  \item If $AB = BA$, $\rank A = r-1$ and $A$ is nilpotent, then $A^{r-1} \ne
    0$ and
    \begin{displaymath}
      B \in \K [A] = \langle I, A, \dots, A^{r-1} \rangle.
    \end{displaymath}
  \end{enumerate}
\end{lem}

\begin{proof}
  (a) The right inequality follows from the inclusions $\ker (AB) \supseteq
  \ker B$ and $\im (AB) \subseteq \im A$. To prove the left inequality, let
  $\beta$ be the restriction of the map $B : \K^r \to \K^r$ to $\ker (AB)$.
  Obviously, $\ker \beta = \{ v \in \ker (AB) \suchthat Bv = 0 \} = \ker B$,
  and $\im \beta \subseteq \ker A$. Hence
  \begin{displaymath}
    \dim_\K \ker (AB) = \dim_\K \ker \beta + \dim_\K \im \beta \le \dim_\K
    \ker B + \dim_\K \ker A,
  \end{displaymath}
  which is equivalent to $\rank (AB) \ge \rank A + \rank B - r$.

  (b) Assume that $\rank (AB) = \rank A$. We know that $\im (AB) \subseteq \im
  A$, hence equal ranks implies $\im (AB) = \im A$. It follows that $\im
  (AB^k) = \im A$ for all $k$ by induction on $k$. Indeed, since $AB = BA$, we
  have
  \begin{displaymath}
    \im AB^{k+1} = \im BAB^k = B (\im AB^k) = B (\im A) = \im BA = \im AB =
    \im A.
  \end{displaymath}
  But $B$ is nilpotent, implying $\im A = \im AB^r = \im 0 = 0$. Hence $A=0$.
  Therefore, when $A \ne 0$, it follows that $\rank AB < \rank A$.

% Below we use $A^r = 0$. This is well know, and follows from (a) and (b):
% First prove $\rank A^{k+1} \ge \rank A^k - 1$, hence $\rank A^k \ge r-k$.
% Thus $A^k \ne 0$ for all $k<r$, and by (b) $\rank A^{k+1} < \rank A^k$ for
% these $k$. Hence $\rank A^k = r-k$ for all $1 \le k \le r$.

  (c) Let $A^0 = I$. Part (a) implies for all $k \ge 0$ that
  \begin{displaymath}
    \rank A^{k+1} \ge \rank A^k + \rank A - r = \rank A^k - 1.
  \end{displaymath}
  Since $A$ is nilpotent, we know that $A^r = 0$. Therefore,
  \begin{displaymath}
    0 = \rank A^r \ge \rank A^{r-1} - 1 \ge \rank A^{r-2} - 2 \ge \dots \ge
    \rank A - (r-1) = 0.
  \end{displaymath}
  It follows that all inequalities must be equalities, that is, $\rank A^k =
  r-k$ for all $0 \le k \le r$. In particular, $A^{r-1} \ne 0$. Moreover, the
  quotient $\ker A^k / \ker A^{k-1}$ has dimension 1 for all $1 \le k \le r$.
  Consider the filtration
  \begin{displaymath}
    0 = \ker I \subsetneq \ker A \subsetneq \ker A^2 \subsetneq \dots
    \subsetneq \ker A^{r-1} \subsetneq \ker A^r = \K^r.
  \end{displaymath}
  Choose $v_1 \notin \ker A^{r-1}$, and let $v_k = A^{k-1} v_1$ for $k = 2,
  \dots, r$. Then $\{ v_1, \dots, v_r \}$ is a basis for $\K^r$. To prove
  this, note that $v_k \notin \ker A^{r-k}$ because $A^{r-1} v_1 \ne 0$, but
  $v_k \in \ker A^{r-k+1}$ since $A^r = 0$. Assume that $v_1, \dots, v_r$ are
  linearly dependent. Then there exist $c_1, \dots, c_r \in \K$, not all zero,
  such that $\sum_{i=1}^r c_i v_i = 0$. If we let $k = \min \{ i \suchthat c_i
  \ne 0 \}$, then $v_k = c_k^{-1} ( \sum_{i=k+1}^r c_i v_i )$. But $v_i \in
  \ker A^{r-k}$ for all $i>k$, implying $v_k \in \ker A^{r-k}$, a
  contradiction.

  There exist $c_1, \dots, c_r \in \K$ such that $B v_1 = \sum_{i=1}^r c_i v_i
  = \sum_{i=1}^r c_i A^{i-1} v_1$ since $\{ v_1, \dots, v_r \}$ is a basis for
  $\K^r$. Since $AB=BA$ it follows for all $k$ that
  \begin{align*}
    B v_k & = B A^{k-1} v_1 = A^{k-1} B v_1 \\
    & = A^{k-1} \sum_{i=1}^r c_i A^{i-1}
    v_1 = \sum_{i=1}^r c_i A^{i-1} A^{k-1} v_1 = \sum_{i=1}^r c_i A^{i-1} v_k.
  \end{align*}
  Since $\{ v_i \}$ is a basis, it follows that $B = \sum_{i=1}^r c_i
  A^{i-1}$, that is, $B \in \K [A]$.
\end{proof}

The following theorem gives a positive answer to question \ref{que:Q} when $r
\le 4$.

\begin{thm} \label{thm:rle4}
  Assume that $r \le 4$ and $\bar{\K} = \K$. Let $f \in \cR_d$, $d \ge 3$,
  satisfy $\ann_R (f)_1 = 0$. Then for some $n \ge 1$ there exists $f_t \in
  \cR_d [t_1, \dots, t_n]$ such that $f_0 = f$ and $f_t$ splits regularly
  $\dim_\K M_f - 1$ times over $\K (t_1, \dots, t_n)$.
\end{thm}

%\begin{proof}
\noindent\textit{Proof:}
We may assume that $M_f$ does not contain any non-trivial idempotent, because
if it does, we apply theorem \ref{thm:regsplitE} first, and then the following
proof on each additive component. Since $\bar{\K} = \K$, it follows by
proposition \ref{prop:eigen} that $M_f = \langle I \rangle \dsum M_f^{\nil}$
where $M_f^{\nil} = \{ A \in M_f \suchthat A \text{ is nilpotent} \}$.

The conclusion follows from theorem \ref{thm:onematrix} if $M_f$ is generated
by a single matrix. And if $M_f^{\nil}$ contains a matrix $A$ of rank $r-1$,
then $M_f = \K [A]$ by lemma \ref{lem:ranks}. Therefore, we now assume that
$M_f$ is \emph{not} generated by a single matrix, and in particular, that all
matrices in $M_f^{\nil}$ have rank $\le r-2$.

If $r=1$, then $f = c x_1^\pd d$ and $M_f = \langle I \rangle$, thus there is
nothing to prove. If $r=2$, then $M_f$ must be generated by a single matrix,
and we are done.

If $r=3$, then $M_f^{\nil}$ may only contain matrices of rank 1. Since $M_f$
cannot be generated by a single matrix, $M_f^{\nil}$ must contain two matrices
$A \nparallel B$ of rank 1. We may write $A = u_1 v_1^\T$ and $B = u_2 v_2^\T$
for suitable vectors $u_i, v_j \in \K^r$. Note that $A^2 = B^2 = AB = BA = 0$
since their ranks are $< 1$ by lemma \ref{lem:ranks}b. Thus $u_i^\T v_j = 0$
for all $i,j = 1,2$. If $u_1 \nparallel u_2$, then this implies $v_1 \parallel
v_2$ since $r=3$. Similarly, $v_1 \nparallel v_2$ implies $u_1 \parallel u_2$.
However, both cases are impossible, since each imply $\ann_R (f)_1 \ne 0$ by
corollary \ref{prop:UV}. (These are essentially the two cases in example
\ref{ex:UV}.)

% Note that $\im A^\T = \langle v_1 \rangle$ and $\im A^\T = \langle u_1
% \rangle^\perp$. If $u_1 \nparallel u_2$ and $v_1 \parallel v_2$, then
% $\isect \im = \langle v_1 \rangle$ and $\sum \ker = ( \langle u_1 \rangle
% \isect \langle u_2 \rangle )^\perp = \K^3$. If $u_1 \parallel u_2$ and $v_1
% \nparallel v_2$, then $\sum \im = \langle v_1, v_2 \rangle$ and $\isect \ker
% = \langle u_1 \rangle^\perp$ has codimension 1.

Suppose $r=4$ and that $M_f^{\nil}$ only contains matrices of rank $\le 2$. We
will break down the proof of this case into four subcases.

\smallskip \textit{Case 1.} Assume $M_f^{\nil}$ contains two matrices $A
\nparallel B$ of rank 1, i.e. $A = u_1 v_1^\T$ and $B = u_2 v_2^\T$. Then $u_i
v_j^\T = 0$ for all $i,j = 1,2$ as above. Again, both $u_1 \parallel u_2$ and
$v_1 \parallel v_2$ lead to contradictions by corollary \ref{prop:UV}. Thus
we may up to a base change assume $u_1 = [1000]^\T$ and $u_2 = [0100]^\T$.
Hence $v_i = [00{*}{*}]^\T$, and after another change of basis, $v_1 =
[0010]^\T$ and $v_2 = [0001]^\T$. In other words,
\begin{displaymath}
  A = \left(
    \begin{smallmatrix}
      0 & 0 & 1 & 0 \\
      0 & 0 & 0 & 0 \\
      0 & 0 & 0 & 0 \\
      0 & 0 & 0 & 0
    \end{smallmatrix}
  \right) \qquad \text{and} \qquad
  B = \left(
    \begin{smallmatrix}
      0 & 0 & 0 & 0 \\
      0 & 0 & 0 & 1 \\
      0 & 0 & 0 & 0 \\
      0 & 0 & 0 & 0
    \end{smallmatrix}
  \right).
\end{displaymath}
Since $I_2 (\p\: A\p\: B\p) \subseteq \ann f$, this already implies that there
exist $c_1, c_2 \in \K$ and $g \in \K [x_1, x_2]^{DP}$ such that $f = c_1 x_3
x_1^\pd {d-1} + c_2 x_4 x_2^\pd {d-1} + g$. Note that $c_1, c_2 \ne 0$ since
$\ann (f)_1 = 0$, and we may assume $c_1 = c_2 = 1$.

Suppose that $M_f^{\nil}$ contains a matrix $C$ in addition to $A$ and $B$.
Then $CA = AC = CB = BC = 0$ because their ranks are $<1$. This implies that
\begin{displaymath}
  C = \bigl(
  \begin{smallmatrix}
    0 & \star \\
    0 & 0
  \end{smallmatrix}
  \bigr)
  \text{ as a $2 \times 2$ block matrix using $2 \times 2$ blocks,}
\end{displaymath}
and modulo $A$ and $B$ we may assume that
\begin{math}
  \star = \bigl(
  \begin{smallmatrix}
    0 & a \\
    b & 0
  \end{smallmatrix}
  \bigr).
\end{math}
It follows that
\begin{displaymath}
  I_2 (\p\: C\p) = (b \p_1 \p_3 - a \p_2 \p_4) \subseteq \ann_R f.
\end{displaymath}
Hence $0 = (b \p_1 \p_3 - a \p_2 \p_4) (f) = b x_1^\pd{d-2} - a x_2^\pd{d-2}$.
This implies $a=b=0$ since $d \ge 3$. Thus we have proven that $M_f = \langle
I,A,B \rangle$. Let
\begin{displaymath}
  f_t = \tfrac 1t \Bigl( (x_1 + t x_3)^\pd d - x_1^\pd d + (x_2 + t x_4)^\pd d
  - x_2^\pd d \Bigr) + g.
\end{displaymath}
Then $f_0 = f$, and $f_t \sim x_3^\pd d + x_4^\pd d - \bigl( x_1^\pd d +
x_2^\pd d - tg \bigr)$ obviously splits twice.

\smallskip \textit{Case 2.} Suppose $M_f$ does not contain any matrix of rank
1. If $A,B \in M_f^{\nil}$, then both have rank 2 and $A^2 = B^2 = AB = BA =
0$. We may assume that
\begin{math}
  A = \bigl(
  \begin{smallmatrix}
    0 & I \\
    0 & 0
  \end{smallmatrix}
  \bigr),
\end{math}
which implies that
\begin{math}
  B = \bigl(
  \begin{smallmatrix}
    0 & B' \\
    0 & 0
  \end{smallmatrix}
  \bigr).
\end{math}
But then $B - \lambda A$ has rank 1 when $\lambda$ is an eigenvalue for $B'$,
a contradiction. Therefore, for the rest of the proof we may assume that $M_f$
contains exactly one matrix of rank 1.

\smallskip \textit{Case 3.} Assume $M_f$ does not contain any $A$ of rank 2
satisfying $A^2=0$. Then $M_f$ must contain an $A$ such that $\rank A = 2$ and
$A^2 \ne 0$. Note that $\rank A^2 = 1$. Because $M_f \ne \K [A]$, there exists
$B \in M_f$, $B \notin \K [A]$. Then $\rank B = 2$ since $M_f$ cannot contain
several matrices of rank 1. Thus $B^2 \ne 0$, and therefore $B^2 = b A^2$, $b
\ne 0$. Also $\rank AB \le 1$, hence $AB = BA = a A^2$. Let $t$ be a root of
$t^2 + 2at + b$. Since $\rank (tA+B) \le 1$ implies $B \in \K [A]$, we get
$\rank (tA+B) = 2$. But $(tA+B)^2 = (t^2 + 2at + b) A^2 = 0$, contradicting
our assumption.

\smallskip \textit{Case 4.} Hence $M_f$ contains a matrix $A$ of rank 2
satisfying $A^2=0$ and a matrix $B$ of rank 1. We may assume that
\begin{math}
  A = \bigl(
  \begin{smallmatrix}
    0 & I \\ 0 & 0
  \end{smallmatrix}
  \bigr).
\end{math}
From $AB=BA$ it follows that
\begin{math}
  B = \bigl(
  \begin{smallmatrix}
    B_1 & B_2 \\ 0 & B_1
  \end{smallmatrix}
  \bigr),
\end{math}
and $B_1 = 0$ since $\rank B = 1$. Modulo a similarity transformation $B
\mapsto PBP^{-1}$ with
\begin{math}
  P = \bigl(
  \begin{smallmatrix}
    Q & 0 \\ 0 & Q
  \end{smallmatrix}
  \bigr)
\end{math}
we may assume that
\begin{displaymath}
  B_2 = \Bigl(
  \begin{smallmatrix}
    \lambda_1 & 0 \\ 0 & \lambda_2
  \end{smallmatrix}
  \Bigr) \qquad \text{or} \qquad
  B_2 = \Bigl(
  \begin{smallmatrix}
    \lambda & 1 \\ 0 & \lambda
  \end{smallmatrix}
  \Bigr),
\end{displaymath}
and modulo $A$ this becomes
\begin{math}
  B_2 \in \left\{ \bigl(
    \begin{smallmatrix}
      1 & 0 \\ 0 & 0
    \end{smallmatrix}
    \bigr), \bigl(
    \begin{smallmatrix}
      0 & 1 \\ 0 & 0
    \end{smallmatrix}
    \bigr) \right\}.
\end{math}
Since $B$ is the only matrix in $M_f$ of rank 1 (up to a scalar), the first
must be disregarded. (It reduces to case 1 above.) Hence
\begin{math}
  B_2 = \bigl(
  \begin{smallmatrix}
    0 & 1 \\ 0 & 0
  \end{smallmatrix}
  \bigr).
\end{math}
It follows that 
\begin{displaymath}
  f = x_4 x_1^\pd {d-1} + x_3 x_2 x_1^\pd {d-2} + g \quad \text{ where } \quad
  g \in \K [x_1, x_2]^{DP},
\end{displaymath}
up to a base change. Define $f_t \in \cR_d [t]$ by
\begin{displaymath}
  f_t = \tfrac{1}{st} \Bigl( (x_1 + s x_2 + t x_3 + st x_4)^\pd d - (x_1 + s
  x_2)^\pd d - (x_1 + t x_3)^\pd d + x_1^\pd d \Bigr) + g.
\end{displaymath}
%% CHANGED 2013-07-13: $f_t \iso $ -> $f_t$
Then $f_0 = f$, and $f_t$ splits twice. If $M_f = \langle I, A, B
\rangle$, then we are done.

% Above, it is not necessary to use the stronger similarity transformation
% with
% \begin{math}
%   P = \bigl(
%   \begin{smallmatrix}
%     Q & 0 \\ 0 & Q'
%   \end{smallmatrix}
%   \bigr).
% \end{math}
% It only implies that
% \begin{displaymath}
%   B_2 = \Bigl(
%   \begin{smallmatrix}
%     \lambda_1 & 0 \\ 0 & \lambda_2
%   \end{smallmatrix}
%   \Bigr) \qquad \text{or} \qquad
%   B_2 = \Bigl(
%   \begin{smallmatrix}
%     0 & 1 \\ 0 & 0
%   \end{smallmatrix}
%   \Bigr),
% \end{displaymath}
% with the same result modulo $A$.

Thus assume that $M_f^{\nil}$ contains a matrix $C \notin \langle A,B
\rangle$. Because $CA=AC$ and $CB=BC$, we have
\begin{displaymath}
  C = \left(
    \begin{smallmatrix}
      c_1 & c_2 & c_3 & c_4 \\
      0 & c_1 & c_5 & c_6 \\
      0 & 0 & c_1 & c_2 \\
      0 & 0 & 0 & c_1
    \end{smallmatrix}
  \right)\!.
\end{displaymath}
Clearly, $c_1=0$ since $C$ is nilpotent. If $c_2=0$, then $\rank (C - c_3 A -
c_4 B) \le 1$, thus $C \in \langle A,B \rangle$ since $B$ is the only matrix
in $M_f$ of rank 1. This contradiction allows us to assume that $c_2=1$. It
also implies that $M_f^{\nil}$ cannot contain yet another matrix, since we
then would have to get another one of rank 1. Therefore, $M_f = \langle
I,A,B,C \rangle$. Now, $\rank C < 3$ implies $c_5 = 0$, and modulo $B$ we may
assume $c_4 = 0$. If $\chr \K \ne 2$, we may also assume $c_3 = c_6 = 0$. This
follows from the similarity transformation $C \mapsto PCP^{-1}$ where
\begin{math}
  P = \bigl(
  \begin{smallmatrix}
    I & Q \\ 0 & I
  \end{smallmatrix}
  \bigr)
\end{math}
with
\begin{math}
  Q = \bigl(
  \begin{smallmatrix}
    0 & 0 \\ q & 0
  \end{smallmatrix}
  \bigr)
\end{math}
and $q = \frac12 (c_3 - c_6)$. % And modulo $A$.
It follows that
\begin{displaymath}
  f = x_4 x_1^\pd {d-1} + x_3 x_2 x_1^\pd {d-2} + c x_1^\pd d
\end{displaymath}
up to a base change. (We may even assume $c=0$ if $\chr \K \nmid d$.) Let
\begin{displaymath}
  f_t = \tfrac{1}{st} \Bigl( (x_1 + s x_2 + t x_3 + st x_4)^\pd d - (x_1 + s
  x_2)^\pd d - (x_1 + t x_3)^\pd d + x_1^\pd d \Bigr) + c x_1^\pd d.
\end{displaymath}
Then $f_0 = f$, and $f_t \sim x_1^\pd d + x_2^\pd d + x_3^\pd d + x_4^\pd d$
splits regularly three times.

If $\chr \K = 2$, then the case $(c_3, c_6) = (0,1)$ is not in the $\GL_\K
(4)$ orbit of $(c_3, c_6) = (0,0)$. A base change shows that this additional
case is isomorphic to $M_f = \langle I, A, B, A^2 \rangle$ where $A^2 = B^2$
and
\begin{displaymath}
  A = \left(
    \begin{smallmatrix}
      0 & 1 & 0 & 0 \\
      0 & 0 & 0 & 1 \\
      0 & 0 & 0 & 0 \\
      0 & 0 & 0 & 0
    \end{smallmatrix}
  \right) \qquad
  B = \left(
    \begin{smallmatrix}
      0 & 0 & 1 & 0 \\
      0 & 0 & 0 & 0 \\
      0 & 0 & 0 & 1 \\
      0 & 0 & 0 & 0
    \end{smallmatrix}
  \right).
\end{displaymath}
This implies that $f = x_4 x_1^\pd{d-1} + x_3^\pd 2 x_1^\pd{d-2} + x_2^\pd 2
x_1^\pd{d-2} + c x_1^\pd d$. Let
\begin{align*}
  f_t & = t^{-3} \Bigl( t (x_1 + t x_2 + t^2 x_4)^\pd d + t (x_1 + t x_3)^\pd
  d \\
  &\qquad - (x_1 + t^2 x_2 + t^2 x_3)^\pd d + (1-2t+ct^3) x_1^\pd d \Bigr).
\end{align*}
Again, $f_0 = f$, and $f_t \sim x_1^\pd d + x_2^\pd d + x_3^\pd d + x_4^\pd d$
splits regularly three times. Hence in each case we have found an $f_t \in
\cR_d [t_1, \dots, t_n]$ such that $f_0 = f$ and $f_t$ splits regularly
$\dim_\K M_f - 1$ times over $\K (t_1, \dots, t_n)$, and we are done. \qed
%\end{proof}

% The last $f_t$ is very asymmetric, but I have not found a symmetric
% splitting that works.

\begin{rem}
  Note that the last case of the proof says the following. Suppose $M_f$
  contains two matrices of rank 2 that are non-proportional. If $\chr \K \ne
  2$, then $M_f$ contains exactly two of rank 2 such that $A^2=0$. If $\chr \K
  = 2$, then there are two possibilities. Either every matrix in $M_f$ of rank
  2 satisfies $A^2=0$, or only one matrix is of this type, and the rest
  satisfy $A^2 \ne 0$.
\end{rem}

We will end this section with a generalization of theorem \ref{thm:onematrix}.

\begin{thm} \label{thm:morematrices}
  Suppose $d \ge 3$ and $f \in \cR_d$. Let $A_1, \dots, A_m \in \Mat_\K (r,r)$
  be nonzero and nilpotent, and assume there exist orthogonal idempotents
  $E_1, \dots, E_m$ such that $E_i A_i = A_i E_i = A_i$ for all $i$. Let $n_i
  = \indx A_i$ and $1 \le a_i < n_i$. Assume that $A_i^k \in M_f$ for all $k
  \ge a_i$ and all $i = 1, \dots, m$. Let $n = \sum_{i=1}^m (n_i - a_i)$. Then
  $f$ is a specialization of some $f_t \in \cR_d [t_1, \dots, t_n]$ that
  splits regularly $n$ times over $\K (t_1, \dots, t_n)$.
\end{thm}

\begin{proof}
  The proof uses the same ideas as the proof of theorem \ref{thm:onematrix},
  with some modifications. Fix one $i \in \{ 1, \dots, m \}$, and choose $Q
  \in E_i \Mat_\K (r,r) E_i$ such that $A_i^{n_i-1} Q A_i^{n_i-1} =
  A_i^{n_i-1}$. Define matrices $P = I + \sum_{k=1}^{n_i-1} t^k A_i^{n_i-1-k}
  Q$ and $A_{it} = A_i + t Q A_i^{n_i-1}$. It follows that $A_{it}^{n_i-1} = P
  A_i^{n_i-1}$ and $A_{it}^{n_i} = t A_{it}^{n_i-1}$. Because $A_i^{n_i-1} \in
  M_f$, there exists $g \in \cR_d$ such that $\p g = A_i^{n_i-1} \p f$. Define
  \begin{displaymath}
    g_t = \phi_P (g) = \sum_{k \ge 0} t^k g_k \quad \text{ and } \quad
    f_t = f + \sum_{k \ge 1} t^k g_{n_i-1+k}.
  \end{displaymath}

  For all $i \ne j$, it follows from $E_i E_j = 0$ that $A_i E_j = E_j A_i =
  A_i A_j = 0$. Thus $A_j A_{it} = 0$. Since $\p g_t = P \phi_P (\p g) =
  A_{it}^{n_i-1} \phi_P (\p f)$, it follows that $A_j \p g_t = 0$, and
  therefore, $A_j \p g_k = 0$ for all $k \ge 0$. Hence $A_j^k \in M_{f_t}$ for
  all $j \ne i$ and $k \ge a_j$.

% Let $M_i = E_i M = \langle E_i, A_i^{a_i}, \dots, A_i^{n_i-1} \rangle$, a
% $\K$-algebra with identity element $E_i$. If $m>1$, then $E_i \notin M_f$
% and $M \subsetneqq M_1 \dsum \dots \dsum M_m$. If $E_i$ had been an element
% of $M_f$, we would get a regular splitting of $f$. But we still get a
% degenerate splitting involving all $A_i^k$.

  We will now prove that $A_{it}^k \p f_t = A_i^k \p f + \sum_{j \ge 1} t^j \p
  g_{n_i-1-k+j}$ for all $k \ge 0$. Assume it is true for some $k \ge 0$. The
  arguments following equation \eqref{eq:pgt} in the proof of theorem
  \ref{thm:onematrix} apply here and show that $A_i^{n_i-1} \p g_{n_i-j} = 0$
  for all $j>1$, $A_i^{n_i-1} \p f = A_i^{n_i-1} \p g_{n_i-1}$ and $A_i \p
  g_{n_i-1+j} + Q A_i^{n_i-1} \p g_{n_i-2+j} = \p g_{n_i-2+j}$ for all $j$. It
  follows that
  \begin{align*}
    A_{it}^{k+1} \p f_t & = \Bigl( A_i + t Q A_i^{n_i-1} \Bigr) \biggl(
    A_i^k \p f + \sum_{j \ge 1} t^j \p g_{n_i-1-k+j} \biggr) \\
%   & = A_i^{k+1} \p f + t Q A_i^{n_i-1+k} \p f + \sum_{j \ge 1} t^j A_i \p
%   g_{n_i-1-k+j} + \sum_{j \ge 1} t^{j+1} Q A_i^{n_i-1} \p g_{n_i-1-k+j} \\
    & = A_i^{k+1} \p f + \sum_{j \ge 1} t^j \Bigl( A_i \p g_{n_i-1-k+j} + Q
    A_i^{n_i-1} \p g_{n_i-2-k+j} \Bigr) \\
    & = A_i^{k+1} \p f + \sum_{j \ge 1} t^j \p g_{n_i-2-k+j}.
  \end{align*}
  Since $A_i^k \in M_f$ for all $k \ge a_i$ it follows that $A_{it}^k \in
  M_{f_t}$ for all $k \ge a_i$. In particular, $E' = I - (t^{-1}
  A_{it})^{n_i-1} \in M_{f_t}$.

  % It is a bit hidden, but all three equations are needed.

  Since $E'$ is idempotent, we may apply theorem \ref{thm:regsplitE}. It tells
  us that $f_t$ has a regular splitting with the following two additive
  components, $t^{-n_i+1} g_t$ and
  \begin{displaymath}
    f' = t^{-n_i+1} \bigl( t^{n_i-1} f - g_0 - t g_1 - \dots - t^{n_i-1}
    g_{n_i-1} \bigr),
  \end{displaymath}
  and furthermore that $(A_{it} E')^k = A_{it}^k E' \in M_{f'}^{E'}$ for all
  $k \ge a_i$. Hence we may repeat our procedure on $f'$. By induction on
  $n_i$ and $i$, we arrive at some $f_{\underline{t}} \in \cR_d [t_1, \dots,
  t_n]$ such that $f_0 = f$ and $f_{\underline{t}}$ splits regularly $n$
  times.
\end{proof}

% Can we assume $(c_3, c_6) = (0,1)$? We may at least assume $(c_3, c_6) =
% (0,a)$. This implies that $f = x_4 x_1^\pd{d-1} + x_3 x_2 x_1^\pd{d-2} + a
% x_2^\pd 2 x_1^\pd{d-2} + c x_1^\pd d$. Since $\bar\K = \K$, we may assume
% $a=1$.

\begin{rem}
  We assume in theorem \ref{thm:morematrices} that $A_i^k \in M_f$ for all $k
  \ge a_i$. It is in fact enough to assume $A_i^{a_i} B_i, A_i^{a_i+1} C_i \in
  M_f$ for some invertible $B_i, C_i \in \K [A_i]$. Indeed, apply proposition
  \ref{prop:Mexpo} with $A = A_i B_i^{-1} C_i$, $B = I$ and $C = A_i^{a_i}
  B_i$. It follows that $A^k C = A_i^{a_i+k} B_i^{1-k} C_i^k \in M_f$ for all
  $k \ge 0$. In particular, with $k = n_i-a_i-1$, we get $A_i^{n_i-1} P \in
  M_f$ where $P \in \K [A_i]$ is invertible. This implies $A_i^{n_i-1} \in
  M_f$ since $A_i^{n_i} = 0$. Now letting $k = n_i-a_i-2$ implies $A_i^{n_i-2}
  \in M_f$. By descending induction on $k$ we get $A_i^k \in M_f$ for all $k
  \ge a_i$.
\end{rem}

\section{Counter examples} \label{section:negative}

In this section we will produce examples of $f \in \cR_d$ in which we cannot
find an $f_t \in \cR_d [t_1, \dots, t_n]$ such that $f_0 = f$ and $f_t$ splits
regularly $\dim_\K M_f - 1$ times over $\K (t_1, \dots, t_n)$. Thus question
\ref{que:Q} has a negative answer for these $f$. There exist many such
examples due to purely numerical reasons, and the following theorem enables us
to find some.

% The following proposition is also true for $d=2$. (Different proof.)

\begin{thm} \label{thm:nilpotrank}
  Let $d \ge 3$, $s \le r$ and $\cS = \K [x_1, \dots, x_s]^{DP} \subseteq
  \cR$. Suppose $h \in \cS_d$ does not split regularly. Let $f = h +
  x_{s+1}^\pd d + \dots + x_r^\pd d \in \cR_d$. Assume that there exists an
  $f_t \in \cR_d [t_1, \dots, t_n]$ such that $f_0 = f$ and $f_t$ splits
  regularly $m-1$ times over $\K (t_1, \dots, t_n)$. Suppose $m > r-s+1$. Then
  $M_h$ must contain a non-zero nilpotent matrix of rank $\le s/(m-r+s)$.
\end{thm}

% Alternatively:
% Given $d \ge 3$ and $f \in \cR_d$ that satisfies $\ann_R (f)_1 = 0$. Let $f
% = h_1 + \dots + h_k$ be a maximal regular splitting of $f$. Then $\p h_i =
% E_i \p f$ for some orthogonal idempotents $E_i$. Let $s = \max \{ \rank E_i
% \}$. Assume that there exists an $f_t \in \cR_d [t_1, \dots, t_n]$ such that
% $f_0 = f$ and $f_t$ splits regularly $m-1 \ge k-1$ times over $\K (t_1,
% \dots, t_n)$. If $k<m$, that is, if $f$ does not itself split regularly
% $m-1$ times, then $M_f$ must contain a non-zero nilpotent matrix of rank
% $\le \mu = \max \{ (r - \sum_{j \in J} \rank E_j)/(m-|J|) \suchthat J
% \subsetneq \{ 1, \dots, k \} \}$. Note that $\mu \ge s/(m-k+1)$, which is
% max when $|J| = k-1$. Depending on $\{ \rank E_i \}$, $\mu$ might be larger.
% If $\rank E_i = 1$ for all but one, then $\mu = s/(m-k+1)$.

\begin{proof}
  Clearly, $m \le r$. Note that if $\ann_R (f)_1 \ne 0$, then $\ann_S (h)_1
  \ne 0$. In this case $M_h$ will contain nilpotent matrices of rank 1, and we
  are done. Therefore, we may assume $\ann_R (f)_1 = 0$. This implies $\ann_{R
    (t_1, \dots, t_n)} (f_t)_1 = 0$. It also implies that $f \ne 0$ since $s >
  r-m+1 \ge 1$.

  % The most direct and in a sense easiest proof of this is: Assume $0 \ne D
  % \in \ann_{R (t_1, \dots, t_n)} (f_t)_1 = 0$. May assume $D \in R [t_1,
  % \dots, t_n]$. Write $D = \sum_{k \ge 0} D_k t_n^k$ and $f_t = \sum_{k \ge
  % 0} f'_k t_n^k$. May assume $D_0 \ne 0$, otherwise divide $D$ by $t_n$.
  % Then $0 = D (f_t) = \sum ...$, hence $D_0 (f'_0) = 0$. Divide $D_0$ by
  % $t_{n-1}$ if possible, and proceed by induction.

  For each $k = 1, \dots, r-s$, define $E_k \in \Mat_\K (r,r)$ by
  \begin{displaymath}
    (E_k)_{ij} =
    \begin{cases}
      1 & \text{if $i=j=k+s$,} \\
      0 & \text{otherwise.}
    \end{cases}
  \end{displaymath}
  Clearly, $E_k$ is a diagonal idempotent of rank 1. Furthermore, $\p \bigl(
  x_{s+k}^\pd d \bigr) = E_k \p f$, thus $E_k \in M_f$. Let $E_0 = I -
  \sum_{k=1}^{r-s} E_k \in M_f$. It follows by theorem \ref{thm:regsplitE}
  that $M_f = M_0 \dsum M_1 \dsum \dots \dsum M_{r-s}$ where $M_k = M_f E_k =
  \langle E_k \rangle$ for $k = 1, \dots, r-s$, and $M_0 = M_f E_0 \iso M_h$.
  To be precise,
  \begin{math}
    M_0 = \left\{ \bigl(
      \begin{smallmatrix}
        A & 0 \\
        0 & 0
      \end{smallmatrix}
    \bigr) \suchthat A \in M_h \right\}.
  \end{math}

  Choose a multiplicative (monomial) order on $\K' = \K [t_1, \dots, t_n]$
  with $1$ as the smallest element. If $V$ is any $\K$-vector space and $v \in
  V' = V \otimes_\K \K [t_1, \dots, t_n]$, $v \ne 0$, denote by $\lc(v) \in V$
  the leading coefficient of $v$, which to us the coefficient of the
  \emph{smallest} non-zero term of $v$ in the ordering. Note that if $\varphi
  : U \times V \to W$ is a $\K$-bilinear map, it induces a $\K'$-bilinear map
  $\varphi' : U' \times V' \to W'$. Then $\lc \bigl( \varphi' (u,v) \bigr) =
  \varphi \bigl( \lc (u), \lc (v) \bigr)$ as long as $\varphi \bigl( \lc (u),
  \lc (v) \bigr) \ne 0$.

  % Note that $\K'$ is not a field. And $\varphi' (\sum_i a_i u_i, \sum_j b_j
  % v_j) = \sum_{i,j} a_i b_j \varphi (u_i, v_j)$ for all $a_i, b_j \in \K'$.
  
  There exist orthogonal idempotents $A_1, \dots, A_m \in M_{f_t}$ and
  non-zero polynomials $g_1, \dots, g_m \in \cR_d (t_1, \dots, t_n)$ such that
  $\sum_{i=1}^m A_i = I$ and $\p g_i = A_i \p f_t$. Let the common denominator
  of the entries of $A_i$ be $\lambda_i \in \K [t_1, \dots, t_n]$.  We may
  scale $\lambda_i$ such that $\lc (\lambda_i) = 1$. Replace $A_i$ by
  $\lambda_i A_i$. Then $A_i \in \Mat_{\K [t_1, \dots, t_n]} (r,r)$ and $A_i^2
  = \lambda_i A_i$. Moreover, replace $g_i$ by $\lambda_i g_i$ to preserve the
  relation $\p g_i = A_i \p f_t$. This implies that $g_i \in \cR_d [t_1,
  \dots, t_n]$.

  Let $A_{i0} = \lc (A_i) \ne 0$. Note that $\lc (f_t) = f$, and $A_{i0} \p f
  \ne 0$ because $\ann_R (f)_1 = 0$. It follows that
  \begin{displaymath}
    \p \lc (g_i) = \lc (\p g_i) = \lc (A_i \p f_t) = \lc (A_i) \p \lc (f_t) =
    A_{i0} \p f.
  \end{displaymath}
  Hence $A_{i0} \in M_f$. If $A_{i0}^2 \ne 0$, then $A_{i0}^2 = \lc (A_i^2) =
  \lc (\lambda_i A_i) = A_{i0}$. Thus $A_{i0}^2 = 0$ or $A_{i0}^2 = A_{i0}$
  for all $i$. Furthermore, $A_{i0} A_{j0} = 0$ for all $i \ne j$, because
  $A_i A_j = 0$. In addition, $\rank A_{i0} \le \rank A_i$. (If some minor of
  $A_i$ is zero, then the corresponding minor of $A_{i0}$ must also be zero.)

% Note that $\lc (v)$ is \emph{not} computed componentwise. For example, $\lc
% (t + t^3, t^2 + t^4) = (t,0)$.

  Since $h$ does not split regularly, $M_h$ does not contain any non-trivial
  idempotents. Hence $\{ E_i \}$ is the unique maximal coid in $M_f$, and any
  idempotent in $M_f$ is a sum of some of the $E_i$'s. Assume $A_{i0}$ is
  idempotent. We want to prove that $A_{i0} \in \langle E_1, \dots, E_{r-s}
  \rangle$. If it is not, then $A_{i0} E_0 = E_0$. For all $j \ne i$, we have
  $A_{j0} A_{i0} = 0$, and therefore $A_{j0} E_0 = 0$ and $A_{j0} \ne A_{i0}$.
  This implies $A_{j0} \in \dsum_{i=1}^{r-s} M_i = \langle E_1, \dots, E_{r-s}
  \rangle$, and it follows that $A_{j0}^2 \ne 0$. Hence $A_{j0}$ must be an
  idempotent! Therefore $\{ A_{j0} \}_{j=1}^m$ is a set of orthogonal
  idempotents, but $\{ E_j \}_{j=0}^{r-s}$ is maximal, hence $m \le r-s+1$, a
  contradiction.

  Let $J = \{ i \suchthat A_{i0}^2 = A_{i0} \}$ and $k = \sum_{i \in J} \rank
  A_{i0} \ge |J|$. By the last paragraph, $k \le r-s$. Clearly, the number of
  nilpotents among $\{ A_{i0} \}_{i=1}^m$ is
  \begin{displaymath}
    m - |J| \ge m-k \ge m-r+s \ge 2.
  \end{displaymath}
  Now suppose that $M_h$ does not contain any non-zero nilpotent matrix of
  rank $\le s/(m-r+s)$. Then $\rank A_{i0} > s/(m-r+s)$ for all $i \notin J$.
  It follows that
  \begin{align*}
    r & = \sum_{i=1}^m \rank A_i \ge \sum_{i=1}^m \rank A_{i0} > k + (m-k)
    \frac{s}{m-r+s} \\
    & = \frac{ms - (r-m)k}{m-r+s} \ge \frac{ms - (r-m)(r-s)}{m-r+s} = r,
  \end{align*}
  which is the contradiction we sought.
\end{proof}

% Is $\rank A_{i0} = \rank A_i$ when $A_{i0}^2 = A_{i0}$?

% Note that $\sum_{i=1}^m \lambda_i^{-1} A_i = I$ implies that $\sum_{i=1}^m
% (\sprod_{j \ne i} \lambda_j) A_i = (\sprod_{j=1}^m \lambda_j) I$. Taking
% ``$\lc$'', this implies that there exists $J \subseteq \{ 1, \dots, m \}$
% such that $\sum_{j \in J} A_{j0} = I$ or $0$. If $\sum_{j \in J} A_{j0} =
% I$, then $J = \{ 1, \dots, m \}$ and $A_{j0}^2 = A_{j0}$ for all $j$. If
% $\sum_{j \in J} A_{j0} = 0$, then $A_{j0}^2 = 0$ for all $j \in J$.

\begin{rem}
  It is not correct that if $M_{f_t}$ contains $m$ idempotents of rank $\le
  k$, then $M_{f_0}$ must contain $m$ idempotents or nilpotents of rank $\le
  k$. A simple example is $f = x_2 x_1^\pd {d-1}$, $r=2$. Then $M_f = \langle
  I,A \rangle$ where
  \begin{math}
    A_1 = \bigl(
    \begin{smallmatrix}
      0 & 1 \\
      0 & 0
    \end{smallmatrix}
    \bigr).
  \end{math}
  Let $f_t = t^{-1} [ (x_1 + tx_2)^\pd d - x_1^\pd d]$, so that $f_0 = f$.
  Then $M_{f_t} = \langle A_t, B_t \rangle$ where
  \begin{math}
    A_t = \left(
      \begin{smallmatrix}
        -t & 1 \\
        0 & 0
      \end{smallmatrix}
    \right)
  \end{math}
  and
  \begin{math}
    B_t = \left(
      \begin{smallmatrix}
        0 & 1 \\
        0 & t
      \end{smallmatrix}
    \right).
  \end{math}
  Thus both $A_0 = B_0 = A$. We see that $M_{f_t}$ can contain two idempotents
  of rank 1 even though $\dim_\K \{ A \in M_f \suchthat \rank A \le 1 \} = 1$.
\end{rem}

Now that we have theorem \ref{thm:nilpotrank} at our disposal, we are ready to
give the first example in which question \ref{que:Q} has a negative answer.

% in our arsenal

\begin{exmp} \label{ex:r5}
  Suppose $r=5$ and $a,b \ge 2$. Let
  \begin{displaymath}
    f = x_1^\pd {a-1} x_2^\pd {b+1} x_3 + x_1^\pd a x_2^\pd b x_4 + x_1^\pd
    {a+1} x_2^\pd {b-1} x_5.
  \end{displaymath}
  Then $f \in \cR_d$ where $d = a+b+1 \ge 5$. The annihilator ideal is
  \begin{align*}
    \ann_R (f) & = (\p_3, \p_4, \p_5)^2 + (\p_1 \p_4 - \p_2 \p_3, \p_1 \p_5 -
    \p_2 \p_4) \\
    & \phantom{==} + (\p_1^a \p_3, \p_2^b \p_5, \p_1^{a+2}, \p_2^{b+2}) +
    (\p_1^{a+1} \p_2^b, \p_1^a \p_2^{b+1}).
  \end{align*}
  
  It is easy to check that $\ann_R f$ contains the right-hand side. For the
  converse, assume that $D \in \ann_R (f)_e$. Modulo $(\p_3, \p_4, \p_5)^2$
  there exist $D_i \in \K [\p_1, \p_2]$ such that $D = \p_3 D_1 + \p_4 D_2 +
  \p_5 D_3 + D_4$, and modulo $(\p_1 \p_4 - \p_2 \p_3, \p_1 \p_5 - \p_2 \p_4)$
  we may assume that $D_2 = 0$ and $D_3 = c_1 \p_1 \p_2^{e-2} + c_2
  \p_2^{e-1}$. Computing $Df$, we see that $Df = 0$ is equivalent to $D_1
  (x_1^\pd {a-1} x_2^\pd {b+1}) + D_3 (x_1^\pd {a+1} x_2^\pd {b-1}) = D_4 (f)
  = 0$. This implies that $D_1 \in (\p_1^a, \p_2^{b+2})$, $D_3 \in (\p_2^b)$
  and $D_4 \in (\p_1^{a+2}, \p_2^{b+2}, \p_1^{a+1} \p_2^b, \p_1^a
  \p_2^{b+1})$, and proves that $D$ is contained in the right-hand side.
  
  Since $a,b \ge 2$, we see that $\ann_R f$ has two generators of degree $d$.
  Thus $\dim_\K M_f = 3$. Let $g_1 = x_1^\pd a x_2^\pd {b+1}$, $g_2 = x_1^\pd
  {a+1} x_2^\pd b$ and
  \begin{displaymath}
    A_1 = \left(
      \begin{smallmatrix}
        0 & 0 & 1 & 0 & 0 \\
        0 & 0 & 0 & 1 & 0 \\
        0 & 0 & 0 & 0 & 0 \\
        0 & 0 & 0 & 0 & 0 \\
        0 & 0 & 0 & 0 & 0
      \end{smallmatrix}
    \right)\!, \qquad
    A_2 = \left(
      \begin{smallmatrix}
        0 & 0 & 0 & 1 & 0 \\
        0 & 0 & 0 & 0 & 1 \\
        0 & 0 & 0 & 0 & 0 \\
        0 & 0 & 0 & 0 & 0 \\
        0 & 0 & 0 & 0 & 0
      \end{smallmatrix}
    \right)\!.
  \end{displaymath}
  A simple calculation shows that $\p g_1 = A_1 \p f$ and $\p g_2 = A_2 \p f$.
  This implies that $A_1, A_2 \in M_f$, and it follows that $M_f = \langle I,
  A_1, A_2 \rangle$. (Note that $g_1 = \p_1 h$ and $g_2 = \p_2 h$ where $h =
  x_1^\pd{a+1} x_2^\pd{b+1}$.)

% Suppose there is an $f_t \in \cR_d [t]$ such that $f_0 = f$ and $f_t$ splits
% twice over $\K (t)$. That is, $f_t = g_{1t} + g_{2t} + g_{3t}$ where $\p
% g_{it} = A_{it} \p f_t \ne 0$ and the $A_{it}$'s are orthogonal idempotents.
% ($\sum_i A_{it} = I$ follows since $\p f_t = \sum_i A_{it} \p f_t$ and
% $\ann_{R(t)} (f_t)_1 = 0$.) Then at least one $\{ A_{it} \}$ must have rank
% equal to 1, say $\rank A_{1t} = 1$. We may multiply $f, g_{1t}, g_{2t}$ and
% $g_{3t}$ by their common denominator and assume $f_t, g_{it} \in \cR_d [t]$
% and $A_{it}^2 = t^{n_i} A_{it}$. Next find $a_i, b_i \in \bN$ and $\lambda_i
% = \lambda_{i0} + t \lambda_{i1} + \dots + t^{b_i} \lambda_{ib_i} \in \K[t]$
% such that $\lambda_{i0} = 1$ and $\lambda_i A_{it} = A_{i0} + t A_{i1} +
% \dots + t^{a_i} A_{ia_i}$ where $A_{ij} \in \Mat_{\K[t]}$ and $A_{i0} \ne
% 0$. Then $\lambda_i \p g_{it} = A_{i0} \p f_0 + t [ A_{it} \p f_0 + A_{i0}
% \p f_1 ] + \dots$, therefore, $\p g_{i0} = A_{i0} \p f_0$. Thus $A_{i0} \in
% M_f$ and $g_{i0} \in G_f$. But $\rank A_{1t} = 1$, hence $\rank A_{10} \le
% 1$. Since $A_{10} \ne 0$, we must have $\rank A_{10} = 1$. But no $A \in
% M_f$ has rank equal to 1. Reductio ad absurdum.
  
  Since $M_f$ does not contain any non-zero nilpotent matrix of rank 1,
  theorem \ref{thm:nilpotrank} implies that there does not exist an $f_t \in
  \cR_d [t_1, \dots, t_n]$ such that $f_0 = f$ and $f_t$ splits regularly
  $\dim_k M_f - 1$ times over $\K (t_1, \dots, t_n)$. Moreover, by adding
  terms $x_i^\pd d$ with $i>5$, we have produced such examples for all $r \ge
  5$ and $d \ge 5$.
\end{exmp}

\begin{exmp} \label{ex:r7r9}
  Let us consider the following two polynomials.
  \begin{enumerate}
    \setlength{\itemsep}{2pt}
    \setlength{\parskip}{2pt}
    \renewcommand{\theenumi}{\alph{enumi}}
    \renewcommand{\labelenumi}{\normalfont(\theenumi)}
% \item $f = x_1^\pd {a-1} x_2^\pd {b+1} x_3 + x_1^\pd a x_2^\pd b x_4 +
%   x_1^\pd {a+1} x_2^\pd {b-1} x_5$, $a,b \ge 2$.
% \item $f = x_4 ( x_1 x_2 x_3 + x_2^\pd 3 ) + x_5 ( x_1^\pd 2 x_3 + x_1
%   x_2^\pd 2 ) + x_6 ( x_1^\pd 2 x_2 + x_3^\pd 3 ) + x_7 ( x_1^\pd 3 + x_2
%   x_3^\pd 2 )$.
  \item $f_1 = x_4 \bigl( x_2 x_3^\pd 2 \bigr) + x_5 \bigl( x_1 x_3^\pd 2 +
    x_2^\pd 2 x_3 \bigr)$

    $\qquad\phantom{0} + x_6 \bigl( x_1 x_2 x_3 + x_2^\pd 3 \bigr) + x_7
    \bigl( x_1^\pd 2 x_3 + x_1 x_2^\pd 2 \bigr) \in \cR_4$, $r=7$.
  \item $f_2 = x_5 \bigl( x_3 x_4 \bigr) + x_6 \bigl( x_2 x_4 + x_3^\pd 2
    \bigr)$

    $\qquad\phantom{0} + x_7 \bigl( x_1 x_4 + x_2 x_3 \bigr) + x_8 \bigl( x_1
    x_3 + x_2^\pd 2 \bigr) + x_9 \bigl( x_1 x_2 \bigr) \in \cR_3$, $r=9$.
  \end{enumerate}
  Tedious but simple computations show that the annihilators are:
  \begin{align*}
    \ann_R (f_1) = \; & (\p_4, \p_5, \p_6, \p_7)^2 + (\p_1 \p_4, \p_2 \p_4 -
    \p_1 \p_5, \p_3 \p_4 - \p_2 \p_5, \p_2 \p_5 - \p_1 \p_6, \\
    & \: \p_3 \p_5 - \p_2 \p_6, \p_2 \p_6 - \p_1 \p_7, \p_3 \p_6 - \p_2 \p_7)
    + (\p_1 \p_3 - \p_2^2) \\
    & + (\p_2 \p_3 \p_7, \p_3^2 \p_7) + (\p_1^3, \p_1^2 \p_2, \p_3^3) +
    (\p_2^4, \p_2^3 \p_3) \\
    \ann_R (f_2) = \; & (\p_5, \dots, \p_9)^2 + (\p_1 \p_5, \p_2 \p_5, \p_1
    \p_6, \p_3 \p_5 - \p_2 \p_6, \p_4 \p_5 - \p_3 \p_6, \\
    & \: \p_2 \p_6 - \p_1 \p_7, \p_3 \p_6 - \p_2 \p_7, \p_4 \p_6 - \p_3 \p_7,
    \p_2 \p_7 - \p_1 \p_8, \p_3 \p_7 - \p_2 \p_8, \\
    & \: \p_4 \p_7 - \p_3 \p_8, \p_2 \p_8 - \p_1 \p_9, \p_3 \p_8 - \p_2 \p_9,
    \p_4 \p_8, \p_3 \p_9, \p_4 \p_9) \\
    & + (\p_1^2, \p_2^2 - \p_1 \p_3, \p_2 \p_3 - \p_1 \p_4, \p_3^2 - \p_2
    \p_4, \p_4^2) + (\p_2^2 \p_3, \p_2 \p_3^2)
  \end{align*}
  In both cases, $\dim_\K M_{f_i} = 3$. It is easy to check that the two
  nilpotent matrices in $M_{f_1}$ are of rank 3, and of rank 4 in $M_{f_2}$.
  By theorem \ref{thm:nilpotrank}, there does not exist an $f_t \in \cR_d
  [t_1, \dots, t_n]$ such that $f_0 = f_i$ and $f_t$ splits regularly $\dim_k
  M_{f_i} - 1$ times over $\K (t_1, \dots, t_n)$. Again, we may add terms
  $x_i^\pd d$ to produce such examples for all $r \ge 7$ when $d=4$ and all $r
  \ge 9$ when $d=3$.
\end{exmp}

The next proposition allows us to construct $f$ such that $M_f$ does not
contain nilpotent matrices of small rank. The previous examples are special
cases of this proposition.

% Should I introduce the matrices before the proposition? F.ex.: In fact,
% under some conditions we will find $f$ such that $M_f$ consists of the
% following matrices.
% \begin{displaymath}
%   B'_k =
%   \begin{pmatrix}
%     0 & \dots & 0 & 1 & 0 & \hdotsfor{3} & 0 \\
%     \vdots &&& \ddots & \ddots & \ddots &&& \vdots \\
%     0 & \hdotsfor{3} & 0 & 1 & 0 & \dots & 0
%   \end{pmatrix}
% \end{displaymath}

% How about the following. Let $A$ be a set of statements. Define $\Delta (A)
% = 1$ if all statements in $A$ are true, and $\Delta (A) = 0$ otherwise. Then
% $(B_k)_{ij} = \Delta (j = s+k+i, i \le s)$.

\begin{prop} \label{prop:counterex}
  Suppose $d \ge 3$, $s \ge 2$, $q \ge 1$ and $r = 2s+q$. Let $\cS = \K [x_1,
  \dots, x_s]^{DP} \subseteq \cR = \K [x_1, \dots, x_r]^{DP}$. Let $g_1,
  \dots, g_{s+q} \in \cS_{d-1}$ satisfy $\p_{i+1} g_j = \p_i g_{j+1} =
  h_{i+j-2} \in \cS_{d-2}$ for all $1 \le i < s$ and $1 \le j < s+q$. Define
  $f = \sum_{i=1}^{s+q} x_{s+i} g_i \in \cR_d$. Assume that $h_i = 0$ for all
  $i < s-1$, and that $h_{s-1}, \dots, h_{s+q+1}$ are linearly independent.
  Then $M_f = \langle I, B_0, \dots, B_q \rangle$ where, for each $k = 0
  \dots, q$,
  \begin{displaymath}
    (B_k)_{ij} =
        \begin{cases}
          1, & \text{if $i \le s$ and $j = s+k+i$,} \\
          0, & \text{otherwise}.
        \end{cases}
  \end{displaymath}
\end{prop}

% Note: if I had defined $f = \sum_{i=1}^{s+q} x_{r+1-i} g_i$, then $X_2$
% would become Toeplitz instead of Hankel, and the $B'_k$ would be
% ``reversed''.

\begin{proof}
  For each $k$ we note that $B_k$ is block matrix of the form
  \begin{math}
    \bigl(
    \begin{smallmatrix}
      0 & \smash{B'_k} \\ 0 & 0
    \end{smallmatrix}
    \bigr),
  \end{math}
  where $B'_k \in \Mat_\K (s,s+q)$ is a ``displaced'' identity matrix. That
  is, $B'_k$ is a block matrix of the form $\bigl( O_1\: I\: O_2 \bigr)$,
  where $O_1$ is an $s \times k$ zero matrix, $I$ is an $s \times s$ identity
  matrix, and $O_2$ is an $s \times (q-k)$ zero matrix. In particular, $\rank
  B_k = s$.

% Thus $(B'_k)_{ij} = \delta_{j,k+i}$.

% For each $i = 0, \dots, q$, there exists an $f_i \in \cS_d$ such that $\p_j
% f_i = g_{i+j}$. Since $\p f_i = B'_i \p f$, we get $B_i \in M_f$.

  By computing $\p\p^\T f$, we see that it has a block decomposition,
  \begin{displaymath}
    \p \p^\T f =
    \begin{pmatrix}
      X_1 & X_2 \\ X_3 & 0
    \end{pmatrix},
  \end{displaymath}
  where $X_1 \in \Mat_\K (s,s)$ and $X_2 \in \Mat_\K (s,s+q)$. $X_2$ is a
  Hankel matrix in the sense that $(X_2)_{ij} = \p_i g_j = h_{i+j-1}$ for all
  $1 \le i \le s$ and $1 \le j \le s+q$, i.e
  \begin{displaymath}
    X_2 = X_3^\T =
    \begin{pmatrix}
      h_1 & \dots & h_{s+q} \\
      : && : \\
      h_s & \dots & h_{r-1}
    \end{pmatrix}.
  \end{displaymath}
  We note that the columns and rows of $X_2$ are linearly independent over
  $\K$. This implies that $\ann_R (f)_1 = 0$.

  By lemma \ref{lem:main2},
  \begin{math}
    A = \left(
      \begin{smallmatrix}
        A_1 & A_2 \\
        A_3 & A_4
      \end{smallmatrix}
    \right) \in M_f
  \end{math}
  if and only if
  \begin{displaymath}
    A \p\p^\T f =
    \begin{pmatrix}
      A_1 X_1 + A_2 X_3 & A_1 X_2 \\
      A_3 X_1 + A_4 X_3 & A_3 X_2
    \end{pmatrix}
  \end{displaymath}
  is symmetric. Since the entries of $X_1$ and $X_2 = X_3^\T$ are linearly
  independent, this is equivalent to both
  \begin{displaymath}
    \begin{pmatrix}
      A_1 X_1 & 0 \\
      A_3 X_1 & 0
    \end{pmatrix}
    \quad \text{ and } \quad
    \begin{pmatrix}
      A_2 X_3 & A_1 X_2 \\
      A_4 X_3 & A_3 X_2
    \end{pmatrix}
  \end{displaymath}
  being symmetric. In particular, it implies that $A_3 X_1 = 0$. Let $a^\T$ be
  a row in $A_3$, and define $\delta = \sum_{i=1}^s a_i \p_i$. Then $0 = a^\T
  X_1 = [ \delta \p_1 f, \dots, \delta \p_s f]$, i.e. $0 = \p_i \delta f =
  \sum_{j=1}^{s+q} x_{s+j} \p_i \delta g_j$ for all $i \le s$. This implies
  $\delta g_j = 0$ for all $j$, and therefore, $\delta f = 0$. Since $\ann_R
  (f)_1 = 0$, it follows that $A_3 = 0$.
  
  \smallskip Next we investigate $A_4 X_3 = (A_1 X_2)^\T$. We will use
  induction to prove that both $A_1$ and $A_4$ are identity matrices, up to a
  scalar. Let $a_{ij} = (A_4)_{ij}$ for all $1 \le i,j \le s+q$ and $b_{ij} =
  (A_1)_{ij}$ for all $1 \le i,j \le s$. Then
  \begin{displaymath}
    (A_4 X_3)_{ij} = \sum_{k=1}^{s+q} (A_4)_{ik} (X_3)_{kj} = \sum_{k=1}^{s+q}
    a_{ik} h_{j+k-1},
  \end{displaymath}
  and similarly, $(A_1 X_2)_{ji} = \sum_{k=1}^s b_{jk} h_{i+k-1}$. Thus $A_4
  X_3 = (A_1 X_2)^\T$ is equivalent to the following set of equations;
  \begin{equation} \label{eq:*1} \tag{$*_1$}
    \sum_{k=1}^{s+q} a_{ik} h_{j+k-1} = \sum_{k=1}^s b_{jk} h_{i+k-1} \text{
      for all } 1 \le i \le s+q \text{ and } 1 \le j \le s.
  \end{equation}
  
  Let $c = a_{11}$. Consider first the equation $\sum_{k=1}^{s+q} a_{1k} h_k =
  \sum_{k=1}^s b_{1k} h_k$, which we get from \eqref{eq:*1} by letting
  $i=j=1$. Since the non-zero $h_k$'s involved are linearly independent, it
  follows that $a_{1k} = 0$ for all $k>s$. Next put $i=1$ into \eqref{eq:*1}
  to get $\sum_{k=1}^{s+q} a_{1k} h_{j+k-1} = \sum_{k=1}^s b_{jk} h_k$. If
%% CHANGED 2013-07-13: a_{s-j+2} -> a_{1,s-j+2}
  $a_{1k} = 0$ for all $k \ge s-j+3$, then this equation implies $a_{1,s-j+2} =
  0$. By induction on $j$, $a_{1k} = 0$ for all $k>1$. Hence \eqref{eq:*1}
%% CHANGED 2013-07-13: j=1 -> i=1
  with $i=1$ reduces to $a_{11} h_j = \sum_{k=1}^s b_{jk} h_k = b_{j,s-1}
  h_{s-1} + b_{js} h_s$ for all $j$. This implies that $b_{jk} = c
  \delta_{jk}$ for $k=s-1$ and $k=s$. The symbol $\delta_{jk}$ is defined by
  $\delta_{jj} = 1$ for all $j$, and $\delta_{jk} = 0$ for all $j \ne k$.
  
  Now assume for some $2 \le i \le s+q$, that $b_{jk} = c \delta_{jk}$ for all
  $1 \le j \le s$ and $k>s-i$. Consider the right-hand side of \eqref{eq:*1}.
  If $k<s-i$, then $h_{i+k-1} = 0$. When $k>s-i$, all $b_{jk}$ are zero by the
  induction hypothesis, except $b_{jj} = c$. Thus $\sum_{k=1}^s b_{jk}
  h_{i+k-1}$ consists of at most two terms, $b_{j,s-i} h_{s-1}$ ($k=s-i$,
  requires $i<s$) and $c h_{i+j-1}$ ($k=j$, requires $s-i < j \le s$). Hence
  if $j=1$ and $i \ge s$, then \eqref{eq:*1} becomes $\sum_{k=1}^{s+q} a_{ik}
  h_k = c h_i$. Since $h_{s-1}, \dots, h_{s+q}$ are linearly independent, it
  follows that $a_{ik} = c \delta_{ik}$ for all $k \ge s$ and $b_{1,s-i} =
  a_{i,s-1}$.
  
  Assume for some $2 \le j \le s$ that we know $a_{ik} = c \delta_{ik}$ for
  all $k > s-j+1$. Then the left-hand side of \eqref{eq:*1} consist of at most
  three terms, corresponding to $k=s-j$, $k=s-j+1$ and $k=i>s-j+1$. Hence
  \eqref{eq:*1} reduces to
% $\sum_{k=1}^{s+q} a_{ik} h_{j+k-1} = a_{i,s-j} h_{s-1} + a_{i,s-j+1} h_s + c
% h_{i+j-1} = \sum_{k=1}^s b_{jk} h_{i+k-1} = b_{j,s-i} h_{s-1} + c h_{i+j-1}$
% \begin{displaymath}
%   a_{i,s-j} h_{s-1} + a_{i,s-j+1} h_s + c h_{i+j-1} = b_{j,s-i} h_{s-1} + c
%   h_{i+j-1},
% \end{displaymath}
  \begin{equation*}
    \begin{gathered}[t]
      a_{i,s-j} h_{s-1} \\ (j<s)
    \end{gathered}
    + a_{i,s-j+1} h_s +
    \begin{gathered}[t]
      c h_{i+j-1} \\ (i>s-j+1)
    \end{gathered} =
    \begin{gathered}[t]
      b_{j,s-i} h_{s-1} \\ (i<s)
    \end{gathered} +
    \begin{gathered}[t]
      c h_{i+j-1}. \\ (i>s-j)
    \end{gathered}
  \end{equation*}
  We have written under each term what it requires. The two terms $c
  h_{i+j-1}$ cancel each other, except when $i=s-j+1$. It follows that
  $a_{i,s-j+1} = c \delta_{i,s-j+1}$ and $b_{j,s-i} = a_{i,s-j}$. By induction
  on $j$, $a_{ik} = c \delta_{ik}$ for all $k \ge 1$, and $b_{j,s-i} =
  a_{i,s-j} = c \delta_{j,s-i}$ for all $j \ge 1$. By induction on $i$,
  $b_{jk} = c \delta_{jk}$ for all $1 \le j,k \le s$, and $a_{ik} = c
  \delta_{ik}$ for all $1 \le i,k \le s+q$. This means that $A_1 = cI$ and
  $A_4 = cI$.
  
  \smallskip Finally, to finish the proof, we need to show that $A_2 X_3$ is
  symmetric if and only if $A_2 \in \langle B'_0, \dots, B'_q \rangle$. Let
  $a_{ij} = (A_2)_{ij}$ for all $1 \le i \le s$ and $1 \le j \le s+q$, and let
  $a_{ij} = 0$ for $j \le 0$. $A_2 X_3$ is symmetric if and only if
  \begin{equation} \label{eq:*2} \tag{$*_2$}
    \sum_{k=1}^{s+q} a_{ik} h_{j+k-1} = \sum_{k=1}^{s+q} a_{jk} h_{i+k-1}
    \text{ for all } 1 \le j < i \le s.
  \end{equation}
  Assume for some $2 \le i \le s$ that $a_{1k} = 0$ for all $k > s+q+2-i$.
  Equation \eqref{eq:*2} with $j=1$ says that $\sum_{k=1}^{s+q} a_{ik} h_k =
  \sum_{k=1}^{s+q} a_{1k} h_{i+k-1}$. Since $h_k = 0$ for $k < s-1$ and
  $h_{s-1}, \dots, h_{s+q+1}$ are linearly independent, it follows that
  $a_{1,s+q+2-i} = 0$ and $a_{ik} = a_{1,k-i+1}$ for all $k = s-1, \dots,
  s+q$. By induction on $i$, $a_{1k} = 0$ for all $k \ge q+2$ and $a_{ik} =
  a_{1,k-i+1}$ for all $(i,k) \in \{ 2, \dots, s \} \times \{ s-1, \dots, s+q
  \}$.
  
  Assume for some $2 \le \alpha < s$ that
  \begin{equation} \label{eq:*3} \tag{$*_3$}
    a_{ij} = a_{1,j-i+1} \text{ for all pairs } \bigl\{ (i,j) \,\big\vert\: i
    < \alpha \text{ or } j > s - \alpha \bigr\}.
  \end{equation}
  This is true for $\alpha = 2$. For some $\alpha < \beta \le s$ assume in
  addition that
  \begin{equation} \label{eq:*4} \tag{$*_4$}
    \begin{gathered}
      a_{ij} = a_{1,j-i+1} \text{ for all pairs} \\
      \bigl\{ (i,j) \,\big\vert\: (i \le \beta-2 \text{ and } j = s - \alpha)
      \text{ or } (i = \alpha \text{ and } j \ge s - \beta + 2) \bigr\},
    \end{gathered}
  \end{equation}
  and also that
  \begin{equation} \label{eq:*5} \tag{$*_5$}
    a_{\beta-1, s-\alpha} = a_{\alpha, s-\beta+1}.
  \end{equation}
  These assumptions hold for $\beta = \alpha+1$. For all $k \ge s-\beta+2$ it
  follows in particular that $a_{\beta, k-\alpha+\beta} = a_{1, k-\alpha +1} =
  a_{\alpha k}$ by putting $(i,j) = (\beta, k-\alpha+\beta)$ in \eqref{eq:*3}
  and $(i,j) = (\alpha,k)$ in \eqref{eq:*4}. Therefore, any term on the
  left-hand side of $\sum_{k=1}^{s+q} a_{\alpha k} h_{\beta+k-1} =
  \sum_{k=1}^{s+q} a_{\beta k} h_{\alpha+k-1}$ with $ s-\beta+2 \le k \le s+q$
  cancel the corresponding term on the right-hand side. In addition, we
  already know that any term on the right-hand side with $k \ge q+2$ are zero.
  Hence the equation reduces to
  \begin{displaymath}
    a_{\alpha, s-\beta} h_{s-1} + a_{\alpha, s-\beta+1} h_s = a_{\beta,
      s-\alpha} h_{s-1} + a_{\beta, s-\alpha+1} h_s.
  \end{displaymath}
  
  This implies that $a_{\beta, s-\alpha} = a_{\alpha, s-\beta}$ and
  $a_{\alpha, s-\beta+1} = a_{\beta, s-\alpha+1}$. And because $a_{\beta-1,
    s-\alpha} = a_{\alpha, s-\beta+1}$ by \eqref{eq:*5} and $a_{\beta,
    s-\alpha+1} = a_{1, s-\alpha-\beta+2}$ by \eqref{eq:*3}, it follows that
  $a_{\beta-1, s-\alpha} = a_{1, s-\alpha-\beta+2}$. These equations are
  exactly what we need to proceed with induction on $\beta$. This induction
  ends after $\beta=s$, proving \eqref{eq:*4} and \eqref{eq:*5} with $\beta =
  s+1$. In order to continue with induction on $\alpha$, we need \eqref{eq:*3}
  with $\alpha \mapsto \alpha+1$. Now \eqref{eq:*4} with $\beta=s+1$ contains
  all these equations, except $a_{s, s-\alpha} = a_{1,1-\alpha}$. But $a_{s,
    s-\alpha} = a_{\alpha 0}$ by \eqref{eq:*5} with $\beta=s+1$, implying
  $a_{s, s-\alpha} = a_{\alpha 0} = 0 = a_{1,1-\alpha}$. Hence we may do
  induction on $\alpha$, finally proving \eqref{eq:*3} with $\alpha=s$. Since
  $\alpha_{1k} = 0$ for all $k \le 0$ and all $k \ge q+2$, this gives us
  exactly what we wanted, namely $A_2 = \sum_{k=0}^q a_{1,k+1} B'_k$.
  
  The converse statement, that $A_2 \in \langle B'_0, \dots, B'_q \rangle$
  implies that $A_2 X_3$ is symmetric, follows easily from equation
  \eqref{eq:*2}. This completes the proof.
\end{proof}

% How about: $h_2 = \dots = h_{s-1} = 0$ and $h_1, h_s, \dots, h_{s+q+1}$
% linearly independent?

% Note that $A_1 X_2 = X_2 A_4^\T$ implies that $A_1 \bar{\p} g_i = \bar{\p}
% ( \sum_{j=1}^s a_{ij} g_j )$ for all $i$, hence $A_1 X_1$ is symmetric.

% For $0 \le k \le q$, define the $s \times s$ matrix $Y_k$ by $(Y_k)_{ij} =
% \p_i g_{j+k}$, i.e. $Y_k$ consists of columns $k+1, \dots, k+s$ from $X_2$.
% The condition $\p_{i+1} g_j = \p_i g_{j+1}$ for all $1 \le i < s$ and $1 \le
% j < s+q$ is equivalent ($q>0$) to all $Y_k$ being symmetric! And this is
% equivalent to $\exists f_0, \dots, f_q \in \cS_d$ s.t. $Y_k = \p \p^\T f_k$.

\begin{rem} \label{rem:counterex}
  Proposition \ref{prop:counterex} involves polynomials $g_1, \dots, g_{s+q}
  \in \cS_{d-1}$ that satisfy $\p_{i+1} g_j = \p_i g_{j+1}$ for all $1 \le i <
  s$ and $1 \le j < s+q$. Using the $\{ g_i \}$ we defined $h_1, \dots,
  h_{r-1} \in \cS_{d-2}$ by $h_{i+j-1} = \p_i g_j$. This actually implies that
  $\p_{i+1} h_j = \p_i h_{j+1}$ for all $1 \le i < s$ and $1 \le j < r-1$.
  Indeed, if $i<s$ and $j<r-1$, then we may choose $k<s+q$ such that $h_j =
  \p_{j-k+1} g_k$. Hence
  \begin{displaymath}
    \p_{i+1} h_j = \p_{i+1} \p_{j-k+1} g_k = \p_i \p_{j-k+1} g_{k+1} = \p_i
    h_{j+1}.
  \end{displaymath}
  
  Assume conversely that we have polynomials $h_1, \dots, h_{r-1} \in
  \cS_{d-2}$ satisfying $\p_{i+1} h_j = \p_i h_{j+1}$ for all $1 \le i < s$
  and $1 \le j < r-1$. For some $k \in \{ 1, \dots, s+q \}$, consider $\{ h_k,
  \dots, h_{k+r-1} \}$. Since this set satisfies $\p_i h_{k-1+j} = \p_j
  h_{k-1+i}$ for all $1 \le i,j \le r$, it follows that there exists $g_k$
  such that $\p_i g_k = h_{k-1+i}$ for all $1 \le i \le r$. This defines $g_1,
  \dots, g_{s+q} \in \cS_{d-1}$, and $\p_{i+1} g_j = h_{i+j} = \p_i g_{j+1}$.
\end{rem}

\begin{rem} \label{rem:Hankel}
  Let $f_{ij} = (A_2 X_3)_{ij} = \sum_{k=1}^{s+q} a_{ik} h_{j+k-1}$ for $1 \le
  i,j \le s$. $A_2 X_3$ is symmetric if and only if it is a Hankel matrix,
  i.e. $f_{i+1,j} = f_{i,j+1}$ for all $1 \le i,j < s$. One implication is
  obvious. To prove the other, assume that $A_2 X_3$ is symmetric. Note that
  $\p_{i+1} h_j = \p_i h_{j+1}$ by remark \ref{rem:counterex}. Therefore,
  $\p_{k+1} f_{ij} = \p_k f_{i,j+1}$ for all $1 \le i \le s$ and all $1 \le
  j,k < s$. Assume for some $2 \le k \le 2s-2$ that $f_{i+1,j} = f_{i,j+1}$
  for all $1 \le i,j < s$ such that $i+j=k$. The following now follows for all
  $1 \le i < s$ and $1<j<s$ such that $i+j=k+1$.
  
  If $l<s$, then $\p_l f_{i+1,j} = \p_{l+1} f_{i+1,j-1} = \p_{l+1} f_{ij} =
  \p_l f_{i,j+1}$. Similarly, if $l>1$, then $\p_l f_{i+1,j} = \p_{l-1}
  f_{i+1,j+1} = \p_{l-1} f_{j+1,i+1} = \p_l f_{j+1,i} = \p_l f_{i,j+1}$. Here
  we also used that $A_2 X_3$ is symmetric. Together this shows that $\p_l
  f_{i+1,j} = \p_l f_{i,j+1}$ for all $l$, and therefore $f_{i+1,j} =
  f_{i,j+1}$. We have assumed $j>1$ here, thus we still need to prove that
  $f_{k+1,1} = f_{k,2}$ when $k<s$. But this follows by the symmetry of $A_2
  X_3$, which implies $f_{k+1,1} = f_{1,k+1}$. By induction on $k$, $A_2 X_3$
  is Hankel.
\end{rem}

% The $j=1$ case is an empty statement if $k \ge s$.

\begin{rem}
  The assumption in proposition \ref{prop:counterex} that $\p_{i+1} g_j = \p_i
  g_{j+1}$ for all $1 \le i < s$ and $1 \le j < s+q$ ensures that $B_k \in
  M_f$ for all $k = 0, \dots, q$. The extra restrictions on the $h_i$'s
  guarantee that $M_f = \langle I, B_0, \dots, B_q \rangle$. There are other
  restrictions we could impose on $\{ h_i \}$ to achieve the same ends, but at
  least $q+3$ of the $h_i$'s must be linearly independent. To prove this, let
  $\nu = \dim_\K \langle h_1, \dots, h_{r-1} \rangle$. Let us count the number
  of linearly independent equations that the symmetry of $A_2 X_3$ imposes on
  the entries of $A_2$. Let $f_{ij} = (A_2 X_3)_{ij}$. By remark
  \ref{rem:Hankel} we may use the equivalent statement that $A_2 X_3$ is a
  Hankel matrix.

  For every $i=1, \dots, s-1$, the equation $f_{i2} = f_{i+1,1}$ reduces to at
  most $\nu$ equations over $\K$. For every $j = 3, \dots, s$, the equation
  $f_{ij} = f_{i+1,j-1}$ gives at most one more equation, namely $\p_s^{d-2}
  f_{ij} = \p_s^{d-2} f_{i+1,j-1}$. All others are covered by $f_{i,j-1} =
  f_{i+1,j-2}$ since $\p_k f_{ij} = \p_{k+1} f_{i,j-1}$ for all $k<s$. Thus we
  get at most $(s-1)(\nu+s-2)$ linearly independent equations. In order to
  make $\dim_\K M_f = q+2$, we need to reduce the $s(s+q)$ entries of $A_2$ to
  $q+1$. We can only hope to achieve this if
  \begin{displaymath}
    (s-1) (\nu+s-2) \ge s(s+q) - (q+1) = (s-1) (s+q+1).
  \end{displaymath}
  Since $s \ge 2$, this is equivalent to $\nu \ge q+3$.
\end{rem}

When using proposition \ref{prop:counterex}, we need to construct the $g_i$'s
involved. By remark \ref{rem:counterex}, the condition on the $g_i$'s is
equivalent to the corresponding condition on the $h_i$'s. Since the $h_i$'s
have extra restrictions, it is easier to work directly with them. The next
lemma tells us how the $\{ h_i \}$ can and must be chosen.

\begin{lem} \label{lem:bigraded}
  Let $f \in \cR_d$. Define a homogeneous ideal $J \subseteq R$ by
  \begin{displaymath}
    J = I_2 \left(
      \begin{matrix}
        \p_1 & \dots & \p_{r-1} \\
        \p_2 & \dots & \p_r
      \end{matrix}
    \right) = \Bigl( \Bigl\{ \p_i \p_{j+1} - \p_{i+1} \p_j \,\Big\vert\: i,j =
    1, \dots, r-1 \Bigr\} \Bigr).
  \end{displaymath}
  Then the following statements are equivalent.
  \begin{enumerate}
    \setlength{\itemsep}{2pt}
    \setlength{\parskip}{0pt}
    \renewcommand{\theenumi}{\alph{enumi}}
    \renewcommand{\labelenumi}{\normalfont(\theenumi)}

  \item $J \subseteq \ann_R f$.

  \item There exists $g \in \cR_d$ such that $\p_i g = \p_{i+1} f$ for all $i
    = 1, \dots, r-1$.

    This $g$ is unique modulo $\langle x_r^\pd d \rangle$.
  \item There exists $h \in \cR_d$ such that $\p_i h = \p_{i-1} f$ for all $i
    = 2, \dots, r$.

    This $h$ is unique modulo $\langle x_1^\pd d \rangle$.

  \item $f$ is a linear combination of the terms in $(x_1 + t x_2 + \dots
    + t^{r-1} x_r)^\pd d$.

  \item $f$ is a linear combination of the terms in $(x_r + t x_{r-1} + \dots
    + t^{r-1} x_1)^\pd d$.

  \end{enumerate}
  Furthermore, if $n \ge 2$, then $f_1, \dots, f_n \in \cR_d$ satisfy $\p_i
  f_{j+1} = \p_{i+1} f_j$ for all $1 \le i < s$ and $1 \le j < n$ if and only
  if $f_1, \dots, f_n$ are $n$ consecutive terms in $c_t (x_r + t x_{r-1} +
  \dots + t^{r-1} x_1)^\pd d$ for some $c_t \in \K [t]$.
\end{lem}

% Note that (e) is equivalent to $f$ being a term in $c_t (x_r + t x_{r-1} +
% \dots + t^{r-1} x_1)^\pd d$ for some $c_t \in \K [t]$.

\begin{rem}
  For any $\alpha = (\alpha_1, \dots, \alpha_r) \in \bN_0^r$ define $\sigma
  (\alpha) = \sum_{i=1}^r (r-i) \alpha_i$. Let $|\alpha| = \sum_{i=1}^r
  \alpha_i$ and $m = \max \{ \sigma (\alpha) \suchthat \sum_{i=1}^r \alpha_i =
  d \} = (r-1)d$, and define
  \begin{displaymath}
    g_{dk} = \sum_{\substack{|\alpha| = d\\[1pt] \sigma (\alpha) = k}} x^\pd
    \alpha \in \cR_d
  \end{displaymath}
  for all $0 \le k \le m$. Clearly, $g_{d0}, \dots, g_{dm}$ are linearly
  independent, and
  \begin{displaymath}
    \bigl( x_r + t x_{r-1} + \dots + t^{r-1} x_1 \bigr)^\pd d = \sum_{k=0}^m
    t^k g_{dk}.
  \end{displaymath}
  Thus $\{ g_{dk} \}$ are the terms we speak of in lemma \ref{lem:bigraded}e.
  The lemma implies that $J_d^\perp = \{ f \in \cR_d \suchthat J \subseteq
  \ann_R f \} = \langle g_{d0}, \dots, g_{dm} \rangle$, hence $\dim_\K (R/J)_d
  = m+1$ for all $d \ge 0$.
\end{rem}

\begin{proof}[Proof of lemma \ref{lem:bigraded}:]
  The implications (b) $\implies$ (a), (c) $\implies$ (a) and (d) $\implies$
  (a) are all obvious. Furthermore, (d) $\iff$ (e), because the two expansions
  have the same terms, just in opposite order, since
  \begin{displaymath}
    \bigl( x_r + t x_{r-1} + \dots + t^{r-1} x_1 \bigr)^\pd d = t^{(r-1)d}
    \bigl( x_1 + \tfrac 1t x_2 + \dots + (\tfrac 1t)^{r-1} x_r \bigr)^\pd d.
  \end{displaymath}
  
  To prove (a) $\implies$ (b), assume that $J \subseteq \ann_R f$. For any $i
  = 1, \dots, r$ let $e_i \in \K^r$ be the \te{i} unit vector, i.e. $(e_i)_j =
  1$ if $j=i$, and $(e_i)_j = 0$ otherwise. In particular, $\alpha =
  (\alpha_1, \dots, \alpha_r) = \sum_{i=1}^r \alpha_i e_i$. For any $\alpha$
  such that $|\alpha|=d$, let
  \begin{displaymath}
    g_\alpha =
    \begin{cases}
      \p^{\alpha - e_i + e_{i+1}} (f), & \text{if $\alpha_i>0$ for some
        $i<r$,}\\
      0, & \text{if $\alpha_r = d$.}
    \end{cases}
  \end{displaymath}
  This is well defined since $J \subseteq \ann_R f$. Note that $g_\alpha$ is
  an element of $\K$. Define a polynomial $g \in \cR_d$ by $g =
  \sum_{|\alpha|=d} g_\alpha x^\pd \alpha$. It follows that $\p_i g = \p_{i+1}
  f$ for all $i<r$. Indeed, for all $|\alpha| = d-1$ we get $\p^\alpha \p_i g
  = g_{\alpha + e_i} = \p^{\alpha + e_{i+1}} f = \p^\alpha \p_{i+1} f$.
  Obviously, if both $g$ and $g'$ satisfy (b), then $\p_i g' = \p_{i+1} f =
  \p_i g$ for all $i<r$, hence $g'-g \in \langle x_r^\pd d \rangle$. This
  proves (a) $\implies$ (b). Moreover, we obtain a proof of (a) $\implies$ (c)
  by renaming the variables $(x_1, \dots, x_r) \mapsto (x_r, \dots, x_1)$.
  
  Note that (a) $\implies$ (e) follows from (a) $\implies$ (b) and the last
  statement. Thus we are done when we prove the last statement. One
  implication is obvious. To prove the other, let $n \ge 2$ and assume that
  $f_1, \dots, f_n \in \cR_d$ satisfy $\p_i f_{j+1} = \p_{i+1} f_j$ for all $1
  \le i < s$ and $1 \le j < n$. In particular, $J \subseteq \ann_R (f_i)$ for
  all $i$. From what we have already proven, we may for $k>n$ inductively
  choose $f_k \in \cR_d$ such that $\p_i f_{j+1} = \p_{i+1} f_j$ for all $i<r$
  and $\p_r^d (f_k) = 0$, and similarly for $k \le 0$, except then $\p_1^d
  (f_k) = 0$. For all $\alpha = (\alpha_1, \dots \alpha_r)$, $\alpha_i \ge 0$,
  let $\sigma (\alpha) = \sum_{i=1}^r (r-i) \alpha_i$. Since $\p_i (f_k) =
  \p_r (f_{k-(r-i)})$, it follows that $\p^\alpha (f_k) = \p_r^d (f_{k -
    \sigma (\alpha)})$ for all $k$. Obviously, $\max \{ \sigma (\alpha)
  \suchthat \sum_{i=1}^r \alpha_i = N \} = (r-1) \cdot N$. If $k > n +
  (r-1)N$, then for all $|\alpha| \ge N$ we have $\p^\alpha (f_k) = \p_r^d
  (f_{k - \sigma (\alpha)}) = 0$, hence $f_k = 0$. Similarly, $f_k = 0$ for
  all $k \ll 0$.
  
  Pick $a,b \ge 0$ such that $f_{-a}, f_b \ne 0$ and $f_{-a-1} = f_{b+1} = 0$.
  (In fact, $f_{-a} = c_1 x_r^\pd d$ and $f_b = c_2 x_1^\pd d$.) Define $f_t =
  \sum_{k=0}^{a+b} t^k f_{k-a} \in \cR_d [t]$. It follows for all $i<r$ that
  \begin{align*}
    (\p_i - t \p_{i+1}) (f_t) & = \p_i f_t - t \p_{i+1} f_t = \sum_{k=0}^{a+b}
    t^k \p_i f_{k-a} - t \sum_{k=0}^{a+b} t^k \p_{i+1} f_{k-a} \\
    & = \sum_{k \in \Z} t^k \p_i f_{k-a} - \sum_{k \in \Z} t^{k+1} \p_i
    f_{k-a+1} = 0.
  \end{align*}
  Thus $\ann_{R(t)} (f_t) \supseteq ( \p_1 - t \p_2, \dots, \p_{r-1} - t \p_r,
  \p_r^{d+1} )$. Note that
  \begin{displaymath}
    \ann_{R(t)} \bigl( (x_r + \dots + t^{r-1} x_1)^\pd d \bigr) = \bigl( \p_1
    - t \p_2, \dots, \p_{r-1} - t \p_r, \p_r^{d+1} \bigr).
  \end{displaymath}
  By lemma \ref{lem:apolar} there exists $c_t \in \K (t)$ such that $f_t = c_t
  (x_r + \dots + t^{r-1} x_1)^\pd d$. Since $f_t \in \cR_d [t]$, it follows
  that $c_t = \p_r^d f_t \in \K [t]$, finishing the proof.
\end{proof}

\begin{rem}
  By remark \ref{rem:counterex} and lemma \ref{lem:bigraded}, the polynomials
  $h_1, \dots, h_{r-1}$ in proposition \ref{prop:counterex} must be $r-1$
  consecutive terms in $c_t \bigl( \sum_{k=0}^{s-1} t^k x_{s-k}
  \smash{\bigr)}^\pd {d-2}$ for some $c_t \in \K [t]$. We also need $h_i = 0$
  for all $i<s-1$ and $h_{s-1}, \dots, h_{s+q+1}$ linearly independent. Since
  there are $(d-2)(s-1) + 1$ linearly independent terms in $\bigl(
  \sum_{k=0}^{s-1} t^k x_{s-k} \smash{\bigr)}^\pd {d-2}$, those conditions can
  be met if and only if
  \begin{displaymath}
    q+2 \le (d-2)(s-1).
  \end{displaymath}
  
  In particular, it is possible to construct such examples with $q=1$ as long
  as $(d-2)(s-1) \ge 3$, i.e. $s \ge 4$ when $d=3$, $s \ge 3$ when $d=4$, and
  $s \ge 2$ when $d \ge 5$. This is what we did in examples \ref{ex:r5} and
  \ref{ex:r7r9}. We may now also construct examples having $q>1$.
\end{rem}

\begin{rem}
  We started this chapter with the following question \ref{que:Q}. Given a
  polynomial $f \in \cR_d$, $d \ge 3$, is it possible to find $f_t \in \cR_d
  [t_1, \dots, t_n]$ such that $f_0 = f$ and $f_t$ splits regularly $\dim_k
  M_f - 1$ times over $\K (t_1, \dots, t_n)$? When $r \le 4$ we proved in
  theorem \ref{thm:rle4} that this is always possible. When $r \ge 5$ and $d
  \ge 5$, or $r \ge 7$ and $d=4$, or $r \ge 9$ and $d=3$, we have found
  examples that this is not always possible. This leaves only the six pairs
  \begin{displaymath}
    (r,d) \in \bigl\{ (5,3), (6,3), (7,3), (8,3), (5,4), (6,4) \bigr\}.
  \end{displaymath}
\end{rem}

We end this chapter with the following example. It is basically the first
degenerate splitting example we ever considered, and theorem
\ref{thm:onematrix} was formulated and proven with this example as a model.

\begin{exmp}
  Let $A \in \Mat_\K (r,r)$ be the fundamental Jordan block, i.e. 
  \begin{displaymath}
    A_{ij} =
    \begin{cases}
      1, & \text{if $j=i+1$,}\\
      0, & \text{otherwise.}
    \end{cases}
  \end{displaymath}
% $A_{ij} = \delta_{i+1,j}$.
  
  Let the ideal $J \subseteq R$ be defined as in lemma \ref{lem:bigraded}, and
  let
  \begin{displaymath}
    I = I_2 ( \p\: A\p ) = I_2 \left(
      \begin{smallmatrix}
        \p_1 & \p_2 & \dots & \p_{r-1} & \p_r \\
        \p_2 & \p_3 & \dots & \p_r & 0
      \end{smallmatrix}
    \right) = J + \p_r \cdot (\p_2, \dots \p_r).
  \end{displaymath}
  
  For all $d \ge 0$ and $k = 0, \dots, (r-1)d$, define $h_{dk} \in \cR_d$ by
  \begin{equation} \label{eq:hdk}
    \bigl( x_1 + t x_2 + \dots + t^{r-1} x_r \bigr)^\pd d =
    \sum_{k=0}^{(r-1)d} t^k h_{dk}.
  \end{equation}
  If we let $\tau
  (\alpha) = \sum_{i=1}^r (i-1) \alpha_i$, then this simply means that
  \begin{displaymath}
    h_{dk} = \sum_{\substack{|\alpha| = d\\[1pt] \tau (\alpha) = k}} x^\pd
    \alpha.
  \end{displaymath}
  
  Note that $\p_i h_{dk} = h_{d-1,k-i+1}$ for all $i = 1, \dots, r$. Let $f
  \in \cR_d$. It follows from lemma \ref{lem:bigraded} that $I \subseteq
  \ann_R f$ if and only if $f \in \langle h_{d0}, \dots, h_{d,r-1} \rangle$.
  This implies that
  \begin{displaymath}
    I_d^\perp = \bigl\{ f \in \cR_d \,\big\vert\: I \subseteq \ann_R f \bigr\}
    = \langle h_{d0}, \dots, h_{d,r-1} \rangle,
  \end{displaymath}
  and therefore $\dim_\K (R/I)_d = r$ for all $d>0$. Note that $\p_r (h_{dk})
  = 0$ for all $k<r-1$, thus $\ann_R (f)_1 \ne 0$ if $f \in \langle h_{d0},
  \dots, h_{d,r-2} \rangle$.
  
  Let $d \ge 3$ and $f = h_{d,r-1}$. Clearly $\ann_R (f)_1 = 0$, hence
  proposition \ref{prop:Mf} implies that $M_f$ is a commutative $\K$-algebra.
  Since $A \in M_f$, it follows by lemma \ref{lem:ranks}c that $M_f = \K [A]$.
  Let us prove that
  \begin{equation} \label{eq:max}
    \ann_R f = I + \p_1^{d-1} \cdot (\p_1, \dots, \p_{r-1}).
  \end{equation}
  Since $\p_i h_{dk} = h_{d-1,k-i+1}$, it follows that $\p_1^{d-2} \p_i f =
  h_{1,r-i} = x_{r+1-i}$ for all $i = 1, \dots, r$. These are linearly
  independent, and it follows that $\{ \p_1^k \p_i f \}_{i=1}^r$ are linearly
  independent for all $0 \le k \le d-2$. Hence for all $0<e<d$ we get $\dim_\K
  (R/ \ann_R f)_e \ge r = \dim_\K (R/I)_e$. Since $I \subseteq \ann_R f$, it
  follows that $\ann_R (f)_e = I_e$ for all $e<d$ and $H(R/ \ann_R f) =
  (1,r,r,\dots,r,1)$. In degree $d$ $\ann_R f$ needs $r-1$ extra generators.
  Since $\p_1^{d-1} \p_i f = 0$ for all $i<r$, equation \eqref{eq:max}
  follows. Note that $\ann_R f$ is generated in degree two and $d$ only.

  % Note that $\p_i h_{dk} = h_{d-1, k-i+1}$. Therefore, $\p_1^{e-1} \p_i
  % h_{d,r-1} = h_{d-e,r-i}$, and these are linearly independent for $i=1,
  % \dots, r$ and $e<d$. Thus $\ann_R (f)_e = I_e$ for all $e<d$. And
  % $\p_1^{d-1} \p_i h_{d,r-1} = 0$ for all $i<r$.
  
  Equation \eqref{eq:hdk} can be used to define a degenerate splitting of
  length $r$ of $f$. Indeed, substituting $k+1$ for $r$, the equation may be
  rewritten as
  \begin{displaymath}
    h_{dk} + \sum_{i>k} t^{i-k} h_{di} =
    t^{-k} \Bigl( \bigl( x_1 + t x_2 + \dots + t^k x_{k+1} \bigr)^\pd d -
    \sum_{i<k} t^i h_{di} \Bigr).
  \end{displaymath}
  Since $h_{di} \in \K [x_1, \dots, x_k]^{DP}$ for all $i<k$, we may proceed
  carefully by induction and prove that there exists a polynomial $h'_t \in \K
  [t_1, \dots, t_k] [x_1, \dots, x_{k+1}]^{DP}$ such that $h'_0 = h_{dk}$ and
  $h'_t$ splits $k$ times inside $\K (t_1, \dots, t_k) [x_1, \dots,
  x_{k+1}]^{DP}$. In particular, there exists $f_t \in \cR_d [t_1, \dots,
  t_{r-1}]$ such that $f_0 = f$ and $f_t$ splits $r-1$ times over $\K (t_1,
  \dots, t_{r-1})$, which is also what theorem \ref{thm:onematrix} guarantees.
  In fact, the degenerate splitting $f_t$ we get from equation \eqref{eq:hdk}
  is essentially the same as the one theorem \ref{thm:onematrix} gives us,
  since $A^k \p f = \p h_{d,r-k-1}$ for all $k$.

% Theorem \ref{thm:onematrix} mimics this example in a more general setting.

% F.ex. $h_{d2} + tX + sY = (st^2)^{-1} \bigl( s (x_1 + tx_2 + t^2x_3)^\pd d -
% (x_1 + stx_2)^\pd d + (1-s) x_1^\pd d \bigr)$.
  
  Note that $f_t \sim x_1^\pd d + \dots + x_r^\pd d$, thus this example is an
  extremal case. Other examples of $f \in \cR_d$ such that $M_f = \K [A]$ and
  $A$ is in Jordan normal form can be constructed from this one.
\end{exmp}

%% \input{general}
%% Generalizations (M^f and others)
\chapter{Generalizations}

A central object in this paper has been $M_f$, the matrix algebra that we have
associated to any $f \in \cR_d$. In this chapter we consider how to generalize
the construction of $M_f$ and some of the results in section \ref{section:Mf}.
In fact, we will define two different generalizations of $M_f$, and both give
is new algebras. Indeed, we show that both $\widehat{M}^f = \bigl(
\dsum_{e=0}^{d-3} M^f_e \bigr) \dsum \bigl( \dsum_{e \ge d-2} \Mat_{R_e} (r,r)
\bigr)$, where $M^f_e$ is defined below, and $M_{f,D} = \{ A \in \Mat_\K (N,N)
\suchthat I_2 (D\: AD) \subseteq \ann_R f \}$ are (non-commutative)
$\K$-algebras, see propositions \ref{prop:Mfe} and \ref{prop:MfD}.

We start by defining a $\K$-vector space $M^f_e$ that generalizes $M_f$ in the
sense that $M^f_0 = M_f$.

\begin{defn} \label{def:Mfe}
  Let $d \ge 0$ and $f \in \cR_d$. For all $e \ge 0$ define $M^f_e$ by
  \begin{displaymath}
    M^f_e = \{ A \in \Mat_{R_e} (r,r) \suchthat I_2 (\p\:A\p) \subseteq \ann_R
    f \}.
  \end{displaymath}
\end{defn}

Lemmas \ref{lem:main1} and \ref{lem:main2} were important tools in the study
of $M_f$. They provided a connection between $M_f$ and polynomials $g \in
\cR_d$ that we later used to find regular and degenerate splittings of $f$.
Lemma \ref{lem:main3} updates both lemmas, connecting $M^f_e$ to polynomials
$g \in \cR_{d-e}$ that are related to $f$.

\begin{lem} \label{lem:main3}
  Suppose $d \ge e \ge 0$ and $f \in \cR_d$.
  \begin{enumerate}
    \setlength{\itemsep}{2pt}
    \setlength{\parskip}{0pt}
    \renewcommand{\theenumi}{\alph{enumi}}
    \renewcommand{\labelenumi}{\normalfont(\theenumi)}

  \item Let $A \in \Mat_{R_e} (r,r)$. The following are equivalent.
    \begin{enumerate}
      \setlength{\itemsep}{2pt}
      \setlength{\parskip}{0pt}
      \renewcommand{\theenumii}{\roman{enumii}}
      \renewcommand{\labelenumii}{\normalfont(\theenumii)}
    \item $I_2 (\p\: A\p) \subseteq \ann_R f$.
    \item $A \p\p^\T f$ is a symmetric matrix.
    \item There exists $g \in \cR_{d-e}$ such that $\p g = A \p f$.
    \end{enumerate}
    Furthermore, this $g$ is unique if $e<d$.

  \item Let $g \in \cR_{d-e}$. The following are equivalent.
    \begin{enumerate}
      \setlength{\itemsep}{2pt}
      \setlength{\parskip}{0pt}
      \renewcommand{\theenumii}{\roman{enumii}}
      \renewcommand{\labelenumii}{\normalfont(\theenumii)}
    \item There exists $A \in \Mat_{R_e} (r,r)$ such that $\p g = A \p f$.
    \item $R_1 (g) \subseteq R_{e+1} (f)$.
    \item $\ann_R (f)_{d-e-1} \subseteq \ann_R (g)_{d-e-1}$.
    \end{enumerate}
  \end{enumerate}
\end{lem}

\begin{proof}
  The proof of the equivalences in (a) is an exact copy of the proof of lemma
  \ref{lem:main2}, and the uniqueness of $g$ is obvious. To prove (b), the
  existence of an $A$ such that $\p g = A \p f$ simply means that $R_1 (g)
  \subseteq R_{e+1} (f)$. By duality this is equivalent to $\ann_R (g)_{d-e-1}
  = R_1 (g)^\perp \supseteq R_{e+1} (f)^\perp = \ann_R (f)_{d-e-1}$.
\end{proof}

% There exists $g \in \cR_{d-e}$ such that $\p g = A \p f$ if and only if $A
% \p\p^\T f$ is symmetric. The symmetry is equivalent to $(A \p\p^\T - \p\p^\T
% A^\T)(f) = 0$, i.e. $I_2 (\p\: A\p) \subseteq \ann f$. (This is only a
% condition on $\ann (f)_{e+2}$.)

\begin{defn} \label{def:gamma_fe}
  If $d > e \ge 0$ and $f \in \cR_d$, let
  \begin{displaymath}
    \gamma^f_e : M^f_e \to \cR_{d-e}
  \end{displaymath}
  be the $\K$-linear map defined by sending a matrix $A \in M^f_e$ to the
  unique poly\-nomial $g \in \cR_{d-e}$ satisfying $\p g = A \p f$, cf. lemma
  \ref{lem:main3}a.
\end{defn}

$\gamma^f_e$ is indeed a map of $\K$-vector spaces since $\p g = A \p f$ is
$\K$-linear in both $A$ and $g$. In chapters \ref{chapter:regular} and
\ref{chapter:limits} we used elements in the image of $\gamma_f = \gamma^f_0$
to produce regular and degenerate splittings of $f$. Even though we do not
find such an explicit use of the polynomials in $\im \gamma^f_e$ when $e>0$,
we are still interested in its image. We start by calculating the kernel and
image of $\gamma^f_e$.

\begin{lem} \label{lem:gamma_fe}
  Suppose $d > e \ge 0$ and $f \in \cR_d$. Then
  \begin{align*}
    \im \gamma^f_e & = (\m_R \ann_R f)_{d-e}^\perp, \\
    \ker \gamma^f_e & = \{ A \in \Mat_{R_e} (r,r) \suchthat A \p f = 0 \}.
  \end{align*}
  Moreover, if we let $\beta_{1j}$ be the minimal number of generators of
  $\ann_R (f)$ of degree $j$, then
  \begin{align*}
    \dim_\K \im \gamma^f_e & = \dim_\K (R/\ann f)_{d-e} + \beta_{1,d-e}, \\
    \dim_\K \ker \gamma^f_e & = r e \cdot \tbinom{r-1+e}{e+1} + r \cdot
    \dim_\K \ann (f)_{e+1}.
  \end{align*}
\end{lem}

\begin{proof}
  By lemma \ref{lem:main3}b, $\im \gamma^f_e = \{ g \in \cR_{d-e} \suchthat
  \ann_R (f)_{d-e-1} \subseteq \ann_R (g)_{d-e-1} \}$. Since $\ann_R g$ is
  determined by its degree $d-e$ piece by lemma \ref{lem:ann}a, it follows
  that $\im \gamma^f_e = (R_1 \cdot \ann_R (f)_{d-e-1})^\perp = (\m_R \ann_R
  f)_{d-e}^\perp$. Evidently, $R_1 \ann_R f_{d-e-1}$ is a $\K$-vector subspace
  of $\ann_R (f)_{d-e}$ of codimension $\beta_{1,d-e}$. Hence
  \begin{displaymath}
    \dim_\K \im \gamma^f_e = \codim_\K (R_1 \cdot \ann_R (f)_{d-e-1}) =
    \dim_\K (R/\ann f)_{d-e} + \beta_{1,d-e}.
  \end{displaymath}

% In particular, $\dim_\K G^f_{d-2} = \dim_\K R_2 / (R_1 \cdot \ann_R (f)_1) =
% \binom {r-\beta_{11}+1} {2}$.

% To compute the dimension of $M^f_e$ we also need $\dim_\K \ker \gamma^f_e$.

  Since $\p \gamma^f_e(A) = A \p f$, we get $\ker \gamma^f_e = \{ A \in
  \Mat_{R_e} (r,r) \suchthat A \p f = 0 \}$. If we let $V_e = \{ D = [D_1
  \dots D_r]^\T \in R_e^{r} \suchthat \sum_i D_i \p_i \in \ann (f)_{e+1} \}$,
  we see that $\dim_\K \ker \gamma^f_e = r \cdot \dim_\K V_e$. We note that
  $V_e$ is the kernel of the map $R_e^r \to \cR_{d-e-1}$ given by $D \mapsto
  \sum_i D_i \p_i (f)$. This map is the composition $R_e^r \to R_{e+1} \to
  \cR_{d-e-1}$, and its image is $R_{e+1}(f)$ since $R_e^r \to R_{e+1}$ is
  surjective. It follows that
  \begin{displaymath}
    \dim_\K V_e = r \cdot \tbinom {r-1+e}{e} - \dim_\K R_{e+1} (f) = e \cdot
    \tbinom {r-1+e}{e+1} + \dim_\K \ann (f)_{e+1}.\qedhere
  \end{displaymath}
\end{proof}

The first significant property that $M_f$ possesses is that it is closed under
matrix multiplication when $d \ge 3$. Our definition of $M^f_e$ allows us to
transfer this to $M^f = \dsum_{e \ge 0} M^f_e$, with a similar restriction.
The following proposition should therefore come as no surprise.

\begin{prop} \label{prop:Mfe}
  Suppose $a+b \le d-3$. Matrix multiplication defines a map
  \begin{displaymath}
    M^f_a \times M^f_b \to M^f_{\smash{a+b}},
  \end{displaymath}
  and all commutators belong to $\ker \gamma^f_{\smash{a+b}}$. In particular,
  the augmentation
  \begin{displaymath}
    \widehat{M}^f = \left( \dsum_{e=0}^{d-3} M^f_e \right)
    \dsum \left( \dsum_{e \ge d-2} \Mat_{R_e} (r,r) \right)
  \end{displaymath}
  is a (non-commutative) graded $\K$-algebra with unity.
\end{prop}

% Commutators: if $A \in M^f_a$ and $B \in M^f_b$, then $(AB-BA) \p f = 0$.

\begin{proof}
  The proof of proposition \ref{prop:Mf} generalizes immediately.
\end{proof}

% Let $A,B \in M^f$. Then $AB \in M^f \iff AB \p\p^\T f = BA \p\p^\T f$, and
% this is $\iff AB \p f = BA \p f$ when $a+b \ne d-1$.

Since $M^f_e = \Mat_{R_e} (r,r)$ for all $e \ge d-1$, we see that
$\widehat{M}^f$ differs from $M^f$ only in degree $d-2$. It is interesting
that the image of the multiplication map $M^f_a \times M^f_b \to
\Mat_{R_{a+b}} (r,r)$ is generally not contained in $M^f_{\smash{a+b}}$ if
$a+b = d-2$. An easy example is $r=2$ and $f = x_1^\pd 2 + x_2^\pd 2 \in
\cR_2$. Then $\p\p^\T f = I$, thus $M^f_0$ consists of all symmetric matrices.
But the product of two symmetric matrices is not symmetric, unless they
commute.

We now want to study $\im \gamma^f_e$ in more detail. To help us do that we
define the following graded $R$-modules.

\begin{defn} \label{def:FG}
  If $f \in \cR_d$, let $F^f = \dsum_e F^f_e$ and $G^f = \dsum_e G^f_e$ where
  \begin{align*}
    F^f_e & = \{ g \in \cR_{d-e} \suchthat \ann (f)_k \subseteq \ann (g)_k
    \:\forall\: k \le d-e \}, \\
    G^f_e & = \{ g \in \cR_{d-e} \suchthat \ann (f)_k \subseteq \ann (g)_k
    \:\forall\: k < d-e \}.
  \end{align*}
\end{defn}

In the following we will often drop the superscripts ($^f$). Obviously, $G_d =
\K$ and $G_e = F_e = 0$ for all $e>d$. Note that $G_e = \{ g \in \cR_{d-e}
\suchthat \ann_R (f)_{d-e-1} \subseteq \ann_R (g)_{d-e-1} \}$ for all $e$ by
lemma \ref{lem:ann}a. In particular, lemma \ref{lem:main3}b implies that
\begin{displaymath}
  G_e = \im \gamma^f_e \quad \text{ for all } \quad 0 \le e < d.
\end{displaymath}
The next lemma summarizes some nice properties of $F$ and $G$.

\begin{lem} \label{lem:Gf}
  Suppose $f \in \cR_d$. Then the following are true.
  \begin{enumerate}
    \setlength{\itemsep}{2pt}
    \setlength{\parskip}{0pt}
    \renewcommand{\theenumi}{\alph{enumi}}
    \renewcommand{\labelenumi}{\normalfont(\theenumi)}

  \item $G = \{ g \in \cR \suchthat \p_i g \in F \:\forall\: i \} \supseteq F
    = R(f)$,

  \item $\dim_\K (G/F)_e = \beta_{1,d-e}$ for all $e$, and

  \item $G \iso \Hom_\K (R / \m_R \ann_R f,\K)$.
  \end{enumerate}
  In particular, $G$ is a graded canonical module for $R / \m_R \ann_R f$, and
  we can get a free resolution of $G$ (as a graded $R$-module) by computing
  one for $R / \m_R \ann_R f$ and dualizing.
\end{lem}

\begin{proof}
  Recall that $R_e (f)^\perp = \ann_R (f)_{d-e}$ by lemma \ref{lem:ann}b.
  Dualizing this equation gives $R_e (f) = \{ g \in \cR_{d-e} \suchthat Dg = 0
  \:\forall\: D \in \ann_R (f)_{d-e} \}$, which equals $F_e$ by lemma
  \ref{lem:ann}a. Combining this with lemma \ref{lem:main3}b, we get $G_e = \{
  g \in \cR_{d-e} \suchthat R_1 (g) \subseteq R_{e+1} (f) = F_{e+1} \}$. This
  proves (a).

  (b) follows from lemma \ref{lem:gamma_fe} if $0 \le e < d$, and it is
  trivial otherwise.

% Note that $G_{-1} \ne 0$ if $f = l^\pd d$, but $G_e = 0$ for all $e<-1$.
  
  Before we prove (c), we want to say something about dualizing $F$. Note that
  $\cR_e = \Hom_\K (R_e, \K)$ since $\cR$ by definition is the graded dual of
  $R$. This implies $R_e = \Hom_\K (\cR_e, \K)$. Since $F_{d-e} \subseteq
  \cR_e$, the map $R_e \to \Hom_\K (F_{d-e}, \K)$ is clearly surjective, and
  its kernel is $\{ D \in R_e \suchthat D(g) = 0 \:\forall\: g \in F_{d-e} \}
  = F_{d-e}^\perp = \ann_R (f)_e$. Thus $\Hom_\K (F_{d-e}, \K) \iso (R /
  \ann_R f)_e$, and therefore $\Hom_\K (F,\K) \iso R / \ann_R f$. This
  explains why $F^* \iso F$, which is the Gorenstein property of $F$.

% $F^* \iso R / \ann_R f$ is degree-reversing, $R / \ann_R f \iso F$ is not.
  
  Turning to $G$, the map $R_e \to \Hom_\K (G_{d-e},\K)$ is surjective as
  above. Its kernel is $\{ D \in R_e \suchthat D(g) = 0 \: \forall \: g \in
  G_{d-e} \} = G_{d-e}^\perp$, and $G_{d-e}^\perp = (\m_R \ann_R f)_e$ by
  lemma \ref{lem:gamma_fe}. This shows that $\Hom_\K (G,\K) \iso R / \m_R
  \ann_R f$, proving (c). The last statements follow since $R / \m_R \ann_R f$
  is Artinian.
\end{proof}

% Therefore, $G$ is a graded canonical module for $R / \m_R \ann_R f$, and we
% can get a free resolution of $G$ (as a graded $R$-module) by computing one
% for $R / \m_R \ann_R f$ and dualizing.

Since $F = R (f)$, multiplication in $R$ induces a ring structure on $F$ given
by $D(f) \star E(f) = DE(f)$. For all $a,b$ such that $a+b \ne d$, we can
extend $\star$ to a bilinear map $F_a \times G_b \to G_{a+b}$ by $D(f) \star g
= D(g)$. This is well defined because $a \ne d-b$ implies $\ann_R (f)_a
\subseteq \ann_R (g)_a$. The equation $D(f) \star g = D(g)$ is not
well defined when $a = d-b$ and $g \in G_b \setminus F_b$, thus $G$ is not
quite an $F$-module.

In order to extend the multiplication to all of $G$, we need an even larger
restriction on the degrees, as seen in the following proposition. Note that
$M^f$ contains $R \cdot I = \{ D \cdot I \suchthat D \in R \}$, the subalgebra
consisting of all multiples of the identity matrix. Clearly, if $D \in R_e$,
then $\gamma^f_e (D \cdot I) = D(f)$. Thus $\gamma^f_e : M^f_e \to G_e$ maps
$R_e \cdot I$ onto $F_e$.

\begin{prop} \label{prop:star}
  $\gamma = \dsum_e \gamma_e$ induces a multiplication $\star : G_a \times G_b
  \to G_{a+b}$ for $a+b \le d-3$ that is associative, commutative and
  $\K$-bilinear. $f \in G_0$ acts as the identity. Furthermore, $D(f) \star h
  = D(h)$ for all $D \in R_a$ and $h \in G_b$.
\end{prop}

\begin{proof}
  Given $g \in G_a$ and $h \in G_b$, we can find $A \in M_a$ and $B \in M_b$
  such that $g = \gamma_a (A)$ and $h = \gamma_b (B)$ since $G_e = \im
  \gamma_e$. Since $a+b \le d-3$ it follows from proposition \ref{prop:Mfe}
  that $AB \in M_{a+b}$ and $BA \p f = AB \p f$. We define $g \star h$ to be
  \begin{displaymath}
    g \star h = \gamma_{a+b} (AB) \in G_{a+b}.
  \end{displaymath}
  
  First we prove that this is well defined. Assume that $\gamma_a (A') =
  \gamma_a (A)$ and $\gamma_b (B') = \gamma_b (B)$. Then $A' \p f = A \p f$
  and $B' \p f = B \p f$, and therefore
  \begin{align*}
    \p \bigl( \gamma_{a+b} (A'B') \bigr) & = A'B' \p f = A'B \p f \\
    & = BA' \p f = BA \p f = AB \p f = \p \bigl( \gamma_{a+b} (AB) \bigr).
  \end{align*}
  Hence $\gamma_{a+b} (A'B') = \gamma_{a+b} (AB)$.
  
  Now, $AB \p f = BA \p f$ is equivalent to $\gamma_{a+b} (AB) = \gamma_{a+b}
  (BA)$, which implies $g \star h = h \star g$. Associativity follows from
  associativity of matrix multiplication, and the bilinearity is obvious.
  Furthermore, from $f = \gamma_0 (I)$ it follows that $f \star g = g$ for all
  $g \in G_a$, $a \le d-3$. Finally, if $D \in R_a$, then $D(f) = \gamma_a (D
  \cdot I)$. Hence $D(f) \star h = \gamma_a (D \cdot I) \star \gamma_b (B) =
  \gamma_{a+b} (D \cdot B) = D(h)$.
\end{proof}

The last statement, $D(f) \star h = D(h)$, says that $\star$ restricts to the
``module'' action $F_a \times G_b \to G_{a+b}$, but with the stronger
requirement $a+b \le d-3$. Let us extend the multiplication $\star : G_a
\times G_b \to G_{a+b}$ by zero if $a+b \ge d-2$. We do this to get an
algebra, but note that $\star$ no longer restricts to $D(f) \star E(f) =
DE(f)$ on $F$ when $a+b \ge d-2$.

% That is, if $g \in G_a$, $h \in G_b$ and $a+b \ge d-2$, we define $g \star h
% = 0$.

\begin{cor} \label{cor:truncG}
  The truncation $\widetilde{G} = \dsum_{e=0}^{d-3} G_e$ is a commutative
  $\K$-algebra.
\end{cor}

\begin{proof}
  This is immediate from proposition \ref{prop:star}.
\end{proof}

% Once we extend $\star$ by zero, we get an algebra structure on all of $G$.
% But we need to trucate since we want $f$ to act as the identity.

\begin{rem}
  Proposition \ref{prop:star} implies in particular that $G_e$ is a module
  over $G_0$ for all $e \le d-3$. We first discovered this the following way.
  Let $N = \binom{r+e}{e+1}$, and fix a basis $\{ D_1, \dots, D_N \}$ be for
  $R_{e+1}$. Define $D = [D_1, \cdots, D_N]^\T$ and $M'_e = \{ A \in \Mat_\K
  (r,N) \suchthat I_2 (\p\: AD) \subseteq \ann f \}$. Just slightly modifying
  ideas in this chapter, it is easy to see that there is a surjective map
  $M'_e \to G_e$, and that matrix multiplication $M'_0 \times M'_e \to M'_e$
  induces the same module action $G_0 \times G_e \to G_e$ as above.
\end{rem}

There are other ways, in addition to $M^f$, to generalize the construction of
$M_f$. We feel the following is worth mentioning. Fix some $e \ge 1$, and let
$N = \dim_\K R_e = \binom{r-1+e}{e}$. Choose a basis $\D = \{ D_1, \dots, D_N
\}$ for $R_e$, and let $D = [D_1 \dots D_N]^\T$. For any $d \ge 0$ and $f \in
\cR_d$, we define
\begin{displaymath}
  M_{f,D} = \{ A \in \Mat_\K (N,N) \suchthat I_2 (D\: AD) \subseteq \ann_R f
  \}.
\end{displaymath}

$M_{f,D}$ is clearly a $\K$-vector space containing the identity matrix. We
note that $M_{f,\p} = M_f$, thus this is another generalization of $M_f$.
However, one of the basic lemmas we used to study $M_f$, lemma
\ref{lem:main2}, does not generalize to $M_{f,D}$ when $e \ge 2$. That is,
$I_2 (D\: AD) \subseteq \ann f$ does not imply that there exists $g \in \cR_d$
such that $Dg =A Df$. The converse implication is obviously still true. On the
other hand, lemma \ref{lem:main1} generalizes, i.e. $\ann (f)_{d-e} \subseteq
\ann (g)_{d-e}$ if and only if there exists $A \in \Mat_\K (N,N)$ such that
$Dg = A Df$. But the reason for including $M_{f,D}$ here, is that proposition
\ref{prop:Mf} generalizes.

\begin{prop} \label{prop:MfD}
  Suppose $e \ge 1$ and $d \ge 3e$. Let $f \in \cR_d$. Then $M_{f,D}$ is
  closed under matrix multiplication. If furthermore $\ann (f)_e = 0$, then
  $M_{f,D}$ is a commutative $\K$-algebra.
\end{prop}

\begin{proof}
  Pick $A,B \in M_{f,D}$. Note that for all $i,j,k$ the $3 \times 3$ minor 
  \begin{displaymath}
    \begin{vmatrix}
      D_i & (AD)_i & (BD)_i \\
      D_j & (AD)_j & (BD)_j \\
      D_k & (AD)_k & (BD)_k
    \end{vmatrix}
  \end{displaymath}
  belongs to $\ann (f)_{3e}$ by expansion along the third column. Expanding
  along the third row proves that
  \begin{displaymath}
    D_k \cdot \left| 
      \begin{matrix}
        (AD)_i & (BD)_i \\
        (AD)_j & (BD)_j
      \end{matrix}
    \right| \in \ann (f)_{3e}
  \end{displaymath}
  for all $i,j$ and $k$. Since $d \ge 3e$ it follows that $I_2(AD\:BD)
  \subseteq \ann f$. Hence $(AD)(BD)^\T(f) = ADD^\T(f)B^\T$ is symmetric, and
  therefore
  \begin{displaymath}
    AB DD^\T (f) = A DD^\T (f) B^\T = B DD^\T (f) A^\T = DD^\T (f) B^\T A^\T =
    DD^\T (f)(AB)^\T.
  \end{displaymath}
  This means that $AB \in M_{f,D}$. Moreover,
  \begin{displaymath}
    AB DD^\T (f) = DD^\T (f) B^\T A^\T = B DD^\T (f) A^\T = BA DD^\T (f),
  \end{displaymath}
  which implies that $(AB-BA) Df = 0$. If $\ann (f)_e = 0$, then $AB = BA$.
\end{proof}

%% End of body of paper

\cleardoublepage
\nocite{*}%		Includes everything in phd.bib
\addcontentsline{toc}{chapter}{Bibliography}
\bibliography{phd}

\end{document}